\documentclass[12pt]{amsart}
\usepackage[headings]{fullpage}
\usepackage{amssymb,amsmath,amscd,amsthm}
\usepackage{ifthen,comment}
\usepackage{graphicx,subcaption}
\usepackage{mathrsfs,scalerel}

\usepackage{xypic}

\usepackage[shortlabels]{enumitem}
\newlist{enuma}{enumerate}{1}
\setlist[enuma]{label=(\alph*)}
\setlist{font=\normalfont}

\usepackage{tikz}
\usetikzlibrary{cd,positioning,math,shapes,calc}
\usetikzlibrary{decorations.markings,decorations.pathreplacing,decorations.pathmorphing}
\usetikzlibrary{knots,intersections}
\usepgfmodule{decorations}
\usepackage{bbm}
\usepackage{mathtools}
\usepackage{subcaption} 
\usepackage[bookmarks=true,colorlinks=true,linkcolor=blue,citecolor=blue,filecolor=blue,menucolor=blue,urlcolor=blue]{hyperref}

\tikzset{knot diagram/every knot diagram/.style={background color=gray!20,clip width=6,end tolerance=5pt,clip radius=0.2cm}}

\tikzset{edge/.style={line width=0.8}}
\tikzset{wall/.style={very thick}}
\tikzset{det/.style={edge,decoration={markings,mark=at position .5 with
	{\node[draw,fill=gray!20,isosceles triangle,sharp corners,transform shape,inner sep=0.1cm] (det){};}},postaction={decorate}}}
\tikzset{-o-/.code 2 args={
\ifthenelse{\equal{#2}{}}{
}{
	\ifthenelse{\equal{#2}{>}\OR\equal{#2}{<}}{
		\pgfkeysalso{decoration={markings,mark=at position #1 with {\arrow[scale=0.8]{#2}}},postaction={decorate}}
	}{
		\pgfkeysalso{decoration={markings,mark=at position #1 with {\draw[black, fill={#2}] circle[radius=2pt];}},postaction={decorate}}
	}
}
}}

\newcommand{\picmargin}{\mathop{}\!}
\newcommand{\stsize}{\footnotesize}

\newcommand{\TanglePic}[5]{
\picmargin
\begin{tikzpicture}[baseline=(ref.base)]
\tikzmath{\xw=#1; \yh=#2; \yd=0; \yu=0;}
\ifthenelse{\equal{#3}{<-}\OR\equal{#4}{<-}}{
	\tikzmath{\yd=-0.1;}
}{}
\ifthenelse{\equal{#3}{->}\OR\equal{#4}{->}}{
	\tikzmath{\yu=0.1;}
}{}
\tikzmath{\yt=\yh+\yu;}
\fill[gray!20] (0,\yd)rectangle(\xw,\yt);
\node(ref) at ({\xw/2},{\yh/2}) {\phantom{$-$}};
\begin{scope}[wall]
\ifthenelse{\equal{#3}{w}}{
	\draw (0,\yd) --(0,\yt);
}{\ifthenelse{\equal{#3}{}}{}{
	\draw[#3] (0,\yd) --(0,\yt);
}}
\ifthenelse{\equal{#4}{w}}{
	\draw (\xw,\yd) --(\xw,\yt);
}{\ifthenelse{\equal{#4}{}}{}{
	\draw[#4] (\xw,\yd) --(\xw,\yt);
}}
\end{scope}
#5
\end{tikzpicture}
\picmargin
}

\newcommand{\HorizontalTangle}[7]{
\TanglePic{0.9}{0.9}{#1}{#2}{
\tikzmath{\ya=\yh/2+0.2; \yb=\yh/2-0.2;}
\path (0,\ya) coordinate (C) (0,\yb) coordinate (D);
\ifthenelse{\equal{#3}{}\AND\equal{#4}{}}{}{
	\draw[left,inner sep=2pt] (C)node{\stsize #3} (D)node{\stsize #4};
}
\path (\xw,\ya) coordinate (A) (\xw,\yb) coordinate (B);
\ifthenelse{\equal{#5}{}\AND\equal{#6}{}}{}{
	\draw[right,inner sep=2pt] (A)node{\stsize #5} (B)node{\stsize #6};
}
#7
}}

\newcommand{\crossSt}[9]{
\HorizontalTangle{#1}{#2}{#6}{#7}{#8}{#9}{
\ifthenelse{\equal{#3}{a}}{
	\draw[edge] (C) ..controls ({\xw/2},\ya) and ({\xw/2},\yb).. (B);
	\draw[edge] (D) ..controls ({\xw/2},\yb) and ({\xw/2},\ya).. (A);
}{
	\begin{knot}
	\strand[edge] (C) ..controls ({\xw/2},\ya) and ({\xw/2},\yb).. (B);
	\strand[edge] (D) ..controls ({\xw/2},\yb) and ({\xw/2},\ya).. (A);
	\ifthenelse{\(\equal{#3}{n}\AND\equal{#4}{#5}\)\OR\(\equal{#3}{p}\AND\NOT\equal{#4}{#5}\)}{
		\flipcrossings{1}
	}{}
	\end{knot}
}
\ifthenelse{\equal{#4}{>}\OR\equal{#4}{<}}{
	\tikzmath{\pos=0.9;}
}{
	\tikzmath{\pos=0.8;}
}
\path[edge,-o-={\pos}{#4}] (C) ..controls ({\xw/2},\ya) and ({\xw/2},\yb).. (B);
\path[edge,-o-={\pos}{#5}] (D) ..controls ({\xw/2},\yb) and ({\xw/2},\ya).. (A);
}}

\newcommand{\cross}[5]{\crossSt{#1}{#2}{#3}{#4}{#5}{}{}{}{}}

\newcommand{\crosswall}[6]{\crossSt{}{#1}{#2}{#3}{#4}{}{}{#5}{#6}}

\newcommand{\walltwowallSt}[8]{
\HorizontalTangle{#1}{#2}{#5}{#6}{#7}{#8}{
\draw[edge,-o-={0.5}{#3}] (C) -- (A);
\draw[edge,-o-={0.5}{#4}] (D) -- (B);
}}

\newcommand{\walltwowall}[4]{\walltwowallSt{#1}{#2}{#3}{#4}{}{}{}{}}

\newcommand{\twowall}[5]{\walltwowallSt{}{#1}{#2}{#3}{}{}{#4}{#5}}

\newcommand{\kink}{
\TanglePic{0.9}{0.9}{}{}{
\begin{knot}
\strand[edge] (\xw,0.3) -- (0.6,0.3)
	..controls (0.2,0.3) and (0.2,0.7).. (0.45,0.7);
\strand[edge] (0,0.3) -- (0.2,0.3)
	..controls (0.7,0.3) and (0.7,0.7).. (0.45,0.7);
\end{knot}
\path[edge,-o-={0.5}{>}] (0.6,0.3) -- (\xw,0.3);
}}

\newcommand{\horizontaledge}[1]{
\TanglePic{0.9}{0.9}{}{}{
\draw[edge,-o-={0.8}{#1}] (0,{\yh/2}) --(\xw,{\yh/2});
}}

\newcommand{\circlediag}[2][]{
\TanglePic{0.9}{0.9}{}{}{
\draw[edge,-o-={0.1}{#2}] (0.45,0.45) circle (0.25);
}}

\newcommand{\ThreeStates}[8]{
\ifthenelse{\equal{#6}{}\AND\equal{#7}{}\AND\equal{#8}{}}{}{
	\draw (#2,#3)node[#1,inner sep=2pt]{{\stsize #6}};
	\draw (#2,#4)node[#1,inner sep=2pt]{{\stsize #7}};
	\draw (#2,#5)node[#1,inner sep=2pt]{{\stsize #8}};
}}

\newcommand{\MultiStrand}[7][b]{
\TanglePic{#2}{#3}{#4}{#5}{
\tikzmath{\xw=#2; \y1=0.2; \y4=\yh-\y1; 
\s=0.35; \y2=\y1+\s; \y3=\y4-\s; \v=(\y2+\y3)/2;}
\ifthenelse{\equal{#1}{b}}{
	\tikzmath{\ym=\y2; \yds=(\y2+\y4)/2;}
}{
	\tikzmath{\ym=\y3; \yds=(\y1+\y3)/2;}
}
\ifthenelse{\equal{#6}{1}\OR\equal{#6}{3}}{
	\draw ({\xw-0.15},\yds)node[rotate=90]{...};
}{}
\ifthenelse{\equal{#6}{2}\OR\equal{#6}{3}}{
	\draw (0.15,\yds)node[rotate=90]{...};
}{}
#7
}}

\newcommand{\sinksourcethree}[1]{
\MultiStrand{1.4}{1.3}{}{}{3}{
\coordinate (V0) at (0.6,\v);
\coordinate (V1) at (0.8,\v);
\draw[edge,-o-={.50}{#1}] (0,\y1) to [out=0, in=-120] (V0);
\draw[edge,-o-={.57}{#1}] (0,\y2) to [out=0, in=-150] (V0);
\draw[edge,-o-={.50}{#1}] (0,\y4) to [out=0, in=120] (V0);
\draw[edge,-o-={.65}{#1}] (V1) to [out=-60, in=180] (\xw,\y1);
\draw[edge,-o-={.58}{#1}] (V1) to [out=-30, in=180] (\xw,\y2);
\draw[edge,-o-={.65}{#1}] (V1) to [out=60, in=180] (\xw,\y4);
}}

\newcommand{\coupon}[2]{
\MultiStrand{1.4}{1.3}{}{}{3}{
\node (V) at (\xw/2,\yh/2) [ellipse, minimum height=0.8cm,inner sep=0pt, draw]
	{\small{#1}};
\draw[edge,-o-={.35}{#2}] (0,\y1) ..controls (\xw*0.2,\y1).. (V.-120);
\draw[edge,-o-={.50}{#2}] (0,\y2) ..controls (\xw*0.2,\y2).. (V.-170);
\draw[edge,-o-={.35}{#2}] (0,\y4) ..controls (\xw*0.2,\y4).. (V.120);
\draw[edge,-o-={.70}{#2}] (V.-60) ..controls (\xw*0.8,\y1).. (\xw,\y1);
\draw[edge,-o-={.60}{#2}] (V.-10) ..controls (\xw*0.8,\y2).. (\xw,\y2);
\draw[edge,-o-={.70}{#2}] (V.60) ..controls (\xw*0.8,\y4).. (\xw,\y4);
}}

\newcommand{\widecoupon}[8]{
\MultiStrand[#1]{1.5}{1.4}{<-}{<-}{3}{
\ThreeStates{left}{0}{\y4}{\ym}{\y1}{#3}{#4}{#5}
\ThreeStates{right}{\xw}{\y4}{\ym}{\y1}{#6}{#7}{#8}
\node (V) at (\xw/2,\v) [circle,minimum height=1cm,
	inner sep=0pt,draw]{#2};
\begin{scope}[edge]
\draw (0,\y1) ..controls (\xw*0.2,\y1).. (V.-120);
\draw (0,\y4) ..controls (\xw*0.2,\y4).. (V.120);
\draw (V.-60) ..controls (\xw*0.8,\y1).. (\xw,\y1);
\draw (V.60) ..controls (\xw*0.8,\y4).. (\xw,\y4);
\ifthenelse{\equal{#1}{b}}{
	\draw (0,\ym) ..controls (\xw*0.1,\ym).. (V.-165);
	\draw (V.-15) ..controls (\xw*0.9,\ym).. (\xw,\ym);
}{
	\draw (0,\ym) ..controls (\xw*0.1,\ym).. (V.165);
	\draw (V.15) ..controls (\xw*0.9,\ym).. (\xw,\ym);
}
\end{scope}
}}

\newcommand{\vertexnearwall}[2][b]{
\MultiStrand[#1]{1}{1.3}{}{w}{2}{
\coordinate (V) at (0.7,\v);
\draw[edge,-o-={.40}{#2}] (0,\y4) to [out=0, in=120] (V);
\draw[edge,-o-={.40}{#2}] (0,\y1) to [out=0, in=-120] (V);
\draw[edge,-o-={.45}{#2}] (0,\ym) ..controls (0.3,\ym).. (V);
}}

\newcommand{\nedgewall}[6][b]{
\MultiStrand[#1]{1}{1.3}{}{#2}{1}{
\ThreeStates{right}{\xw}{\y4}{\ym}{\y1}{#4}{#5}{#6}
\draw[edge, -o-={.5}{#3}] (0,\y4) -- (\xw,\y4);
\draw[edge, -o-={.5}{#3}] (0,\ym) -- (\xw,\ym);
\draw[edge, -o-={.5}{#3}] (0,\y1) -- (\xw,\y1);
}}

\newcommand{\capwall}[6][]{
\HorizontalTangle{}{#2}{}{}{#5}{#6}{
\draw[edge,-o-={0.8}{#4}] (\xw,\yb) ..controls (0.1,\yb) and (0.1,\ya).. (\xw,\ya);
}}

\newcommand{\capnearwall}[1]{
\HorizontalTangle{}{w}{}{}{}{}{
\draw[edge,-o-={0.8}{#1}] (0,\yb) ..controls (0.7,\yb) and (0.7,\ya).. (0,\ya);
}}

\newcommand{\bubblestrand}[6]{
\TanglePic{1}{1}{<-}{<-}{
\ifthenelse{\equal{#1}{b}}{\tikzmath{\b=1;}}{\tikzmath{\b=-1;}}
\tikzmath{\yb=\yh/2-\b*0.2; \ys=\yh/2+\b*0.2;}
\node (B) at ({\xw/2},\yb) [circle,draw,inner sep=1pt] {\stsize #2};
\begin{scope}[every node/.style={inner sep=2pt}]
\begin{knot}
\strand[edge] (0,\ys)node[left]{\stsize #5}
	-- (\xw,\ys)node[right]{\stsize #6};
\strand[edge] (B) -- ({\xw/2},\yd);
\end{knot}
\draw[edge] (0,\yb)node[left]{\stsize #3} -- (B)
	-- (\xw,\yb)node[right]{\stsize #4};
\end{scope}
}}

\allowdisplaybreaks

\newcommand\no[1]{}

\definecolor{pinky}{rgb}{1.0, 0, 1.0}

\theoremstyle{plain}
\newtheorem{theorem}{Theorem}[section]
\newtheorem*{theorem*}{Theorem}

\newtheorem{lemma}[theorem]{Lemma}
\newtheorem{corollary}[theorem]{Corollary}
\newtheorem{proposition}[theorem]{Proposition}
\newtheorem{conjecture}{Conjecture}
\newtheorem{question}{Question}

\newtheorem{definition}[theorem]{Definition}

\theoremstyle{definition}
\newtheorem{remark}[theorem]{Remark}
\newtheorem{example}[theorem]{Example}

\newcommand{\bcon}{\begin{conjecture}}
\newcommand{\econ}{\end{conjecture}}
\newcommand{\bcor}{\begin{corollary}}
\newcommand{\ecor}{\end{corollary}}
\newcommand{\bdf}{\begin{definition}}
\newcommand{\edf}{\end{definition}}
\newcommand{\benu}{\begin{enumerate}}
\newcommand{\eenu}{\end{enumerate}}
\newcommand{\beq}{\begin{equation}}
\newcommand{\eeq}{\end{equation}}
\newcommand{\be}{\begin{equation}}
\newcommand{\ee}{\end{equation}}
\newcommand{\bexa}{\begin{example}}
\newcommand{\eexa}{\end{example}}
\newcommand{\bexe}{\begin{exercise}}
\newcommand{\eexe}{\end{exercise}}
\newcommand{\bfac}{\begin{fact}}
\newcommand{\efac}{\end{fact}}
\newcommand{\bite}{\begin{itemize}}
\newcommand{\eite}{\end{itemize}}
\newcommand{\blem}{\begin{lemma}}
\newcommand{\elem}{\end{lemma}}
\newcommand{\bmat}{\begin{pmatrix}}
\newcommand{\emat}{\end{pmatrix}}
\newcommand{\bprb}{\begin{problem}}
\newcommand{\eprb}{\end{problem}}
\newcommand{\bpro}{\begin{proposition}}
\newcommand{\epro}{\end{proposition}}

\newcommand{\bque}{\begin{question}}
\newcommand{\eque}{\end{question}}
\newcommand{\brem}{\begin{remark}}
\newcommand{\erem}{\end{remark}}
\newcommand{\bthm}{\begin{theorem}}
\newcommand{\ethm}{\end{theorem}}
\newcommand{\bpr}{\begin{proof}}
\newcommand{\epr}{\end{proof}}

\makeatletter
\renewcommand*{\fps@figure}{htbp}
\makeatother

\newcommand{\term}[1]{\textbf{#1}}

\DeclareMathOperator{\tr}{tr}
\DeclareMathOperator{\id}{id}
\DeclareMathOperator{\pr}{pr}

\DeclareMathOperator{\Hom}{Hom}
\DeclareMathOperator{\Fr}{Fr}

\DeclareMathOperator{\Mat}{Mat}
\newcommand{\detq}{\operatorname{det}_q}

\def\Id{\mathrm{Id}}

\def\BN{\mathbb N}
\def\BZ{\mathbb Z}
\def\BQ{\mathbb Q}

\def\BC{\mathbb C}

\def\bT{\mathbb T}

\def\Zq{{ \BZ[\qq^{\pm 1}]}}

\def\cS{\mathscr S}

\let\OldS\S 
\renewcommand{\S}[0]{{\mathcal S}} 

\def\cF{\mathcal F}

\def\cV{\mathcal V}
\def\cX{\mathcal X}
\def\cO{\mathcal O}
\def\cA{\mathcal A}

\def\cI{\mathcal I}

\def\Mqn{\mathcal M_q(n)}
\def\Oq{{\mathcal O_q({\rm SL}_n)}} 

\def\SLn{{\rm SL}_n} 

\def\Ibad{\mathcal I^{\mathrm {bad}}}

\def\Oe{\mathcal O_{\eta}({\rm SL}_n)} 

\def\Weyl{{\mathrm{Weyl}}}

\def\ord{{\mathsf{ord}}}

\def\fS{\mathfrak{S}}

\def\fm{{\mathfrak{m}}}

\def\Cut{\mathsf{Cut}}

\def\Mon{{\mathsf {Mon}}}
\def\Pol{{\mathsf {Pol}}}

\def\JJ{{\mathbb{J}}}

\def\sign{{\mathsf{sign}}}

\def\bk{\mathbf k}

\def\buu{{\mathbf u}}

\def\bi{{\bar i}}
\def\bj{{\bar j}}

\def\bV{{\rd V}}
\def\bfS{\rd{\fS}}

\def\bPhi{\rd{\Phi}}
\def\btr{\rdtr}
\def\bA{\rd{\cA}}

\def\tfS{{\widetilde \fS}}

\def\pfS{\partial \fS}

\def\al{\alpha}
\def\ve{\varepsilon}

\def\pal{{\partial \al}}

\def\qq{{\hat q}}
\def\hq{{\hat q}}
\def\Rx{R^\times}

\def\hxi{{\hat \xi}}

\def\la{\langle}
\def\ra{\rangle}
\def\ot{\otimes}
\def\embed{\hookrightarrow}
\def\onto{\twoheadrightarrow}

\def\ttt{{\mathbbm t}}
\def\aaa{{\mathbbm a}}
\def\ccc{{\mathbbm c}}

\def\gaa{\mathsf g}
\def\ag{{ \gaa }}

\newcommand{\ints}{\mathbb{Z}}

\newcommand{\reals}{\mathbb{R}}

\newcommand{\surface}{\fS}
\newcommand{\face}{\cF}

\newcommand{\rdo}{\overline}
\DeclareMathOperator{\rdtr}{\rdo{tr}}
\newcommand{\mat}{\mathsf}
\newcommand{\rdm}[1]{\rdo{\mat{#1}}}

\newcommand{\exm}[1]{\rdm{#1}_{\ext{\lambda}}}

\newcommand{\rd}[1]{\protect\ThisStyle{\makebox[0pt][l]{\ensuremath{\protect\SavedStyle\overline{\phantom{#1}}}}}#1}

\newcommand{\lv}[1]{#1'}
\newcommand{\ext}[1]{#1^\ast}

\newcommand{\FG}{\mathcal{X}}

\newcommand{\bad}{{\mathrm{bad}}}

\newcommand{\skein}{\mathscr{S}}

\newcommand{\reduceS}{\rd{\skein}}

\newcommand{\stdT}{{\poly_3}}
\newcommand{\poly}{\mathbb{P}}
\def\PP{\poly}

\let\avec=\vec
\renewcommand{\vec}{\mathbf}
\newcommand{\cev}[1]{\protect\ThisStyle{\reflectbox{\ensuremath{\protect\SavedStyle\avec{\reflectbox{\ensuremath{\protect\SavedStyle#1}}}}}}}
\newcommand{\ceC}{\cev{C}}

\def\XS{\FG(\fS,\lambda)}

\def\bXS{\rd{\FG}(\fS,\lambda)}

\def\bmQ{\rdm{Q}}
\def\mQ{\mat{Q}}

\def\bmP{\rdm{P}}

\def\bsX{{\rd{\FG}}}
\def\fT{\mathfrak{T}}

\def\norm{{ \mathrm{norm}}}
\def\rdV{{\rd{V}}}

\def\tY{\tilde Y}

\def\SS{\cS_{\hat{q}}(\fS)}
\def\SE{\cS_{\hat{\eta}}(\fS)}
\def\bSS{\reduceS(\fS)}
\usepackage{extpfeil}

\usepackage[normalem]{ulem}

\begin{document}

\title{Frobenius homomorphisms for stated ${\rm SL}_n$-skein modules}

\author[Hyun Kyu Kim]{Hyun Kyu Kim}
\address{School of Mathematics, Korea Institute for Advanced Study (KIAS), 85 Hoegi-ro, Dongdaemun-gu, Seoul 02455, Republic of Korea}
\email{hkim@kias.re.kr}

\author[Thang  T. Q. L\^e]{Thang  T. Q. L\^e}
\address{School of Mathematics, 686 Cherry Street,
 Georgia Tech, Atlanta, GA 30332, USA}
\email{letu@math.gatech.edu}

\author[Zhihao Wang]{Zhihao Wang}
\address{Zhihao Wang, School of Physical and Mathematical Sciences, Nanyang Technological University, 21 Nanyang Link Singapore 637371}
\email{ZHIHAO003@e.ntu.edu.sg}
\address{University of Groningen, Bernoulli Institute, 9700 AK Groningen, The Netherlands}
\email{wang.zhihao@rug.nl}

\begin{abstract}
	The stated ${\rm SL}_n$-skein algebra $\mathscr{S}_{\hat{q}}(\mathfrak{S})$ of a surface $\mathfrak{S}$ is a quantization of the ${\rm SL}_n$-character variety, and is spanned over $\mathbb{Z}[\hat{q}^{\pm 1}]$ by framed tangles in $\mathfrak{S} \times (-1,1)$. If $\hat{q}$ is evaluated at  
    a root of unity $\hat{\omega}$ with the order of $\hat{\omega}^{4n^2}$ being $N$, 
    then for $\hat{\eta} = \hat{\omega}^{N^2}$, the Frobenius homomorphism $\Phi : \mathscr{S}_{\hat{\eta}}(\mathfrak{S}) \to \mathscr{S}_{\hat{\omega}}(\mathfrak{S})$ is a surface 
    generalization of the well-known Frobenius homomorphism between quantum groups. We show that the image under $\Phi$ of a framed oriented knot $\alpha$ is given by threading along $\alpha$ of the reduced power elementary polynomial, which is an ${\rm SL}_n$-analog of the Chebyshev polynomial $T_N$. This generalizes Bonahon and Wong's result for $n=2$, and confirms a conjecture of Bonahon and Higgins. Our proof uses representation theory of quantum groups and its skein theoretic interpretation, and does not require heavy computations. We also extend our result to marked 3-manifolds.
\end{abstract}

\maketitle
\tableofcontents{}

\section{Introduction}\label{s.Intro}

\def\AAA{{\mathbb A}}
\def\heta{{\hat \eta}}
\def\SeA{\cS_\heta(\AAA)}
\def\SA{\cS(\AAA)}
\def\Cx{{\mathbb C ^\times}}
\def\ppp{\mathbbm {p}}
\def\W{{\mathbb W}}
\def\An{{\mathsf A}}
\def\tAn{{\tilde \An}}
\def\cN{{\mathcal N}}
\def\SB{{\cS(\Bi)}}

\def\Bi{{\mathsf B}}
\def\tBi{{\tilde \Bi}}
\def\MN{{(M,\cN)}}
\def\SMN{{\cS(M,\cN)}}
\def\SMNR{{\cS(M,\cN;R)}}
\def\SeMN{{\cS_\heta(M,\cN)}}
\newcommand{\qbinom}[2]{ \begin{bmatrix}
#1 \\ #2
\end{bmatrix}_q}

\newcommand{\green}[1]{{\color{green}#1}}

\def\Zhq{{\BZ[\hq^{\pm1}]}}

\def\sl{\mathfrak{sl}}

\subsection{Review of ${\rm SL}_n$-skein modules and ${\rm SL}_n$-skein algebras} Let $n\ge 2$. For an oriented surface $\fS$ the {\it $\SLn$-skein algebra}, introduced by Przytycki \cite{Prz} and Turaev \cite{Turaev} for $n=2$ and Sikora \cite{sikora2005skein} for $n \ge 2$, is a quantization of the $\SLn$-character variety of the surface along the Atiyah-Bott-Goldman's Poisson structure.
The $\SLn$-skein algebra of $\fS$ is an algebra over the ring $\Zhq$ of Laurent polynomials in an indeterminate $\hq$, and is spanned by isotopy classes of framed oriented links in the thickened surface $\fS \times(-1,1)$. 

To introduce excision into skein algebra theory, for $n=2$ the second author \cite{Le:triangulation} (see also \cite{BW}) defined the 
{\it stated ${\rm SL}_n$-skein algebra} $\mathscr{S}(\fS)$ (\OldS\ref{ss.marked}--\OldS\ref{subsec:punctured_bordered_surface_and_n-web}). This notion was generalized for $n=3$ by Higgins \cite{Hig23}, and for general $n$ by the second author and Sikora \cite{LS21}. Here the surface $\fS$ is a {\it punctured bordered surface} (or {\it pb surface} for short), which is the result of removing a finite number of points, called 
{\it punctures} or ideal points, from a compact oriented surface such that each component of the boundary
 $\pfS$  is diffeomorphic to an open interval $(0,1)$. As a $\Zhq$-module, $\mathscr{S}(\fS)$ is spanned by framed oriented tangles, consisting of properly embedded knots and arcs, in the thickened surface $\fS \times (-1,1)$, with a height order above each component of $\pfS$. The endpoints of the tangles are {\it stated}, meaning that they are colored by elements of $\{1, \dots, n\}$. The product of two stated tangles in $\mathscr{S}(\fS)$ is the result of stacking the first above the second. The stated ${\rm SL}_n$-skein module can be straightforwardly extended to 
 {\it marked 3-manifolds}, which are pairs $\MN$, where $M$ is a smooth oriented 3-manifold with possibly-empty boundary, and a {\it marking} $\mathcal{N}$  is a union of oriented open intervals embedded in $\partial M$, whose closures are disjoint. The thickening of a pb surface has a natural structure of a marked 3-manifold; see  \OldS\ref{subsec:punctured_bordered_surface_and_n-web}. In general, the stated ${\rm SL}_n$-skein module $\cS\MN$ for $\MN$ is a $\mathbb{Z}[\hat{q}^{\pm 1}]$-module, but not viewed as an algebra. Usually $\cS(\fS)$ and $\cS\MN$ are described using not just framed oriented tangles but also certain directed graphs with framing, called ${\rm SL}_n$-webs or $n$-webs, which results in simplifying the relations (\OldS\ref{ss.marked}--\OldS\ref{subsec:punctured_bordered_surface_and_n-web}). We avoid webs in this introduction to make the discussion easier.

When $e$ is an ideal arc of $\fS$, cutting $\fS$ along $e$ gives another pb surface $\Cut_e \fS$. An important feature of the stated skein theory is that there exists a {\it cutting homomorphism}
$$
\Theta_e : \mathscr{S}(\mathfrak{S}) \to \mathscr{S}(\mathsf{Cut}_e(\mathfrak{S}))
$$
given by a simple formula. If $\fS$ is {\it essentially bordered}, meaning that each connected component of $\fS$ has non-empty boundary, then $\Theta_e$ is an algebra embedding. The cutting homomorphism also extends to the marked 3-manifold case, where it is 
merely a $\Zhq$-module homomorphism.

If $\fS$ has an ideal triangulation $\lambda$, then cutting $\fS$ along all the interior edges of $\lambda$ we get a collection $\cF_\lambda$ of ideal triangles. The cutting homomorphisms give an algebra map
\be  \mathscr{S}(\fS) \to \bigotimes_{\tau \in \cF_\lambda}\cS(\tau).
\label{eq.TriDecomp}
\ee

For an ideal  triangle $\tau$, its stated ${\rm SL}_n$-skein algebra is totally understood in the framework of quantum groups. First, for the bigon $\PP_2$, which is the standard closed disk with two punctures on its boundary removed, there is a natural algebra isomorphism
\be 
\mathscr{S}(\mathbb{P}_2) \cong \mathcal{O}_q({\rm SL}_n)
\label{eq.P222}
\ee
 where $\mathcal{O}_q({\rm SL}_n)$ is the quantized coordinate algebra of the algebraic group ${\rm SL}_n$, with $q = \hat{q}^{2n^2}$.
 It is known that $\mathcal{O}_q({\rm SL}_n)$ is a cobraided Hopf algebra, meaning that it is a Hopf algebra equipped with a so-called co-$R$-matrix (see \OldS\ref{subsec-the-bigon-case}). There are geometrically defined cobraided Hopf algebra structure on $\cS(\PP_2)$, and the isomorphism of \eqref{eq.P222} preserves the cobraided Hopf algebra structure.

Second, the stated ${\rm SL}_n$-skein algebra $\cS(\tau)$ of the ideal triangle $\tau$ is the cobraided tensor product of two copies of $\mathscr{S}(\mathbb{P}_2) = \mathcal{O}_q({\rm SL}_n)$. This means that, as a $\Zhq$-module, $\cS(\tau)= \mathscr{S}(\mathbb{P}_2) \ot \mathscr{S}(\mathbb{P}_2) $, and the product is twisted by the cobrading structure.

There are many applications of the algebra map \eqref{eq.TriDecomp}. One of them is the construction of the {\it quantum trace map} \cite{LY23}
\begin{align}
    \nonumber
    {\bf tr}_\lambda : \mathscr{S}(\mathfrak{S}) \to \overline{\mathcal{X}}(\mathfrak{S},\lambda)
\end{align}
that satisfies several favorable properties. Here $\overline{\mathcal{X}}(\mathfrak{S},\lambda)$ is the ($n$-th root version of the) Fock-Goncharov quantum torus algebra \eqref{eq.reduced_FG_algebra}, appearing in the theory of quantum cluster algebras  \cite{FG2,FG3}. For $n=2$ the quantum trace was constructed by Bonahon and Wong \cite{BW}. For $n=3$ see also \cite{kim2020rm,douglas2021quantum}.

\subsection{Root of unity case, and Frobenius homomorphisms}

This paper concerns the stated ${\rm SL}_n$-skein module $\mathscr{S}_{\hat{\xi}}(M,\mathcal{N})$ and the stated ${\rm SL}_n$-skein algebra $\mathscr{S}_{\hat{\xi}}(\mathfrak{S})$, obtained from $\mathscr{S}(M,\mathcal{N})$ and $\mathscr{S}(\mathfrak{S})$ by evaluating $\hat{q}$ at a non-zero complex number $\hat{\xi}$, especially when $\hat{\xi}$ is a root of unity, with the ground ring $\mathbb{C}$. We use the symbol $\hat{\omega}$ for a root of unity, and let
$$
\omega = \hat{\omega}^{2n^2}, \qquad N = \ord(\omega^2), \qquad \hat{\eta} = \hat{\omega}^{N^2}, \qquad \eta = \hat{\eta}^{2n^2} = \omega^{N^2}.
$$
The definition of $N$ means that $N$ is the smallest positive integer such that $(\omega^2)^N=1$. One can check that $\eta$ is $\pm 1$.

The stated ${\rm SL}_n$-skein modules evaluated at $\heta$ and $\hat{\omega}$ are related to each other by the {\it Frobenius homomorphisms}, which are of our primary interest. The motivating case is the bigon $\mathbb{P}_2$, for which the Frobenius homomorphism
\begin{align}
    \label{eq.intro.bigon_Phi}
    \Phi : \mathscr{S}_{\hat{\eta}}(\mathbb{P}_2) \to \mathscr{S}_{\hat{\omega}}(\mathbb{P}_2)
\end{align}
is the well-studied Frobenius homomorphism \cite{PW}
\begin{align}
    \label{eq.intro.Phi_O}
    \Phi^{\mathcal{O}} : \mathcal{O}_\eta({\rm SL}_n) \to \mathcal{O}_\omega({\rm SL}_n)
\end{align}
between the quantum groups, sending the generator $u_{ij}$ to the power $u_{ij}^N$, where $u_{ij} \in \mathcal{O}_q({\rm SL}_n)$ is what quantizes the function on ${\rm SL}_n$ reading the $(i,j)$-th entry.

\begin{theorem}[Theorem \ref{thmFrob}]
\label{thm.intro.Frobenius_map}
For each marked 3-manifold $(M,\mathcal{N})$ that is {\em essentially marked}, meaning that each connected component of $M$ intersects $\mathcal{N}$, there exists a $\mathbb{C}$-linear map
$$
\Phi : \mathscr{S}_{\hat{\eta}}(M,\mathcal{N}) \to \mathscr{S}_{\hat{\omega}}(M,\mathcal{N})
$$
between the stated ${\rm SL}_n$-skein modules at roots of unity, called the {\bf Frobenius homomorphism}, such that
\begin{enumerate}[label=\rm (\roman*)]
    \item when $(M,\mathcal{N})$ is the thickening of the bigon $\mathfrak{S} = \mathbb{P}_2$, the map $\Phi$ coincides with the map in \eqref{eq.intro.bigon_Phi}, i.e. the Frobenius homomorphism for quantized coordinate algebras for ${\rm SL}_n$,

    \item (Functoriality) if $f:(M,\mathcal{N}) \hookrightarrow (M',\mathcal{N}')$ is an embedding between essentially marked 3-manifolds, then the induced maps $f_* : \mathscr{S}_{\hat{\xi}}(M,\mathcal{N}) \to \mathscr{S}_{\hat{\xi}}(M',\mathcal{N}')$ for $\hat{\xi} \in \{\hat{\eta},\hat{\omega}\}$ commute with $\Phi$, i.e.  the following diagram commutes: 
    $$
    \xymatrix{
    \mathscr{S}_{\hat{\eta}}(M,\mathcal{N}) \ar[r]^-{\Phi} \ar[d]_{f_*} & \mathscr{S}_{\hat{\omega}}(M,\mathcal{N}) \ar[d]^{f_*} \\
    \mathscr{S}_{\hat{\eta}}(M',\mathcal{N}') \ar[r]^-{\Phi} & \mathscr{S}_{\hat{\omega}}(M',\mathcal{N}').
    }
    $$
\end{enumerate}
These properties 
{\rm (i)}--{\rm (ii)} make $\Phi$ unique. Moreover, $\Phi$ satisfies the following:
\begin{enumerate}[label=\rm (\roman*)]
    \setcounter{enumi}{2}
    \item when $(M,\mathcal{N})$ is the thickening of a pb (i.e. punctured bordered) surface $\mathfrak{S}$ that is {\em essentially bordered}, meaning that each connected component of $\mathfrak{S}$ has non-empty boundary, then $\Phi$ is an algebra embedding,

    \item the Frobenius homomorphisms $\Phi$ commute with the cutting homomorphisms,

    \item the Frobenius homomorphisms $\Phi$ are compatible via the quantum trace maps ${\bf tr}_\lambda$ with the Frobenius homomorphisms for quantum torus algebras, i.e. for each triangulable pb surface $\mathfrak{S}$ and its ideal triangulation $\lambda$, the following diagram commutes:
    $$
    \xymatrix@C+10mm{
    \mathscr{S}_{\hat{\eta}}(\mathfrak{S}) \ar[r]^-{\Phi} \ar[d]_{{\bf tr}_\lambda} & \mathscr{S}_{\hat{\omega}}(\mathfrak{S}) \ar[d]^{{\bf tr}_\lambda} \\
    \overline{\mathcal{X}}_{\hat{\eta}}(\mathfrak{S},\lambda) \ar[r]^-{\Phi^\mathbb{T}} & \overline{\mathcal{X}}_{\hat{\omega}}(\mathfrak{S},\lambda),
    }
    $$
    where the algebra homomorphism $\Phi^\mathbb{T}$ sends each generator $x_v$ of $\overline{\mathcal{X}}_{\hat{\eta}}(\mathfrak{S},\lambda)$ to its $N$-th power $x_v^N$ of the corresponding generator $x_v$ of $\overline{\mathcal{X}}_{\hat{\omega}}(\mathfrak{S},\lambda)$ (see \eqref{eqFrDef}).
\end{enumerate}
\end{theorem}
The parts 
(i)--(iv) of Theorem \ref{thm.intro.Frobenius_map} have  
already been proven in the literature. For $n=2$ \cite{BL22,KQ}, for $n=3$ \cite{Hig23} when $\gcd(N,6)=1$ and $(M,\mathcal N)$  is the thickening of a pb surface. For general $n$ it is proved in \cite{Wan23} when $\gcd(N,2n)=1$, and in \cite{KW24} when $(M,\mathcal N)$ is the thickening of a pb surface, with its construction relying heavily on the results in \cite{Wan23}. We provide a proof that is different from \cite{Wan23}, and is based on the triangulation map \eqref{eq.TriDecomp} as well as on the arguments of \cite{LS21,LY23}. Our proof is more structural, so that its asserted properties can be understood more clearly. Besides, we are able to drop the assumption on the order of the roots of unity, which was imposed in \cite{Wan23}. 

\vspace{2mm}

The part (v) of Theorem \ref{thm.intro.Frobenius_map} is proved for $n=2$ by Bonahon and Wong \cite{BW16}, and for general $n$ for the first time here. It is one of the main original results of the present paper, and is what shall be a bridge to connect the skein theory to the theory of quantum cluster varieties of Fock and Goncharov \cite{FG2,FG3}; for 
such a connection in the case $n=2$, see \cite{AK17}. In fact, we prove the compatibility of the Frobenius homomorphisms $\Phi$ with four different types of quantum trace maps, not just ${\bf tr}_\lambda$; see Theorem \ref{thmFrob}(f).

\subsection{Image of a loop under Frobenius homomorphism} 
For a concrete understanding of the Frobenius homomorphism $\Phi: \mathscr{S}_{\hat{\eta}}(M,\mathcal{N}) \to \mathscr{S}_{\hat{\omega}}(M,\mathcal{N})$ one needs to know how to describe the image $\Phi(\alpha)$ where $\al$ is a stated tangle. The functoriality reduces this question to the case when $\al$ is connected, i.e. when it is a stated 
framed arc with endpoints in $\mathcal{N}$ or a framed oriented knot.  From Theorem \ref{thm.intro.Frobenius_map} it follows that for a stated framed 
arc $\alpha$, we have
\begin{align}
\label{eq.intro.Phi_on_arc}    
\Phi(\alpha) = \alpha^{(N)},
\end{align}
where $\alpha^{(N)}$ stands for the union of $N$ parallel copies of $\alpha$, taken in the direction of the framing of $\alpha$.

\vspace{2mm}

The major concern of this paper is to describe $\Phi(\alpha)$ for $\al$ a framed oriented knot. To explain the result we should first recall the notion of {\it threading} a polynomial $P$ in $n-1$ variables along a framed oriented knot $\alpha$ in $(M,\mathcal{N})$. We must first assume that
\begin{align}
    \label{eq.intro.condition_on_omega}
    [n]_\omega! \neq 0,
\end{align}
where $[k]_\omega! = \prod_{i=1}^k[i]_\omega$, and $[i]_\omega = \frac{\omega^i-\omega^{-i}}{\omega-\omega^{-1}} = \sum_{j=0}^{i-1} \omega^{-(i-1)+2j}$ is the quantum integer. For each $k=1,2,\ldots,n-1$, we shall define $\alpha_{\varpi_k}$ as 
an element of $\cS_{\hat{\omega}}\MN$ as follows, which can be regarded as $\alpha$ carrying the $k$-th fundamental weight $\varpi_k$ of ${\rm SL}_n$. First, observe that a small open neighborhood $U$ of $\alpha$ is diffeomorphic to the thickening of the open annulus $\mathsf{A}$. Note that $\mathsf{A}$ is diffeomorphic to a sphere with two punctures, which is a pb surface. Let $
{\bf a}$ be the core curve of $\mathsf{A}$ (see Figure \ref{fig1}(a)), so that the element $
{\bf a} \in \mathscr{S}_{\hat{\omega}}(\mathsf{A})$ is sent to $\alpha \in \mathscr{S}_{\hat{\omega}}(M,\mathcal{N})$ via the map $$\iota_* : \mathscr{S}_{\hat{\omega}}(\mathsf{A}) \to \mathscr{S}_{\hat{\omega}}(M,\mathcal{N})$$ induced by the embedding $\iota : \mathsf{A} \times (-1,1) \stackrel{\sim}{\to} U \hookrightarrow M$. Let $
{\bf a}_{\varpi_k}' \in \mathscr{S}_{\hat{\omega}}(\mathsf{A})$ be 
given as in Figure \ref{fig1}(b); it involves an ${\rm SL}_n$-web given as a certain directed $n$-valent graph.
Then let $
{\bf a}_{\varpi_k} := (-1)^{\binom{k}{2} + \binom{n-k}{2}} \frac{1}{[n-k]_\omega! [k]_\omega!} 
{\bf a}_{\varpi_k}' \in \mathscr{S}_{\hat{\omega}}(\mathsf{A})$. Finally, let $$\alpha_{\varpi_k} := \iota_*(
{\bf a}_{\varpi_k}) \in \mathscr{S}_{\hat{\omega}}(M,\mathcal{N}).$$ For an $n-1$ variable polynomial $P(x_1,\ldots,x_{n-1})$, replace each monomial $x_1^{c_1} \cdots x_{n-1}^{c_{n-1}}$ by the element $\alpha_{\varpi_1}^{(c_1)} \cup \cdots \cup \alpha_{\varpi_{n-1}}^{(c_{n-1})}$ of $\mathscr{S}_{\hat{\omega}}(M,\mathcal{N})$, where $\alpha_{\varpi_k}^{(c_k)}$ is obtained from $\alpha^{(c_k)}$, i.e. the $c_k$ parallel copies of $\alpha$, by replacing each copy $\alpha$ by $\alpha_{\varpi_k}$. We denote the final resulting element of $\mathscr{S}_{\hat{\omega}}(M,\mathcal{N})$ by $\alpha^{[P]}$, which is said to be the result of threading $P$ along the framed oriented knot $\alpha$. 

\vspace{2mm}

We apply the threading operation to specific polynomials, which are ${\rm SL}_n$ generalizations of the Chebyshev polynomials. 
For any matrix $A \in {\rm SL}_n(\mathbb{C})$ and for each subset $I$ of $\{1,\ldots,n\}$, let $A_{I,I}$ be the $I\times I$ submatrix of $A$. For each $k=1,\ldots,n-1$ let
\begin{align}
    \label{eq.intro.D_k_A_def}
    D_k(A) := \sum_{I \subset \{1,\ldots,n\}, \, |I|=k} \det(A_{I,I}),
\end{align}
i.e. the sum of all principal minors of $A$ of size $k$. Note that $D_1(A)$ is the usual trace of $A$, and $D_k(A)$ is the trace for the $k$-th fundamental representation of ${\rm SL}_n$. For each $m \in \mathbb{Z}_{>0}$ and $k=1,\ldots,n-1$, there is a unique $n-1$ variable polynomial $\bar{P}_{m,k}$ with integer coefficients (in fact, $n$ should also be one of the parameters, which we omit), satisfying
\begin{align}
    \label{eq.intro.role_of_P_m_k_for_matrices}
    D_k(A^m) = \bar{P}_{m,k}(D_1(A),D_2(A),\ldots,D_{n-1}(A)), \qquad \forall A\in {\rm SL}_n(\mathbb{C}).
\end{align}
If $A$ is diagonalizable with distinct eigenvalues $\lambda_1,\ldots,\lambda_n$, then $D_k(A)$ coincides with the $k$-th elementary symmetric function in the eigenvalues
$$
e_k = \sum_{1\le i_1<i_2<\ldots<i_k\le n} \lambda_{i_1} \lambda_{i_2} \cdots \lambda_{i_k}.
$$
Consider the power-sum expression, by applying the $m$-th {\it Adams operation}, i.e. by replacing each $\lambda_i$ by its $m$-th power $\lambda_i^m$:
$$
e_k^{(m)} := \sum_{1\le i_1<i_2<\ldots<i_k\le n} \lambda_{i_1}^m \lambda_{i_2}^m \cdots \lambda_{i_k}^m,
$$
which is a symmetric polynomial in $\lambda_1,\ldots,\lambda_n$. Then $\bar{P}_{m,k}$ is what expresses this function $e_k^{(m)}$ as a polynomial in the elementary symmetric polynomials $e_1,\ldots,e_{n-1}$:
\begin{align}
    \label{eq.intro.role_of_P_m_k_for_symmetric_functions}
    e_k^{(m)} = \bar{P}_{m,k}(e_1,e_2,\ldots,e_{n-1}),
\end{align}
which is why $\bar{P}_{m,k}$ is called the {\it reduced power elementary polynomial} in \cite{BH23}. Here `reduced' indicates that we are working under the condition $\lambda_1 \lambda_2 \cdots \lambda_n = 1$.

\begin{theorem}[Theorem \ref{thm11}]
\label{thm.intro.Phi_of_knot}
    Assume that the root of unity $\hat{\omega}$ satisfies $[n]_\omega! \neq 0$ \eqref{eq.intro.condition_on_omega}. For any framed oriented knot $\alpha$ in an essentially marked 3-manifold $(M,\mathcal{N})$, the image $\Phi(\alpha) \in \mathscr{S}_{\hat{\omega}}(M,\mathcal{N})$ of $\alpha \in \mathscr{S}_{\hat{\eta}}(M,\mathcal{N})$ under the Frobenius homomorphism $\Phi$ is given by the result of threading the reduced power elementary polynomial $\bar{P}_{N,1}$ along $\alpha$:
    \begin{align}
        \label{eq.intro.image_of_framed_oriented_knot}
        \Phi(\alpha) = \alpha^{[\bar{P}_{N,1}]} \in \mathscr{S}_{\hat{\omega}}(M,\mathcal{N}). 
    \end{align}
    For any stated 
    framed tangle
    $\alpha = \alpha_1 \cup \cdots \cup \alpha_r$ given by disjoint union of stated framed 
    arcs and framed oriented knots, we have
    $$
    \Phi(\alpha) = \Phi(\alpha_1) \cup \cdots \cup \Phi(\alpha_r),
    $$
    where the symbols $\cup$ in the right hand side $(\sim) \cup \cdots \cup (\sim)$ is understood as being multilinear like the symbol $\otimes$, and each $\Phi(\alpha_j)$ is given by \eqref{eq.intro.Phi_on_arc} if $\alpha_j$ is a stated framed 
    arc and by \eqref{eq.intro.image_of_framed_oriented_knot} if $\alpha_j$ is a framed oriented knot. 
\end{theorem}

Theorem \ref{thm.intro.Phi_of_knot}, which is the culminating result of the present paper, gives a complete skein-theoretic description of the image of a stated 
framed tangle in an essentially marked 3-manifold $(M,\mathcal{N})$  
under the Frobenius homomorphism $\Phi : \mathscr{S}_{\hat{\eta}}(M,\mathcal{N}) \to \mathscr{S}_{\hat{\omega}}(M,\mathcal{N})$. It generalizes the $n=2$ result of Bonahon and Wong \cite{BW16}, and confirms the conjecture of Bonahon and Higgins \cite{BH23}. 
We also prove 
\begin{align}     \label{eq.intro.image_of_framed_oriented_knot_k}
        \Phi(\alpha_{\varpi_k}) = \alpha^{[\bar{P}_{N,k}]} \in \mathscr{S}_{\hat{\omega}}(M,\mathcal{N}).
\end{align}
Note that the proof of \cite{BW16} for $n=2$ is by heavy computation. One general philosophy taken by the current paper is to avoid any serious computations, be it skein-theoretic or about non-commutative Laurent polynomials, and instead to come up with conceptual and structural arguments. As such, we believe that our proof gives a better understanding even for the known case $n=2$.

For a geometric description of $\Phi(\al)$ where $\al$ is a framed oriented knot see \OldS\ref{subsec-geomatric-description}.  For a connected surface without boundary we construct in \OldS\ref{subsec.Frobenius_without_boundaries} a Frobenius map for a quotient of $\cS(\fS)$, called the {\it projected skein algebra} defined in \cite{LS21} (see \eqref{def-projected-skein-punctured} and \eqref{def-projected-skein-closed}).

\subsection{Proof by representation theory of quantum groups}

Our proof of Theorem \ref{thm.intro.Phi_of_knot} is through representation theory of quantum groups. We first establish the following result about the quantum group Frobenius homomorphism $\Phi^\mathcal{O} : \mathcal{O}_\eta({\rm SL}_n) \to \mathcal{O}_\omega({\rm SL}_n)$ \eqref{eq.intro.Phi_O}, which is of independent interest; we notice that Bonahon and Higgins are also recently working on a version of this result, by a different approach than ours.
\begin{theorem}[Theorem \ref{thm-L-comp-F}(b)]
\label{thm.intro.Phi_O_on_D_k_q}
    One has the following, as an equality of elements of $\mathcal{O}_\omega({\rm SL}_n)$:
    \begin{align}
        \label{eq.intro.Phi_O_eq}
        \Phi^\mathcal{O}(D_k^\eta({\bf u})) = \bar{P}_{N,k} (D_1^\omega({\bf u}), \cdots, D_{n-1}^\omega({\bf u})).
    \end{align}
\end{theorem}
Here $D^\xi_k({\bf u}) \in \mathcal{O}_\xi({\rm SL}_n)$ is obtained by evaluating the element $D_k^q({\bf u}) \in \mathcal{O}_q({\rm SL}_n)$ at a root of unity $q = \xi$, where $D_k^q({\bf u})$ is the sum of all principal $k\times k$ `quantum' minors (see \eqref{eq.D_k_q}). So $D_k^q({\bf u})$ is a quantization of $D_k({\bf u})$, or of the $k$-th elementary symmetric polynomial $e_k$, and thus Theorem \ref{thm.intro.Phi_O_on_D_k_q} is a quantum lift of \eqref{eq.intro.role_of_P_m_k_for_matrices}--\eqref{eq.intro.role_of_P_m_k_for_symmetric_functions}. To prove this we take advantage of the duality between $\mathcal{O}_q({\rm SL}_n)$ and the quantized enveloping algebra $\mathcal{U}_q(\mathfrak{sl}_n)$. The evaluation of the `integral form' of the latter at a number $q = \xi$ is denoted by $\dot{\mathcal{U}}_\xi(\mathfrak{sl}_n)$. Denote Lusztig's Frobenius homomorphism \cite{GL93} by
$$
\Psi^\mathcal{U} : \dot{\mathcal{U}}_\omega(\mathfrak{sl}_n) \to \dot{\mathcal{U}}_\eta(\mathfrak{sl}_n)
$$
which is dual to and has appeared earlier in the literature than $\Phi^\mathcal{O}$ \eqref{eq.intro.Phi_O}. It induces a map
$$
\Psi^{\mathsf{Rep}} : \mathsf{Rep}(\dot{\mathcal{U}}_\eta(\mathfrak{sl}_n)) \to \mathsf{Rep}(\dot{\mathcal{U}}_\omega(\mathfrak{sl}_n))
$$
between the representation rings, each isomorphic to the ring of symmetric polynomials $(\mathbb{C}[\lambda_1,\ldots,\lambda_n]/(\lambda_1\cdots\lambda_n=1))^{S_n}$ through the character map. A straightforward computation on the $k$-th fundamental representation $V_{\varpi_k}$ shows that $\Psi^{\mathsf{Rep}}$ is given in terms of characters as the $N$-th Adams operation. This yields a version of \eqref{eq.intro.Phi_O_eq} for $\Psi^{\mathsf{Rep}}$. Then Theorem \ref{thm.intro.Phi_O_on_D_k_q} follows by applying the natural module-trace map $\mathcal{T}_\xi : \mathsf{Rep}(\dot{\mathcal{U}}_\xi(\mathfrak{sl}_n)) \to \mathcal{O}_\xi({\rm SL}_n)$, coming from the duality between $\dot{\mathcal{U}}_\xi(\mathfrak{sl}_n)$ and $\mathcal{O}_\xi({\rm SL}_n)$.

\vspace{2mm}

We now recall the result 
of Queffelec and Rose \cite{QR18}
(see also \cite{Pou22,DS})  saying that the (stated) ${\rm SL}_n$-skein algebra $\mathscr{S}_{\hat{\xi}}(\mathsf{A})$ for the twice-punctured sphere $\mathsf{A}$ is isomorphic to the representation ring $\mathsf{Rep}(\dot{\mathcal{U}}_\xi(\mathfrak{sl}_n))$, where ${\bf a}_{\varpi_k}$ is sent to $V_{\varpi_k}$ (Theorem \ref{twice_sphere}). We then verify by a straightforward computation that when we cut $\mathsf{A}$ through an ideal arc $c$ into the bigon $\mathbb{P}_2$, the cutting homomorphism $\Theta_c : \mathscr{S}_{\hat{\xi}}(\mathsf{A}) \to \mathscr{S}_{\hat{\xi}}(\mathbb{P}_2)$ coincides with the module-trace map $\mathcal{T}_\xi$ (Theorem \ref{thm.cutting_homomorphism_and_module-trace}). This morally gives the Frobenius homomorphism (Proposition \ref{r-L-comp-F})
$$
\Psi : \mathscr{S}_{\hat{\eta}}(\mathsf{A}) \to \mathscr{S}_{\hat{\omega}}(\mathsf{A})
$$
for the twice-punctured sphere $\mathsf{A}$, which sends $
{\bf a}_{\varpi_k} \in \mathscr{S}_{\hat{\eta}}(\mathsf{A})$ to $\bar{P}_{N,k}(
{\bf a}_{\varpi_1},\ldots,
{\bf a}_{\varpi_{n-1}}) \in \mathscr{S}_{\hat{\omega}}(\mathsf{A})$. By embedding the open annulus $\mathsf{A}$ into the once punctured disk with one puncture on its boundary component (Theorem \ref{thm10}), and by using Theorems \ref{thm.intro.Frobenius_map} and \ref{thm.intro.Phi_O_on_D_k_q}, we prove the sought-for Theorem \ref{thm.intro.Phi_of_knot} and \eqref{eq.intro.image_of_framed_oriented_knot_k}.

We note that, in a recent work of Higgins \cite{Hig24}, Theorems \ref{thm.intro.Phi_of_knot} and \ref{thm.intro.Phi_O_on_D_k_q}  are established for $n=3$ under the condition $\gcd(N,6)=1$.

\subsection{Acknowledgments} The authors thank 
Francis Bonahon, Valentin Buciumas, Francesco Costantino, Vijay Higgins, Myungho Kim, You Qi and  Adam Sikora
for valuable discussions. 

H.K. has been supported by KIAS Individual Grant (MG047204) at Korea Institute for Advanced Study. T.L. has been work supported by the NSF grant DMS-2203255.
Z.W. is supported by research scholarships from Nanyang Technological University and the University of Groningen.

\def\home{{\hat \omega}}
\def\Ud{{\dot {\mathcal U}_q}}
\def\Udq{{\dot{\mathcal U}_q}(sl_n)}
\def\Ude{{\dot{\mathcal U}_\heta}(sl_n)}
\def\Udo{{\dot {\mathcal U}_\home}(sl_n)}

\def\Oe{{\cO_\eta}}
\def\Oo{{\cO_\omega}}
\def\Oq{{\mathcal O}_q({\rm SL}_n)}
\def\SoM{{\cS_\home\MN}}
\def\SeM{{\cS_\heta\MN}}

\def\SAR{{\cS(\An;R)}}

\def\SeM{\cS_\heta(M, \cN)}
\def\cN{{\mathcal N}}
\def\MN{{(M,\cN)}}
\def\homega{{\hat \omega}}
\def\Rep{{\mathsf{Rep}}}
\def\Ch{{\mathsf{Ch}}}
\def\SoA{{ \cS_\homega(\An)}}
\def\Zq{{\BZ[q^{\pm1}]}}

\def\H{\mathfrak{h}}

\section{Notations and conventions} \label{ss.ground}

Let $\mathbb{N}, \BZ, \BC$ be respectively the set of all non-negative integers, the set of all integers, and the set of all complex numbers. 
A complex number $\xi$ is a 
root of 
unity if there exists a positive integer $d$ such that $\xi^d=1$, and the smallest such $d$ is called the order of $\xi$ and denoted by $\ord(\xi)$.

We will fix an integer $n
\ge 2$ and  work with the Lie algebra $
\sl_n(\BC)$. Denote $$\JJ=\{ 1, 2, \dots, n\}.$$ For $i\in \JJ$ its conjugate is defined by $$\bar i= n+1-i.$$ Let $\JJ_k$ be the set of $k$-elements subsets of $\JJ$. Let $S_n$ be the symmetric group on $n$ letters in $\JJ$.

For a sequence $\avec i= (i_1, \dots, i_m)$ of distinct numbers, define
 \begin{align}
     \label{eq.length_of_sequence}
     \ell (\avec i)= \# \{ (j,k) \mid \,  1\le j < k \le m
 , \,\, i_j > i_k   \}.
 \end{align} 
 If $\sigma\in S_n$, then $\ell(\sigma):= \ell (\sigma(1), \dots, \sigma(n))$ is the usual length of $\sigma$.

We use $\mathbb{Z}_n$ to denote the group $\mathbb{Z}/n\mathbb{Z}$ of integers modulo $n$.

Throughout the paper, $\Zhq$ is the ring of Laurent polynomial in the variable $\hq$ with integer coefficients. The element $$q= \qq^{2n^2}$$ is the usual quantum parameter appearing  
in the theory of quantum groups.  
For $m \in \mathbb{Z}$, by $q^{m/(2n^2)}$ we mean $\qq^m$. 
For  $m\in \BN$ we define the quantum integer $[m]_q$ and its factorial 
by
\begin{align}
    \label{eq.quantum_integer}
    [m]_q = \frac{q^m - q^{-m}}{q - q^{-1}} = q^{-m+1}+q^{-m+3}+\cdots +q^{m-1}, \quad
[m]_q!= \prod_{i=1}^m [i]_q, \quad [0]_q!=1.
\end{align}

The ground ring $R$ is usually a  commutative $\Zhq$-algebra which is a domain.

We will often use the following constants in $R$:
\begin{equation}\label{def-constants-tac}
	\ttt= (-1)^{n-1} q^{\frac{n^2-1}{n}}, \qquad
	\aaa= q^{(1-n)(2n+1)/4}, \qquad
	\ccc_i= q^{\frac{n-1}{2n}}(-q)^{n-i}, \quad i \in \JJ.
\end{equation}

We also use Kronecker’s delta notation and its sibling
\begin{align*}
    \delta_{i,j}=\begin{cases}
        1 & \text{if }i=j\\
        0 &  \text{if }i\neq j
    \end{cases},\qquad
    \delta_{i<j}=\begin{cases}
        1 & \text{if }i<j\\
        0 &  \text{if }i\geq j
    \end{cases}.
\end{align*}

\section{Quantized algebra of regular functions on ${\rm SL}_n$} \label{sec.Oqsln}

In this section we review the quantized algebra $\Oq$ of regular functions on ${\rm SL}_n$, and the quantum  Frobenius homomorphism for $\Oq$. Note that our $\Oq$ is defined over the ring $\Zq$, not over the field $\BQ(q)$ like in many other texts.

\subsection{Quantum matrices and the quantized function algebra $\mathcal{O}_q({\rm SL}_n)$} 
\def\baa{{\mathbf a}}

\begin{definition}[see \cite{KS,PW}]
\begin{enuma}

\item A $k \times k$ matrix $
{\bf u}= (
u_{ij})_{i,j=1}^k$ with entries in a $\Zq$-algebra is  a {\bf $q$-quantum matrix} if any $2\times 2$ submatrix $\begin{pmatrix}a & b \\ c& d\end{pmatrix}$ of it satisfies the relations
\begin{equation}\label{eq.Mq2}
ab = q ba, \quad ac = q ca, \quad bd = q db, \quad cd = q dc,\quad
bc = cb, \quad ad - da = (q-q^{-1}) bc.
\end{equation}
\item For such a $q$-quantum matrix ${\bf u}$, its {\bf quantum determinant} is defined by
\begin{equation}
\detq(\buu)
:=\sum_{\sigma\in 
S_k} (-q)^{\ell(\sigma)}u_{1\sigma(1)}\cdots u_{k\sigma(k)}
=\sum_{\sigma\in S_k} (-q)^{\ell(\sigma)}u_{\sigma(1)1}\cdots u_{\sigma(k)k}.
\notag
\end{equation}
\end{enuma}
\end{definition}
It is known that $\det_q({\bf u})$ commutes with each entry $u_{ij}$.

\bdf[see \cite{PW}]\label{def.OqSLn} 
\begin{enuma}
    \item The {\bf $q$-quantum matrix algebra} $\Mqn$ is the $\Zq$-algebra generated by entries $u_{ij}$ of the $n\times n$ matrix $\buu= (u_{ij})_{i,j=1}^n$ subject to the relations \eqref{eq.Mq2} for any $2\times 2$ submatrix $\begin{pmatrix}a & b \\ c& d\end{pmatrix}$.

    \item The {\bf quantized function algebra} $\Oq$ is defined as 
\[\Oq := \Mqn/({\det}_q \buu-1)\]
\end{enuma}

\edf

\def\Mz{\mathcal M_{\eta}}
\def\Mxn{\mathcal M_{\xi}}

\def\tOq{\widetilde{\Oq}}
\def\tu{{\tilde u}}
\def\degn{{ \mathsf{deg}^{(n)}}}
\def\Ox{{\cO_\xi}}

 By abuse of notation, we denote the image of $u_{ij}\in \Mqn$ under the natural projection $\Mqn \to \Oq$ also by $u_{ij}$.

It is known (e.g. \cite{PW}) that $\Oq$ is a Hopf algebra, with the comultiplication $\Delta$, the counit $\varepsilon$, and the antipode $S$
given by
\begin{align}
\Delta(u_{ij}) & = \sum_{k=1}^n u_{ik} \ot u_{kj}, \qquad \ve(u_{ij})= \delta_{ij}, \qquad 
S({u}_{ij} ) 
=
(-q)^{i-j} {\det}_q(\buu^{ji}). \label{e-Oq-ops}
\end{align}
Here $\buu^{ji}$ is the result of removing the $j$-th row and the $i$-th column from $\buu$.

For a non-zero $\xi\in \BC$, define 
$$\Mxn = \Mqn \ot_{\Zq} 
\BC, \qquad  \Ox= \Oq\ot_\Zq \BC,$$
 where $\BC$ is considered as a $\Zq$-module by $q\mapsto \xi$.

\def\uq{\mathcal U_{\omega}}

\def\OA{{\Oq}}
\def\flip{{\mathsf{flip}}}
\def\coin{{\mathsf{coin}}}

\def\UA{{\dot{\mathcal U  }}_q}
\def\DDD{{ \mathcal D} }
\def\Do{{\DDD_\omega}}
\def\Deta{{\DDD_\eta}}
\def\CA{{\mathcal A}}
\def\bY{{\Lambda}}
\def\Yo{\bY_\omega}
\def\Ye{\bY_\eta}
\def\Frac{{\mathsf{Fr}}}
\def\Uo{\dot{\mathcal U}_\omega}
\def\Uoq{\dot{\mathcal U}_q}
\def\Ue{\dot{\mathcal U}_\eta}
\def\Ch{{\mathsf {Ch}}}
\def\Rep{{\mathsf{Rep}}}
\def\Lo{{L_\omega}}
\def\Leta{{L_\eta}}

\no{
\subsection{$\BZ_n$-grading} For a word $w$ in the letters $u_{ij}$, define $\degn(w)\in \BZ_n$ to be its length mod $n$. Since $\degn$ is respected by all the defining relations of $\Oq$, it descends to a $\BZ_n$-grading of the algebra $\Oq$:
\be 
\Oq = \bigoplus_{d \in \BZ_n} \Oq _d, \qquad \Oq _d \, \Oq _
{d'}\subset \Oq _{d+d'}.
\ee
Here $\Oq_d$ is the 
span of all words $w$ with $\degn(w)=d$.
}

\subsection{Quantum elementary symmetric function}\label{subsec:quantum_elementary_symmetric_function} We review a quantum analog of the elementary symmetric function of the eigenvalues of a matrix; see \cite{DL}.

Recall that $\JJ_k$ is the set of $k$-elements subsets of $\JJ=\{1,2,\ldots,n\}$.
For $I,J\in \JJ_k$, let $M^I_J(\buu)\in \Oq$, called a {\bf quantum minor}, be the quantum determinant of the $I\times J$ submatrix of $\buu$. 
For $k\in\mathbb J$, we define 
\begin{align}
    \label{eq.D_k_q}
    D_k^{q}(\textbf{u}) = \sum_{I\in\mathbb J_k} M^I_I(\buu) \in \Oq,
\end{align}
the sum of all `principal' quantum minors of size $k$. Fix a non-zero  $\xi\in \BC$. Let $D_k^{\xi}(\buu)$ be the image of $ D_k^{q}(\buu)$ under the map $\Oq \to \Ox$.

 By  \cite{DL}, the elements $D_k^{q}(\bf u)$, $1\leq k\leq n-1 $, pairwise commute, and if $\xi$ is not a root of 
 unity then 
 the elements $D_k^\xi({\bf u})$, $1\le k \le n-1$, are algebraically independent over $\mathbb C$.
  We will remove the restriction on $\xi$:
\begin{lemma}\label{r.indep} 
For a non-zero $\xi \in \mathbb{C}$, the elements $D_k^{\xi}(\buu)$, 
$1\le k \le n-1$, are 
	 algebraically independent over $\mathbb C$.
\end{lemma}
\bpr By checking the defining relations it is easy to see that there is a $\BC$-algebra map
$$f: 
\mathcal{O}_\xi \to \BC[x_1, \dots, x_{n}]/(x_1 \cdots x_n=1)$$
given by $f(u_{ij})=0$ if $i\neq j$ and $f(u_{ii})= x_i$. Then $f(D_k^{\xi}(\buu))=e_k$, the $k$-th elementary symmetric polynomial defined by
\begin{align}
    \label{elementary_symmetric_polynomial}
    e_k = \sum_{1\le i_1 < \dots < i_k \le n} x_{i_1} \dots x_{i_k}.
\end{align} 
As $e_1, \dots , e_{n-1}$ are algebraically independent \cite{Macdonald}, so are $D_1^{\xi}(\buu), \dots, D_{n-1}^{\xi}(\buu)$.
\epr

We will refer to the elements $D_k^q({\bf u})\in \mathcal{O}_q({\rm SL}_n)$ and $D_k^\xi({\bf u}) \in \mathcal{O}_\xi$ as the {\bf quantum elementary symmetric functions}.

\def\Rx{R[x_1,x_2,\cdots,x_n]/(x_1x_2\cdots x_n=1)}
\def\Ru{R[\textbf{u}]}
\def\oq{\Oq}
\def\toq{\tOq}

\def\Mon{\mathcal M_{\omega}}

\subsection{Frobenius homomorphisms for $\Mqn$ and $\Oq$}\label{ss.Frobenius_map_for_Oq}

Recall that $\Mxn= \Mqn \ot _\Zq \BC$, where $q 
\mapsto \xi$ in the tensor factor $\mathbb{C}$.

\begin{theorem}\label{Fro_Oq} Assume that $\omega\in \BC$ is a root of unity. Let $N= \ord(\omega^2)$ and $\eta = \omega^{N^2}$. 
\begin{enuma}
\item 	There exists a  unique $\BC$-algebra homomorphism
	$\Phi  :\Mz\rightarrow \Mon$, given by 
	\begin{align}
	    \label{eq.Phi_on_u_ij}
        \Phi  (u_{ij}) = u_{ij}^N \quad\mbox{for all $i,j\in \JJ$.}
	\end{align}
	The map $\Phi$ satisfies
	\be 
	\Phi  ({\det}_\eta(\buu)) = ({\det}_\omega(\buu))^N.
	\label{eq.det}
	\ee
	
	\item  
   The map $\Phi  $ descends to an injective  Hopf algebra homomorphism $\Oe \to \Oo$, also denoted by 
    $$\Phi: \Oe \to \Oo.$$
\end{enuma}
\end{theorem}
The above maps $\Phi$ are the quantized-function-algebra versions of what can be called the {\bf Frobenius homomorphisms}.

For odd $\ord(\omega^2)$, Parshall and Wang \cite{PW} gave a proof of Theorem \ref{Fro_Oq}, which can be modified for all roots of unity. For completeness we present a proof in 
 Appendix \ref{appendix-bigon}.

\def\cT{\mathcal T}

\section{Lusztig's quantum Frobenius homomorphism}\label{sec.Lusztig_Frobenius_map}

In this section we use Lusztig's results on quantum groups
 to calculate $\Phi(D_k^\eta(\buu))$, i.e. the image of the quantum 
elementary symmetric function $D^\eta_k({\bf u}) \in \Oe$ under the Frobenius homomorphism $\Phi : \Oe \to \Oo$.

Towards this goal, we first review Lusztig's modified quantum group $\UA$, which is defined over $\Zq$, and is a perfect dual of $\Oq$. We review its specialization $\Uo$ at a root of unity $\omega$.   
The Grothendieck ring $\Rep_\omega$ 
of the category of finite-dimensional unital representations of 
$\Uo$ embeds naturally into
 $\Oo$ via the so-called module-trace map. The image of this embedding is exactly the subalgebra generated by the elements $D^\omega_k(\buu)$.
 The main result of the section, Theorem \ref{thm-L-comp-F}, which we establish essentially by assembling the results known in the literature, is that the restriction of the Frobenius homomorphism $\Phi: \Oe \to \Oo$ onto 
 $\Rep_\eta$ is the well-known Adams operation defined using the power sum in symmetric function theory. From this we find out that the `reduced power elementary polynomials' $\bar{P}_{N,k}$, which are $(n-1)$-variable polynomials serving as ${\rm SL}_n$ analogs of the Chebyshev polynomials of the first type for ${\rm SL}_2$, and which are known to express the `power sum' elementary symmetric functions as polynomials in the elementary symmetric functions, in fact work in the same way for the quantum elementary symmetric functions. More precisely, we obtain the identity \begin{align}
 \label{role_of_P_n_k_for_quantum_elementary_symmetric}
 \Phi(D^\eta_k({\bf u})) = \bar{P}_{N,k}(D^\omega_1({\bf u}), D^\omega_2({\bf u}), \ldots, D^\omega_{n-1}({\bf u})).
 \end{align}
 In principle, for each parameter $n$, $k$ and $N$, this identity can be verified by a direct computation involving elements of $\mathcal{O}_\omega$. However, in reality, as the parameters grow large, the computation quickly gets too complicated to perform by hand, nor such a computation would easily yield a proof for general parameters. We provide a conceptual proof using representation theory, which practically requires no computation.
 For proofs we use the fact that the Frobenius homomorphism for $\Oq$ and Lusztig's Frobenius homomorphism for $\UA$ are dual to each other.

The results of this section, including \eqref{role_of_P_n_k_for_quantum_elementary_symmetric},
will be interpreted skein-theoretically in later sections, and will play an important role in our paper.

\def\W{X}
\def\classic{{}}

\def\rk{{\mathsf{rank}}}

\subsection{Representations of $\sl_n(\BC)$, symmetric polynomials, and Adams operations} 
\label{ss.representations_of_sln}

We quickly review some representation theory of the Lie algebra $\sl_n(\BC)$, with the goal to fix the notations. For details see \cite{FH}.

Recall that $\sl_n(\BC)$ consists of all traceless $n\times n$ matrices with entries in $\BC$. The Cartan subalgebra $\H\subset \sl_n(\BC)$  consists of all diagonal traceless matrices. The {\bf weight lattice} $X\subset \Hom_\BC(\H, \BC)$ (which is usually denoted by $P$) is the free abelian group generated by $w_1, \dots, w_n$ subject to $w_1+ \dots + w_n=0$. Here 
$w_i: \H \to \BC$ is defined by $$w_i(D)= D_{ii},$$ the $i$-th diagonal entry of $D\in \H$. The fundamental weights $\varpi_
{k}$, are
\be\label{def-fundamental-weight}
\varpi_
{k}= w_1 + \dots + w_
{k}, \qquad  
k=1, \dots, n-1.
\ee
The $\BN$-submodule $X_+$ of $X$ spanned by $\varpi_1, \dots, \varpi_{n-1}$ is the set of {\bf  dominant weights}.

    The standard partial order $\le$ on $X$ is defined so that  $\lambda'\leq \lambda$ if and only $\lambda-\lambda'  = \sum_{1\leq i\leq n-1}
    c_i(w_i - w_{i+1})$, where $
    c_i\in\mathbb N$ for $1\leq i\leq n-1$. We write $\lambda' < \lambda$ if $\lambda' \le \lambda$ and $\lambda' \neq \lambda$.
    
The Weyl group is the symmetric group $S_n$, which acts on $\W$ by $\sigma(w_i) = w_{\sigma(i)}$ for $\sigma\in S_n$. The group ring $\BC[\W]$ has a basis given by the set of all $e^x$, $x \in \W$; here, $e^x e^{x'} = e^{x+x'}$. Using $$x_i= e^{w_i},$$ 
 we have $$\BC[\W]=\BC[x_1, \dots, x_n]/(x_1 
 \cdots x_n=1).$$ 
The subset $\BC[\W]^{S_n}\subset \BC[\W] $ of all elements fixed by $S_n$ is a $\BC$-subalgebra  isomorphic to $\BC[e_1, \dots, e_{n-1}]$, where 
$e_k$ is the $k$-th elementary symmetric polynomial in $x_1,\ldots,x_n$, as defined in \eqref{elementary_symmetric_polynomial}.

Let $V$ be a finite-dimensional $\sl_n(\BC)$-module. For $\mu\in X$  the {\bf  $\mu$-weight subspace} is
$$ V^\mu:= \{ v \in V \mid D v = \mu(D) v
\ \text{for all} \ D \in \H\}.$$
 The {\bf character} of $V$ is defined by
\be
\Ch(V) = \sum _{\mu \in X} \dim(V^\mu) e^\mu \in \BZ[X] \subset \BC[X].
\ee
We say that $\mu \in X$ is a {\bf weight of $V$} if $\dim(V^\mu) \neq 0$.

An $\sl_n(\BC)$-module  $V$ is said to be {\bf of highest weight $\lambda \in X$} if $V^\lambda$ is spanned by a non-zero vector $v$ and $N_+ v=0$, where $N_+\subset \sl_n(\BC)$ consists of all strictly upper triangular matrices in $\sl_n(\BC)$.
Note that $\lambda$ is a weight of $V$. It is well known that any other weight $\lambda'$ of $V$ is strictly less than $\lambda$, i.e. $\lambda' < \lambda$, which justifies the terminology `highest' weight.

For every  $\lambda\in  X_+$ there is a simple $\sl_n(\BC)$-module $\Lambda^\classic(\lambda)$ 
of highest weight $\lambda$.
Let $\Rep_\BC$ be the Grothendieck ring of the category of finite-dimensional $\sl_n(\BC)$-modules, with the ground ring extended from $\BZ$ to $\BC$. The following is well known; see e.g. \cite{FH}.

\bpro \label{rRepC}
\begin{enuma}
\item The set $ \{\Lambda^\classic(\lambda)  )\mid \lambda\in X_{+}\}$ is a 
vector space basis of $\Rep_\BC$. 

\item The character map 
 yields a $\BC$-algebra isomorphism 
 \be 
 \Ch: \Rep_\BC 
 \xrightarrow{\cong} \BC[\W]^{S_n}. \label{eqRepC}
 \ee
Moreover, for each $k=1,\ldots,n-1$ we have
 \be 
 \label{eq.Ch_sends_Lambda_varpi_i_to_e_i}
 \Ch( \Lambda^\classic(\varpi_
 {k})  ) = e_
 {k}.
 \ee
\end{enuma}
\epro

For  $
m \in \BN$, the $\BC$-algebra 
endomorphism of $\BC[x_1, \dots, x_n]$ given on the generators by $$ x_i 
\mapsto x_i^
{m} \quad\mbox{for each $i=1,\ldots,n$,}$$ descends to a $\BC$-algebra 
endomorphism 
 \begin{align}
     \label{Adams}
     \Psi_
     {m}: \BC[\W]^{S_n} \to \BC[\W]^{S_n},
 \end{align}
 known as the {\bf $
 m$-th Adams operation} \cite{BtD}. Since $\BC[\W]^{S_n} = \BC[e_1, \dots,e_{n-1}]$, for $
 m\in \BN$ and $
 k=1,\dots, n-1$ 
 there exists a unique polynomial  $\bar P_{
 m,k}\in \mathbb Z[y_1,y_2,\cdots,y_{n-1}]$ such that
\be \label{eq-def-barP} 
\Psi_m(e_k) = \bar P_{
m,k}(e_1, \dots, e_{n-1}).
\ee
This polynomial $\bar P_{
m,k}$, defined in the context of symmetric polynomials, is called a {\em reduced power elementary polynomial} in \cite{BH23}.

It will become useful to recall the following role of reduced power elementary polynomials in the context of minors of matrices. Suppose that $A$ is an $n$ by $n$ matrix with entries in $R$. 
For each $k=1,\ldots,n-1$, we define $D_k(A) = \sum_{I\in\mathbb J_{k}}\det(A_{I,I})$ as in \eqref{eq.intro.D_k_A_def}, where $\mathbb{J}_k$ is the set of all $k$-subsets of $\mathbb{J}=\{1,\ldots,n\}$, and $A_{I,I}$ is the $I\times I$ submatrix of $A$. 

\begin{lemma}[{\cite[Lem.5]{BH23}}]\label{lem-PD}
	For any two integers $m$ and $k$ with $m>0$, $1\leq k\leq n-1$,
	the polynomial $\bar{P}_{m,k}$ is the unique polynomial in $\mathbb Z[y_1,y_2,\cdots,y_{n-1}]$  such that \\ $$
	\bar{P}_{m,k}(D_1(A),D_2(A),\cdots,D_{n-1}(A)) = D_{k}(A^{m})$$ for every
	$A\in \SLn(\mathbb C)$.
\end{lemma}

In this paper we will be studying various versions of Frobenius homomorphisms for different contexts, eventually including the stated ${\rm SL}_n$-skein algebras, and investigate how 
these different versions are related to 
one another. As an example, we shall see that the $N$-th Adams operation $\Phi_N$ can be interpreted as the representation-theoretic Frobenius homomorphism; in a sense, the entire present section is devoted to understanding of this statement. The corresponding reduced power elementary poylnomials $\bar{P}_{N,
k}$ will play the role of ${\rm SL}_n$ analogs of the Chebyshev polynomials for ${\rm SL}_2$. Indeed, as pointed out in \cite{BH23}, one can see that when $n=2$ and $
k=1$, the single-variable polynomial $\bar{P}_{N,1}$ coincides with the $N$-th Chebyshev polynomial $T_N$ of the first type.

\subsection{Lusztig's modified quantized enveloping algebra} 
\label{sub-modified-qaantized-algebra}
We now review Lusztig's integral modified quantized enveloping algebra of $\sl_n(\BC)$, following \cite{GL93,GL09}. This integral version of the quantum group is the one 
having $\Oq$ as a perfect dual (in the sense we explain in \OldS\ref{ss.duality_between_UA_and_Oq}).

Let $\UA$ be Lusztig's modified quantized enveloping algebra, denoted by $_\cA \dot { \mathbf U }$ in \cite{GL93}. It is a $\Zq$-algebra generated by symbols $1_\lambda, E_i^{(
r)}1_\lambda, F_i^{(
r)}1_\lambda$, $i=1, \dots, n-1,\,
r\in \BN_{\ge 1}, \lambda \in X$, subject to relations described in \cite{GL93}. We use the convention $E_i^{(0)}1_\lambda = F_i^{(0)}1_\lambda = 1 _\lambda$. We have
$1_\lambda 1 _{\lambda'} = \delta_{\lambda, \lambda'} 1 _\lambda$. In particular,
each $1_\lambda$ is an idempotent.

\no
{Moreover
\begin{align*}
	1_\lambda 1 _{\lambda'} & = \delta_{\lambda, \lambda'} 1 _\lambda\\
	\UA &= \bigoplus_{\lambda, \lambda'\in X} 1_\lambda \UA 1_{\lambda'}.
\end{align*}
}

Note that $\UA$ does not have
the element 1 which should be the sum of all $1_\lambda, \lambda \in X$. 
Moreover $\UA$ is a topological Hopf algebra where the coproduct $\Delta(x)$ is in general an infinite sum of elements of $\UA \ot \UA$ and belongs to a certain completion of $\UA \ot \UA$.

For each dominant weight $\lambda\in X_+$ there exists a unique $\UA$-module $\bY_q(\lambda)$, denoted by $_\cA \Lambda(\lambda)$ in \cite{GL93}, 
which can be interpreted as a quantum deformation of the simple $\sl_n(\BC)$-module $\Lambda^\classic(\lambda)$ of highest weight $\lambda$ and 
which satisfies the following property:
	for each  $\mu\in X$ the $\Zq$-module $M^\mu:=1_\mu (\bY_q(\lambda))$ is free with rank 
	\be 
    \rk_{\Zq} M^
    {\mu}= \dim_\lambda(\mu),
	\label{eq.rank}
	\ee
	where $\dim_\lambda(\mu)$ is the dimension of the $\mu$-weight subspace of  
    $\Lambda^\classic(\lambda)$.

\subsection{Duality between $\UA$ and $\Oq$}\label{ss.duality_between_UA_and_Oq}

The algebras
$\UA$ and  $\OA$ are Hopf dual to each other. This means that there is a 
$\Zq$-bilinear form 
\be \la \cdot, \cdot \ra:  \OA  \ot _\Zq \UA   \to \Zq
\label{eq.formZ}
\ee satisfying certain conditions, among 
which are: for $x, x_1, x_2 \in \UA$ and $y, y_1, y_2\in \Oq$, 
\begin{align}
\text{if} \ \Delta(y) = \sum y'\otimes y'', \ \text{then }
\langle y, x_1x_2\rangle &= \sum
		\langle y', x_1\rangle\langle y'',x_2\rangle,   \notag  \\
\text{if} \ \Delta(x) = \sum x'\otimes x'', \ \text{then }		\langle y_1y_2,x\rangle &= 
\sum \langle y_1,x'\rangle \langle y_2,x''\rangle \label{eq-dot-paring},
\end{align}
using the usual Sweedler notation. 
Although the last sum of \eqref{eq-dot-paring} is a priori  an infinite sum in general, there are only finitely many non-zero terms.

\def\can{{\mathsf{can}}}
\def\cT{\mathcal T}

The bilinear form \eqref{eq.formZ} is {\em perfect} in the following sense. The $\Zq$-module $\UA$ has  a free basis $\dot{ \mathbf B}$,  the {\em canonical basis} \cite{GL93}. For each $b \in \dot{ \mathbf B}$ there exists $b^*\in \OA$ such that the set $\{b^* 
\,|\, b \in \dot{ \mathbf B}\}$ is a $\Zq$-basis of $\OA$ (the {\em dual canonical basis})  \cite{Kashiwara93, GL93}. Moreover \cite{GL93}
\be\label{eq-cononical-basis}
 \la (b')^*, b \ra = \delta_{b, b'}, \ \text{for all} \ b, b'\in \dot{ \mathbf B}.
\ee

In particular, the $\Zq$-linear map $\Oq \to (\UA)^*:= \Hom_\Zq(\UA, \Zq)$, induced by the form \eqref{eq.formZ}, is injective. We will identify $\Oq$ with the image of this map.

\subsection{Module-trace map}\label{ss.module-trace_map} We explain how certain $\UA$-modules give rises to elements in $\Oq$.

Let $V$ be an $\UA$-module 
that is free over $\Zq$ of finite rank. For $v\in V$ and $f\in \Hom_\Zq(V, \Zq)$, the $\Zq$-linear function $\UA \to \Zq$, defined by
$u \to f(u v)$, is called a {\em matrix entry} function on $\UA$.

Define $\cT(V) \in (\UA)^*$, called the {\bf module-trace map} of $V$ 
as follows: for  $u\in \UA$,
$$
\cT(V)(u)={\rm tr}_V(u),
$$  
where the right hand side is the trace of the action of $u$ on $V$.
All the matrix entry functions are in $\Oq$.
In particular,  $\cT(V)\in \Oq$, as $\cT(V)$ is the sum of the diagonal matrix entry functions.

\blem[module-trace of a fundamental representation of $\UA$] \label{lem-trace-fundamental-representations-dot}
For $k=1,\ldots,n-1$, we have
\be 
\cT( \Lambda_q(\varpi_k))= D_{k}^{q} ({\bf u}).
\ee
\elem
\bpr Consider the anti-symmetric $q$-polynomial module over $\Zq$:
$$A_q[y_1,\dots, y_n] := \Zq\la y_1, \dots, y_n\ra/ (y_i y_j + q y_j y_i, y^2_i; 1\le  i< j\le n).$$
There is an $\Oq$-coaction $\al: A_q[y_1,\dots, y_n]  \to A_q[y_1,\dots, y_n]  \ot \OA$ given by $$\al (y_i) 
:=\sum _{j=1}^{n} y_j\otimes_{\Zq} u_{ji}.$$ The $\Zq$-module $
A_q[y_1,\dots, y_n] $ is free of rank $2^n$ with basis $y_J:= y_{j_1} \dots y_{j_
{k}}$, $J= (j_1 < \dots < j_
{k})\in \JJ_
{k}$, $1\le k \le n-1$. For each $k=1,\ldots,n-1$, the $\Zq$-submodule $A^k_q[y_1,\dots, y_n]$ spanned by $y_J, J \in \JJ_k$, is a subcomodule. Dually, $A^k_q[y_1,\dots, y_n]$ is a $\UA$-module of highest weight $\varpi_k$. Hence 
\begin{align}\label{eq.Dk}
	 \cT( \Lambda_q(\varpi_k))= \sum _{J \in \JJ_k} \text{coefficient of $y_J$ in $
     \alpha(y_J)$  } = \sum _{J \in \JJ_k} M^J_J = D_{k}^{
     q} ({\bf u}).
\end{align}
\epr

\def\cU{{\mathcal U}}
\def\Yx{\Lambda_\xi}
\def\Ux{\dot\cU_\xi}
\subsection{Specializing $q$ to a complex number}
We now review some of Lusztig's results \cite{GL92,GL93} on representations of $\UA$ when $q$ is evaluated at a non-zero complex number.

For a non-zero $\xi\in \BC$ define, similarly to the case of $\Ox$, 
$$\Ux:= \UA \ot _\Zq \BC, \qquad \Yx(\lambda):= \bY_q(\lambda) \ot _\Zq \BC$$ where 
$\BC$ is considered as an $\Zq$-module by $q 
\mapsto\xi$. Then $\Yx(\lambda)$ is an $\Ux$-module.

An $\Ux$-module $M$ is {\em unital} if \begin{itemize}
	\item for any $m\in M$ we have $1_\lambda m =0$ for all but finitely many $\lambda\in X$, and
	\item for any $m\in M$ we have $ m= \sum _{\lambda \in X } 1_\lambda m$.
\end{itemize}
Let $M$ be a finite-dimensional unital $\Ux$-module. Then $M^\lambda: = 1_\lambda M$ is called the $\lambda$-weight subspace. If $M^\lambda\neq 0$ then $\lambda$ is called a weight of $M$.
We have the {\em weight decomposition }
$$
M = \bigoplus_{\lambda \in X} M^\lambda.
$$

Each $M^\lambda$ is finite-dimensional, and $M^\lambda  =0$ for all but finitely many $\lambda\in X$. Define the formal character
$$ 
\Ch_{\xi}(M) = \sum _{\lambda\in X} (\dim M^\lambda) \, e^{\lambda} \in \BZ[X].
$$

{
Note that $\Ux$-submodules of $M$ and $\Ux$-quotients of $M$ are also unital. 
Let $\Rep_\xi$ be the Grothendieck ring of {the category of} finite-dimensional unital $\Ux$-representations, with the ground ring extended from $\BZ$ to $\BC$. 
}

\bthm[\cite{GL92,GL93}]\label{thm-rep-dotU}

\begin{enuma}
\item
For $\lambda\in X_+$ the $\Ux$-module $\Lambda_\xi(\lambda)$ has a unique simple quotient $L_\xi(\Lambda)$. Moreover $\dim_\BC(L_\xi(\lambda))^\lambda =1 $. 
The map $\lambda \mapsto \Lambda_\xi(\lambda)$ is a bijection from $X_+$ to the set of all isomorphism classes of simple unital $\Ux$-modules.

\item For $
k=1, \dots, n-1$, we have
$ 
L_\xi(\varpi_
{k}) = \Yx(\varpi_
{k}) 
$ and 
consequently $\Ch_\xi( L_\xi(\varpi_
{k}) )= e_
{k}$.

\item The algebra $\Rep_\xi$ is the $\BC$-polynomial algebra in the variables $L_\xi(\varpi_1), \dots, L_\xi(\varpi_{n-1})$,
\be \Rep_\xi = \BC[L_\xi(\varpi_1), \dots, L_\xi(\varpi_{n-1})].
\label{eq.iso15}
\ee
 Moreover, $\Ch_\xi$ gives an algebra isomorphism  $\Ch_\xi: \Rep_\xi \rightarrow \BC[X]^{S_n}$.

\end{enuma}
\ethm
\bpr 
We refer to \cite{GL93} for parts (a) and (b), and here prove (c) only. From the definition it is easy to show that $\Ch_\xi: \Rep_\xi \rightarrow \BC[X]$ is an algebra homomorphism.

 If $\Ux$-modules $M, M_1, M_2$ satisfy $M= M_1+ M_2$ in $\Rep_\xi$,  then for any $\mu \in X$,
\be 
\dim_\BC(M^\mu) = \dim_\BC(M_1^\mu)+ \dim_\BC(M_2^\mu).
\label{eq24}
\ee

By part (a) we have $ \Rep_\xi ={\rm span}_\mathbb{C}\{ L_\xi(\lambda) \, |\, \lambda \in X_+\}.
$

Let $\lambda\in X_+$.
There is an $\Ux$-submodule $M_\xi$ of $\Lambda_\xi(\lambda)$ such that 
\be \Lambda_\xi(\lambda)= L_\xi(\lambda) + M_\xi,\ \text{as an equality in $\Rep_\xi$}.
\label{eq25}
\ee

Since $\dim_\BC(\Lambda_\xi(\lambda)^\lambda ) = \dim_\BC(L_\xi(\lambda)^\lambda )=1$ and all weights $\mu \neq \lambda$ of $\Lambda_\xi(\lambda)$ are  less than $\lambda$, 
identity \eqref{eq24} shows that all weights of
$M_\xi$ are less than $\lambda$. Hence in $\Rep_\xi$ we have
$$ M_\xi = \sum_{ \mu \in X_+, \, \mu < \lambda } c_\mu L_\xi(\mu),
$$
for some $c_\mu \in \mathbb{Z}$.

Combining with \eqref{eq25}, and 
by induction on $\lambda$, we get that each $L_\xi(\mu)$ is a $\BZ$-linear combination of $\{ \Lambda_\xi(\mu) \mid \mu\in X_+\}$.
This shows
\be 
{\rm span}_\mathbb{C}\{ \Lambda_\xi(\lambda) \, |\, \lambda \in X_+\} = {\rm span}_\mathbb{C}\{ L_\xi(\lambda) \, |\, \lambda \in X_+\} = \Rep_\xi.
\ee
Since $\Rep_\BC$ has $\{ \Lambda(\lambda) \, |\, \lambda \in X_+\}$ as a 
$\BC$-vector space basis (Proposition \ref{rRepC}(a)), there is a surjective $\BC$-linear map $f: \Rep_\BC \onto \Rep_\xi$ given by $f(\Lambda(\lambda))= \Lambda_\xi(\lambda)$. The composition of the following $\BC$-linear maps
$$ \Rep_\BC \xtwoheadrightarrow{f} \Rep_\xi \xrightarrow[]{\Ch_
{\xi}} \BC[X]$$
is the classical character map $\Ch$, which maps $\Rep_\BC$ bijectively onto $\BC[X]^{S_n}$ by Proposition \ref{rRepC}(b). It follows that $f$ is 
a vector space isomorphism, and $\Ch_\xi$  maps $\Rep_\xi$ bijectively onto $\BC[X]^{S_n}$. As $\Rep_\BC$ is the free $\BC$-polynomial algebra in the variables $\Lambda(\varpi_1), \dots, \Lambda(\varpi_{n-1})$, we get \eqref{eq.iso15}.
\epr

\def\Ox{\mathcal O_{\xi}}

The form \eqref{eq.formZ} specializes to a $\BC$-bilinear  form $$\la \cdot, \cdot \ra_\xi : 
\Ox \ot_\BC \Ux
\to \BC.$$ From  
\eqref{eq-cononical-basis} we have

\blem
\label{lem.inj}
 The map $\Ox
 \to (\Ux)^*$ induced by the form $\la \cdot, \cdot \ra_\xi$ is injective.
\elem
We will identify 
$\Ox$ with a subset of $(\Ux)^*$ by the above embedding.

\def\cT{{\mathcal T}}

\subsection{Module-trace map at a complex number}\label{ss.module-trace_map_at_a_complex_number}

For a finite-dimensional $\Ux$-module $V$ define $\cT_\xi(V)\in (\Ux)^*$, called the {\bf module-trace map} of $V$, so that for $u\in \Ux$ the value $\cT_\xi(V)(u)$ is 
 the trace of the action of $u$ on $V$:
 $$
 \mathcal{T}_\xi(V)(u) = {\rm tr}_V(u).
 $$
 From the properties of the trace, we have
\begin{align*}
\cT_\xi(V \ot V')& = \cT_\xi(V) \cT_\xi(V'), \\
\cT_\xi(V_2)& = \cT_\xi(V_1) + \cT_\xi(V_3), \quad \text{if } \ 
0 \to V_1 \to V_2 \to V_3 \to 0 \quad \text{is exact}.
\end{align*}
 Hence we can extend $\cT_\xi$ to a $\BC$-algebra homomorphism
$
\cT_\xi: \Rep_\xi \to (\Ux)^*.
$
\bpro \label{r-inj}
\begin{enuma}
\item We have $\cT_\xi (\Rep_\xi) \subset \Ox$.

\item The map $\cT_\xi: \Rep_\xi \to \Ox$ is injective.
\end{enuma}
\epro
\bpr (a) Since the algebra $\Rep_\xi$ is generated by $\Lambda_\xi(\varpi_
{k})$, $1\le k\le n-1$, it is enough to show that $\cT_\xi(\Lambda_\xi(\varpi_
{k}))\in \Ox$ for each $
k=1, \dots, n-1$. Since $\cT(\Lambda_q(\varpi_
{k})) \in \Oq$, we have $\cT_\xi(\Lambda_\xi(\varpi_
{k}))\in \Ox$.

(b) By Lemma \ref{lem-trace-fundamental-representations-dot}, we have $\cT_\xi(\Lambda_\xi(\varpi_
{k}))= D_
{k}^\xi(\buu)$. By 
Lemma \ref{r.indep}, the elements 
$D_
{k}^\xi(\buu), 
k=1, \dots, n-1$, are algebraically independent. Hence by
\eqref{eq.iso15} and Theorem \ref{thm-rep-dotU}(b), it follows that $\cT_\xi$ is injective.
\epr

\def\Ad{{\mathsf{\Psi}}}

\subsection{Roots of 
unity: the Adams operation}\label{subsection-root-Adams}
Assume that $\omega$ is a root of 
unity. Let $N= \ord(\omega^2)$ and $\eta = \omega^{N^2}$. Let $\Ad\colon\Rep_\eta \to \Rep_\omega$ be the following composition
$$  \Ad\colon \Rep_\eta \xrightarrow{\Ch_\eta } \BC[X]^{S_n} \xrightarrow{\Psi_N } \BC[X]^{S_n} \xrightarrow{\Ch_\omega^{-1} } \Rep_\omega,$$
which is a $\mathbb C$-algebra embedding because the $N$-th Adams operation $\Psi_N$ in \eqref{Adams} is  a $\mathbb C$-algebra embedding. As hinted before, we shall be seeing that this map $\Ad\colon\Rep_\eta \to \Rep_\omega$ can be viewed as the Frobenius homomorphism for the representation rings.

By Theorem \ref{thm-rep-dotU} we have $\Ch_\xi( L_\xi(\varpi_{k}) )= e_{k}$, while by definition \eqref{eq-def-barP}, we have $\Psi_N(e_{k}) = \bar P_{N,k}(e_1, \dots, e_{n-1})$. Hence
\be 
\Ad( L_\eta(\varpi_{k}))= \bar P_{N,k} (L_\omega(\varpi_1) , \dots, L_\omega(\varpi_{n-1})). \label{eq57a}
\ee
This equation presents the role of the reduced power elementary polynomial $\bar{P}_{N,k}$ in the context of the fundamental simple modules in the representation rings of quantized enveloping algebras at roots of unity, in relation to the Adams operation.

\def\Fr{{\mathsf{Fr}}}
\def\pb{{\mathsf{pb}}}
\subsection{Roots of unity: Lusztig's quantum Frobenius homomorphism} We now discuss Lusztig's quantum Frobenius homomorphism.
Assume that $\omega$ is a root of unity. Let $N= \ord(\omega^2)$ and $\eta = \omega^{N^2}$.

By \cite[Theorem 35.1.9]{GL93},  there exists a unique Hopf $\BC$-algebra homomorphism 
$$\Fr: \Uo \to \Ue,$$ called  
Lusztig's quantum Frobenius homomorphism,  such that
\begin{align}\label{eq.Lusztig_Frobenius}
    \begin{split}
    \Fr(E_i^{(r)} 1 _\lambda) &= \begin{cases} (E_i^{(r/N)} 1 _{\lambda/N})  &  \mbox{if } N|r\  
    \mbox{and} \ \lambda \in N X \subset X, 
	\\
		0 &\text{otherwise,}
	\end{cases} \\
	\Fr(F_i^{(r)} 1 _\lambda) &= \begin{cases} (F_i^{(r/N)} 1 _{\lambda/N})  &  \mbox{if } N|r\  
    \mbox{and} \ \lambda \in N X \subset X, \\
		0 &\text{otherwise.}
	\end{cases}
    \end{split}
\end{align}

If $V$ is a $\Ue$-module, then by pulling back via $\Fr: \Uo \to \Ue$ we can consider $V$ as an $
\Uo$-module, which is denoted by $\Fr^\pb(V)$.  On the other hand, the map $\Fr: \Uo \to \Ue$ has the linear  dual $\Fr^*: (\Ue)^* \to (\Uo)^*$, where $\mathcal{V}^*= \Hom_\BC(
\mathcal{V}, \BC)$ for a $\BC$-vector space $
\mathcal{V}$.

\blem\label{lem.module-trace_and_Fr}
For a $\Ue$-module $V$, we have
\be 
\cT_\omega( \Fr^\pb(V)) = \Fr^*( \cT_\eta(V)  ).
\label{eq48}
\ee
\elem
\bpr This follows straightforwardly from the definition; for $u\in \Uo$,
$$ \cT_{\omega}( \Fr^\pb(V)) (u)=  
{\rm tr}_{\Fr^\pb(V)}(u) = 
{\rm tr}_V(\Fr(u))
=  \cT_{\eta}(V)(\Fr(u)   ) 
= \Fr^*( \mathcal{T}_\eta(V)) (u).$$
This proves the lemma.
\epr

\bpro\label{r-PhiPsi} 
We have $\Fr^\pb= \Ad\colon \Rep_\eta \to \Rep_\omega$.
\epro
\begin{proof}

	Let a finite-dimensional unital $\Ue$-module  $V$ have weight 
   space decomposition
	$$V=
    \textstyle \bigoplus_{\lambda\in X} V^{\lambda},  \quad V^{\lambda} = 1_{\lambda}\cdot V.$$ 
	Because $\Fr(1_{N\lambda})= 1 _\lambda$, when considered as a subspace of the pull back $\Fr^\pb(V)$, 
 each $V^{\lambda}$ is a $N\lambda$-weight subspace.
 Hence, 
	$$\Ch_\omega(\Fr^\pb(V)) = \sum_{\lambda\in X}(\text{dim}_{\mathbb C} V^{\lambda}) e^{N\lambda} =  \Psi_N(\sum_{\lambda\in X}(\text{dim}_{\mathbb C} V_{\lambda}) e^{\lambda}) = \Psi_N (\Ch_\eta (V)).$$
	This proves the Proposition.
\end{proof}

\def\End{{\mathsf{End}}}

\renewcommand\vec{\avec}
\def\bb{{\avec b}}

\subsection{Duality between two Frobenius homomorphisms} 
The following is probably well known.

\bthm\label{thm-dual2}
Assume that $\omega$ is a root of unity. Let $N= \ord(\omega^2)$ and $\eta = \omega^{N^2}$. Then Lusztig's quantum Frobenius homomorphism $\Fr: \Uo \to \Ue$ is dual to the Frobenius homomorphism $\Phi: \Oe \to \Oo$. This means that for all $x\in \Uo$ and $y\in \Oe$,
\be
\la y,  \Fr(x) \ra_{\eta} =  \la \Phi(y), x \ra_{\omega}
\label{eq.dual0}
\ee
\ethm
By the injectivity of Lemma \ref{lem.inj}, the theorem means that the linear dual map $\Fr^*: (\Ue)^* \to (\Uo)^*$ is equal to $\Phi$ on the subset $\Oe\subset (\Ue)^*$. We give a proof of this theorem, for completeness.

\bpr  

The Hopf duality reduces \eqref{eq.dual0} to the case when $x$ is one of $E^{(r)}_t 1_\lambda$, $F^{(r)}_t 1_\lambda$ and $y$ is one of $u_{ij}$, which will be proved by  explicit calculations.

Let $\Mat_n(\BC)$ be the set of all $n\times n$ matrices with entries in $\BC$. Let $E_{i,j}\in \Mat_n(\BC)$ be the matrix whose entries are 0 except for the 
$(i,j)$-th entry, which is 1.
The $\Uo$-module $\Lambda_\omega(\varpi_1)$ is $\BC^n$ as a vector space, and its representation $\rho: \Uo\to  \Mat_n(\BC)$ is given by
	\begin{equation}\label{eq-dot-representation-first-weigt}
		\begin{split}
			\rho(E_i^{(r)} 1_{\lambda}) = \left\{
			\begin{array}{ll}
				E_{t,t} & \mbox{if $\lambda = w_{t}$ for some $1\leq t\leq n$ and $r=0$},\\
				E_{i,i+1} & \mbox{if $\lambda = w_{i+1}$ and $r=1$}, \\
				0 & \mbox{otherwise},
			\end{array}
			\right.\\
			\rho(F_i^{(r)} 1_{\lambda})= \left\{
			\begin{array}{ll}
				E_{t,t} & \mbox{if $\lambda = w_{t}$ for some $1\leq t\leq n$ and $r=0$},\\
				E_{i+1,i} & \mbox{if $\lambda = w_i$ and $r=1$}, \\
				0 & \mbox{otherwise}.
			\end{array}
			\right.
		\end{split}
	\end{equation}
	For any $x\in\Uo$ and $i,j \in \JJ=\{1,\ldots,n\}$, 
    we have
    \begin{align}
        \label{eq.pair_with_u_ij}
        \langle u_{ij},x \rangle_\omega = \mbox{the $(i,j)$-th entry of $\rho(x)$.}
    \end{align}
	
	Recall that $ \Delta(x)$, for $x\in 
    \Uo$,  is in general an infinite sum of elements in $\Uo \ot \Uo$ (with  only a finite number of them acting non-trivially on the tensor product of any two finite-dimensional unital $\Uo$-modules). Define $\Delta^{(m)}(x)$ for $m \ge 2$ recursively by
	$$ \Delta^{(2)}= \Delta, \qquad \Delta^{(m)}= (\Delta \ot \id^{\ot (m-2)  }) \circ \Delta^{(m-1)}, \quad  m \ge 3.$$
	\def\vl{{\avec \lambda}}
	For example, for $\lambda\in X$.
	\be 
		\Delta^{(N)}(1_{\lambda}) = \sum_{\vl =(\lambda_1,\ldots,\lambda_N) \in X^N(\lambda)} 1_{\lambda_1}\otimes\cdots\otimes 1_{\lambda_N}.
		\label{eq.coprod0}
		\ee
where $X^N(\lambda)= \{ \vl=
(\lambda_1,\ldots,\lambda_N) \in X^N \mid \sum \lambda_j= \lambda\} $. 

Let $1\leq t\leq n-1$, $r$ be a positive integer, and $\lambda\in X$. 	
To describe $\Delta^{(N)}(E_t^{(r)} 1_{\lambda})$ we use the following notations. Let $\varsigma_t = E_{t,t} - E_{t+1, t+1}$, an element of the Cartan subalgebra.
For  ${\vec b}=(b_1,b_2,\cdots,b_N)\in\mathbb N^{N}$, $\vl=(\lambda_1,\ldots,\lambda_N)$, and $0\leq i\leq N-1$, define 
$$ s_i({\vec  b}) = \sum_{j=i+1}^N b_j, 
	\qquad
	s({\vec  b}) = \sum_{1\leq i\leq N-1} b_is_i({\vec  b}), \qquad f_t(\bb, \vl)= \sum_{1\leq l\leq N-1} s_l({\vec  b}) \lambda_l( \varsigma_t). $$

From \cite[Proposition 4.8]{GL90} and the definition of the coproduct  we have, in $\Uo$,
\begin{align}
\Delta^{(N)}(E_t^{(
r)} 1_{\lambda})
			&=\sum_{{\vec  b}\in\mathbb N^N, s_0({\vec  b})=
            r; \ \vl \in X^N(\lambda)}\omega^{s({\vec  b})+
			f_t(\bb, \vl)
			} 
			\big(
			E_t^{(b_1)} 1_{\lambda_1}
			\otimes 
			\cdots
			\otimes 		E_t^{(b_{N})} 1_{\lambda_N}\big),  \label{eq.DeltaE}\\
			\Delta^{(N)}(F_t^{(
            r)} 1_{\lambda})
			&=\sum_{{\vec  b}\in\mathbb N^N, s_0({\vec  b})=
            r; \ \vl \in X^N(\lambda)}\omega^{-s({\vec  b})-
			f_t(\bb, \vl)
			} 
			\big(
			F_t^{(b_N)} 1_{\lambda_N}
			\otimes 
			\cdots 
			\otimes 		F_t^{(b_{1})} 1_{\lambda_1}\big),  \label{eq.DeltaF}
\end{align}

Let us prove \eqref{eq.dual0} for $y=u_{ij}$ and $x= 1_\lambda$. By \eqref{eq.Phi_on_u_ij}, \eqref{eq-dot-paring} and \eqref{eq.coprod0}, the right hand side of \eqref{eq.dual0} equals
\begin{align*}
	\langle u_{ij}^N, 1_{\lambda} \rangle_{\omega} = 
	\sum_{\vl\in X^N(\lambda)} 	\langle u_{ij}, 1_{\lambda_1}  \rangle_{\omega}\cdots 	\langle u_{ij}, 1_{\lambda_N}  \rangle_{\omega}=
	\begin{cases}
		1 & \mbox{if } i=j \mbox{ and } \lambda = N w_i,\\
		0 & \text{otherwise},
	\end{cases}
\end{align*}
where for the second equality we used \eqref{eq-dot-representation-first-weigt} and \eqref{eq.pair_with_u_ij}.
The left hand side of \eqref{eq.dual0} is
\begin{align*}
	 \langle u_{ij}, \Fr(1_{\lambda})  \rangle_{\eta} =	
	\begin{cases}
		1 & \mbox{if } i=j \mbox{ and } \lambda = N w_i,\\
		0 & \text{otherwise},
	\end{cases}
\end{align*}
where we used \eqref{eq.Lusztig_Frobenius} and \eqref{eq.pair_with_u_ij}.
This proves \eqref{eq.dual0} for $y=u_{ij}$ and $x= 1_\lambda$.

Let us prove \eqref{eq.dual0} for $y=u_{ij}$ and $x=  E^{(r)}1_\lambda$ with $r>0$.
Assume that ${\vec  b}=(b_1,\cdots,b_N)\in\mathbb N^N$ such that
$s_0({\vec  b}) = r$.
Then equations \eqref{eq-dot-representation-first-weigt} and \eqref{eq.pair_with_u_ij} imply that 
$$\langle u_{ij}, 	E_t^{(b_l)} 1_{\lambda_l} \rangle_{\omega}=
\left\{
\begin{array}{ll}
	1& \mbox{if $\lambda_l = w_{t}$, $b_l=0$, and $i=j=t$},\\
	1 & \mbox{if $\lambda_l = w_{t+1}$, $b_t=1$, $i=t$, and $j=t+1$}, \\
	0 & \mbox{otherwise}.
\end{array}
\right.$$
Since $r>0$, we have $b_l\neq 0$ for at least one $l$. This shows that
$\prod_{1\leq l\leq N} \langle  u_{ij}, E_t^{(b_l)}1_{\lambda_l}  \rangle_{\omega} = 0$ if $i=j$. 
Then we get
\begin{align*}
    \prod_{1\leq l\leq N} \langle u_{ij}, E_t^{(b_l)}1_{\lambda_l} \rangle_{\omega} =
	\left\{
	\begin{array}{ll}
		1 & \mbox{if $\lambda_l = w_{t+1}$ and $b_l=1$ for all $1\leq l \leq N$, \, $i=t$, and $j=t+1$}, \\
		0 & \mbox{otherwise}.
	\end{array}
	\right.
\end{align*}

If $b_1=\cdots=b_N=1$ and $\lambda_1 =\cdots =\lambda_N = w_{t+1}$, then  $s({\vec  b}) + f_t(\bb, \vl)=0$. Hence equations \eqref{eq.Phi_on_u_ij}, \eqref{eq-dot-paring} and \eqref{eq.DeltaE} show that the right hand side of \eqref{eq.dual0} equals
\begin{align}\label{pairing-E}
	\langle u_{ij}^N, E_t^{(r)} 1_{\lambda} \rangle_{\omega}
	=
	\left\{
	\begin{array}{ll}
		1 & \mbox{if $\lambda =N w_{t+1}$, $r=N$, $i=t$, and $j=t+1$}, \\
		0 & \mbox{otherwise}.
	\end{array}
	\right.
\end{align}
Meanwhile, from the value of $\Fr$ given by \eqref{eq.Lusztig_Frobenius}, and from equations \eqref{eq-dot-representation-first-weigt} and \eqref{eq.pair_with_u_ij}, we see that the left hand side of \eqref{eq.dual0} equals
\begin{align}\label{pairing-e}
\langle u_{ij}, \Fr(E_t^{(r)} 1_{\lambda}) \rangle_{\eta}
	=
	\left\{
	\begin{array}{ll}
		1 & \mbox{if $\lambda=Nw_{t+1}$, $r=N$, $i=t$, and $j=t+1$}, \\
		0 & \mbox{otherwise}.
	\end{array}
	\right.
\end{align}
Equations \eqref{pairing-E} and \eqref{pairing-e} show that \eqref{eq.dual0} holds for $y=u_{ij}$ and $x=  E^{(r)}1_\lambda$ with $r>0$.

Similarly, we can show  \eqref{eq.dual0} for $y=u_{ij}$ and $x=  F^{(r)}1_\lambda$ with $r>0$.
\epr

\def\End{{\mathsf{End}}}
\def\CL{{\mathcal L}}

\subsection{Adams operations and Frobenius homomorphism} We now show that the Adams operation is the restriction of the Frobenius 
 homomorphism in the following sense.

\begin{theorem} 
\label{thm-L-comp-F}Assume that $\omega\in \BC$ is a root of unity.
	Let $N=\ord(\omega^2)$ and $\eta=\omega^{N^2}$.

    \begin{enuma}
	\item The following diagram commutes:  
	\begin{equation}\label{eq57}
	\begin{tikzcd}
		\Rep_\eta \arrow[r,hook, "\cT_\eta"]
		\arrow[d, "\Ad"] 
		&    \Oe \arrow[d, "\Phi"] \\
	\Rep_\omega  \arrow[r, hook, "\cT_\omega"] 
		&  	\Oo.
	\end{tikzcd}
	\end{equation}
\item Consequently, for $k=1, \dots, n-1$,
\be 
\label{eq.role_of_bar_P_N_i_for_quantum_minors}
\Phi(D_{k}^\eta(\buu)) = 
\bar P_{N,k}(D_{1}^\omega(\buu), \dots, D_{n-1}^\omega(\buu)).
\ee	 
\end{enuma}

\end{theorem}

\begin{proof} (a)  The commutativity of the diagram \eqref{eq57} is equivalent to the equation
\begin{align}
\label{eq.module-trace_and_Psi_Phi}
    \Phi\circ \cT_\eta = \cT_\omega \circ \Psi,
\end{align}  
which in turn is equivalent to equation \eqref{eq48} of Lemma \ref{lem.module-trace_and_Fr},
since $\Phi=\Fr^*$ by Theorem \ref{thm-dual2} and $\Ad= \Fr^\pb$ by Lemma \ref{r-PhiPsi}.

(b) By Lemma \ref{lem-trace-fundamental-representations-dot} and Theorem \ref{thm-rep-dotU}(b) we have
\begin{align}
\label{eq.quantum_elementary_symmetric_as_module-trace}
    D^\xi_{k}({\bf u})= \cT_\xi(L_\xi(\varpi_{k}))
\end{align}
for each $k=1,2,\ldots,n-1$ and any non-zero $\xi \in \mathbb{C}$. 
Observe that
\begin{align*}
	\Phi (D^\eta_{k}({\bf u}))  
    & = \Phi(\cT_\eta(L_\eta(\varpi_{k}))) \quad (\because \mbox{equation \eqref{eq.quantum_elementary_symmetric_as_module-trace}}) \\
	& = \mathcal T_{\omega}(\Ad(L_{\eta}(\varpi_{k}))) \quad (\because \mbox{equation \eqref{eq.module-trace_and_Psi_Phi} of part (a)})  \\ 
	& = \mathcal T_{\omega}(\bar P_{N,k} ( L_\omega(\varpi_1), \ldots, L_\omega(\varpi_{n-1}) )) \quad (\because \mbox{equation \eqref{eq57a}}) \\
	& =\bar P_{N,k}( \mathcal{T}_{\omega}(L_\omega(\varpi_1)), \ldots, \mathcal{T}_{\omega}(L_\omega(\varpi_{n-1}))) \quad (\because\mbox{$\mathcal{T}_{\omega}$ is a ring  homomorphism}) \\
	& =\bar P_{N,k}( D^\omega_1({\bf u}), \ldots, D^\omega_{n-1}({\bf u})) \quad (\because\mbox{equation \eqref{eq.quantum_elementary_symmetric_as_module-trace}}).
\end{align*}
\epr

Equation \eqref{eq.role_of_bar_P_N_i_for_quantum_minors} presents the role of the reduced power elementary polynomial $\bar{P}_{N,k}$ in the context of the quantum elementary symmetric polynomials (i.e. sums of principal minors of each size) in the quantized function algebras at roots of unity, in relation to the Frobenius homomorphism.

\begin{remark}
    Francis Bonahon informed us that he and Vijay Higgins prove \eqref{eq.role_of_bar_P_N_i_for_quantum_minors} in their upcoming work, by a proof different from ours.
\end{remark}

The roles of $\bar{P}_{N,i}$ studied in this section will later be translated into the language of stated ${\rm SL}_n$-skein algebras, and we will see that the results obtained this section are what let us bypass difficult skein-theoretic computations.

\def\SO{\cS_{\hat\omega}(\fS)}

\section{Stated ${\rm SL}_n$-skein modules and algebras}\label{sec.stated_SLn-skein_modules_and_algebras}

The elements of the ${\rm SL}_n$-skein module of a $3$-manifold $M$ are linear combinations of isotopy classes of {\it $n$-webs} in $M$, each of which is a union of framed oriented knots and certain graphs embedded in $M$, subject to special local relations called {\it skein relations}, which in turn are modeled on the category of representations of Lie algebras and quantum groups. When $M$ is given as a surface $\fS$ times interval $(-1,1)$, this module is equipped with a natural product structure, and the resulting ${\rm SL}_n$-skein algebra of $\fS$ is known to provide a quantization of the ${\rm SL}_n$-character variety of $\fS$. Later, {\it stated} versions of the ${\rm SL}_n$-skein modules and algebras of `marked' 3-manifolds and `punctured-bordered' (in short, `pb') surfaces have been studied, which behave well under cutting and gluing of the underlying manifolds.

Here we give a review of the literature, on the construction and important properties. Especially, we review the ${\rm SL}_n$-quantum trace map which maps the ${\rm SL}_n$-skein algebra of a surface to a quantum torus algebra.

\def\pM{\partial M}
\def\vk{\varkappa}
\def\drQ{\partial_r(Q)}
\def\pal{\partial \al}
\def\bpp{\beta^\perp}
\def\p{\partial }
\def\MM{\mathbb M}

\newcommand\FIGc[3]{\begin{figure}[htpb]
\includegraphics[height=#3]{draws/#1.eps}
\caption{#2}\label{f.#1}
\end{figure}}

\subsection{Marked $3$-manifolds and $n$-webs} We introduce the category of marked 3-manifolds, and its subcategory of punctured bordered  surfaces.
\label{ss.marked} 

A {\bf marked {$3$}-manifold} is a pair $(M, \cN),$ where $M$ is a smooth oriented $3$-manifold with (possibly empty) boundary $\p M$, and $\cN$, called the marking, is an oriented 1-dimensional submanifold of $\p M$. We require that each connected component of $\cN$ is an open interval. 

An {\em embedding} of a marked 3-manifold $\MN$ into a marked 3-manifold $(M', \cN')$ is an orientation-preserving proper embedding $f: M \embed M'$  that maps $\cN$ into $\cN'$ preserving their orientations. Two such embeddings $f$ and $f'$ are {\em isotopic} if there is a continuous family of diffeomorphisms $H_t: M'\to M', t\in [0,1]$, such that $
H_0=\id, H_1\circ 
F= f'$, and $H_t(\cN')\subset \cN'$ for all $t\in [0,1]$. 

The {\bf category of marked 3-manifolds} $\MM$ has marked 3-manifolds as objects, and 
 a morphism  from $\MN$ to $(M', \cN')$ is an isotopy class of embeddings from $\MN$ to $(M', \cN')$. 

\bdf\label{def-n-web}
An  {\bf $n$-web} $\alpha$ in $\MN$ is a disjoint union of finite number of oriented circles and a finite directed graph properly embedded into $M$ such that
\benu[nosep]
\item  Every vertex of $\alpha$ is either a sink or a source and either $1$-valent or $n$-valent. 
We denote the set of 1-valent vertices, called {\bf endpoints} of $\al$, by $\pal$. 
\item Each edge of the graph  is a smooth embedding of the closed interval $[0,1]$ into $M$.
\item $\alpha$ is equipped with a {\bf framing} which is a continuous non-vanishing vector field transversal to $\alpha$. In particular, the framing at a vertex is  transversal to all incident edges.
\item[(4)] The set of half-edges at every $n$-valent vertex is cyclically ordered. 
\item[(5)] $\al \cap \p M=\al \cap \cN=\p \al$,  and the framing at these endpoints is a tangent vector of $\cN$, pointing in the direction of the orientation of $\cN$. We call such tangent vector {\bf positive}.
\eenu
\edf

The  {\bf height order} on $\p \al$  is the  partial order in which  two points $x,y\in \p \al$ are comparable if and only if they belong to the same 
component of the marking $\mathcal{N}$, and $x > y$, or {\bf $x$ is higher than $y$}, if going along the positive direction of the marking we encounter $y$ first. We say that $x$ and $y$ are {\bf consecutive} if there is no $z\in \p \al$ such that $ x > z > y$ or $ y > z > x$.

A  {\bf state} of an $n$-web $\alpha$ is a map  $\partial \alpha\rightarrow \mathbb J=\{1,2,\ldots,n\}$.  
An $n$-web $\alpha$ together with the choice of a state is called a {\bf stated $n$-web}.
We consider (stated) $n$-webs up to isotopy within their 
classes. 
By convention, the empty set is considered a stated $n$-web which is isotopic only to itself.

 Recall that for $i\in \JJ$ its conjugate $\bi$ is $n+1-i$. Also, $S_n$ is the symmetric group of $\JJ$.

 The {\bf stated ${\rm SL}_n$-skein module}  $\cS_{\hq}\MN$ 
 is the $\Zhq$-module freely spanned by the isotopy classes of stated $n$-webs in $\MN$ modulo the following defining relations using the constants $\ttt, \aaa, \ccc_i$ defined in \eqref{def-constants-tac} and the quantum integer $[n]_q = \frac{q^n-q^{n-1}}{q-q^{-1}} = q^{-n+1}+q^{-n+3}+\cdots+q^{n-1}$ in \eqref{eq.quantum_integer}. The first batch of relations regarding the interior of $M$ are:
\begin{gather}
q^{\frac 1n} \cross{}{}{p}{>}{>} - q^{-\frac 1n}\cross{}{}{n}{>}{>}
= (q-q^{-1})\walltwowall{}{}{>}{>}, \label{e.pm} \\
\kink = \ttt \horizontaledge{>}, \label{e.twist}\\
\circlediag{<} = (-1)^{n-1} [n]_{q} \TanglePic{0.9}{0.9}{}{}{}, \label{e.unknot}\\
\sinksourcethree{>}=(-q)^{\binom{n}{2}}\cdot \sum_{\sigma\in S_n}
(-q^{(1-n)/n})^{\ell(\sigma)} \coupon{$\sigma_+$}{>}. \label{e.sinksource}
\end{gather}
where the ellipse enclosing $\sigma_+$ is the minimum crossing positive braid, whose crossings are positive
$\cross{}{}{p}{>}{>}$,  representing a permutation $\sigma\in S_n$ and $\ell(\sigma)$ is the length of $\sigma\in S_n$, which is the minimum number of transpositions $(1 2)$, $(2 3)$, \ldots, $(n-1 \, n)$ needed to express $\sigma$. For example, the minimum positive braid corresponding to the element $\sigma=(1 2 3) \in S_n$ is 
$
\raisebox{-.2in}{
\begin{tikzpicture}
	\fill[gray!20] 
    (0,0.2) 
    rectangle 
    (1.6,1.4);
	\begin{knot}
		\draw[-o-={0.9}{>}] (0,1.3) -- (1.6,0.3);
		\draw[-o-={0.9}{>}] (0,0.8) -- (1.6,1.3);
		\draw[-o-={0.9}{>}] (0,0.3) -- (1.6,0.8);
		\strand[edge]  (0,1.3) -- (1.6,0.3);
		\strand[edge] (0,0.8) -- (1.6,1.3) (0,0.3) -- (1.6,0.8);
	\end{knot}
\end{tikzpicture}}
$, where $\ell(\sigma)=\ell( (12)(13)) = 2$. 
The second batch of relations regarding $\mathcal N$ are:
\begin{align}
\vertexnearwall{white} & = \aaa \sum_{\sigma \in S_n} (-q)^{\ell(\sigma)} \nedgewall{<-}{white}{$\sigma(n)$}{$\sigma(2)$}{$\sigma(1)$}
\label{e.vertexnearwall}\\
\capwall{<-}{right}{white}{$i$}{$j$} & = \delta_{\bar j,i} \ccc_i,\label{e.capwall}\\
\capnearwall{white} &= 
\sum_{i \in \JJ} (\ccc_{\bar i})^{-1} \twowall{<-}{white}{black}{$i$}{$\bar{i}$}
\label{e.capnearwall}\\
\crosswall{<-}{p}{white}{white}{$i$}{$j$}
&=q^{-\frac{1}{n}}\left(\delta_{{j<i}}(q-q^{-1})\twowall{<-}{white}{white}{$i$}{$j$}+q^{\delta_{i,j}}\twowall{<-}{white}{white}{$j$}{$i$}\right).
\label{e.crossp-wall}
\end{align}
Each shaded rectangle in relations \eqref{e.pm}-\eqref{e.crossp-wall} is the projection of a small open cube 
in $M$. 
The oriented graphs and graphs with small white or black circles contained in the shaded rectangles represent parts of stated $n$-webs with framing  pointing to  readers. 
Small white circles in each equation represent an arbitrary orientation (left-to-right or right-to-left) of the edges, consistent for the entire equation. The black circle represents the opposite orientation.
When a boundary edge of a shaded area is directed, the direction indicates the height order of the endpoints of the stated $n$-webs on that directed line, where going along the direction increases the height, and the involved endpoints are consecutive in the height order. The height order outside the drawn part can be arbitrary.
For detailed explanation for relations \eqref{e.pm}-\eqref{e.crossp-wall},
see \cite{LS21}.

To simplify notations, we will write 
$\cS_{\hq}\MN$ as $\cS\MN$.

\subsection{Punctured bordered surfaces and stated ${\rm SL}_n$-skein algebras}
\label{subsec:punctured_bordered_surface_and_n-web}

\begin{definition}\label{def.pb_surface}
A {\bf punctured bordered
surface $\fS$} (or a {\bf pb surface}) is a surface of the form $\fS =\bfS\setminus \cV$, where $\bfS$ is a compact oriented 2-dimensional manifold with (possibly empty) boundary $\partial \bfS$, and $\cV\subset 
\bfS$ is a finite set such that every component of $\partial \bfS $ intersects $\cV$. Each connected component of $\partial \fS =\partial 
\bfS \setminus \cV$ is diffeomorphic to the open interval $(0,1)$ and is called a {\bf boundary edge}. A point $x\in \cV$ is called an {\bf ideal point}, or a {\bf puncture}, of $\fS$. 
A point $x \in \mathcal{V}$ that does not lie in $\p \bfS$ is called an {\bf interior puncture}.

A pb surface $\fS$ is {\bf essentially bordered} if every connected component of it has non-empty boundary.
\end{definition}

An \term{ideal arc} in $\fS$ is an embedding $c:(0,1)\embed\fS$ which can be extended to an immersion $\bar c : [0,1] \to \bfS$ such that $\bar c(0), \bar c(1) \in 
\mathcal{V}$. An ideal arc $c$ is \term{trivial} if the extended map $\bar c$ can be homotoped
rel its boundary to a point.

A closed interval properly embedded in $\fS$ is called a \term{$\pfS$-arc}. A $\pfS$-arc is \term{trivial} if it is homotopic 
rel its boundary points to a subinterval of $\pfS$.

Let $\fS$ be a pb surface.
For each boundary component $b$ of $\fS$, we select one point $x_b\in b$. 
Define 
\begin{align}
    \label{marked_3-manifold_from_surface}
    \mbox{$M_\fS=\fS\times(-1,1)$ \quad and \quad
$\cN_\fS=
\bigcup_{b}(\{x_b\}\times (-1,1))$,}
\end{align}
where the union 
goes over all boundary components $b$ of $\fS$ and the orientation of $\mathcal N_\fS$ is the positive orientation of $(-1,1)$. 
Then $(M_\fS,\mathcal N_\fS)$ is a marked $3$-manifold defined up to isotopy, called the {\bf thickening} of ${\fS}$, denoted as $\widetilde{\fS}$.
Define 
\begin{align}
    \label{stated_SLn-skein_algebra}
    \cS_\hq(\fS)=
\cS(\fS) := \cS(\tfS) = \cS(M_\fS,\mathcal N_\fS),
\end{align}
called the {\bf stated ${\rm SL}_n$-skein algebra} of $\fS$.
For two stated $n$-webs $\al, \beta$ in 
$(M_\fS,\mathcal N_\fS)$, its product $\al \beta\in
\cS(\fS)$ is defined as the result of stacking $\al$ above $\beta$. This means that
we first isotope so that $\al \subset \fS \times (0,1)$ and $\beta \subset \fS\times(-1,0)$, then let $\al \beta := \al \cup \beta\in \cS(\fS)$.

We often identify $\fS$ as the subset $\fS \times \{0\}$ of $\tfS $. For a point $(x,t)\in \tfS = \fS \times (-1,1)$, its {\bf height} is $t$. A tangent vector to $\tfS$ at $(x,t)$ is \term{upward vertical} if it is along the positive direction of the component $(-1,1)$. We denote by $\pr: \tfS \to \fS$ the projection onto the first component. 
When we consider (stated) $n$-webs in $\widetilde{\fS}$, we are allowed to use isotopies preserving the height (partial) order on endpoints to isotope their endpoints away from  $\{x_b\}\times (-1,1)$ for any boundary component $b$ of $\fS$.

Every (stated) $n$-web $\alpha$ in $\widetilde{\fS}$ can be isotoped to a \term{vertical position}, where
\begin{itemize}
\item the framing is upward vertical everywhere,
\item $\al$ is in general position with respect to the projection $\pr: \tfS \to \fS$, and
\item at every $n$-valent vertex, the cyclic order of half edges, after projected onto $\fS$, is compatible with the positive orientation of $\fS$ (counterclockwise if drawn on the pages of the paper).
\end{itemize}

\begin{definition}\label{def-diag}
Suppose that $\al$ is an $n$-web in $\tfS$ in a vertical position.
The projection $
\pr(\al)\subset \fS$, together with the usual over/underpassing information at each double point (forming a crossing of $\alpha$), and the partial order on $\partial 
(\pr(\al)) = \pr(\partial \al)$ induced from the height order on $\partial \al$, is called the {\bf diagram} of $\al$.

An {\bf $n$-web diagram} is the diagram of an $n$-web.

A {\bf stated $n$-web diagram} is the diagram of a stated $n$-web, where we also remember the induced state.
\end{definition}
Whenever clear, we often identify an $n$-web or a stated $n$-web with its diagram, by a slight abuse of notation.

The orientation of a boundary edge $e$ of $\fS$ is \term{positive} if it is induced from the orientation of $\fS$. In pictures the convention is that the positive orientation of a boundary edge is the counterclockwise one (considered from inside $\fS$). If the height order of an $n$-web diagram $\al$ is given by the positive orientation, i.e. the height order increases when following the positive direction on each boundary edge, then we say that $\al$ has \term{positive order}. One defines \term{negative order} similarly, using the \term{negative orientation}, which is the opposite of the positive orientation.

We use $\mathsf A$ to denote the twice-punctured sphere, which is diffeomorphic to a once-punctured open disc; see Figure \ref{fig-twice-punctured-sphere}(a).
There is an $n$-web ${\bf a}$ in $\mathsf{A}\times (-1,1)$ as shown in Figure \ref{fig-twice-punctured-sphere}(b).
An $n$-web $
\alpha$ in a marked $3$-manifold $(M,\mathcal N)$ is called a {\bf framed oriented knot} if $
\alpha$ consists of a single oriented circle. 
Then there exists an embedding $f_
{\alpha}\colon \mathsf{A}\times (-1,1)\rightarrow \mathring{M}$ into the interior of $M$ such that 
$f_
{\alpha}(
{\bf a}) = 
\alpha$.

\begin{figure}
\centering
\begin{tikzpicture}[baseline=0cm,every node/.style={inner sep=2pt}]
\draw[wall,fill=gray!20,dotted] (0,0) circle[radius=1];
    \draw[fill=white] (0,0) circle[radius=0.1];
\path (0,-1.5)node{(a)};

\begin{scope}[xshift=3.5cm]
\draw[wall,fill=gray!20,dotted] (0,0) circle[radius=1];
    \draw[fill=white] (0,0) circle[radius=0.1];
    \draw[decoration={markings, mark=at position 0.25 with {\arrow{>}}}, postaction={decorate}] 
        (0,0) circle[radius=0.5];
    \path (0,-1.5)node{(b)};
\end{scope}
\end{tikzpicture}
\caption{(a) The twice punctured sphere $\mathsf A$ \qquad (b) A knot in $\mathsf A$}\label{fig-twice-punctured-sphere}
\end{figure}
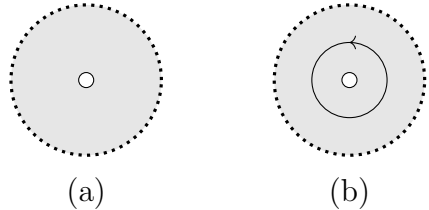

The {\bf bigon} $\PP_2$ is the pb surface obtained from a closed disk by removing two points on its boundary (see  
\OldS\ref{subsec.polygons} for general polygons). The bigon
will play an important role in the present paper. 
A {\bf based bigon} is a bigon together with the choice of a distinguished ideal point, called the {\bf based vertex}. 
In pictures
a based bigon is depicted with the based vertex at the top, and we can define the \term{left edge $e_l$} the \term{right edge $e_r$}, as in Figure \ref{fig-bigon}(a). We often depict $\PP_2 $ as the square $[-1,1] \times (-1,1)$, as in Figure \ref{fig-bigon}(b).
There is an $n$-web $u$ in $\mathbb P_2\times (-1,1)$ as shown in Figure \ref{fig-bigon}(c).
An $n$-web $\gamma$ in $\MN$ that consists of a connected graph with no $n$-valent vertex, two $1$-valent vertices and one edge is called a
{\bf framed $\mathcal N$-arc} in $\MN$.
Then there exists a proper embedding 
$f_\gamma\colon \mathbb P_2\times (-1,1) \rightarrow M$ such that $f_\gamma(u) = \gamma$.

\begin{figure}
\centering
\begin{tikzpicture}[baseline=0cm,every node/.style={inner sep=2pt}]
\draw[wall,fill=gray!20] (0,0) circle[radius=1];
\draw[fill=white] (0,1) circle[radius=0.1] (0,-1) circle[radius=0.1];
\path (-135:1)node[below left]{$e_l$} (-45:1)node[below right]{$e_r$};
\path (0,-1.5)node{(a)};

\begin{scope}[xshift=3.5cm]
\fill[gray!20] (-0.9,-1) rectangle (0.9,1);
\draw[wall] (-0.9,-1) -- (-0.9,1) (0.9,-1) -- (0.9,1);
\path (-0.9,-0.5)node[below left]{$e_l$} (0.9,-0.5)node[below right]{$e_r$};
\path (0,-1.5)node{(b)};
\end{scope}

\begin{scope}[xshift=7cm]
\fill[gray!20] (-0.9,-1) rectangle (0.9,1);
\draw[wall] (-0.9,-1) -- (-0.9,1) (0.9,-1) -- (0.9,1);
\draw[edge,-o-={0.5}{>}] (-0.9,0)node[left]{$i$} -- (0.9,0)node[right]{$j$};
\path (0,-1.5)node{(c)};
\end{scope}

\end{tikzpicture}
\caption{(a) \& (b) (based)
bigon $\PP_2$. \qquad (c) The stated arc in $\PP_2$}\label{fig-bigon}
\end{figure}
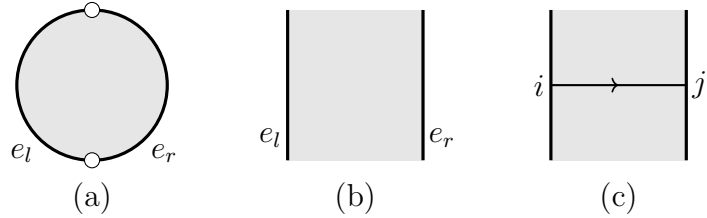

Framed oriented knots and framed $\cN$-arcs are special $n$-webs, which are important in our considerations.

\def\sign{\mathrm{sign}}

\subsection{Change of ground ring}
\label{ss.change-coef} For a commutative $\Zhq$-algebra $R$ define
$$ \cS(M,\cN;R)= \SMN \otimes_\Zhq  R,\quad
\cS(\fS;R)= \cS(\fS) \otimes_\Zhq  R.$$

When the $\Zhq$-algebra $R$ is clear, for any element $x\in \cS(M,\cN)$, we still use $x$ to denote the projection of $x$ in $\cS(M,\cN;R)$.

For a non-zero complex number $\hat\xi$, we have a $\Zhq$-algebra structure 
on $\mathbb C$ given by $\hat q\mapsto \hat\xi$; denote this $\mathbb{C}$ by $\mathbb{C}_{\hat\xi}$.  
We use $\cS_{\hat\xi}\MN$ (resp. $\cS_{\hat\xi}(\fS)$) to denote 
$\cS(M,\cN;\mathbb C_{\hat\xi})$ (resp. $\cS(\fS;\mathbb C_{\hat\xi})$).
For $x\in \SMN$ let $$[x]_{\hat\xi}\in \cS_{\hat\xi}\MN$$ be the image of $x$ under the natural map $\SMN \to \cS_{\hat\xi}\MN$.
When the complex number $\hat\xi$ is clear, we will use $x$ to denote $[x]_{\hat\xi}$.

\subsection{Functoriality}
\label{ss.functor} 
An embedding of marked 3-manifolds $f: \MN\embed (M', N')$ induces 
 a $\Zhq$-module homomorphism $$f_{\ast}=\cS(f): \SMN \to \cS(M', \cN'),$$ 
mapping each $n$-web $\alpha$ to $f(\alpha)$. 
This homomorphism depends only on the isotopy class of $f$.

Hence,  $\cS(\cdot )$ defines a functor from the category $\MM$ of marked $3$-manifolds to the category of $\Zhq$-modules.

\subsection{Embedding of punctured bordered surfaces} \label{sec.embed}

A proper embedding $f:\fS _1 \embed \fS _2$ of pb surfaces yields a $\Zhq$-linear map $f_\ast: \cS(\fS_1) \to \cS(\fS_2)$ as follows.  
Suppose that $\al$ is a stated $n$-web diagram in $\fS_1$ with negative order; see Definition \ref{def-diag} and the paragraph following it.  
Let $[\al]\in \cS(\fS_1)$ be the element determined by $\al$. Define $f_\ast([\al]) = [f(\al)] \in \cS(\fS_2)$, where $f(\al)$ is given the negative boundary order. Clearly, $f_\ast$ is a well-defined $R$-linear map and does not change under ambient isotopies of $f$. In general, $f_\ast$ is not an algebra homomorphism.  
If $f$ maps the set of the boundary components of $\fS_1$ injectively to the set of the boundary components of $\fS_2$, we say $f$ is a {\bf strict embedding}. If $f$ is a strict embedding, then  
$f_\ast: \cS(\fS_1) \to \cS(\fS_2)$ is an algebra homomorphism \cite{LS21}.

\def\cal{\mathcal}

\def\ori{{\mathrm{or}}}
\subsection{Mod $n$ grading for $\SMN$} \label{sec-grading}
We introduce a $\BZ_n$-grading on $\SMN$.

A $n$-web $\al$  defines a homology class  $[\al]_{\mathsf {homol}  }\in H_1(M, \cN;\BZ_n)$ as follows.  
First on each loop component of $\al$ choose a point and declare it a vertex. Now $\al$ is a directed graph, and as such it defines a 1-chain  in $(M, \cN)$ with integer coefficients. This 1-chain is closed mod $n$ and hence defines an element in $H_1(M, \cN;\BZ_n)$, which does not depend on the choice of the vertices on loop components.

For each defining relation of the stated ${\rm SL}_n$-skein module $\cS\MN$, all involved (stated) $n$-webs represent the same element in $H_1(M,\cN; \mathbb{Z}_n)$.
Hence we have a $\BZ_n$-grading 
\be 
 \SMN = \bigoplus _{x \in H_1(M, \cN;\BZ_n)  } \SMN_x,
 \label{eq-grading}
\ee
where $\SMN_x$ is spanned by elements represented by stated $n$-webs of homology class  $x$.

Clearly a morphism $f: \MN\to (M', \cN')$ respects the $\BZ_n$-grading, meaning that
\be 
\cS(f)  (\SMN_x  ) \subset  \cS(M', \cN')_{f_*(x) }.
\ee

\def\SS{\cS(\fS)}

For a pb surface $\fS$, the grading \eqref{eq-grading} is an algebra grading, meaning that
\be 
\SS_x\,  \SS_{x'}  \subset \SS _{x+ x'}.
\ee

\subsection{Essentially marked 3-manifolds}\label{subsec:essentially_marked_3-manifolds} 
A marked 3-manifold $\MN$ (\OldS\ref{ss.marked}) is called an {\bf essentially marked} 3-manifold if each connected component of $M$ intersects $\cN$.

\begin{lemma}\label{lem-span-arcs} \begin{enuma}
\item Assume that $\MN$ is an essentially marked 3-manifold. Then $\SMN$ is spanned by stated $n$-webs whose components consist only of stated framed $\cN$-arcs.

\item Assume that $\fS$ is an essentially bordered pb surface (Definition \ref{def.pb_surface}). Then $\cS(\fS)$ is spanned by stated $n$-web diagrams of the form $\alpha = \bigsqcup_{i=1}^r \alpha_i$, where each $\alpha_i$ is a stated $\partial \fS$-arc 
and
moreover $\alpha = \alpha_1 \alpha_2 \cdots \alpha_{r}$ as elements of $\cS(\fS)$. In particular, the algebra $\cS(\fS)$ is generated by stated $\pfS$-arcs.
\end{enuma}
\end{lemma}
\begin{proof}
(a) Using relation \eqref{e.vertexnearwall}, we can kill all sinks and sources. Then relation \eqref{e.capnearwall} helps to remove all knot components.

(b) By Lemma 5.1 of \cite{LS21}, the $R$-module $\cS(\fS)$ is spanned by stated $n$-web diagrams $\alpha$ without sinks, sources, nor crossings. Relation \eqref{e.capnearwall} helps to remove all knot components from $\alpha$, so that we can assume that each $\alpha$ is the disjoint union of stated arcs. Then relation \eqref{e.crossp-wall} helps to reorder the height order of boundary points, so that $\cS(\fS)$ is spanned by stated $n$-web diagrams $\alpha$ of the form $\alpha = 
\alpha_1 
\alpha_2 \cdots 
\alpha_
{r}$, where the $
\alpha_i$'s are disjoint stated arcs.
\end{proof}

\subsection{Reversing orientations of webs} 
An {\bf orientation} of a web
consists of orientations of all its loop components and directions of all its edges.
Let $\cev{\alpha}$ denote an $n$-web $\alpha$ with its orientation reversed (and unchanged framing). 
\begin{lemma}\cite{LS21,LY23}
   The orientation-reversal map $\cev {\,\cdot\,}: \cS\MN\to \cS\MN$
is a well-defined $\Zhq$-module automorphism.
In particular, it is a $\Zhq$-algebra automorphism when $\MN$ is the thickening of a pb surface.
\end{lemma}

\subsection{Reflection} \label{sec.reflection}

\def\refl{{\bf r}}

  A \term{$\Zhq$-algebra with reflection} is a $\Zhq$-algebra $A$ equipped with a $\BZ$-linear anti-involution $\refl$, called the \term{reflection}, such that $\refl(\qq)=\qq^{-1}$. In other words, $\refl: A \to A$ is a $\BZ$-linear map such that for all $x,y \in A$,
	\[
    \refl(xy)=\refl(y)\refl(x),\qquad 
    \refl(\qq x)=\qq^{-1} \refl(x),\qquad 
    \refl^2=\id.
    \]

\begin{proposition}[Theorem 4.9 of \cite{LS21}]\label{r.reflection}
	Assume that $\fS$ is a pb surface. There is a unique reflection $\refl: \SS \to \SS$ such that if $\al$ is a stated $n$-web diagram then $\refl(\al)$ is obtained from $\al$ by switching all the crossings and reversing the height order on each boundary edge. 
\end{proposition}

Let $\fS$ be an essentially bordered pb surface.
An element $x\in \cS(\fS)$ is \term{reflection-normalizable} if $\refl(x) = \hq^{2m} x$ for some $m\in\mathbb Z$. 
In this case, such $m$ is unique because $\cS(\fS)$ is a free module over $\Zhq$ (see \cite{LS21}).
We define the \term{reflection-normalization} by
\begin{equation}\label{eq.reflec}
	[x]_{\norm}:= \hq^{m}x.
\end{equation}
We have $
\refl([x]_{\norm})= [x]_{\norm}$, i.e. $[x]_{\norm}$ is reflection invariant. 
We will see in \OldS\ref{sub-frame} that
the reflection-normalization 
agrees with the Weyl-normalization of a Laurent monomial in a quantum torus algebra.

\subsection{Cutting homomorphisms}\label{subsec.cutting_homomorphism}

Let $(M,\mathcal{N})$ be any marked $3$-manifold, and $D$ be a properly embedded closed disk in $M$ such that there is no intersection between $D$ and the closure of $\cN$.
 After removing an open neighborhood of $D$,  which is diffeomorphic to $D\times (-1,1)$, we get a new 3-manifold $M'$ whose boundary contains two copies $D_1$ and $D_2$ of $D$
such that gluing $D_1$ to $D_2$ yields $M$ together with a surjective smooth map
$\text{pr}\colon M'\rightarrow M$.

Let $\beta\subset D$ be an embedded oriented open interval. Suppose that $\text{pr}^{-1}(\beta) = \beta_{1}\cup \beta_2$ with
   $\beta_1\in D_1$ and $\beta_2\in D_2$. We cut 
   $M$ along 
   $D$ to  obtain a new marked 3-manifold   $(M', \cN')$, where we put $\cN' = \cN\cup \beta_1\cup \beta_2$. We will  denote   $(M', \cN')$ as  $\mathsf{Cut}_{(D,\beta)}(M, \cN )$, which is defined up to isomorphism.

For a stated $n$-web $\alpha$ in $\MN$, we say that $\alpha$
 is {\bf $(D,\beta)$-transverse} if 
the vertices of $\alpha$ are not in $D$, $\alpha$ is transverse to $D$, $(\alpha\cap D)\subset \beta$, and the framing at every
point of $\alpha \cap D =\alpha\cap \beta$ is a positive tangent vector of $\beta$.
 For any map $s\colon\alpha\cap \beta\rightarrow\mathbb J=\{1,\ldots,n\}$, 
 define a stated $n$-web $\alpha_s$ in $\Cut_{(D, \beta)}(M,\mathcal{N})$ 
 to be the lift of $\alpha$ such that for every $P\in \alpha\cap\beta$ the two newly created boundary points of $\alpha_s$ corresponding to $P$ both  have the state $s(P)$. 
 
 \begin{theorem}[\cite{LS21}]\label{thm.cutting_homomorphism_for_3-manifolds}
Let $D$ be a closed disk properly embedded in a marked $3$-manifold $\MN$
and let $\beta$ be an oriented open interval embedded in $D$.  Then there is a unique $\Zhq$-module homomorphism
\begin{equation*}
\Theta_{(D,\beta)}\colon \cS\MN\rightarrow
\cS(\Cut_{(D, \beta)}(M,\mathcal{N}))
\end{equation*}
sending every $(D,\beta)$-transverse stated 
$n$-web $\alpha$ in $\MN$ to the sum of all of its lifts
$$
\Theta_{(D,\beta)}(\alpha) = \sum_{s\colon \alpha\cap \beta\rightarrow\mathbb J} 
\alpha_s.
$$
 \end{theorem}

We now present the cutting homomorphism for pb surfaces.

Let $e$ be an ideal arc lying in the interior of a pb surface $\fS$; see Definition \ref{def.pb_surface} and the paragraph following it. 
Let $\Cut_e(\fS)$ be a pb surface obtained from $\fS$ by cutting along $e$. In particular, there exist distinct boundary edges $e_1$ and $e_2$ of $\Cut_e(\fS)$ 
such that $\fS= \Cut_e(\fS)/(e_1=e_2)$, with $e=e_1=e_2$. Let 
$$p_e: \Cut_e(\fS) \to \fS$$ 
be the natural projection map.

An $n$-web diagram $\alpha$ in $\fS$ is \term{$e$-transverse} if the $n$-valent vertices of $\alpha$ are not in $e$ and $\alpha$ is transverse to $e$. Assume that $\alpha$ is an $e$-transverse stated $n$-web diagram in $\fS$. Let $h$ be a linear order on the set $\alpha \cap e$. 
For a map $s: \alpha \cap e\to \JJ$, let $\alpha_{h,s}$ be the stated $n$-web diagram in $\Cut_e(\fS)$ given by $p^{-1}_e(\alpha)$ with the height order on its points in $e_1 \cup e_2$ induced (via $p_e$) from $h$, and the states on its points in $e_1 \cup e_2$ induced (via $p_e$) from $s$.

\begin{theorem}\cite{LS21}\label{t.splitting2}
Suppose that $e$ is an interior ideal arc of a pb surface $\fS$. Then there is a unique $\Zhq$-algebra homomorphism 
$$\Theta_e: \skein(\fS) \to \skein(\Cut_e(\fS))$$ 
such that if $\alpha$ is an $e$-transverse stated $n$-web diagram in $\fS$ and $h$ is any linear order on $\alpha \cap e$, then
\begin{equation}\label{eq.cut00}
\Theta_e(\al) =\sum_{s: \alpha \cap e\to \JJ} \alpha_{h,s}.
\end{equation}
If in addition $\fS$ is essentially bordered (Definition \ref{def.pb_surface}), then $\Theta_e$ is injective.
\end{theorem}

If $e,e'$ are two disjoint ideal arcs lying in the interior of $\fS$, then we have $\Theta_e\Theta_{e'}= \Theta_{e'}\Theta_e$.

\subsection{Polygons}\label{subsec.polygons}
We will define polygons as special pb surfaces and explain the relationship between the bigon and $\Oq$, the quantized algebra of functions on $\SLn$ defined in Section \ref{sec.Oqsln}.

For a positive integer $k$, an \term{ideal $k$-gon}, or simply a \term{$k$-gon}, is a pb surface obtained as the result of removing $k$ points on the boundary of the standard closed disk.
A \term{based $k$-gon} is a $k$-gon with one distinguished ideal point, called the \term{based vertex}. Given two based $k$-gons, there is a unique, up to isotopy, orientation-preserving diffeomorphism between them, preserving the based vertices. In this sense, the based $k$-gon is uniquely determined as a pb surface, and we denote it by $\PP_k$.

Thus, $\PP_1$ is the (based) monogon. By \cite[Theorem 6.1]{LS21}, we have an isomorphism 
\begin{align}
\label{monogon_skein_algebra}
    R\cong \skein(\PP_1),
\end{align}
given by $x \mapsto x\cdot \emptyset$. We will often identify $\skein(\poly_1)\equiv R$.

For the based bigon $\mathbb{P}_2$, we have:

\bthm [{\cite[Theorem 6.3]{LS21}}]\label{thm-iso-Oq-P2}
There is a unique $\Zhq$-algebra isomorphism  
$$\Oq \xrightarrow{\cong}\cS(\PP_2)$$ 
sending $u_{ij}$ (in Definition \ref{def.OqSLn}) to the stated arc in Figure \ref{fig-bigon}(c) for all $i,j \in \JJ=\{1,\ldots,n\}$.
\ethm 

We will identify $\Oq$ with $\cS(\PP_2)$ via the above isomorphism.

Then the Hopf algebra structure of $\Oq$ induces a Hopf algebra structure on $\cS(\PP_2)$. We still use $\Delta$, $\varepsilon$, and $S$ to denote the coproduct, the counit, and the antipode of $\cS(\PP_2)$, respectively. The algebra homomorphisms $\Delta$, $\varepsilon$, and the anti-algebra homomorphism $S$ are given by \cite{LS21}
\begin{align}\label{eq-Hopf-bigon}
\Delta(u_{ij}) = \sum_{k} u_{ik}\otimes u_{kj}, \quad \varepsilon(u_{ij}) = \delta_{ij}, \quad S(u_{ij}) = (-q)^{i-j}\cev{u}_{\bj \bi},
\end{align}
where $\cev{u}_{ij}$ is $u_{ij}$ with the reverse orientation. 
Note that the cutting homomorphism $\Theta_e$ for an internal ideal arc $e$ in $\PP_2$ connecting the two punctures matches the coproduct $\Delta$; see \cite{LS21}.

\begin{figure}
	\centering
	\begin{tikzpicture}
	\draw[line width =1.2pt,fill=gray!20] (0,0)--(4,0)--(2,3.4)--cycle;
	\draw[fill=white] (0,0) circle[radius=0.1] ;
	\draw[fill=white] (4,0) circle[radius=0.1] ;
	\draw[fill=white] (2,3.4) circle[radius=0.1];
	\draw[edge,-o-={0.5}{>}] (0.75,1.25)node[left]{$i_1$} -- (1.5,0)node[below]{$j_1$};
	\draw[edge,-o-={0.5}{>}] (2.5,0)node[below]{$i_2$} -- (3.25,1.25)node[right]{$j_2$};
	\draw[edge,-o-={0.5}{>}] (1.5,2.5)node[left]{$i_3$} -- (2.5,2.5)node[right]{$j_3$};
	\draw[wall,<-] (3.5,0) -- (3,0);
	\node [right]  at (1.125,0.75) {$b$} ;
	\node [left]  at (2.825,0.75) {$c$} ;
	\node [below]  at (2,2.5) {$d$} ;
\end{tikzpicture}
	\caption{Oriented arcs $b,c,d$ in $\PP_3$. We use $b_{i_1j_1}$ to denote the stated arc in $\PP_3$ as shown in the picture (the same for $c_{i_2j_2}$ and $d_{i_3j_3}$).}\label{P3}
\end{figure}
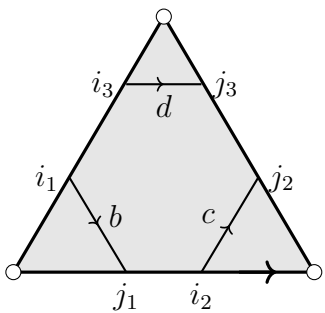

In Figure \ref{P3}, there are three oriented arcs in the based triangle $\mathbb P_3$; the based vertex is the top one.
An open tubular neighborhood $N(b)$ of $b$ is homeomorphic to the based bigon $\mathbb P_2$, where the left edge of $N(b)$ is
the one containing the beginning point of $b$ (determined by the direction of $b$). Note that $b$ in this bigon $N(b) \cong \mathbb{P}_2$ then looks like the arc in Figure \ref{fig-bigon}(c). Let $\cS(b) := \cS(N(b))$,
which is isomorphic to $\mathcal{O}_{q}(\text{SL}_n)$ by Theorem \ref{thm-iso-Oq-P2}. The embedding 
\begin{align}
\label{P2_to_P3}
    N(b) \hookrightarrow \mathbb{P}_3
\end{align} 
induces an algebra homomorphism
$\iota^* : \cS(b) \to \cS(\mathbb P_3)$.

\begin{lemma}[{\cite[Corollary 8.14]{LS21}}]\label{lem_P3}
	The $\Zhq$-linear map $\cS(b) \otimes_{\Zhq} \cS(c) \to \cS(\mathbb P_3)$ given by $x \otimes y \mapsto \iota^*(x)\iota^*(y)$ is bijective.
	
	In particular, the embedding $\PP_2 \embed \PP_3$, induced by one of the oriented arcs $b$, $c$, or $d$ analogously to \eqref{P2_to_P3}, induces an injective algebra homomorphism $\cS(\PP_2;R) \embed \cS(\PP_3; R)$ for any commutative $\Zhq$-algebra $R$. 
\end{lemma}

\subsection{Reduced stated $\SLn$-skein algebras}\label{ss.reduced_stated_SLn-skein_algebras}

Let $v$ be an ideal point of a pb surface $\fS$ that is not an interior puncture; see Definition \ref{def.pb_surface}.  
The \term{corner arcs} $C(v)_{ij}$ and $\ceC(v)_{ij}$, where $i,j \in \JJ$, are special stated $n$-web diagrams in $\fS$ as depicted in Figure \ref{fig-corner}. We also denote by $C(v)$ (resp. $\ceC(v)$) the arcs $C(v)_{ij}$ (resp. $\ceC(v)_{ij}$) without states. Note that $C(v)$ is a trivial $\pfS$-arc if and only if $v$ is a \textbf{monogon vertex}, i.e., if the connected component of $\fS$ containing $v$ is a monogon $\mathbb{P}_1$, having $v$ as its only ideal point.

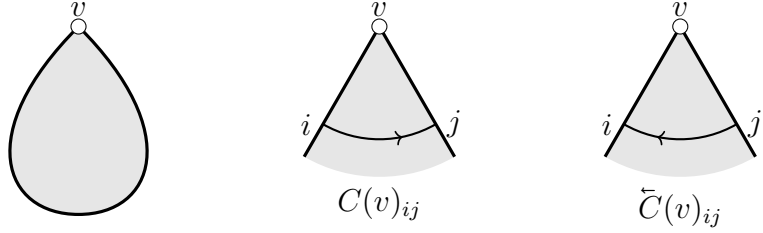
\begin{figure}
	\centering
	\begin{tikzpicture}
\draw[wall,fill=gray!20] (0,0) ..controls (-135:2) and (-1,-2.5)..
	(0,-2.5) ..controls (1,-2.5) and (-45:2).. (0,0);
\draw[fill=white] (0,0)circle(0.1) node[above]{$v$};

\begin{scope}[xshift=4cm]
\fill[gray!20] (0,0) -- (-120:2)
	arc[radius=2,start angle=-120,end angle=-60] -- cycle;
\draw[wall] (-120:2) -- (0,0) -- (-60:2);
\draw[edge,-o-={0.7}{>}] (-120:1.5) node[left]{$i$}
	arc[radius=1.5,start angle=-120,end angle=-60] node[right]{$j$};
\draw (0,-2)node[below]{$C(v)_{ij}$};
\draw[fill=white] (0,0)circle(0.1) node[above]{$v$};
\end{scope}

\begin{scope}[xshift=8cm]
\fill[gray!20] (0,0) -- (-120:2)
	arc[radius=2,start angle=-120,end angle=-60] -- cycle;
\draw[wall] (-120:2) -- (0,0) -- (-60:2);
\draw[edge,-o-={0.7}{>}] (-60:1.5) node[right]{$j$}
	arc[radius=1.5,start angle=-60,end angle=-120] node[left]{$i$};
\draw (0,-2)node[below]{$\cev{C}(v)_{ij}$};
\draw[fill=white] (0,0)circle(0.1) node[above]{$v$};
\end{scope}
\end{tikzpicture}
	\caption{Monogon vertex and corner arcs}\label{fig-corner}
\end{figure}

For an ideal point $v$ that is neither an interior puncture nor a monogon vertex, let
\[C_v = \{C(v)_{ij} \mid i <j \} , \quad \ceC_v = \{\ceC(v)_{ij} \mid i <j\}.\]
If $v$ is a monogon vertex, let $C_v = \ceC_v = \emptyset$.
An element of $C_v$ or $\ceC_v$ is called a \term{bad arc} at $v$. Let $\Ibad \lhd \SS$ be the two-sided ideal generated by all bad arcs. The quotient algebra
\begin{align}
    \label{reduced_stated_skein_algebra}
    \overline{\cS}(\fS)=\overline{\cS}_{\hat q}(\fS):= \cS(\fS)/\Ibad
\end{align}
is called the \term{reduced stated $\SLn$-skein algebra} of $\fS$.

Suppose that $
e$ is an ideal arc in $\fS$. It is easy to see that the cutting homomorphism $\Theta_
{e}: \skein(\fS) \to \skein(\Cut_
{e}(\fS))$ reduces to an algebra homomorphism  
$\overline \Theta_
{e}: \overline{\cS}(\fS) \to \overline{\cS}(\Cut_
{e}(\fS))$, called the cutting homomorphism for the reduced stated $\SLn$-skein algebras.

\subsection{Quantum tori and Quantum torus frame}\label{sub-frame}

For an integer antisymmetric $r\times r$ matrix $Q=(Q_{ij})_{i,j\in\{1,\ldots,r\}}$, the {\bf quantum torus} is defined as the associative algebra  
\be
\bT(Q)=\bT_\hq(Q):= \Zhq\la x_1^{\pm1}, \dots, x_r^{\pm 1}\ra/( x_i x_i = \hq^{2 Q_{ij}} x_j x_i)
\label{eqTorus}
\ee
For $\bk=(k_1, \dots, k_r)\in \BZ^r$, let the {\bf Weyl-normalized} Laurent monomial $x^\bk$ be
\begin{align}\label{Weyl-ordering}
    x^\bk=[x_1^{k_1} \dots x_r^{k_r}] := \hq^{- \sum_{i<j} Q_{ij} k_i k_j} x_1^{k_1} \dots x_r^{k_r}.
\end{align} 
Then $\{ x^\bk \mid \bk \in \BZ^r\}$ is a free $\Zhq$-basis of $\bT_\hq(Q)$.

\begin{remark}\label{rem.Weyl-ordering} We say that the elements $a_1,\ldots,a_r$ of a $\Zhq$-algebra $A$   are {\bf $\hq$-commuting} if there are integers $m_{ij}$ such that $a_i a_j = \hq^{2m_{ij}} a_ja_i$ for all $i,j\in\{1,\ldots,r\}$. Assume further that $A$ is torsion free over $\Zhq$. Then the numbers $m_{ij}$ are unique, and we define  the Weyl-normalized product
    $$
    [a_1 \cdots a_r]_{\rm Weyl} := \hq^{- \sum_{i<j}  m_{ij}} a_1 \dots a_r.
    $$
    An example is the above $x^{\bf k}$ in $\mathbb{T}(Q)$.

    One minor remark is that what we call $\hq$-commuting is written as `$q$-commuting' in \cite{LY23}.
\end{remark}

For a non-zero complex number $\home$ define $\bT_\home(Q)= \bT_\hq(Q) \ot_\Zhq \BC$, where $\BC$ is considered as a $\Zhq$-module by $\hq \mapsto \home$. For $\heta = \home ^{N^2}$ the $\BC$-linear map
\be 
\Phi^\bT: \bT_\heta(Q) \to \bT_\home(Q), \qquad  \Phi^\bT(x^\bk) = x^{N \bk}, \label{eqFrDef}
\ee
is an algebra monomorphism, called the Frobenius homomorphism 
for quantum tori.

Assume that $A$ is an $R$-domain and $S\subset A\setminus\{0\}$ is a subset of non-zero elements. 
We say that $S$ \term{weakly generates} $A$ if we have the following: 
for any element $a\in A$, there exists a product $\fm$ of elements in $S$ such that $a\fm$ is contained in the subalgebra of $A$ generated by $S$. The following notion, studied in \cite{LY23}, will become useful:

\begin{definition}\label{def.qtf}
    Let $A$ be an $R$-domain. A finite set $S=\{a_1, \dots, a_r\} \subset A$ is a {\bf quantum torus frame} for $A$ if the following conditions are satisfied:
    \begin{enumerate}
        \item the elements of the set $S$ are $\hq$-commuting and each element of $S$ is non-zero,
        \item the set $S$ weakly generates $A$,
        \item the set $\{a_1 ^{n_1} \dots a_r^{n_r} \mid n_i \in \BN \}$ is $R$-linearly independent.
    \end{enumerate}
\end{definition}

\subsection{$\SLn$-quantum trace maps}\label{subsec.SLn_quantum_trace_maps}

In this subsection we review the ${\rm SL}_n$-quantum trace maps constructed in \cite{LY23}. Certain compatibility regarding the ${\rm SL}_n$-quantum trace maps will constitute part of the main results of the present paper.

\begin{definition}\label{def.triangulation}
    A pb surface is called {\bf triangulable} if each connected component contains an ideal point and falls into none of the following: a sphere with less than three punctures, $\mathbb P_1$, or $\mathbb P_2$. An {\bf ideal triangulation} of a triangulable pb surface is a collection of ideal arcs with the following properties: 
    \begin{enumerate}[label={\rm (\arabic*)}]
        \item any two of these ideal arcs are disjoint and non-isotopic,
        \item this collection of ideal arcs is maximal under condition (1).
    \end{enumerate}
\end{definition}

We may realize $\mathbb{P}_3$ as the following standard ideal triangle:
\begin{equation}
    \stdT=\{(i,j,k)\in\reals^3\mid i,j,k\ge0,i+j+k=n\}\setminus\{(0,0,n),(0,n,0),(n,0,0)\},
\end{equation}
where $(n,0,0)$ is the based vertex. Here $(i,j,k)$ (or $(ijk)$ for brevity) are the barycentric coordinates. Let $v_1=(n,0,0)$, $v_2=(0,n,0)$, $v_3=(0,0,n)$. The edge following $v_i$ in the clockwise orientation is denoted $e_i$. We will draw $\stdT$ in the standard plane as an equilateral triangle with $v_1$ at the top. See the left picture in Figure~\ref{fig-coords} for an example.

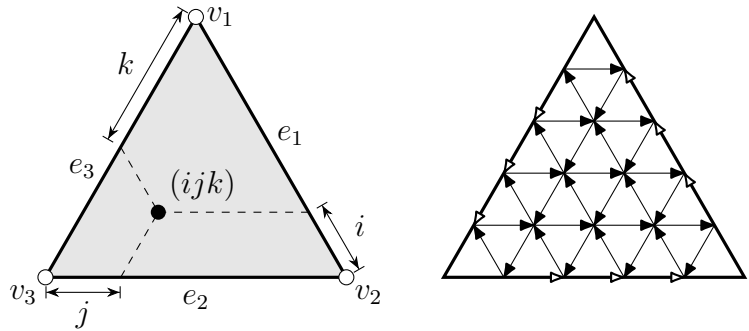
\begin{figure}
    \centering
    \begin{tikzpicture}[baseline=0cm]
\draw[wall,fill=gray!20] (0,0) -- (4,0) -- (60:4) -- cycle;
\draw[dashed] (1,0) -- ++(60:1)coordinate(ijk) node[above right]{$(ijk)$}
	-- ++(120:1) (ijk) -- ++(2,0);
\fill[black] (ijk)circle[radius=0.1];
\begin{scope}[{Bar[]Stealth[]}-{Stealth[]Bar[]}]
\draw (0,0)++(0,-0.2) -- ++(1,0);
\draw (4,0)++(30:0.2) -- ++(120:1);
\draw (60:4)++(150:0.2) -- ++(-120:2);
\end{scope}
\path (0.5,-0.5)node{$j$} (4,0)++(120:0.5)++(30:0.5)node{$i$}
	(60:3)++(150:0.5)node{$k$};
\path (2,-0.3)node{$e_2$} (4,0)++(120:2)++(30:0.3)node{$e_1$}
	(60:1.5)++(150:0.3)node{$e_3$};

\draw[fill=white] (60:4)circle(0.1)node[right]{$v_1$}
	(4,0)circle(0.1)node[below right,inner sep=3pt]{$v_2$}
	(0,0)circle(0.1)node[below left,inner sep=3pt]{$v_3$};
\end{tikzpicture}
    \quad
    \begin{tikzpicture}[baseline=0cm]
\draw[wall] (0,0) -- (4,0) -- (60:4) -- cycle;

\foreach \x in {0.8,1.6,2.4,3.2}
\draw (4,0)++(120:\x) edge (4-\x,0) -- (60:\x) -- (\x,0);

\begin{scope}[thick,tips,{Latex[round]}-]
\foreach \i in {0,...,3} {
\foreach \j in {\i,...,3} {
\path[tips=false] ({0.8*(\j-\i+1)},0)++(60:{0.8*\i}) coordinate (A)
	++(60:0.8) coordinate (B) ++(-0.8,0) coordinate (C);
\path (A) -- (B);
\path (B) -- (C);
\path (C) -- (A);
}}
\end{scope}

\begin{scope}[thick,tips,{Latex[round,fill=white]}-]
\foreach \i in {1,...,3} {
\path (60:{0.8*\i}) -- ++(60:0.8);
\path (60:4) ++ (-60:{0.8*\i}) -- ++(-60:0.8);
\path ({4-0.8*\i},0) -- ++(-0.8,0);
}
\end{scope}
\end{tikzpicture}
    \caption{Barycentric coordinates $(ijk)$ and a $5$-triangulation with its quiver}\label{fig-coords}
\end{figure}

The \term{$n$-triangulation} of $\stdT$ is obtained by subdividing $\stdT$ into $n^2$ small triangles using lines $i,j,k=\text{constant integers}$. An example of a $5$-triangulation is shown in the right picture in Figure~\ref{fig-coords}.

The vertices and edges of all small triangles, except for the vertices of $\stdT$ and the small edges adjacent to them, form a directed graph (or quiver) $\Gamma_\stdT$. Here the direction of a small edge, also called an \term{arrow}, is defined as follows. If the small edge $u$ is in the boundary $\partial\stdT$ then $u$ has the positive (or counterclockwise) orientation of $\partial \stdT$. If $u$ is interior then the direction of $u$ is the same as that of a boundary edge parallel to $u$. Assign weight $1$ to any boundary arrow and weight $2$ to any interior arrow.

The vertex set $\rd{V}_\stdT$ of $\Gamma_\stdT$ is the set of points with integer barycentric coordinates:
\begin{equation}
\label{ol_V_P3}
    \rd{V}_\stdT=\{(ijk)\in\stdT\mid i,j,k\in\ints\}.
\end{equation}
Elements of $\rdV_\stdT$ are called \term{small vertices}, and small vertices on the boundary of $\rd{V}_\stdT$ are called the \term{edge-vertices}.

The triangle $\stdT$ has a unique triangulation consisting of the 3 boundary edges, up to isotopy. 

Fix an ideal triangulation $\lambda$ of $\fS$. An element of $\lambda$ is called a \term{boundary} edge if it is isotopic to a boundary edge of $\fS$. By cutting $\fS$ along all non-boundary edges we get a disjoint union of ideal triangles, each called a \term{face} (or, a triangle) of the triangulation. Let $\face_\lambda$ denote the set of all faces of $\lambda$. Then
\begin{equation}\label{eq.glue}
\fS = \Big( \bigsqcup_{\tau\in\face_\lambda} \tau \Big) /\sim,
\end{equation}
where each face $\tau$ is a copy of $\stdT$, and $\sim$ is the identification of certain pairs of edges of the faces. Note that one might glue two edges of the same face. Each face $\tau$ comes with a \term{characteristic map} $f_\tau: \tau \to \fS$, which is a homeomorphism when restricted to the interior of $\tau$ or the interior of each edge of $\tau$.

An \term{$n$-triangulation} of $\lambda$ is a collection of $n$-triangulations of the faces $\tau$ which are {\bf compatible} with the gluing $\sim$, meaning that whenever an edge $b$ is glued to another edge $b'$, the edge-vertices on $b$ are glued to the edge-vertices on $b'$ \cite{FG,GS2}. Then define the \term{reduced vertex set}
\begin{align}
    \label{eq.rdV}
\rd{V}_\lambda=\bigcup_{\tau\in\face_\lambda}\rd{V}_\tau, \qquad \rd{V}_\tau=f_\tau(\rd{V}_\stdT).
\end{align}
The images of the weighted quivers $\Gamma_\tau$ under $f_\tau$ together form a quiver $\Gamma_\lambda$ on $\fS$ via gluing. Let $\bmQ_\lambda: \rd{V}_\lambda\times \rd{V}_\lambda \to \BZ$ be the signed adjacency matrix of the weighted quiver $\Gamma_\lambda$, i.e. 
\begin{align}
\label{eq.bmQ}
\bmQ_\lambda(v,v') & = a_{\Gamma_\lambda}(v,v') - a_{\Gamma_\lambda}(v',v), \quad \mbox{where} \\
\nonumber
a_{\Gamma_\lambda} (v,v') & = \mbox{sum of weights of the arrows from $v$ to $v'$.}
\end{align}
What matters about the quiver $\Gamma_\lambda$ is only the matrix $\bmQ_\lambda$. One usually tidy up the quiver $\Gamma_\lambda$ by canceling or adding the weighted arrows, while keeping $\bmQ_\lambda$ unchanged. For example, when edges $b$ and $b'$ are glued, we may cancel the boundary arrows on $b$ by the boundary arrows on $b'$ after gluing.

The ($n$-th root version of the) \term{Fock-Goncharov algebra} \cite{BZ,FG2,FG2,GS2} is the quantum torus 
for $\bmQ_\lambda$:
\begin{equation}
\label{eq.reduced_FG_algebra}
\bXS= \bT(\bmQ_\lambda) = \Zhq \la x_v^{\pm 1}, v \in \rd{V}_\lambda \ra / (x_v x_{v'}= \hq^{\, 2 \bmQ_\lambda(v,v')} x_{v'}x_v \ \text{for } v,v'\in \rd{V}_\lambda ).
\end{equation}

The overlines in $\overline{V}_\lambda$, $\bmQ_\lambda$ and $\bXS$ refer to the `non-extended' versions in the language of \cite{LY23}, which pertains to the reduced stated ${\rm SL}_n$-skein algebra $\overline{\cS}(\fS)$ in \eqref{reduced_stated_skein_algebra} and which is in fact relevant to the original setting of Fock, Goncharov and Shen \cite{FG,FG2,FG3,GS2}. We will soon get to the statements about some versions of the so-called ${\rm SL}_n$-quantum trace homomorphisms, where the reduced version in particular maps $\overline{\cS}(\fS)$ to $\bXS$.

For the `extended' version studied in \cite{LY23}, attach a copy of $\stdT$ to each boundary edge of $\surface$. The resulting surface is denoted $\ext{\surface}$. We adopt the convention that in an attached triangle, the attaching edge is the $e_1$ edge; so $e_1$ is not a boundary edge of $\fS^*$. See Figure~\ref{fig-attach-T}.

\begin{figure}
\centering
\begin{tikzpicture}[baseline=0cm]
\fill[gray!20] (-60:0.5) -- (120:3.5) -- ++(30:1) -- ++(-60:4);
\draw[wall] (-60:0.5) -- (120:3.5);
\draw[wall] (-120:0.5) -- (0,0) -- (-3,0) -- (120:3) -- ++(-0.5,0);

\draw (120:1.5) ++(30:0.5)node{$\surface$};
\draw (150:{1.5/sqrt(3)*2})node{$\stdT$};
\draw (-1.5,0)node[below]{$e_2$} ++(120:1.5)node[left]{$e_3$};

\draw[fill=white] (0,0)circle[radius=0.1] (120:3)circle[radius=0.1];
\end{tikzpicture}
\caption{Attaching triangles to boundary edges of $\fS$.}\label{fig-attach-T}
\end{figure}
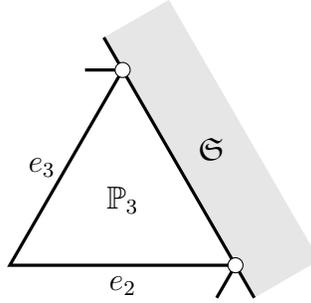

If the surface $\surface$ has an ideal triangulation $\lambda$, then there is an unique extension of $\lambda$ to an ideal triangulation $\lambda^*$ of $\ext{\surface}$ by adding all the new boundary edges. The new faces are exactly the attached triangles. Let $\rd{V}_{\ext{\lambda}}$ be the reduced vertex set of the extended $n$-triangulation. Define the \term{$X$-vertex set} $V_\lambda\subset\rd{V}_{\ext{\lambda}}$ as
\begin{align}
    \label{X-vertex_set}
    V_\lambda= \{\mbox{small vertices of $\overline{V}_{\lambda^*}$ not lying on the $e_3$ edge in the attached triangles}\},
\end{align}
and the \term{$A$-vertex set} $\lv{V}_\lambda\subset\rd{V}_{\ext{\lambda}}$ as
\begin{align}
    \label{A-vertex_set}
    V'_\lambda = \{\mbox{small vertices of $\overline{V}_{\lambda^*}$ not lying on the $e_2$ edge in the attached triangles}\}.
\end{align}
Note that $\rd{V}_\lambda$ is naturally a subset of both $V_\lambda$ and $\lv{V}_\lambda$.

Let $\mQ_\lambda: V_\lambda \times V_\lambda \to \BZ$ be the restriction of $ \bmQ_{\ext{\lambda}}: \rd{V}_{\ext{\lambda}} \times \rd{V}_{\ext{\lambda}} \to \BZ$. The \term{extended $X$-algebra} is defined as
\[\XS = \bT(\mQ_\lambda).\]
There is a natural identification of subalgebras $\rd{\FG}(\surface,\lambda)\subset\FG(\surface,\lambda)\subset\rd{\FG}(\ext{\surface},\ext{\lambda})$.

We need one more prerequisite before recalling the ${\rm SL}_n$ quantum trace map, namely the cutting homomorphism for (reduced) Fock-Goncharov algebras. Let $e$ be a non-boundary edge 
in $\lambda$.
Recall from \OldS\ref{subsec.cutting_homomorphism} that $\Cut_e(\fS)$ denotes the pb surface obtained from $\fS$ by cutting along $e$, where $p_e : \Cut_e(\fS) \to \fS$ is the natural projection, with $p_e^{-1}(e) = e_1 \cup e_2$. Then $\lambda_e = (\lambda \setminus\{e\})\cup \{e_1,e_2\}$ is a triangulation of $\Cut_e(\fS)$; we say that $\lambda_e$ is induced from $\lambda$.
There is 
a natural algebra embedding \cite{LY23}
\begin{align}\label{eq-cutting}
    \overline \Theta_e^{X}: \overline\cX(\fS,\lambda)\rightarrow \overline\cX(
    \Cut_e(\fS), \lambda_e),
\end{align}
the cutting homomorphism for Fock-Goncharov algebras, 
given on the generators by
$$
\overline \Theta_e^{X}(x_v)=
\begin{cases}
    x_v & \mbox{if } v\notin e,\\
    [x_{v'}x_{v''}] & \mbox{if } v\in e\text{ and $
    p_e^{-1}(v)= \{v',v''\}$},
\end{cases}
$$
for any $v\in \rd{V}_\lambda$, where $[\sim]$ is the Weyl-normalization \eqref{Weyl-ordering}.

We now recall the `$X$-version' ${\rm SL}_n$-quantum trace maps:
\begin{theorem}
\cite{LY23}\label{thm-X}
	Assume that $\fS$ is a triangulable 
    pb surface with an ideal triangulation $\lambda$ (Definition \ref{def.triangulation}). 

     \begin{enuma}
     \item There exists an algebra homomorphism
	\begin{equation*}
		\rdtr_\lambda^X\colon \overline{\cS}(\fS) \to \overline \cX(\fS,\lambda)
	\end{equation*}
	such that $\rdtr_\lambda^X$ is compatible with cutting along 
    any non-boundary edge $e$ of $\lambda$; that is, the following diagram commutes (see Theorem \ref{t.splitting2} for $\Theta_e$):
    $$
    \xymatrix{
    \overline{\cS}(\fS) \ar[r]^-{\rdtr_\lambda^X} \ar[d]_{\Theta_e} & \cX(\fS,\lambda) \ar[d]^{\overline{\Theta}_e^X} \\
    \overline{\cS}(\Cut_e(\fS)) \ar[r]^-{\rdtr_{\lambda_e}^X} & \cX(\Cut_e(\fS),\lambda_e).}
    $$
    
    For $n=2,3$ the map $\rdtr_\lambda^X$ is always injective, and for $n>3$, it is injective when $\fS$ is a polygon.

    \item There exists an algebra homomorphism
	\begin{equation*}
		\tr_\lambda^X\colon\cS(\fS) \to  \cX(\fS,\lambda)
	\end{equation*}
  lifting the reduced trace $\rdtr_\lambda^X$ in the following  sense. The image of $\tr_\lambda^X$ is in a subalgebra $\cX'(\fS,\lambda) \subset \cX(\fS,\lambda)$ which comes with a surjective algebra homomorphism $\pr: \cX'(\fS,\lambda) \onto \overline \cX(\fS,\lambda) $, and the following diagram is commutative.
    \begin{equation}\label{eq-tr-lift}
\begin{tikzcd}
\skein(\surface)\arrow[r,"\tr_\lambda^X"]\arrow[d,two heads,"\pr"]
& \cX'(\fS,\lambda)\arrow[d,two heads,"\pr"]\\
\reduceS(\surface)\arrow[r,"\rdtr_\lambda^X"]&\overline \cX(\fS,\lambda)
\end{tikzcd}
\end{equation}

For $n=2$ the map $\tr_\lambda^X$ is  injective, and for $n>2$ it is injective when $\fS$ has no interior punctures.
    \end{enuma}
\end{theorem}
We call $\rdtr_\lambda^X$ the reduced $X$-version ${\rm SL}_n$-quantum trace map, and $\tr_\lambda^X$ the extended $X$-version ${\rm SL}_n$-quantum trace map.
\begin{remark}
   For ${\rm SL}_2$  the reduced $X$-version quantum trace map was first defined in \cite{BW} and the other versions were defined in \cite{LY2,Le:triangulation}. The ${\rm SL}_3$-quantum trace map was studied in \cite{douglas2021quantum,kim2020rm}. For two different triangulations $\lambda$ and $\lambda'$, the quantum trace maps $\rdtr_\lambda^X$ and $\rdtr_{\lambda'}^X$ are compatible with each other via an isomorphism between the skew fields of $\overline \cX(\fS,\lambda)$ and $\overline \cX(\fS,\lambda')$, see \cite{BW,LY2} for $n=2$,
    \cite{kim2024naturality} for $n=3$, and \cite{LY23} for general $n$. It is shown in \cite{kim2024naturality} (for $n=3$) and \cite{kimwang2024naturality} (for general $n$) that the isomorphism between the skew fields of $\overline \cX(\fS,\lambda)$ and $\overline \cX(\fS,\lambda')$ is given by a sequence of quantum cluster $X$-mutations, when $\lambda$ and $\lambda'$ involve no `self-folded' triangles.
\end{remark}

Dealing with the $X$-version quantum trace maps is often technically quite challenging. We move on to recall the $A$-version, which is more tractable.

Let $\lambda$ be an ideal triangulation of a triangulable pb surface $\fS$. Define the reduced matrix $\rdm{H}_\lambda:\rd{V}_\lambda\times\rd{V}_\lambda\to\ints$ 
as follows:
\begin{itemize}
\item If $v$ and $v'$ are not on the same boundary edge then $\rdm{H}_\lambda(v,v')=-\frac{1}{2}\rdm{Q}_\lambda(v,v')\in\ints$.
\item If $v$ and $v'$ are on the same boundary edge, then
\begin{align}\label{def-overline-H}
    \rdm{H}_\lambda(v,v') = \begin{cases} 1 \qquad &\text{when $v=v'$},\\
-1 &\text{when there is arrow from $v$ to $v'$}, \\
0 &\text{otherwise}.
\end{cases}
\end{align}
\end{itemize}

 Define $\mat{H}_\lambda$ as the restriction of $\exm{H}$ to $V_\lambda\times\lv{V}_\lambda$, which agrees with the restriction of $-\frac{1}{2}\exm{Q}$ since the domain of $\mat{H}_\lambda$ does not contain pairs $(v,v')$ on the same boundary edge of $\ext{\surface}$.

\begin{lemma}\cite[Lemma 11.9]{LY23}\label{lem-matrix-HK}
    We have the following:

    \begin{enuma}
    \item There exists a unique matrix  $\overline{\mathsf{K}}_\lambda\colon \rd{V}_\lambda\times\rd{V}_\lambda\to\ints$
    such that $\overline{\mathsf{H}}_\lambda \overline{\mathsf{K}}_\lambda = n\,{\rm id}$.

    \item There exists a unique matrix  $\mathsf{K}_\lambda\colon V_\lambda'\times V_\lambda\to\ints$
    such that $\mathsf{H}_\lambda \mathsf{K}_\lambda = n\,{\rm id}$.
    \end{enuma}
\end{lemma}

Define the matrices $\overline{\mathsf{P}}_\lambda\colon \rd{V}_\lambda\times\rd{V}_\lambda\to\ints$
and 
$\mathsf{P}_\lambda\colon V_\lambda'\times V_\lambda'\to\ints$ by 
\begin{align}\label{eq-relation-P-Q}
    \overline{\mathsf{P}}_\lambda:=
    \overline{\mathsf{K}}_\lambda
    \overline{\mathsf{Q}}_\lambda
    \overline{\mathsf{K}}_\lambda^{\rm t},\quad
    \mathsf{P}_\lambda=
    \mathsf{K}_\lambda
    \mathsf{Q}_\lambda
    \mathsf{K}_\lambda^{\rm t},
\end{align}
where the superscript ${\rm t}$ stands for the transpose. 
Define the following {\bf $A$-version quantum tori}:
\begin{align*}
\overline{\mathcal A}(\fS,\lambda):=
    \mathbb T(\overline{\mathsf{P}}_\lambda)
    =\Zhq \la a_v^{\pm 1}, v \in \rd{V}_\lambda \ra / (a_v a_{v'}= \hq^{\, 2 \bmP_\lambda(v,v')} a_{v'}a_v \ \text{for } v,v'\in \rd{V}_\lambda ),
    \\
    {\mathcal A}(\fS,\lambda):=
    \mathbb T(\mathsf{P}_\lambda)=
    \Zhq \la a_v^{\pm 1}, a \in V_\lambda' \ra / (a_v a_{v'}= \hq^{\, 2 \mathsf{P}_\lambda(v,v')} a_{v'}a_v \ \text{for } v,v'\in V_\lambda').
\end{align*}

Lemma \ref{lem-matrix-HK} and 
equation \eqref{eq-relation-P-Q} imply the following lemma.
\begin{lemma}[{\cite[Theorem 11.7]{LY23}}]\label{lem.LY23_A_to_X}
    There are $\Zhq$-algebra embeddings 
    \begin{align*}
        &\rd\psi_\lambda\colon\overline \cA(\fS,\lambda)\rightarrow \overline \cX (\fS,\lambda),\quad 
        a^{\bf k}\mapsto x^{{\bf k} \overline{\mathsf{K}}}
        \text{ for }{\bf k}\in\mathbb Z^{\rd{V}_\lambda},\\
        &\psi_\lambda\colon  \cA(\fS,\lambda)\rightarrow  \cX (\fS,\lambda),\quad 
        a^{\bf k}\mapsto x^{{\bf k} \mathsf{K}_\lambda}
        \text{ for }{\bf k}\in\mathbb Z^{V_\lambda},
    \end{align*}
    where each ${\bf k}$ is viewed as a row vector.
\end{lemma}

\begin{theorem}
[\cite{LY2} for $n=2$; \cite{LY23} for $n\ge 2$]\label{thm-Atr}
	Assume that $\surface$ is a triangulable 
    pb surface with no interior puncture, and $\lambda$ is an ideal triangulation of $\surface$.
	\begin{enuma}
            \item There is a unique algebra homomorphism
		\begin{equation}\label{eq.bAtr}
			\rdtr_\lambda^A: \overline{\cS}(\fS) \to \bA(\fS,\lambda)
		\end{equation}
		such that
		\begin{equation}\label{eq.bTrAX}
			\rdtr_\lambda^X = \rd\psi_\lambda \circ\rdtr_\lambda^A.
		\end{equation}
		For $n=2,3$ the map $\rdtr_\lambda^A$ is always injective, and for $n>3$, it is injective when $\fS$ is a polygon.

		\item There is a unique algebra embedding
		\begin{equation}\label{eq.Atr}
			\tr_\lambda^A: \SS \to \cA(\fS,\lambda)
		\end{equation}
        such that 
        \begin{equation}\label{eq.AX}
			\tr_\lambda^X = \psi_\lambda \circ\tr_\lambda^A.
		\end{equation}

	\end{enuma}
\end{theorem}
We call $\rdtr_\lambda^A$ the reduced $A$-version ${\rm SL}_n$-quantum trace map, and $\tr_\lambda^A$ the extended $A$-version ${\rm SL}_n$-quantum trace map.

Let $\fS$ be a triangulable pb surface with the triangulation $\lambda$. Suppose that $\fS$ has no interior punctures.
For each $v\in\lv{V}_\lambda$ (resp. $v\in \rd{V}_\lambda$), we review the construction in \cite{LY23} of the element $\gaa_v\in\cS(\fS)$ (resp. $\bar{\gaa}_v\in \rd{\cS}(\fS)$) such that $\tr_\lambda^A(\gaa_v) = a_v\in \cA(\fS,\lambda)$ (resp. $\rd{\tr}_\lambda^A(\bar{\gaa}_v) = a_v\in \overline{\cA}(\fS,\lambda)$). These elements $\gaa_v$ and $\bar{\gaa}_v$ will play an important role in proving the compability between the quantum trace map and the Frobenius homomorphism to be studied in the next section (see Theorem \ref{thmFrob}(f)).

Suppose that $\surface$ does not have interior punctures. Since there is no interior ideal point, each characteristic map $f_\tau : \tau \to \fS$ is an embedding, and we will identify $f_\tau(\tau)$ with $\tau$, 
the latter being a copy of $\PP_3$.

For a small vertex $v\in \rdV_\lambda$,
choose a face $\nu\in\face_\lambda$ 
that contains $v$. There are two such $\nu$ when $v$ is on an interior edge of the triangulation. Otherwise, $\nu$ is unique. Assume that $v=(ijk)\in\overline{V}_\nu$, where we identified $\overline{V}_\nu$ with $\overline{V}_{\mathbb{P}_3}$ in \eqref{ol_V_P3} (for this, we chose a based vertex). Draw a weighted directed graph $Y_v$ properly embedded into $\nu$ as in Figure \ref{fig-skel-main}. Here an edge of $Y_v$ has weight $i$, $j$ or $k$ according as the endpoint lands on the edge $e_1$, $e_2$ or $e_3$ respectively; compare with Figure \ref{fig-coords}. The directed weighted graph $Y_v$ is unique up to ambient isotopy of the ideal triangle $\nu$.

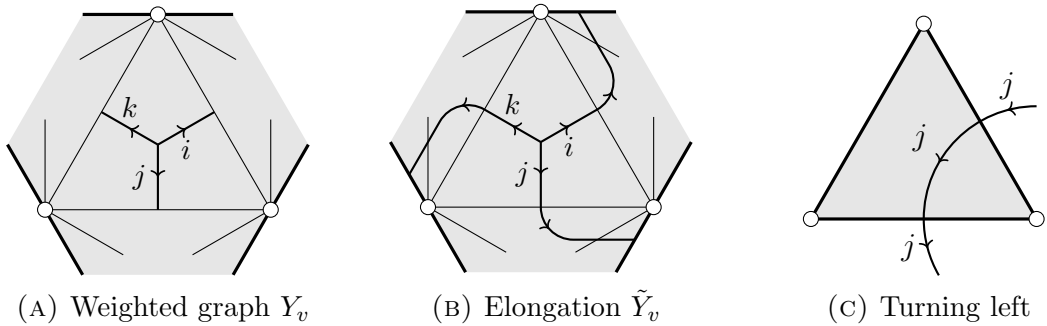
\begin{figure}
\centering
\begin{subfigure}[b]{0.3\linewidth}
\centering
\begin{tikzpicture}[baseline=0cm]
\fill[gray!20] (120:1) -- (-60:1) -- ++(2,0) -- ++(60:2)
	-- ++(120:2) -- ++(180:2) -- cycle;
\draw (0,0) -- (3,0) -- (60:3) -- cycle
	(0,0) edge +(-30:1.2) edge +(90:1.2)
	(3,0) edge +(90:1.2) edge +(-150:1.2)
	(60:3) edge +(-30:1.2) edge +(-150:1.2);
\draw[wall] (120:1) -- (-60:1) ++(2,0) -- ++(60:2) ++(120:2) -- ++(180:2);

\tikzmath{\len=sqrt(3)/2;}
\coordinate (C) at (1.5,\len);
\draw[edge,-o-={0.5}{>}] (C) -- node[midway,below]{\small$i$} +(30:\len);
\draw[edge,-o-={0.5}{>}] (C) -- node[midway,left]{\small$j$} +(-90:\len);
\draw[edge,-o-={0.5}{>}] (C) -- node[midway,above]{\small$k$} +(150:\len);

\draw[fill=white] (0,0)circle[radius=0.1] (3,0)circle[radius=0.1]
	(60:3)circle[radius=0.1];
\end{tikzpicture}
\subcaption{Weighted graph $Y_v$}\label{fig-skel-main}
\end{subfigure}
\begin{subfigure}[b]{0.3\linewidth}
\centering
\begin{tikzpicture}[baseline=0cm]
\fill[gray!20] (120:1) -- (-60:1) -- ++(2,0) -- ++(60:2)
	-- ++(120:2) -- ++(180:2) -- cycle;
\draw (0,0) -- (3,0) -- (60:3) -- cycle
	(0,0) edge +(-30:1.2) edge +(90:1.2)
	(3,0) edge +(90:1.2) edge +(-150:1.2)
	(60:3) edge +(-30:1.2) edge +(-150:1.2);
\draw[wall] (120:1) -- (-60:1) ++(2,0) -- ++(60:2) ++(120:2) -- ++(180:2);

\tikzmath{\len=sqrt(3)/2;\r=sqrt(3)/4;\el=1.25-\r;}
\coordinate (C) at (1.5,\len);
\begin{scope}[edge]
\draw[-o-={0.5}{>}] (C) -- +(30:\len)
	arc[radius=\r,start angle=-60,end angle=30] -- +(120:\el);
\draw[-o-={0.5}{>}] (C) -- +(-90:\len)
	arc[radius=\r,start angle=-180,end angle=-90] -- +(0:\el);
\draw[-o-={0.5}{>}] (C) -- +(150:\len)
	arc[radius=\r,start angle=60,end angle=150] -- +(-120:\el);

\path[-o-={0.5}{>}] (C) -- node[midway,below]{\small$i$} +(30:\len);
\path[-o-={0.5}{>}] (C) -- node[midway,left]{\small$j$} +(-90:\len);
\path[-o-={0.5}{>}] (C) -- node[midway,above]{\small$k$} +(150:\len);
\end{scope}

\draw[fill=white] (0,0)circle[radius=0.1] (3,0)circle[radius=0.1]
	(60:3)circle[radius=0.1];
\end{tikzpicture}
\subcaption{Elongation $\tY_v$}\label{fig-skel-long}
\end{subfigure}
\begin{subfigure}[b]{0.3\linewidth}
\centering
\begin{tikzpicture}[baseline=0cm]
\draw[wall,fill=gray!20] (0,0) -- (3,0) -- (60:3) -- cycle;
\begin{scope}[edge]
\draw[-o-={0.5}{>}] (3,1.5) arc[radius=1.5,start angle=90,end angle=210]
	node[midway,above left]{$j$};
\path[-o-={0.5}{>}] (3,1.5) arc[radius=1.5,start angle=90,end angle=120]
	node[midway,above]{$j$};
\path[-o-={0.5}{>}] (1.5,0)arc[radius=1.5,start angle=180,end angle=210]
	node[midway,left]{$j$};
\end{scope}

\draw[fill=white] (0,0)circle[radius=0.1] (3,0)circle[radius=0.1]
	(60:3)circle[radius=0.1];
\end{tikzpicture}
\subcaption{Turning left}\label{fig-skel-left}
\end{subfigure}
\caption{Graphs associated to a small vertex $v=(ijk) \in \overline{V}_\nu$ 
}\label{fig-skel}
\end{figure}

Elongate the nonzero-weighted edges of $Y_v$ to get an embedded weighted directed graph $\tY_v$ as in Figure~\ref{fig-skel-long}. Here the edge is elongated by using a left turn whenever it enters a triangle; see Figure~\ref{fig-skel-left} for a left turn.

Assume that $v\in\rd{V}_\lambda$. Then $v= (ijk)\in \rdV_\nu$ for an ideal triangle $\nu$ of $\lambda$. 
Turn $\tY_v$ into the stated $n$-web $\ag''_v$ by replacing a $k$-labeled edge of $\tY_v$ with $k$ parallel edges of $\gaa''_v$, adjusted by a sign. See the upper-right picture in Figure~\ref{fig-skel-web}.
We note that \cite[Lemma 4.12]{LY23} implies that the element
$\gaa''_v$ is reflection-normalizable in the sense defined in \OldS\ref{sec.reflection}.

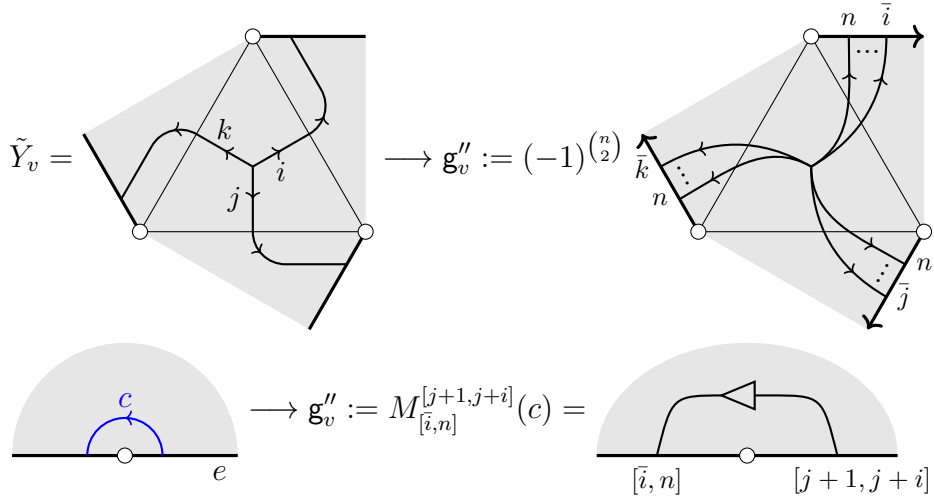
\begin{figure}
	\centering
	\begin{gather*}
\tilde{Y}_v=
\begin{tikzpicture}[baseline=(C.base)]
\fill[gray!20] (0,0) -- (120:1.5) -- (60:3) -- +(1.5,0)
	-- (3,0) -- +(-120:1.5) -- cycle;
\draw (0,0) -- (3,0) -- (60:3) -- cycle;
\draw[wall] (0,0) -- (120:1.5) (60:3) -- +(1.5,0) (3,0) -- +(-120:1.5);
\tikzmath{\len=sqrt(3)/2;\r=sqrt(3)/4;\el=1.25-\r;}
\coordinate (C) at (1.5,\len);
\begin{scope}[edge]
\draw[-o-={0.5}{>}] (C) -- +(30:\len)
	arc[radius=\r,start angle=-60,end angle=30] -- +(120:\el);
\draw[-o-={0.5}{>}] (C) -- +(-90:\len)
	arc[radius=\r,start angle=-180,end angle=-90] -- +(0:\el);
\draw[-o-={0.5}{>}] (C) -- +(150:\len)
	arc[radius=\r,start angle=60,end angle=150] -- +(-120:\el);
\path[-o-={0.5}{>}] (C) -- node[midway,below]{\small$i$} +(30:\len);
\path[-o-={0.5}{>}] (C) -- node[midway,left]{\small$j$} +(-90:\len);
\path[-o-={0.5}{>}] (C) -- node[midway,above]{\small$k$} +(150:\len);
\end{scope}
\draw[fill=white] (0,0)circle[radius=0.1] (3,0)circle[radius=0.1]
	(60:3)circle[radius=0.1];
\node (base) at (C) {\phantom{$-$}};
\end{tikzpicture}
\longrightarrow
\mathsf{g}''_v:=(-1)^{\binom{n}{2}}
\begin{tikzpicture}[baseline=(C.base)]
\fill[gray!20] (0,0) -- (120:1.5) -- (60:3) -- +(1.5,0)
	-- (3,0) -- +(-120:1.5) -- cycle;
\draw (0,0) -- (3,0) -- (60:3) -- cycle;
\draw[wall,->] (0,0) -- (120:1.5);
\draw[wall,->] (60:3) -- +(1.5,0);
\draw[wall,->] (3,0) -- +(-120:1.5);
\tikzmath{\len=sqrt(3)/2;\r=sqrt(3)/4;\el=1.25-\r;}
\coordinate (C) at (1.5,\len);
\begin{scope}[edge]
\draw[-o-={0.3}{<}] (60:3) ++(0.5,0)node[above]{\stsize$n$}
	to[out=-90,in=30] (C);
\draw[-o-={0.3}{<}] (60:3) ++(1,0)node[above]{\stsize$\bar i$}
	to[out=-90,in=30] (C);
\draw[-o-={0.3}{<}] (3,0) ++(-120:0.5)node[right]{\stsize$n$}
	to[out=150,in=-90] (C);
\draw[-o-={0.3}{<}] (3,0) ++(-120:1)node[right]{\stsize$\bar{j}$}
	to[out=150,in=-90] (C);
\draw[-o-={0.3}{<}] (0,0) ++(120:0.5)node[left]{\stsize$n$}
	to[out=30,in=150] (C);
\draw[-o-={0.3}{<}] (0,0) ++(120:1)node[left]{\stsize$\bar{k}$}
	to[out=30,in=150] (C);
\path (60:3) ++(0.75,-0.2) node{...}
	(3,0) ++(-120:0.75) ++(150:0.2) node[rotate=60]{...}
	(120:0.75) ++(30:0.2) node[rotate=-60]{...};
\end{scope}
\draw[fill=white] (0,0)circle[radius=0.1] (3,0)circle[radius=0.1]
	(60:3)circle[radius=0.1];
\node (base) at (C) {\phantom{$-$}};
\end{tikzpicture}\\
\begin{tikzpicture}[baseline=(ref.base)]
\fill[gray!20] (1.5,0) arc[x radius=1.5,y radius=1.5,start angle=0,end angle=180];
\draw[wall] (-1.5,0) -- (1.5,0);
\draw[fill=white] (0,0)circle[radius=0.1];
\draw[edge,-o-={0.5}{>},blue] (0.5,0) arc[radius=0.5,start angle=0,end angle=180] node[midway,above]{$c$};
\path (1.5,0) node[below left]{$e$};
\node (ref) at (0,0.6) {\phantom{$-$}};
\end{tikzpicture}
\longrightarrow
\gaa''_v:=M^{[j+1,j+i]}_{[\bar{i},n]}(c)=
\begin{tikzpicture}[baseline=(ref.base)]
\fill[gray!20] (2,0) arc[x radius=2,y radius=1.5,start angle=0,end angle=180];
\draw[wall] (-2,0) -- (2,0);
\draw[fill=white] (0,0)circle[radius=0.1];
\draw[det] (1.2,0)node[anchor=135]{\stsize$[j+1,j+i]$}
	..controls (1,0.8).. (0.3,0.8) -- (-0.3,0.8)
	..controls (-1,0.8).. (-1.2,0) node[below]{\stsize$[\bar{i},n]$};
\node (ref) at (0,0.6) {\phantom{$-$}};
\end{tikzpicture}
\end{gather*}
	\caption{Definition of $\gaa''_v$}\label{fig-skel-web}
\end{figure}

Now assume that $v\in \lv{V}_\lambda\setminus\rd{V}_\lambda$. Then $v=(ijk)$ is in an attached triangle $\nu\equiv \PP_3$ (of $\fS^\ast$), whose edge $e_1$ is glued to a boundary edge $e$ of $\fS$. Let $c$ be the oriented corner arc of $\fS$ starting on $e$ and going counterclockwise, i.e. turning left all the time. Then \cite[Lemmas 4.10 and 4.13]{LY23} implies that the element 
$\gaa''_v:=M^{[j+1,j+i]}_{[\bar{i},n]}(c)$
is reflection-normalizable. See the lower-right picture in Figure~\ref{fig-skel-web} for the diagram of $\gaa''_v$.

Finally, define $\gaa_v$ to be the reflection-normalization \eqref{eq.reflec} of $\gaa''_v$ for each $v\in\lv{V}_\lambda$. Let $\bar \ag_v$ be the image of $\ag_v$ in $\reduceS(\surface)$. Note that $\bar \ag_v = 0$ if $v \in \lv{V}_\lambda\setminus\rd{V}_\lambda$. 

\begin{theorem}\cite{LY23}\label{trace}
	If $\fS$ has no interior punctures, 
the set $\{\gaa_v\mid v\in V_{\lambda}'\}$ is a quantum torus frame for $\cS(\fS)$, and 
$\tr_\lambda^A(\gaa_v) = a_v$ holds for each $v\in V_{\lambda}'$. 

If $\fS$ is a polygon, the set $\{\bar\ag_v\mid v\in \overline{V}_{\lambda}\}$
is a quantum torus frame for $\overline{\cS}(\fS)$, and 
$\overline{\tr}_\lambda^A(\bar\ag_v) = a_v$ holds for each $v\in \overline{V}_{\lambda}$.
\end{theorem}

As we will see in the next section, when proving a certain compatibility statement regarding the $A$-version quantum trace maps, the above theorem will make it suffice to check it only for the special elements $\gaa_v$.

Major focus of the present paper is on the case when the quantum parameter $\hq$ is specialized at a root of unity. For this purpose, we introduce the following notations for later use, before leaving this section. 
For a non-zero complex number $\hxi$, define 
\begin{align*}
    &\bA_{\hxi}(\fS,\lambda)= \bA(\fS,\lambda)\ot_\Zhq \BC,\qquad
    \cA_{\hxi}(\fS,\lambda)= \cA(\fS,\lambda)\ot_\Zhq \BC,\\
    &\rd\cX_{\hxi}(\fS,\lambda)= \rd\cX(\fS,\lambda)\ot_\Zhq \BC,\qquad
    \cX_{\hxi}(\fS,\lambda)= \cX(\fS,\lambda)\ot_\Zhq \BC,
\end{align*}
where $\BC$ is considered as a $\Zhq$-module by $\hq 
\mapsto \hxi$.

\section{The Frobenius homomorphism for stated ${\rm SL}_n$-skein module}\label{sub-Frobenius-map-SLn}

\def\OA{{{\mathcal O}_{\mathcal A}}}
\def\flip{{\mathsf{flip}}}
\def\coin{{\mathsf{coin}}}
\def\Oo{{\mathcal O}_\omega}
\def\Oe{{\mathcal O}_\eta}
\def\UA{{\dot{\mathcal U  }}_\CA}
\def\DDD{{ \mathcal D} }
\def\Do{{\DDD_\omega}}
\def\Deta{{\DDD_\eta}}
\def\CA{{\mathcal A}}
\def\bY{{\overline Y}}
\def\Yo{\bY_\omega}
\def\Ye{\bY_\eta}
\def\Frac{{\mathsf{Fr}}}
\def\Uo{\dot{\mathcal U}_\omega}
\def\Ue{\dot{\mathcal U}_\eta}
\def\Ch{{\mathsf {Ch}}}
\def\Rep{{\mathsf{Rep}}}
\def\Lo{{L_\omega}}
\def\Leta{{L_\eta}}
\def\home{{\hat \omega}}
\def\heta{{\hat \eta}}
\def\So{{\cS_\home(\fS)}}
\def\Se{{\cS_\heta(\fS)}}
\def\bSo{{\overline\cS_\home(\fS)}}
\def\bSe{{\overline\cS_\heta(\fS)}}

\newcommand\inclp[2]{{\includegraphics[height=#1]{#2-eps-converted-to.pdf}}}
\def\letterx{  \raisebox{-9pt}{\inclp{.9 cm}{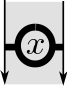}} } 
\def\lettery{  \raisebox{-9pt}{\inclp{.9 cm}{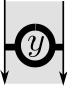}} }
\def\letterxy{  \raisebox{-9pt}{\inclp{.9 cm}{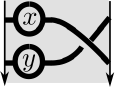}} }
\def\bXSq{{ \overline{\cX}_\hq(\fS,\lambda)   }}
\def\XSq{{ {\cX}_\hq(\fS,\lambda)   }}
\def\ASq{{ {\cA}_\hq(\fS,\lambda)   }}
\def\bASq{{ \overline{\cA}_\hq(\fS,\lambda)   }}
\def\bAPq{{ \overline{\cA}_\hq(\PP_3)   }}
\def\bSS{{  \overline{\cS}_\hq(\fS) }}
\def\OSL{{{\mathcal O}_q(sl_n)}}
\def\bSSe{{  \overline{\cS}_\heta(\fS) }}
\def\bSSo{{  \overline{\cS}_\home(\fS) }}

In this section we study one of the main ingredients of the present paper, namely the Frobenius homomorphism for the stated ${\rm SL}_n$-skein modules
$$
\Phi : \cS_\heta\MN \to \cS_\home\MN
$$
for a marked $3$-manifold $\MN$, where $\home$ is a root of unity, with $N = \ord(\home^{4n^2})$ and $\heta = \home^{N^2}$. Of major interest to many would be the case when $\MN$ is the thickening of a pb surface $\fS$, in which case $\Phi : \cS_\heta(\fS) \to \cS_\home(\fS)$ is required to be 
an algebra homomorphism.

In this section we focus on the case when $\MN$ is an essentially marked $3$-manifold, or when it is the thickening of an essentially bordered pb surface. For this setting, we give a fairly simple construction of the Frobenius homomorphism, which generalizes and improves the previously known constructions. We start with the bigon $\PP_2$ whose stated ${\rm SL}_n$-skein algebra $\cS(\PP_2)$ is a coribbon Hopf algebra isomorphic to $\Oq$.  For an essentially bordered pb surfaces the stated ${\rm SL}_n$-skein algebra is obtained by gluing copies of $\cS(\PP_2)$  together, where the gluing mechanism is governed by the cobraided structure. Since the Frobenius homomorphism for $\cS(\PP_2)$ respects the cobraided structure, we can extend the Frobenius homomorphism from the bigon to the whole surface. This method would work for all similar stated skein theory. 

Theorem \ref{thmFrob}, on the existence of the Frobenius homomorphism $\Phi$ and on its compatibility with various properties of the stated ${\rm SL}_n$-skein modules reviewed in \OldS\ref{sec.stated_SLn-skein_modules_and_algebras}, especially the quantum trace maps, can be regarded as one of the major accomplishments of the present paper.

\def\cF{\mathcal{F}}

\subsection{Assumptions on roots of unity $\home$}\label{ss.assumptions_on_roots_of_unity}

Throughout this section, $\home\in \BC$ is a root of 
unity. Let 
$$\omega= \home^{2n^2}, \quad N =\ord (\omega^2), \quad \heta= \home^{N^2}, \quad \eta= \heta^{2n^2}= \omega^{N^2}.$$

\subsection{Main result of the section}\label{subsec:main_result_of_Frobenius_section} We formulate the main result of this section.

A stated $n$-web $\alpha$ in a marked 3-manifold $(M,\cN)$ (see \OldS\ref{ss.marked}) is {\bf string-ly} if each connected component $\alpha_1, \dots, \alpha_{r}$ of it is a stated framed $\cN$-arc (see \OldS\ref{subsec:punctured_bordered_surface_and_n-web}). For such a string-ly stated $n$-web let $\alpha^{(N)}$ be the result of replacing each $\alpha_j$ with $N$ disjoint parallel copies of $\alpha_j$ placed with respect to each other along
the framing direction. The parallel copies of $\alpha_j$ are supposed to be in a small neighborhood of $\alpha_j$.

\bthm \label{thmFrob} Assume that $\MN$ is an essentially marked 3-manifold (\OldS\ref{subsec:essentially_marked_3-manifolds}).

\begin{enuma}
    
\item
There exists a unique 
$\mathbb{C}$-linear map
$$
\Phi= \Phi_\home: \SeM \to \SoM,
$$
called the {\bf Frobenius homomorphism} for stated ${\rm SL}_n$-skein modules,
such that 
$$\Phi([\al]_{\hat\eta})= [\al^{(N)}]_{\hat \omega}\quad\mbox{whenever $\al$ is a string-ly stated $n$-web in $(M,\cN)$},
$$ 
where $[\sim]_\heta$ and $[\sim]_\home$ are as defined in \OldS\ref{ss.change-coef}.

For an essentially bordered pb surface $\fS$ (Definition \ref{def.pb_surface}), the Frobenius homomorphism 
$$\Phi=\Phi_\home: \Se \to \So$$ associated to the thickening of $\fS$ \eqref{marked_3-manifold_from_surface} is a $\mathbb{C}$-algebra embedding between the stated ${\rm SL}_n$-skein algebras.

\item
If $f: \MN \embed (M', \cN')$ is a  morphism in the category $\mathbb{M}$ (\OldS\ref{ss.marked}), then the induced map $f_*: \cS(M, \cN) \to \cS(M', \cN')$ commutes with $\Phi$. In other words, 
$\Phi\circ f_{*, \heta} = f_{*,\home} \circ \Phi$. 

\item
$\Phi$ commutes with the cutting homomorphisms of \OldS\ref{subsec.cutting_homomorphism}.

\item
$\Phi$ multiplies the $\BZ_n$ grading of \OldS\ref{sec-grading} by $N$: for $k\in \BZ_n$,
\be 
\Phi ( (\SeM)_k  ) \subset  
\SoM)_{Nk}
\ee

\item
$\Phi:\Se\to\So$ for a pb surface $\fS$ descends to a reduced version $$\bPhi: 
\bSe \to 
\bSo$$ between the reduced stated ${\rm SL}_n$-skein algebras \eqref{reduced_stated_skein_algebra}.

\item
The Frobenius homomorphisms $\Phi$ and $\overline{\Phi}$ for stated ${\rm SL}_n$-skein algebras of a triangulable pb surface $\fS$ 
are compatible with the Frobenius homomorphisms $\Phi^\bT$ \eqref{eqFrDef} for quantum tori via the quantum trace 
maps in \OldS\ref{subsec.SLn_quantum_trace_maps}, meaning that for any ideal triangulation $\lambda$ of $\fS$, one has
\begin{align}
	\btr^X_\lambda \circ \bPhi  &= \Phi^\bT \circ \btr^X_\lambda, &\tr^X_\lambda \circ \Phi  &= \Phi^\bT \circ \tr^X_\lambda \\
	\btr^A_\lambda \circ \bPhi  &= \Phi^\bT \circ \btr^A_\lambda,  &\tr^A_\lambda \circ \Phi  &= \Phi^\bT \circ \tr^A_\lambda,
    \label{eq.tr_A_compatible_with_Phi}
\end{align}
where for the $A$-versions in \eqref{eq.tr_A_compatible_with_Phi} we assume that $\fS$ has no interior punctures. 
\end{enuma}
\ethm

\begin{remark}
For $n=2$,
    Theorem \ref{thmFrob} was proved in 
    \cite{BL22,BW16,KQ}. For $n=3$, Theorem \ref{thmFrob}(a), (b), (c), (e)  were proved in \cite{Hig23} when the order of $\hat\omega$ is coprime to $6$ and $(M,\mathcal N)$ is the thickening of an essentially bordered pb surface. 
    For general $n$,
    Theorem \ref{thmFrob}(a), (b), (c), (e) were  proved in \cite{Wan23} when the order of $\hat\omega$ is coprime 
    to $2n$, which was generalized in \cite{KW24} to all roots of unity when $(M,\mathcal N)$ is the thickening of an essentially bordered pb surface.
    
    Our construction of the Frobenius homomorphism for general $n$ (as well as the proof of the properties (a), (b), (c), (e)) is different from \cite{Wan23} and is more structural, as we develop an ${\rm SL}_n$ generalization of the methodology taken for $n=2$ \cite{BL22} and $n=3$ \cite{Hig23}, based on the constructions in \cite{LS21}. Item (f) is proved here for the first time for general $n$, which is a very important property and could be used to formulate the center of the stated ${\rm SL}_n$-skein algebra \cite{FKL,kim2024unicity}.

    Moreover, we have no restriction on the order of $\home$.
\end{remark}

We will prove Theorem \ref{thmFrob} in the remaining parts of this section.

\subsection{Relations for $N$-parallel copies of arcs  in $\SO$}\label{subsec.relations_for_parallel_copies}

We first develop a technical preliminary.

We use
$
\raisebox{-.3\height}{
	\begin{tikzpicture}
		\draw [line width =1pt] (0,0)--(1,0);
		\draw[fill=white] (0.3,-0.2) rectangle (0.7,0.2);
		\node  at (0.5,0) {$N$};
	\end{tikzpicture}
}$
to denote the $N$-parallel copies of the corresponding strand, placed with respect to one another along the framing direction.

When $N$ and $2n$ are coprime,
the third author calculated the height change relation  \cite[Corollary 7.18]{Wan23}, and the relation for removing cups  \cite[Lemma 7.20]{Wan23} for $N$-parallel copies of stated arcs in $\SS$.
These relations can be easily generalized to all $N$ using the same technique. 
We prove Lemmas \ref{lem-height-Nparallel} and \ref{lem-arcs} in Appendix \ref{Appendix-2} for the completeness.

\begin{lemma}\cite[Lemmas 7.6 and 7.7, Corollary 7.18]{Wan23}\label{lem-height-Nparallel}
	In $\cS_{\home}\MN$, we have the following relations:
	$$
	\raisebox{-.40\height}{
		\begin{tikzpicture}
			\draw[color=gray!20,fill=gray!20] (-1,-1.2) rectangle (1,1.2);
			\draw[wall,<-] (1,-1.2)-- (1,1.2);
			\draw[line width =1pt] (-1,-0.8)--(1,0.8);
			\draw[color=gray!20,fill=gray!20] (-0.1,-0.08) rectangle (0.1,0.08);
			\draw[line width =1pt] (-1,0.8)--(1,-0.8);
			\draw[fill=gray!20] (-0.7,0.2) rectangle (-0.3,0.6);
			\node  at (-0.5,0.4) {$N$};
			\draw[fill=gray!20] (-0.7,-0.2) rectangle (-0.3,-0.6);
			\node  at (-0.5,-0.4) {$N$};
			\node [right] at (1,0.8) {$i$};
			\node [right] at (1,-0.8) {$j$};
			\draw[fill=white] (0.5,0.4) circle[radius=0.1] ;
			\draw[fill=white] (0.5,-0.4) circle[radius=0.1] ;
		\end{tikzpicture}
	}=\eta^{-\frac{1}{n}+\delta_{i,j}}
	\raisebox{-.40\height}{
		\begin{tikzpicture}
			\draw[color=gray!20,fill=gray!20] (-1,-1.2) rectangle (1,1.2);
			\draw[wall,<-] (1,-1.2)-- (1,1.2);
			\draw[line width =1pt] (-1,-0.4)--(1,-0.4);
			\draw[color=gray!20,fill=gray!20] (-0.1,-0.08) rectangle (0.1,0.08);
			\draw[line width =1pt] (-1,0.4)--(1,0.4);
			\draw[fill=gray!20] (-0.7,0.2) rectangle (-0.3,0.6);
			\node  at (-0.5,0.4) {$N$};
			\draw[fill=gray!20] (-0.7,-0.2) rectangle (-0.3,-0.6);
			\node  at (-0.5,-0.4) {$N$};
			\node [right] at (1,0.4) {$j$};
			\node [right] at (1,-0.4) {$i$};
			\draw[fill=white] (0.5,0.4) circle[radius=0.1] ;
			\draw[fill=white] (0.5,-0.4) circle[radius=0.1] ;
		\end{tikzpicture}
	}\text{ and }
	\raisebox{-.40\height}{
		\begin{tikzpicture}
			\draw[color=gray!20,fill=gray!20] (-1,-1.2) rectangle (1,1.2);
			\draw[wall,<-] (1,-1.2)-- (1,1.2);
			\draw[line width =1pt] (-1,-0.8)--(1,0.8);
			\draw[color=gray!20,fill=gray!20] (-0.1,-0.08) rectangle (0.1,0.08);
			\draw[line width =1pt] (-1,0.8)--(1,-0.8);
			\draw[fill=gray!20] (-0.7,0.2) rectangle (-0.3,0.6);
			\node  at (-0.5,0.4) {$N$};
			\draw[fill=gray!20] (-0.7,-0.2) rectangle (-0.3,-0.6);
			\node  at (-0.5,-0.4) {$N$};
			\node [right] at (1,0.8) {$i$};
			\node [right] at (1,-0.8) {$j$};
			\draw[fill=black] (0.5,0.4) circle[radius=0.1] ;
			\draw[fill=white] (0.5,-0.4) circle[radius=0.1] ;
		\end{tikzpicture}
	}=\eta^{\frac{1}{n}-\delta_{\bar{i},j}}
	\raisebox{-.40\height}{
		\begin{tikzpicture}
			\draw[color=gray!20,fill=gray!20] (-1,-1.2) rectangle (1,1.2);
			\draw[wall,<-] (1,-1.2)-- (1,1.2);
			\draw[line width =1pt] (-1,-0.4)--(1,-0.4);
			\draw[color=gray!20,fill=gray!20] (-0.1,-0.08) rectangle (0.1,0.08);
			\draw[line width =1pt] (-1,0.4)--(1,0.4);
			\draw[fill=gray!20] (-0.7,0.2) rectangle (-0.3,0.6);
			\node  at (-0.5,0.4) {$N$};
			\draw[fill=gray!20] (-0.7,-0.2) rectangle (-0.3,-0.6);
			\node  at (-0.5,-0.4) {$N$};
			\node [right] at (1,0.4) {$j$};
			\node [right] at (1,-0.4) {$i$};
			\draw[fill=white] (0.5,0.4) circle[radius=0.1] ;
			\draw[fill=black] (0.5,-0.4) circle[radius=0.1] ;
		\end{tikzpicture}
	}
	$$
    for each $i,j \in \JJ = \{1,2,\ldots,n\}$, 
	where each $
	\raisebox{-.3\height}{
		\begin{tikzpicture}
			\draw [line width =1pt] (0,0)--(1,0);
			\draw[fill=white] (0.3,-0.2) rectangle (0.7,0.2);
			\node  at (0.5,0) {$N$};
		\end{tikzpicture}
	}$ is a part of $N$-parallel copies of some stated framed $\cN$-arc.  
\end{lemma}

\begin{lemma}\cite[Lemma 7.20]{Wan23}\label{lem-arcs}
	In $\cS_{\home}\MN$, we have the following relations:
	\begin{align}\label{eq-N-copies-cap-wall}
	\sum_{
    i \in \JJ}\eta^{-\frac{n-1}{2n}}(-\eta)^{\bar i-n}
	\raisebox{-.30in}{
		\begin{tikzpicture}
			\draw[color=gray!20,fill=gray!20] (-1.3,-1) rectangle (0,1);
			\draw [wall,<-] (0,-1)-- (0,1);
			\draw [line width =1pt] (-1.3,-0.5)--(-0.7,-0.5);
			\filldraw[fill=black] (-1,-0.5) circle (0.1);
			\draw [color = black, line width =1pt](-0,-0.5)--(-0.3,-0.5);
			\node [right] at(-0,-0.5) { $\bar{i}$};
			\draw[color=black] (-0.5,-0.5) circle (0.2);
			\draw [line width =1pt] (-1.3,0.5)--(-0.7,0.5);
			\draw [color = black, line width =1pt](-0,0.5)--(-0.3,0.5);
			\node [right] at(-0,0.5) { $i$};
			\draw[fill=gray!20] (-0.7,0.3) rectangle (-0.3,0.7);
			\node at(-0.5,0.5) {$N$};
			\draw[fill=gray!20] (-0.7,-0.3) rectangle (-0.3,-0.7);
			\node at(-0.5,-0.5) {$N$};
			\filldraw[fill=white] (-1,0.5) circle (0.1);
	\end{tikzpicture}}
	= 
	\raisebox{-.30in}{
		\begin{tikzpicture}
			\draw[color=gray!20,fill=gray!20] (-1,-1) rectangle (0.7,1);
			\draw [wall] (0.7,-1)--(0.7,1);
			\draw [line width =1pt] (-1,-0.5)--(-0.7,-0.5);
			\draw [color = black, line width =1pt](-0,-0.5)--(-0.3,-0.5);
			\draw[fill=gray!20] (-0.7,-0.3) rectangle (-0.3,-0.7);
			\node at(-0.5,-0.5) {$N$};
			\draw [color = black, line width =1pt] (-1,0.5)--(-0.7,0.5);
			\draw [color = black, line width =1pt](-0,0.5)--(-0.7,0.5);
			\draw [color = black, line width =1pt] (0 ,-0.5) arc (-90:90:0.5);
			\filldraw[fill=white] (-0.5,0.5) circle (0.1);
	\end{tikzpicture}},
	\end{align}
	where each $
	\raisebox{-.3\height}{
		\begin{tikzpicture}
			\draw [line width =1pt] (0,0)--(1,0);
			\draw[fill=white] (0.3,-0.2) rectangle (0.7,0.2);
			\node  at (0.5,0) {$N$};
		\end{tikzpicture}
	}$ is a part of $N$-parallel copies of some stated framed $\cN$-arc.
	
\end{lemma}

\begin{lemma}\label{lem_change}
	Relations in Lemmas \ref{lem-height-Nparallel} and \ref{lem-arcs} also hold in $\SE$ if we remove the box $
    \fbox{\hspace{-0,8mm}N}$ in these relations. This means that 
	 the height change relation and the relation for removing caps for stated arcs in $\SE$ are compatible with the corresponding relations 
	for $N$-parallel copies of stated arcs in $\SO$. 
\end{lemma}

\def\OSL{\Oq}

\subsection{The bigon case}\label{subsec-the-bigon-case} 
Recall that Theorem \ref{thm-iso-Oq-P2} lets us identify the stated ${\rm SL}_n$-skein algebra $\cS(\PP_2)$ of the based bigon $\mathbb{P}_2$ with the quantized function algebra $\Oq$ (Definition \ref{def.OqSLn}), which is generated by the entries of the quantum matrix $\buu=(u_{ij})_
{i,j\in \JJ}$. 
Recall from \OldS\ref{subsec:quantum_elementary_symmetric_function} that a quantum minor $\fm\in \OSL$ is the quantum determinant of a square submatrix of $\buu$. Let $\fm_\home\in \cS_\home(\PP_2)$ be the image of $\fm$ under the natural map $\cS(\PP_2)\to \cS_\home(\PP_2)$, and define 
$\fm_\heta\in \cS_\heta(\PP_2)$ 
likewise, similarly as in \OldS\ref{ss.change-coef}. Recall that $\cS(\PP_2)$ is a cobraided Hopf algebra, meaning that there is a bilinear form $$\rho_\hq: \cS(\PP_2) \ot \cS(\PP_2) \to \BZ[\hq^{\pm 1}]$$ with favorable properties which make the category of right $\OSL$-comodules a ribbon category; see \cite{LS21}. 
\blem[bigon Frobenius homomorphism]\label{rbigon}
\begin{enuma}
    \item 
    There exists a unique 
    injective Hopf algebra 
    homomorphism  $$\Phi: \cS_\heta(\PP_2) \to \cS_\home(\PP_2)$$ given by $\Phi(u_{ij})= u_{ij}^N$ for all $i,j \in \JJ = \{1,\ldots,n\}$.
For any quantum minor $\fm\in \OSL$, one has
\be 
\Phi(\fm_\heta) = (\fm_\home)^N.
\label{eqminor}
\ee

\item 
For  
any stated $n$-web diagram $\al$ 
in $\PP_2$ that is string-ly (\OldS\ref{subsec:main_result_of_Frobenius_section}), we have $\Phi([\al]_{\hat\eta}) = [\al^{(N)}]_{\hat\omega}$.

\item 
The map $\Phi:\cS_\heta(\PP_2) \to \cS_\home(\PP_2)$ preserves the cobraided structure, meaning that
\be 
\rho_\home(\Phi(x) \ot \Phi(y))= \rho_\heta(x\ot y), \quad\mbox{for all } x,y \in \cS_\heta(\PP_2).
\label{eqrho}
\ee
\end{enuma}

\elem 
\bpr (a) follows from Theorem \ref{Fro_Oq}.

(b) From part (a), we have $\Phi([u_{ij}]_{\hat\eta}) =([u_{ij}]_{\hat\omega})^N = [(u_{ij})^{(N)}]_{\hat\omega}$ for 
all $i,j\in\JJ$; here $u_{ij}\in \cS(\PP_2)$ stands for the stated arc in Figure \ref{fig-bigon}(c). Since $\Phi$ commutes with the antipode $S$ (see \eqref{eq-Hopf-bigon}), we have $\Phi([\cev{u}_{ij}]_{\hat\eta}) = ([\cev{u}_{ij}]_{\hat\omega})^N = [(\cev{u}_{ij})^{(N)}]_{\hat\omega}$ for 
all $i,j\in\JJ$; 
recall that $\cev{u}_{ij}$ is $u_{ij}$ with the reverse orientation. 
The above discussion shows that the claim holds if $[\alpha]_{\hat\eta}$ is a product of $[u_{ij}]_{\hat\eta}$ and $[\cev u_{ij}]_{\hat\eta}$ for 
some collection of $i,j \in \JJ$.
Note that for any stated $n$-web diagram $\alpha$ in $\PP_2$ that is string-ly, by using a sequence of  height exchange relations and relations for removing caps, one can write $[\alpha]_{\hat\eta}$ 
as a linear sum of stated $n$-web diagrams that are products of $[u_{ij}]_{\hat\eta}$ and $[\cev u_{ij}]_{\hat\eta}$ for 
some $i,j \in \JJ$ (see Lemma \ref{lem-span-arcs}(b)). 
Then Lemma \ref{lem_change} completes the proof. 

(c) By \cite{LS21},  $\rho_\hq$ is given by
\begin{align}
	\rho_\hq \left (\letterx  \ot \lettery \right) & = \varepsilon\left(  \letterxy \right),
    \label{eqCobraid} 
\end{align}
where $\varepsilon$ is the counit (see \eqref{eq-Hopf-bigon}). 
Since $\Phi$ 
is compatible with 
$\varepsilon$, in the sense that $\varepsilon \circ \Phi = \varepsilon$, we have \eqref{eqrho}.
\epr

\subsection{Compatibility with the cutting homomorphism, for string-ly diagrams}
\label{ss.compatibility_with_cutting_homomorphism}

We first investigate the bigon case.
Let $
e$ be an oriented interior ideal arc connecting the two ideal points of $\PP_2$. 
Let $\al$ be a 
stated $n$-web diagram in $\mathbb{P}_2$ that is string-ly (\OldS\ref{subsec:main_result_of_Frobenius_section}) and $e$-transverse (\OldS\ref{subsec.cutting_homomorphism}).
By Theorem \ref{t.splitting2}, cutting $\mathbb{P}_2$ along $e$ yields the cutting homomorphism $\Theta_e : \cS(\PP_2) \to \cS(\Cut_e(\PP_2)) = \cS(\PP_2) \otimes \cS(\PP_2)$, which satisfies $\Theta_e(\alpha) = \sum_{s:\alpha \cap e \to \JJ} \alpha_s$, where $\alpha_s$ stands for $\alpha_{h,s}$ of Theorem \ref{t.splitting2} with the linear order $h$ on $\alpha\cap e$ being induced this time by the chosen orientation on $e$ (the positive orientation on $e$ matches the increasing direction of this order). This yields $\Theta_e$ when $\hq$ is specialized at the root 
of unity $\home$
or $\heta$ as well.

Since $\Theta_e$ coincides with the coproduct $\Delta$ of the Hopf algebra $\cS(\PP_2)$, and since $\Phi : \cS_\heta(\PP_2) \to \cS_\home(\PP_2)$ commutes with $\Delta$ in an appropriate sense, we get
\begin{align*}
\Theta_e([\alpha^{(N)}]_\home)
& = \Theta_e(\Phi([\alpha]_\heta)) \quad (\because \mbox{Lemma \ref{rbigon}(b)}) \\
& = \Phi (\Theta_e([\alpha]_\heta)) \\
& = \Phi({\textstyle \sum}_{s:\alpha \cap e \to \JJ} [\alpha_s]_\heta) \\
& = {\textstyle \sum}_{s:\alpha \cap e \to \JJ} [(\alpha_s)^{(N)}]_\home, \quad (\because \mbox{Lemma \ref{rbigon}(b)})
\end{align*}
that is,
\begin{align}
\label{eqcut0_new}
\Theta_e([\alpha^{(N)}]_\home)={\textstyle \sum}_{s:\alpha \cap e \to \JJ} [(\alpha_s)^{(N)}]_\home.
\end{align}

\def\fT{{\PP_3}}

We extend this result to marked 3-manifolds, regarding the cutting homomorhism $\Theta_{(D,\beta)} : \cS\MN \to \cS(\Cut_{(D,\beta)}\MN)$ (in Theorem \ref{thm.cutting_homomorphism_for_3-manifolds}) specialized at roots of unity.

\blem \label{rcut}
Let $\MN$ be an essentially  
marked 3-manifold (\OldS\ref{subsec:essentially_marked_3-manifolds}), 
$D$  
a properly embedded disk in $M$ such that $\partial D\cap\cN=\emptyset$, and 
$\beta$ 
an embedded oriented open interval in $D$.
Suppose that $\al$ is a string-ly stated $n$-web in $\MN$ that is $(D,\beta)$-transverse. Then
\be  \Theta_
{(D,\beta)}( [\al^{(N)}]_{\hat\omega} ) = \sum_{s: \al \cap c \to \JJ} [
(\alpha_s)^{(N)}]_{\hat\omega} \in \cS_\home(\Cut_{(D,\beta)}\MN),
\label{eqcutN}
\ee
where $\alpha_s$ is as in Theorem \ref{thm.cutting_homomorphism_for_3-manifolds}, which is a stated $n$-web in $\Cut_{(D,\beta)}\MN$ induced from $\alpha$ and $s : \alpha \cap c \to \JJ$.

\elem
\bpr A small tubular neighborhood of $\al$ is a disjoint union of the thickenings of bigons. The disk $D$ cuts these thickened bigons into smaller thickened bigons. The lemma follows from a repeated application of \eqref{eqcut0_new}.
\epr

\def\bphi{{\bar \phi}}
\def\fn{{\mathfrak{n}}}
\def\bpsi{{\bar \psi}}
\subsection{Proof of Theorem \ref{thmFrob}(a) for essentially bordered pb surfaces} \label{ss.Proof_of_Thm_6_1a_for_surfaces}

We use induction on $t(\fS)$, defined as follows. If $\fS$ is triangulable, then let $t(\fS)$ be the number of triangles in a triangulation of $\fS$. If $\fS$ is either a monogon or a bigon, then $t(\fS)=0$. Finally if $\fS = \fS_1 \sqcup \fS_2$ then $t(\fS)= t(\fS_1) + t(\fS_2)$.

If $t(\fS)=0$, then $\fS$ is a disjoint union of monogons and bigons, and we are done by \eqref{monogon_skein_algebra} (for monogons) and Lemma \ref{rbigon} (for bigons). 
\begin{figure}[htpb]
	\includegraphics[height=2.8cm]{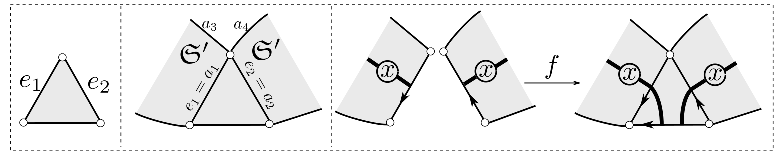} 
	\caption{Left: $\fT$. Middle: Gluing $\fS'$ and $\fT$ by $a_1=e_1$ and $a_2=e_2$ to get $\fS$. Right: web diagram $x\in \cS(\fS')$ and its image $f(x)\in \SS$. Note that $a_3$ may or may not coincide with $a_4$; our figure shows an example when $a_3 \neq a_4$.}  
	\label{figBrtensor0}
\end{figure}

Assume now that $\fS$ is the result of gluing an essentially bordered pb surface $\fS'$ 
and a triangle along two edges, as in the middle picture of
Figure \ref{figBrtensor0}. 
Here $a_1, a_2$ are boundary edges of $\fS'$. Note that $\fS'$ has one less number of triangles than $\fS$. It is not hard to see that for any $\fS$ with $t(\fS)\ge 1$, one can find such $\fS'$. Namely, consider any ideal triangulation $\lambda$ of $\fS$ and consider any triangle $\tau$ of $\lambda$, whose sides are $e_1,e_2,e_3$, with $e_3$ being a boundary edge; there exists such a triangle because $\fS$ is essentially bordered. Let $e_1',e_2'$ be disjoint ideal arcs lying in the interior of $\tau$ that are isotopic to $e_1,e_2$, respectively. Cutting $\fS$ along $e_1'$ and $e_2'$ yields a pb surface given by a disjoint union of a triangle $\PP_3$ and a non-empty pb surface, which we can use as $\fS'$. For example, when $\fS$ is a triangle, then $\fS'$ is a disjoint union of two bigons.

By the induction hypothesis, we may suppose that we have the Frobenius homomorphism 
$$\Phi': \cS_\heta(\fS') \to \cS_\home(\fS'), \qquad \Phi'([\al]_{\hat\eta}) = [\al^{(N)}]_{\hat\omega} \ \text{for each string-ly diagram} \ \al.  $$
By \cite[Proposition 8.1]{LS21} the $R$-linear map $$f: \cS(\fS')\to \SS$$ given in the right picture of Figure \ref{figBrtensor0} is bijective; let $f_\home$ and $f_\heta$ be the specializations of $\hq$ at roots of unity, $\home$ and $\heta$ respectively. We will prove that $\Phi:= f_{\home} \circ \Phi' \circ f_{\heta}^{-1}: \cS_\heta(\fS) \to \cS_\home(\fS)$ is an algebra homomorphism; note that the injectivity of $\Phi$ would follow from that of $\Phi'$. 
Proving that $\Phi$ is an algebra homomorphism can be done easily, since 
the difference between the products on $\cS(\fS')$ and $\SS$ is given by the cobraided structure, and  $\Phi$, for the bigon, preserves the cobraided structure. Here are the details.

To simplify notations we identify $\SS$ with $\cS(\fS')$ by $f$. For two elements $x,y$ in either of these algebras, we use $xy$ to denote the ordinary product of $\cS(\fS')$, and $x*y$ to denote the product of $\SS$.
We need to prove 
\be 
\Phi'(x *y) = \Phi'(x) *  \Phi'( y), \quad \text{ for $x,y \in \cS_\heta(\fS')$}. \label{eqAlg}
\ee
As string-ly diagrams span $\cS_\heta(\fS')$, we can assume that both $x, y$ are string-ly diagrams. For $i=1,2$, cutting $\fS'$ along an ideal arc $a_i'$ in the interior of $\fS'$ isotopic to $a_i$ yields the cutting homomorphism $\Theta_{a_i'} : \cS(\fS') \to \cS(\PP_2) \ot \cS(\fS')$, which equips $\cS(\fS')$ a structure of a left comodule over the Hopf algebra $\cS(\PP_2)$ (see \cite{LS21}); we denote $\Theta_{a_i'}$ by $\Delta_{a_i}$. 
Assume that
\be 
\Delta_{a_1}(x) = \sum x' \ot x'', \qquad \Delta_{a_2}(y) = \sum  y' \ot y''. \label{eq11}
\ee
By \cite[Theorem 8.2]{LS21},  the product in $\cS(\fS)$ is given by
\be 
x*y = \sum \rho_\hq (x' \otimes y') x'' y''. \label{eqProd}
\ee

Using Lemma \ref{rcut} applied to the thickening of $\fS'$ \eqref{marked_3-manifold_from_surface}, where the cutting disk $D$ is $a_i' \times (-1,1)$, we have
$$  \Delta_{a_1}(x^{(N)}) = \sum (x')^{(N)} \ot (x'')^{(N)}, \qquad \Delta_{a_2}(y^{(N)}) = \sum  (y')^{(N)} \ot (y'')^{(N)}.$$
By the induction hypothesis we have $\Phi'(x)= x^{(N)}$ and $\Phi'(y) = y^{(N)}$.
Observe that
\begin{align*}
	\Phi'(x)  * \Phi'( y) &= x^{(N)} * y^{(N)} \\
	&= \sum \rho_\home (  (x')^{(N)} \ot  (y')^{(N)}  )\, \left[   (x'')^{(N)}  (y'')^{(N)} \right] \quad (\because \mbox{\eqref{eqProd} applied to $x^{(N)}$ and $y^{(N)}$}) \\
	&= \sum \rho_\heta (  x' \ot  y'  )\, \left[  \Phi' (x'' y'')  \right] \quad (\because \mbox{\eqref{eqrho} of Lemma \ref{rbigon}(c)}) \\
	&= \Phi' \left ( \sum \rho_\heta (x' \otimes y') x'' y''  \right) \\ 
    & = \Phi'(x*y). \quad (\because \mbox{\eqref{eqProd}})
\end{align*}

Let us now prove that if $\al\in \SS$ is a string-ly diagram, considered as an element of $\cS_\heta(\fS)$, then 
\be \Phi([\al]_{\hat\eta})= [\al^{(N)}]_{\hat\omega}. \label{eq15}
\ee
For this, first recall the inverse $f^{-1}: \cS(\fS) \to \cS(\fS')$. Cutting along $a_1$ and $a_2$ gives the map $\Theta_{a_1, a_2}: \SS \to \cS(\PP_3) \ot \cS(\fS')$. 
Let  $h: \PP_3 \embed \PP_2$ be the embedding obtained by removing (i.e. `filling in') the puncture between $e_1$ and $e_2$, and $
\varepsilon_{\PP_3}$ be the composition $$\varepsilon_{\PP_3}\,:\, \cS(\PP_3) \xrightarrow{h_*} \cS(\PP_2) \xrightarrow{
\varepsilon}R.$$ Then 
\cite[Proposition 8.1]{LS21} says that
\be  f^{-1} = (
\varepsilon_{\PP_3} \ot \id)\circ \Theta_{a_1, a_2}: \SS \to \cS(\fS').
\label{eqfinv}
\ee
Observe that, for any string-ly diagram $\alpha$ in $\PP_3$, 
\begin{align*}
    \varepsilon_{\PP_3,\home}([\alpha^{(N)}]_\home) = \varepsilon_\home(h_{*,\home}([\alpha^{(N)}]_\home))
    & = \varepsilon_{\home}([(h(\alpha))^{(N)}]_\home) \\
    & = \varepsilon_{\home}( \Phi([h(\alpha)]_\heta ) ) \qquad (\because \mbox{Lemma \ref{rbigon}(b)}) \\
    & = \varepsilon_\heta( [h(\alpha)]_\heta ) \qquad (\because \mbox{$\Phi$ for bigon is compatible with $\varepsilon$,} \\
    & \hspace{40mm} \mbox{by Lemma \ref{rbigon}(a)}) \\
    & = \varepsilon_\heta (h_{*,\heta} ([\alpha]_\heta))
    = \varepsilon_{\PP_3,\heta}([\alpha]_\heta),
\end{align*}
hence 
\be
\varepsilon_{\PP_3
,\home}([\al^{(N)}]_{\hat\omega}) = 
\varepsilon_{\PP_3
, \heta}([\al]_{\hat\eta}). 
\label{eqep2}
\ee 
Let us prove \eqref{eq15}.  Assume that $\Theta_{a_1, a_2}([\al]_{\hat\eta}) = \sum [\al']_{\hat\eta} \ot [\al'']_{\hat\eta}$. By Lemma 
\ref{rcut}, 
$$  \Theta_{a_1, a_2}([\al^{(N)}]_{\hat\omega}) = \sum [(\al')^{(N)}]_{\hat\omega} \ot [(\al'')^{(N)}]_{\hat\omega}.
$$

Recall that $\Phi$ is defined to intertwine $f$. Using \eqref{eqfinv} and \eqref{eqep2},
\begin{align*}
	f^{-1} \Phi(([\al]_{\hat\eta})) &= \Phi'( f^{-1}([\al]_{\hat\eta})) = \Phi' \left ( \sum 
    \varepsilon_{\PP_3
    , \heta} ([\al']_{\hat\eta}) \,  [\al'']_{\heta} \right) \\
	&= \sum 
    \varepsilon_{\PP_3
    , \heta} ([\al']_{\heta}) \,  [(\al'')^{(N)}]_{\hat\omega} = \sum 
    \varepsilon_{\PP_3
    , \home} ([(\al')^{(N)}]_{\home}) \,  [(\al'')^{(N)}]_{\home} = f^{-1} ([\al^{(N)}]_{\home}).
\end{align*} 
Applying $f$ to the above, we get \eqref{eq15}. This completes the proof of part (a) for (the thickenings of) essentially bordered pb surfaces.

\subsection{Proof of Theorem \ref{thmFrob}(b) when both $\MN$ and $(M',\cN')$ are thickenings of essentially bordered pb surfaces}\label{sub-proof-Fro-b}
Suppose that $\MN$ and $(M',\cN')$ are thickening of essentially bordered pb surfaces $\fS$ and $\fS'$. 
We need to prove that for every morphism $f: \MN \embed (M', \cN')$ (in the category $\mathbb{M}$ (\OldS\ref{ss.marked})) and all $x\in 
\cS\MN
$, we have
\begin{align}
\label{Thm_6_1b_surface}
    \Phi\circ f_{*, \heta}(
    [x]_\heta)  = f_{*,\home} \circ \Phi (
    [x]_\heta). 
\end{align}
We can assume that $x$ is a string-ly stated $n$-web due to Lemma \ref{lem-span-arcs}(a). Note that $f_* : \cS\MN \to \cS(M',\cN')$ sends a string-ly stated $n$-web to a string-ly stated $n$-web. Thus, the left hand side of \eqref{Thm_6_1b_surface} is
\begin{align*}
    \Phi\circ f_{*, \heta}([x]_\heta) = \Phi [f_*(x)]_\heta = [(f_*(x))^{(N)}]_\home,
\end{align*}
where for the latter equality we applied Theorem \ref{thmFrob}(a) for essentially bordered pb surfaces, which we proved in the last subsection, to the string-ly stated $n$-web $f_*(x)$ in the thickening of $\fS'$. The right hand side of \eqref{Thm_6_1b_surface} is
\begin{align*}
f_{*,\home} \circ \Phi([x]_\heta) = f_{*,\home} ([x^{(N)}]_\home)= [(f_*(x))^{(N)}]_\home,
\end{align*}
where for the former equality we applied Theorem \ref{thmFrob}(a) for the string-ly stated $n$-web $x$ in the thickening of $\fS$. So \eqref{Thm_6_1b_surface} is proved.

 \qed

\subsection{Proof of Theorem \ref{thmFrob}(a) for essentially marked $3$-manifolds}\label{ss.Proof_of_Theorem_6_1a_for_3-manifolds}

We aim to construct a map $\Phi_{\MN} : \cS_\heta\MN \to \cS_\home\MN$, for an essentially marked 3-manifold $\MN$ (\OldS\ref{ss.marked}, \OldS\ref{subsec:essentially_marked_3-manifolds}). Let $\alpha=\bigsqcup_{i=1}^{r}\alpha_i$ be a stated $n$-web in $\MN$, where each $\alpha_i$ is a connected component of $\alpha$.
For every $1\leq i\leq r$, we use an embedded path $p_i\colon[0,1]\rightarrow M$ to connect $\cN$ and $\alpha_i$ such that $p_i\cap \alpha=p_i\cap \alpha_{i} = \{p_i(0)\}$, $p_i\cap \cN= \{p_i(1)\}$ for $1\leq i\leq r$ and 
$p_i\cap p_j=\emptyset$ for $1\leq i\neq j\leq r$.
Let $U(\alpha)$ be a regular open neighborhood of 
$\alpha\cup \left(\bigsqcup_{i=1}^{r}p_i\right)$, which 
is diffeomorphic to the thickening of an essentially bordered pb surface (see Figure \ref{fig-web-neighborhood}). 
Theorem \ref{thmFrob}(a) for essentially bordered pb surfaces which we proved in \OldS\ref{ss.Proof_of_Thm_6_1a_for_surfaces} yields
the Frobenius homomorphism
$\Phi_{U(\alpha)}\colon\cS_\heta(U(\alpha))\rightarrow \cS_\home(U(\alpha))$. 
We use $l$ to denote the proper embedding from 
$
U(\alpha)
$ to $M$, which induces a $\Zhq$-linear homomorphism $l_*\colon\cS(U(\alpha)
)\to\cS\MN$.
For $[\alpha]_\heta \in \cS_\heta\MN$, we define its image under the sought-for map $\Phi_\MN$ as  \begin{align}
    \label{Thm_6_1a_3-manifold_def}
    \Phi_\MN([\alpha]_{\heta}):=l_{*,\home}(\Phi_{U(\alpha)}([\alpha]_{\heta}))\in\cS_\home\MN,
\end{align}
where $[\alpha]_\heta$ in the right hand side is an element of $\cS_\heta(U(\alpha))$.

\begin{figure}[htpb]
	\includegraphics[height=4.5cm]{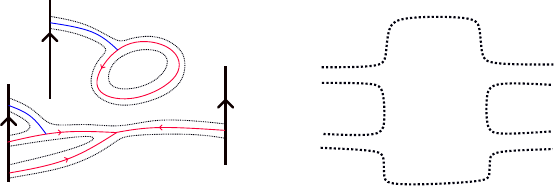} 
	\caption{Left picture: an example of the regular open neighborhood of $\alpha\cup\left(\bigsqcup_{i=1}^{r}p_i\right)$. Right picture: the local picture of $U$ for relation \eqref{e.pm}.}\label{fig-web-neighborhood}
\end{figure}

We will show that $\Phi_\MN([\alpha]_\heta)$ is independent of the choice of $\{p_i\mid 1\leq i\leq r\}$. Let
$\{p_i'\mid 1\leq i\leq r\}$ be another choice.
It is clear that $\Phi_\MN([\alpha]_\heta)$ 
is invariant under the isotopy of $\alpha
\cup \left(\bigsqcup_{i=1}^{r}p_i\right)$.
So, for each $1\leq i\leq r$, we can isotope $p_i(0,1)$ such that $\bigsqcup_{i=1}^{r}p_i$ intersects $\bigsqcup_{i=1}^{r}p_i'$ 
at finitely many points. 
Let $U(\alpha)'$ (resp. $U(\alpha)''$) be a regular open neighborhood of 
$\alpha\cup \left(\bigsqcup_{i=1}^{r}p_i'\right)$
(resp. $\alpha\cup \left(\bigsqcup_{i=1}^{r}p_i\right)\cup \left(\bigsqcup_{i=1}^{r}p_i'\right)$), which is the thickening of an essentially bordered pb surface. Thus we have Frobenius homomorphisms $\Phi_{U(\alpha)'}$ and $\Phi_{U(\alpha)''}$, by the result of \OldS\ref{ss.Proof_of_Thm_6_1a_for_surfaces}. 
We use $l'$ (resp. $l''$) to denote the proper embedding from 
$U(\alpha)'$ (resp. $U(\alpha)''$) to 
$M$. 
We use $f$ (resp. $f'$) to denote the proper embedding from 
$U(\alpha)$ (resp. $U(\alpha)'$) to $U(\alpha)''$. 
Then $l''\circ f=l$ and $l''\circ f'=l'$. 
The maps $l',l'',f,f'$ induce maps between the stated ${\rm SL}_n$-skein algebras and their root of unity versions, denoted with the subscript $*$ and with $\home$ and $\heta$. 
For $[\alpha]_\heta \in \cS_\heta(U(\alpha))$, observe that
\begin{align*}
    l_{*,\home}(\Phi_{U(\alpha)}([\alpha]_{\heta}))&=l''_{*,\home}(f_{*,\home}(\Phi_{U(\alpha)}([\alpha]_{\heta}))) 
    \quad (\because \mbox{$l=l''\circ f$})\\
    &=l''_{*,\home}(\Phi_{U(\alpha)''}(f_{*,\heta}([\alpha]_{\heta}))) \quad (\because \mbox{ 
    \eqref{Thm_6_1b_surface} of \OldS\ref{sub-proof-Fro-b}})\\
    &=l''_{*,\home}(\Phi_{U(\alpha)''}(f'_{*,\heta}([\alpha]_{\heta}))) \quad (\because \mbox{$f(\alpha)=
    f'(\alpha)$; here $[\alpha]_\heta \in \cS_\heta(U(\alpha)')$})\\
    &=l''_{*,\home}(f'_{*,\home}(\Phi_{U(\alpha)'}([\alpha]_{\heta})) \quad (\because \mbox{ 
    \eqref{Thm_6_1b_surface} of \OldS\ref{sub-proof-Fro-b}})\\
    &=l'_{*,\home}(\Phi_{U'(\alpha)}([\alpha]_{\heta})) \quad (\because \mbox{$l'=l''\circ f'$}).
\end{align*}
This shows that $\Phi_\MN([\alpha]_\heta)$ defined in \eqref{Thm_6_1a_3-manifold_def} is independent of the choice of $\{p_i\mid 1\leq i\leq r\}$.

Since any isotopy of $\alpha$ can be extended 
to an isotopy 
of $\alpha\cup \left(\bigsqcup_{i=1}^{r}p_i\right)$, 
it follows that $\Phi_\MN(
[\alpha]_\heta)$  
is invariant under the isotopy of $\alpha$.

We now show that \eqref{Thm_6_1a_3-manifold_def} extends to a well-defined map $\Phi_\MN : \cS_\heta\MN \to \cS_\home\MN$. Let $\sum_{1\leq 
i\leq m} c_i[\alpha_i]_{\heta}=0\in\cS_{\heta}\MN$ be one of the relations \eqref{e.pm}--\eqref{e.crossp-wall}, where $c_i\in\mathbb C$ and $\alpha_i$ is a stated $n$-web in $\MN$ for each $1\leq i\leq m$.
The goal is to show that $\Phi_\MN$ respects this relation, i.e. to show $\sum_{1\le i \le m} c_i \Phi_\MN([\alpha_i]_\heta)=0$. 
Note that all $\alpha_i$, $1\leq i\leq m$, are identical to each other except in a small open 
set $W$, which is diffeomorphic 
to $(0,1)^3$ or $\left((0,1)\times(0,1]
\right)\times (0,1)$. 
As before,
let $p$ be the union of paths that connect components of $\alpha_1$ to $\cN$. We require all these paths have no intersection with $W$.
Then these paths also connect components of $\alpha_i$, $1\leq i\leq m$, to $\cN$.
For each $1\leq i\leq m$,
let $U_i$ be a regular open neighborhood of 
$\alpha_i\cup p$. 
We require that $U_i\setminus W=U_j\setminus W$
for $1\leq i,j\leq m$. 
Then $U_i\cup W=U_j\cup W$
for $1\leq i,j\leq m$, denoted 
by $U$ (see Figure \ref{fig-web-neighborhood}).
We have that $U$ is isomorphic to the thickening of an essentially bordered pb surface. 
For each $1\leq i\leq m$, we use $l^{(i)}$ (resp. $f^{(i)}$) to denote the proper embedding from $U_i$ to $M$ (resp. from $U_i$ to $U$). We 
use $l$ to denote the proper embedding from $U$ to $M$. 
Note that $\sum_{1\leq i\leq m} c_i[\alpha_i]_{\heta}=0\in\cS_{\heta}(U)$. 
Regarding each $[\alpha_i]_\heta$ as an element of $\cS_\heta(U_i)$, we have
\begin{align*}
\sum_{1\leq 
i\leq m} c_i \, l^{(i)}_{*,\home}(\Phi_{U_i}( [\alpha_i]_{\heta}))
   &=\sum_{1\leq 
   i \leq m} c_i \, l_{*,\home}(
   f^{(i)}_{*,\home}
   (\Phi_{U_i}( [\alpha_i]_{\heta})))
    \quad (\because \mbox{$l^{(i)}=l\circ f^{(i)}$})\\
    &=\sum_{1\leq
    i \leq m} c_i \, l_{*,\home}(
    \Phi_{U}
   (f^{(i)}_{*,\heta}( [\alpha_i]_{\heta})))
    \quad  (\because \mbox{ 
    \eqref{Thm_6_1b_surface} of \OldS\ref{sub-proof-Fro-b}})\\
    &= l_{*,\home}\left(\sum_{1\leq 
    i \leq m} c_i \,
    \Phi_{U}
   ([\alpha_i]_{\heta}) \right) \quad \begin{array}{r} (\because \mbox{$f^{(i)}$ is an embedding;} \\ \mbox{here $[\alpha_i]_\heta \in \cS_\heta(U)$}) \end{array} \\
    &=0. \quad (\because \mbox{$\Phi_U$ is well-defined})
\end{align*}
This shows that $\Phi_\MN$ in \eqref{Thm_6_1a_3-manifold_def} 
respects the relations \eqref{e.pm}-\eqref{e.crossp-wall}, hence extends to a well-defined map $\Phi_\MN : \cS_\heta\MN \to \cS_\home\MN$, as desired. The equality $\Phi_\MN([\alpha]_\heta) = [\alpha^{(N)}]_\home$ for each string-ly $n$-web $\alpha$ follows from \eqref{Thm_6_1a_3-manifold_def} and the corresponding property of $\Phi_{U(a)}$ proved in \OldS\ref{ss.Proof_of_Thm_6_1a_for_surfaces} (see \eqref{eq15}).

\subsection{Proof of Theorem \ref{thmFrob}(b) for essentially marked $3$-manifolds}
The arguments in 
\OldS\ref{sub-proof-Fro-b} work verbatim here.

\subsection{Proof of Theorem \ref{thmFrob}(c)}\label{ss.Proof_of_Thm_6_1c}

As in \OldS\ref{subsec.cutting_homomorphism}, let $\MN$ be an essentially marked 3-manifold, and 
assume that $D$ is a properly embedded closed disk in $M$ such that $D$ does not meet the closure of $\cN$, and $\beta$ is an embedded oriented open interval in $D$. Recall that we have a cutting homomorphism $\Theta_{(D,\beta)}: \cS\MN \to \cS(\Cut_{(D,\beta)}\MN)$ (Theorem \ref{thm.cutting_homomorphism_for_3-manifolds}). We need to show that for any
$\alpha\in 
\cS\MN$
\be 
\Phi ( \Theta_{(D,\beta),\heta} (
[\al]_\heta)) = \Theta_{(D,\beta),\home} (\Phi(
[\al]_\heta)). \label{eq16}
\ee
We can assume that $\al$ is a string-ly web in $\MN$ because of Lemma \ref{lem-span-arcs}(a). Then $\Phi([\al]_{\hat\eta})= [\al^{(N)}]_{\hat\omega}$ by Theorem \ref{thmFrob}(a) proved in \OldS\ref{ss.Proof_of_Theorem_6_1a_for_3-manifolds}. 
Now \eqref{eq16} follows from \eqref{eqcutN} of Lemma \ref{rcut} in \OldS\ref{ss.compatibility_with_cutting_homomorphism}. \qed

\subsection{Proof of Theorem \ref{thmFrob}(d)}  The statement follows, since for a string-ly $n$-web $\al$,  we have 
$$[\al^{(N)}]_{\mathsf{homol}  }= N[\al]_{\mathsf{homol}  } \ \ \text{in $H_1(M,\cN;\BZ_n)$}. \qed$$ 

\subsection{Proof of Theorem \ref{thmFrob}(e)}\label{ss.Proof_of_Thm_6_1e}  
If $\al$ is a bad arc, then 
$$\Phi([\al]_{\hat\eta})= [\al^{(N)}]_{\hat\omega} =\home^{m}([\al]_{\hat \omega})^N \in \cI^\bad$$
for some integer $m$.

Thus, $\Phi(\cI^\bad) \subset \cI^\bad$. Hence
$\Phi$ descends to a quotient map $\bPhi: \bSe \to \bSo$. \qed

\subsection{Monomial homomorphisms between quantum tori} \label{secMono} To prove part (f) of Theorem \ref{thmFrob} we make a digression to discuss monomial homomorphisms between quantum tori.

Assume that $Q$ is an integer antisymmetric $r\times r$ matrix, and $Q'$ is an integer antisymmetric $r'\times r'$ matrix. We have the quantum tori $\bT_\hq(Q)$ and $\bT_\hq(Q')$ given by the presentation \eqref{eqTorus}, and the Frobenius homomorphism $\Phi^\bT$ for quantum tori given by \eqref{eqFrDef}.

 An $R$-algebra homomorphism $f: \bT_\hq(Q) \to \bT_\hq(Q')$ is 
 called a {\bf monomial} homomorphism if there are $\bk_i\in \BZ^{r'}$, for $i=1, \dots, r$, such that $f(x_i) = x^{\bk_i}$; where $x^{\bk_i}$ is the Weyl-normalized Laurent monomial defined in \eqref{Weyl-ordering}. For a non-zero complex number $\hat{\xi}$, denote by $f_{\hat{\xi}} : \bT_{\hat{\xi}}(Q) \to \bT_{\hat{\xi}}(Q')$ the map obtained from $f$ by evaluating $\hq$ at $\hat{\xi}$. From the definitions, together with the fact that $f$ sends Weyl-normalized Laurent monomials to Weyl-normalized Laurent monomials (see e.g. \cite[Lem.3.19]{kim2024naturality}), it is clear that
\be 
f_\home \circ \Phi^\bT = \Phi^\bT \circ f_\heta.
\label{eqMono}
\ee

If $K$ is an integer $r \times r'$ matrix such that $K Q' K^
{\rm t}= Q$, where $K^
{\rm t}$ is the transpose of $K$, then there is a monomial algebra homomorphism
$\psi_K: \bT_\hq(Q) \to \bT_\hq(Q')$ given by $\psi_K(x_i)= x^{\bk_i}$, where $\bk_i$ is the $i$-th row of $K$.

\subsection{Proof of Theorem \ref{thmFrob}(f)}  
We will prove the statement in the following order: 
Step 1:~$\fS=\PP_3$ and $\tr= \btr^A$, 
Step 2:~$\fS=\PP_3$ and $\tr= \btr^X$, 
Step 3:~general $\fS$ and $\tr= \btr^X$,  
Step 4:~general $\fS$ and $\tr= \tr^X$, and 
Step 5:~$\tr= \tr^A$ or $\btr^A$. This is the order used in the constructions of quantum traces in \cite{LY23}. Step 1 is 
the most involved, but all other steps are easy.

{\bf Step 1.} 
Let $\fS= \PP_3$. We will show that, for all $x\in \cS_\heta(\PP_3)$,
\be (\btr^A \circ \bPhi) (x)  = (\Phi^\bT \circ\btr^A)(x). \label{eqtrA}
\ee
In this case, $ \btr^A: \cS(\PP_3) \to \bAPq $ is injective (Theorem \ref{thm-Atr}(a)), and the reduced $A$-version quantum torus
is
$$
\bAPq= \bT(\bmP)= R\la a_v ^{\pm 1}, v\in \bV_{\mathbb P_3}\ra/ ( a_v a_{v'} = \hq^{\bmP(v,v')} a_{v'} a_v).
$$
Here $\bV_{\mathbb P_3}$ is a finite set introduced in  
\eqref{ol_V_P3} of  \OldS\ref{subsec.SLn_quantum_trace_maps} and $\bmP:\bV \times \bV$ is an integer antisymmetric matrix $\bmP_\lambda$ in \eqref{eq-relation-P-Q} for the unique ideal triangulation $\lambda$ of $\mathbb{P}_3$. 
From 
Theorem \ref{trace}, for each $v\in \bV_{\mathbb P_3}$ there exists$\bar\gaa_v\in \cS(\PP_3)$ 
such that $\btr^A (\bar\gaa_v)= a_v$. We will first prove \eqref{eqtrA} for $x=\bar\gaa_v$. We will then see that 
the general case 
follows easily  
using Theorem \ref{trace}.

For each oriented arc $c_i$ of Figure \ref{figtri}, an open tubular neighborhood $U_i$ of $c_i$ in $\PP_3$ is diffeomorphic to $\PP_2$, so the embedding $\PP_2 \approx U_i \hookrightarrow \PP_3$ yields
\begin{figure}[htpb]
	\includegraphics[height=1.5cm]{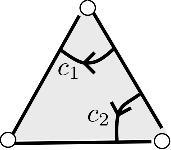} 
	\caption{Arcs $c_1$ and $c_2$}
	\label{figtri}
\end{figure} 
an algebra map $f_{i}: \cS(\PP_2) \to \cS(\PP_3)$, which commutes with $\Phi$ since  if $\al$ is a string-ly diagram 
in $\PP_2$ or $\PP_3$ then $\Phi(
[\al]_\heta) = 
[\al^{(N)}]_\home$ ($\because$ Lemma \ref{rbigon}(b), Theorem \ref{thmFrob}(a)).  
Thus the composition  
$$ \bar f_i: \cS(\PP_2) \xrightarrow{f_i} \cS(\PP_3) \onto  \overline\cS(\PP_3) $$
also commutes with the Frobenius homomorphism ($\because$ Theorem \ref{thmFrob}(e)). Hence by Lemma \ref{rbigon}
(a), if $\fm\in \cS(\PP_2)$ is a quantum minor, then
\be 
\bPhi(\bar f_i(\fm_\heta))= (\bar f_i(\fm_\home))^N. \label{eqminorr}
\ee
By \cite[Lemma 10.3]{LY23}, there are quantum minors $\fm, \fn\in \cS(\PP_2)$ such that $\bar f_1(\fm)$ and $\bar f_2(\fn)$ are $q$-commuting and
\be 
\bar\gaa_v = [ \bar f_1(\fm) \, \bar f_2(\fn)]_\Weyl,  
\label{eqgv}
\ee
where $[\sim]_{\rm Weyl}$ is the Weyl-normalized product defined in Remark \ref{rem.Weyl-ordering}.
Let us prove \eqref{eqtrA} for $x=\bar\gaa_v$. By the definition \eqref{eqFrDef} of $\Phi^\bT$,
\begin{align}
	\Phi^\bT (\btr^A(\bar\gaa_{v,\heta}))&=  \Phi^\bT (a_{v, \heta})= (a_{v, \home})^N.
	\label{eqRS}
\end{align}
Observe that
\begin{align*}
	\btr^A(\bPhi (\bar\gaa_{v,\heta}))&= \btr^A \left(  \left[  \bPhi (  \bar f_1(\fm_\heta)  )\, \bPhi (  \bar f_2(\fn_\heta)  ) \right] _\Weyl \right) \qquad (\because \mbox{\eqref{eqgv}}) \\
	& =
    \btr^A \left(  \left[   (  \bar f_1(\fm_\home)  )\,   \bar f_2(\fn_\home)  ) \right] _\Weyl \right)^N \qquad (\because \mbox{\eqref{eqminorr}}) \\
	&= (\btr^A  (\bar\gaa_{v,\home}))^N \qquad (\because \mbox{\eqref{eqgv}}) \\
    & = (a_{v, \home})^N.
\end{align*}
Comparing with \eqref{eqRS}, we get the sought-for \eqref{eqtrA} for $x=\bar\gaa_v$, and hence also for all $x$ 
in the subalgebra $A$ of $\cS_\heta(\PP_3)$ generated by $\bar\gaa_v, v \in \bV_{\mathbb P_3}$. 

To finish proving \eqref{eqtrA} completely, now
assume that $x\in \overline\cS_\heta(\PP_3)$. From Theorem \ref{trace}, there are $x_1, x_2\in A$ 
such that $x x_1 = x_2$  
holds and that $x_1$$\neq 0$ 
is a product of $\bar\gaa_v$'s.
Denote the homomorphisms of the two sides of \eqref{eqtrA} by $F = \btr^A \circ \bPhi$ and $G = \Phi^\bT \circ\btr^A$.
From $F(x_i)= G(x_i)$  and 
$x x_1 = x_2$ we get
$F(x) G(x_1) = F(x) F(x_1) = F(x x_1) = F(x_2) = G(x_2) = G(x x_1) = G(x) G(x_1)$, hence 
\be 
( F(x) - G(x)
) G(x_1) =0. \label{eq50}
\ee
Note that $G= \Phi^\bT\circ \btr^A$ is injective since  both $\Phi^\bT$ and $ \btr^A$ are (Theorem \ref{thm-Atr}(a)). Hence $G(x_1) \neq 0$. As $\bT_\home(\bmP)$ is a domain, from \eqref{eq50} we have $F(x) = G(x)$. So \eqref{eqtrA} holds.

{\bf Step 2.} 
Still let $\fS= \PP_3$. We will prove 
\be \btr^X \circ \bPhi  = \Phi^\bT \circ \btr^X. \label{eqtrX}
\ee
We have $ \btr^X: \cS(\PP_3) \to \bsX_\hq(\PP_3)$ (Theorem \ref{thm-X}(a)) where 
$$ \bsX_\hq(\PP_3) = \bT_\hq (\bmQ)= R\la x_v ^{\pm 1}, v\in \bV_{\mathbb P_3}\ra/ ( x_v x_{v'} = \hq^{\bmQ(v,v')} x_{v'} x_v),$$
which has $\{ x^\bk \mid \bk \in \BZ^{\bV_{\PP_3}}\}$ as a free $R$-basis; here $\bmQ$ denotes $\bmQ_\lambda$ in \eqref{eq.bmQ} for the unique ideal triangulation $\lambda$ of $\PP_3$. From Theorem \ref{thm-Atr}(a),
there is a monomial algebra homomorphism $\bpsi: \bT_\hq (\bmP) \to  \bT_\hq (\bmQ)$, such that $\btr^X = \bpsi \circ \btr^A$.  Hence \eqref{eqtrX} follows from \eqref{eqtrA}  
proved in Step 1, the fact that $\bpsi$ commutes with $\Phi^\bT$, and the injectivity of $\bpsi$.

{\bf Step 3.} Now let $\fS$ 
be a triangulable essentially bordered pb surface. Let $\lambda$ be an ideal triangulation of $\fS$. We want to prove
\be \btr^X _\lambda \circ \bPhi  = \Phi^\bT \circ \btr^X_\lambda. \label{eqtrXX}
\ee

By cutting $\fS$ along the interior edges in $\lambda$ we get a collection $
\mathcal{F}_\lambda$ of triangles. We can identify  $ \bigotimes_{\tau \in 
\mathcal{F}_\lambda} \bsX_\hq(\tau)$ naturally with the quantum torus $\bT_\hq( \tilde Q )$, where $\tilde Q = \bigoplus _{\tau \in 
\mathcal{F}_\lambda} \bmQ(\tau)$, and $\bmQ(\tau)$ is the $\bmQ$ matrix of $\tau$ \eqref{eq.bmQ}.
In \cite[\OldS 12.2]{LY23}  it is shown that there is a monomial algebra embedding 
$$f: \bXS\embed\bT( \tilde Q )= \bigotimes_{\tau \in 
\mathcal{F}_\lambda} \bsX(\tau),$$ and 
that the quantum trace $\btr^X_\lambda$ is the composition 
$$ 
\bSS  \xrightarrow {\overline \Theta} \bigotimes_{\tau \in 
\mathcal{F}_\lambda} \overline{\cS}_\hq(\tau) \xrightarrow {\ot \btr^X_\tau} \bigotimes_{\tau \in 
\mathcal{F}_\lambda} \bsX_\hq(\tau)\xrightarrow {f^{-1}} \bXS.
$$
Here ${\overline \Theta}$ is the composition of the cutting homomorphisms for all the interior edges of $\lambda$, and $f^{-1}$ denotes the inverse of $f : \bXS \to f(\bXS)$; in particular, part of the above statement says that the image of the composition $(\ot \btr^X_\tau) \circ \overline{\Theta}$ lies in $f(\bXS) \subset \bigotimes_{\tau \in \mathcal{F}_\lambda} \bsX_\hq(\tau)$.
Evaluating at the roots of unity $\heta$ and $\home$, and connecting by the Frobenius homomorphisms, 
we consider the diagram
$$
\begin{tikzcd}
	\bSSe \arrow[r,"\overline \Theta"] \arrow[d,"\overline \Phi"] & \displaystyle{ \bigotimes_{\tau \in 
    \mathcal{F}_\lambda} \overline{\cS}_\heta(\tau)} \arrow[d,"\ot \overline \Phi "]  
	\arrow[r,"\ot \btr^X_\tau"]
	& \displaystyle{\bigotimes_{\tau \in 
    \mathcal{F}_\lambda} \bsX_\heta(\tau)}
	\arrow[d,"\ot  \Phi^\bT "] \arrow[r,"(f_\heta)^{-1}"] & \bsX_\heta(\fS,\lambda) \arrow[d,"  \Phi^\bT "]
	\\
	\bSSo \arrow[r,"\overline \Theta"]  & \displaystyle{ \bigotimes_{\tau \in 
    \mathcal{F}_\lambda} \overline{\cS}_\home(\tau)}   
	\arrow[r,"\ot \btr^X_\tau"]
	& \displaystyle{\bigotimes_{\tau \in 
    \mathcal{F}_\lambda} \bsX_\home(\tau)} \arrow[r,"(f_\home)^{-1}"] &\bsX_
    {\home}(\fS,\lambda)
\end{tikzcd}
$$
The 
left square  is commutative by 
Theorem \ref{thmFrob}(c) proved in \OldS\ref{ss.Proof_of_Thm_6_1c}, Theorem \ref{thmFrob}(e) proved in \OldS\ref{ss.Proof_of_Thm_6_1e}, and the last paragraph of \OldS\ref{ss.reduced_stated_SLn-skein_algebras}. The 
middle square is commutative by Step~2. The 
right square is commutative by \eqref{eqMono} of \OldS\ref{secMono}.
Hence the composition of the three squares is commutative, which proves the sought-for \eqref{eqtrXX}.

{\bf Step 4.} 
Let $\fS$ and $\lambda$ be as in Step 3. 
We 
should prove 
\be \tr^X _\lambda \circ \Phi  = \Phi^\bT \circ \tr^X_\lambda,
\label{eqtrXXX}
\ee
where $\tr^X_\lambda$ is the extended $X$-version quantum trace map (Theorem \ref{thm-Atr}(b)). 
The construction of $\tr^X _\lambda$ in \cite{LY23} (see \cite{LY2} for the ${\rm SL}_2$ case) is via the reduced quantum trace of 
an exteded surface, as follows.
There are (i)  a surface $\fS^*$ with a triangulation $\lambda^*$ (see \OldS\ref{subsec.SLn_quantum_trace_maps}), (ii) a strict embedding $\iota: \fS \embed \fS^*$, and (iii) a monomial algebra embedding $f: \cX_\hq(\fS,\lambda) \embed \bsX_\hq(\fS^*,\lambda^*) $ such that $\tr^X _\lambda$ is the composition 
$$ 
\SS  \xrightarrow {\iota_*   } \cS(\fS^*) \onto  \overline{\cS}_{\hat q}(\fS^*) \xrightarrow {\btr^X_{\lambda^*}   }  \bsX_\hq(\fS^*,\lambda^*) \xrightarrow {f^{-1} } \cX_\hq(\fS,\lambda),
$$  
where $f^{-1}$ is suitably understood as before. 
Since 
the map for each arrow in the above composition commutes with the appropriate Frobenius homomorphism (by Step 3, Theorem \ref{thmFrob}(b) proved in \OldS\ref{sub-proof-Fro-b}, Theorem \ref{thmFrob}(e) proved in \OldS\ref{ss.Proof_of_Thm_6_1e}, \eqref{eqMono} of \OldS\ref{secMono}), we get \eqref{eqtrXXX}.

{\bf Step 5.} 
Let $\fS$ and $\lambda$ be as in Step 3, and we further assume that $\fS$ 
has no interior punctures. We 
should prove 
\be \btr^A _\lambda \circ \bPhi  = \Phi^\bT \circ \btr^A_\lambda, \qquad \tr^A _\lambda \circ \Phi  = \Phi^\bT \circ \tr^A_\lambda. \label{eqtrAA}
\ee

By Lemma \ref{lem.LY23_A_to_X} (proved in \cite{LY23}) we have
monomial algebra embeddings 
$$
\overline{\psi}_\lambda: \overline\cA_\hq(\fS, \lambda) \embed \overline\cX_\hq(\fS, \lambda), \qquad  
\psi_\lambda: \cA_\hq(\fS, \lambda) \embed \cX_\hq(\fS, \lambda), $$
which by Theorem \ref{thm-Atr} (proved in \cite{LY23}) satisfy
$\btr^A _\lambda = 
\overline{\psi}_\lambda^{-1}  \circ \btr^X_\lambda$ and $\tr^A _\lambda =  
\psi_\lambda^{-1}  \circ \tr^X_\lambda$. Hence \eqref{eqtrAA} follows from the corresponding results for the $X$-version quantum traces, established in Steps 3 and 4, together with \eqref{eqMono} of \OldS\ref{secMono}. This completes the proof of Theorem \ref{thmFrob}.

\def\B{{\PP_2}}
\def\SqA{{ \cS(\An;R)}}
\def\SqB{{ \cS(\B;R)}}

\section{The annulus, the bigon, and cutting homomorphism}
\label{sec-annulus}

In this section we exhibit a connection between the 
representation theory of quantum groups, more specifically the materials of \OldS\ref{sec.Lusztig_Frobenius_map}, and the skein theory.
We will show that the module-trace map, defined in representation theory, can be interpreted as a cutting homomorphism in skein theory.

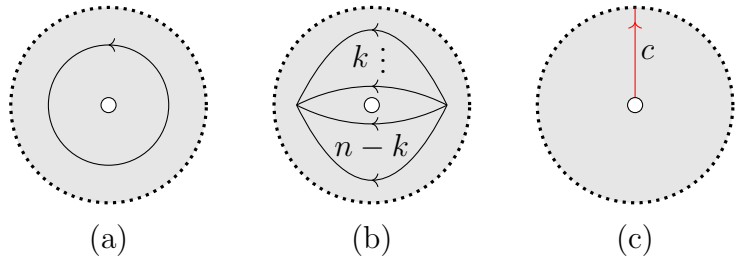
\begin{figure}
\centering
\begin{tikzpicture}[baseline=0cm,every node/.style={inner sep=2pt}]
\draw[wall,fill=gray!20,dotted] (0,0) circle[radius=1.3];
    \draw[fill=white] (0,0) circle[radius=0.1];
    \draw[decoration={markings, mark=at position 0.25 with {\arrow{>}}}, postaction={decorate}] 
        (0,0) circle[radius=0.8];
\path (0,-1.8)node{(a)};

\begin{scope}[xshift=3.5cm]
[baseline=0cm,every node/.style={inner sep=2pt}]
\draw[wall,fill=gray!20,dotted] (0,0) circle[radius=1.3];
    \draw[fill=white] (0,0) circle[radius=0.1];
    \draw[postaction={decorate, decoration={markings, mark=at position 0.5 with {\arrow{>}}}}] (1,0) parabola bend (0,1) (-1,0);
    \draw[postaction={decorate, decoration={markings, mark=at position 0.5 with {\arrow{>}}}}] 
        (1,0) parabola bend (0,0.25) (-1,0);
    \node at (0.15,0.7) {$\vdots$};
    \node at (-0.15,0.65) {$k$};
    \draw[postaction={decorate, decoration={markings, mark=at position 0.5 with {\arrow{>}}}}] (1,0) parabola bend (0,-1) (-1,0);
    \draw[postaction={decorate, decoration={markings, mark=at position 0.5 with {\arrow{>}}}}] 
        (1,0) parabola bend (0,-0.25) (-1,0);
    \node at (0,-0.55) {$n-k$};
\path (0,-1.8)node{(b)};
\end{scope}

\begin{scope}[xshift=7cm]
\draw[wall,fill=gray!20,dotted] (0,0) circle[radius=1.3];
    \draw[red, decoration={markings, mark=at position 0.85 with {\arrow{>}}}, postaction={decorate}] 
        (0,0)--(0,1.3);
    \draw[fill=white] (0,0) circle[radius=0.1];
    \node [right] at (0,0.7) {$c$};
\path (0,-1.8)node{(c)};
\end{scope}
\end{tikzpicture}
\caption{(a) The oriented core curve ${\bf a}$ of the annulus $\An$. 
\\
\hspace*{23mm} (b) The diagram $
{\bf a}_k'$ in $\An$. (c) The essential ideal arc $c$ in $\mathbb P_2$}\label{fig1} 

\end{figure}

\subsection{Stated ${\rm SL}_n$-skein algebra of the annulus}\label{ss.stated_SL_n-skein_algebra_of_the_annulus}

Recall that the {\bf annulus} $\An$, defined in \OldS\ref{subsec:punctured_bordered_surface_and_n-web}, is the twice-punctured sphere, which is a pb surface. In pictures, we depict $\An$ as an open disk with one point removed, see Figure \ref{fig1}.

Recall from \OldS\ref{sub-modified-qaantized-algebra} that for each fundamental weight $\varpi_
{k}, 
k=1, \dots, n-1$, we have the $\Ud$-module $\Lambda_q(\varpi_k)$. Let $R= \Zhq ( [n]_q!  )^{-1}$ , which is $\Zhq$ with the inverse of $ [n]_q!$ adjoined. 
Define  $\Rep_q$ as the free $R$-polynomial ring in the variables $\Lambda_q(\varpi_1), \dots, \Lambda_q(\varpi_{n-1})$:
\begin{equation}\label{def-Rep-q}
    \Rep_q:= R[ \Lambda_q(\varpi_1), \dots, \Lambda_q(\varpi_{n-1})].
\end{equation}
We can think of $\Rep_q$ as the Grothendieck ring of the category of finite-dimensional representations of $\Ud \ot_\Zhq \BQ(\hq)$, with the ground ring extended from $\BZ$ to $R$.

Here we will show that the (stated) ${\rm SL}_n$-skein algebra $\SqA$  is isomorphic to $\Rep_q$. This is due mostly to \cite{QR18} and \cite{Pou22}. See also \cite{DS}.

For each $1\leq k\leq n-1$, let ${\bf a}'_k\in \SqA$ be defined as in Figure \ref{fig1}(b), and let
\begin{align}\label{eq-alphak-loop}
    {\bf a}_k :=
		(-1)^{\binom{k}{2}+\binom{n-k}{2}}\frac{1}{[n-k]_q![k]_q!}\alpha_k'  \in \cS(\An;R)
		\end{align}

\begin{theorem}[\cite{QR18,Pou22}]
\label{twice_sphere} 
  There is a unique $R$-algebra isomorphism 
  $$\kappa\colon \Rep_q= R[ \Lambda_q(\varpi_1), \dots, \Lambda_q(\varpi_{n-1}) ] \xrightarrow{\cong} \SqA  ,$$
   given by $\kappa(\Lambda_q(\varpi_k) ) = {\bf a}_k$,   $k=1, \dots, n-1.$
\end{theorem}

\begin{proof}
For an oriented surface $\fS$, the Cautis-Kamnitzer-Morrison ({\bf CKM}) skein algebra $\cS^{\rm CKM}(\fS)$ is 
one version of an ${\rm SL}_n$-skein algebra; see \cite{Pou22}, where Poudel shows 
that  there is an $R$-algebra  isomorphism  $
\bf I: \cS^{\rm CKM}(\fS;R)\to \cS(\fS;R)$.

Among the generators of the CKM algebra are simple  oriented loops colored by 
$\varpi_k$,
$k=1, \dots, n-1$. Let  $
{\bf z}_k\in \cS^{\rm CKM}(\An;R)$ be the element represented  by the core curve of $\An$,  
as in Figure~\ref{fig1}(a), colored by 
$\varpi_k$.  By \cite[Lemma 5.8]{QR18} the map $
{\bf z}_k
\mapsto \Lambda_{q}(\varpi_
{k})$ gives an $R$-algebra isomorphism
$$\cS^{\rm CKM}(\An;R) \xrightarrow{\cong} R [ \Lambda_{q}(\varpi_1), \dots, \Lambda_{q}(\varpi_{n-1})  ]. $$
From the explicit definition of $
\bf I$ given in \cite[ 
\OldS6]{Pou22}, it is easy to check that  $
{\bf I}(
{\bf z}_k) = 
{\bf a}_k$. This proves the theorem.
\end{proof}

\def\SoB{{ \cS_\home(\B)}}

In principle, the statement of Theorem \ref{twice_sphere} by itself is just saying that $\mathscr{S}(\An;R)$ is a polynomial algebra. The representation-theoretic meaning will become clear by our investigation in the next subsection.

\subsection{Cutting homomorphism and module-trace map} We show that the module-trace map, defined in representation theory (\OldS\ref{ss.module-trace_map}, \OldS\ref{ss.module-trace_map_at_a_complex_number}),  
can be interpreted as a cutting homomorphism, defined in skein theory.

 Cutting the annulus $\An$ along the ideal arc $c$ in Figure \ref{fig1}(c) yields a bigon. This yields the cutting homomorphism (Theorem \ref{t.splitting2})
\begin{align}
\label{eq.cutting_homomorphism_of_annulus}
    \Theta_c: \SqA \to \SqB.
\end{align}
\bthm\label{thm.cutting_homomorphism_and_module-trace} 
Assume that $R= \Zhq ( [n]_q!  )^{-1}$. Under the identification $\SqA= \Rep_q$ of Theorem \ref{twice_sphere}, the cutting homomorphism $\Theta_c$ 
coincides with the module-trace map $\cT_q$ of \OldS\ref{ss.module-trace_map}. 
More precisely, 
the following diagram is commutative: 
\begin{equation}\label{commmmm}
	\begin{tikzcd}
		\Rep_q \arrow[r, "\kappa"]
		\arrow[rd,  "\cT_q" ]  
		&  \SqA  \arrow[d, "\Theta_c"] \\
		& 
        \Oq = \SqB.
	\end{tikzcd}
\end{equation}

\ethm

\brem On the conceptual level, the proof below is very simple: Let $V_k= \Lambda_q(\varpi_k)$.
We have an inclusion  $V_k \embed (V_1)^{\ot k}$ of $
\Ud$-modules. Cutting $
{\bf a}_k$ along $c$, we get a ribbon graph (in the bigon) whose Reshetikhin-Turaev operator $T_k: (V_1)^{\ot k}  \to (V_1)^{\ot k}$ is the idempotent-projection onto $V_k$. By 
\cite[Section 3.4]{LS21},   for $u\in \Ud$ we have
$\Theta_c(
{\bf a}_k)(u) = \tr_{(V_1)^{\ot k}} (T_k \circ u)$, which is
$\tr_{V_k}(u)$.

\erem

\begin{proof} 
As the $R$-algebra $\Rep_q$ is generated by $\Lambda_q(\varpi_1),\dots,\Lambda_q(\varpi_{n-1})$ (Theorem \ref{twice_sphere}), it is enough to show 
\be \cT_q (\Lambda_q(\varpi_k))= \Theta_c (\kappa(\Lambda_q(\varpi_k))), \qquad k=1, \dots, n-1.
\label{eq52}
\ee

By Lemma \ref{lem-trace-fundamental-representations-dot}, the left hand side of \eqref{eq52} 
equals the quantum elementary symmetric function $D^q_k(\buu)\in \Oq$ (defined in \eqref{eq.D_k_q} of \OldS\ref{subsec:quantum_elementary_symmetric_function}).

On the other hand, since $\kappa(\Lambda_q(\varpi_k)) = {\bf a}_k$ (Theorem \ref{twice_sphere}), the right hand side of \eqref{eq52} is $\Theta_c({\bf a}_k)$, which 
equals $D_k^{q}({\bf u})$ by Proposition \ref{r-Thet_alpha_k} below. 

\end{proof}

 \bpro\label{r-Thet_alpha_k} Assume that $R= \Zhq ( [n]_q!  )^{-1}$. In $\SqB$ we have 
 \begin{align}\label{eqal}
 	\Theta_c(
    {\bf a}_k)= D_k^{q}({\bf u}).
 \end{align}
 \epro
 
 \begin{proof}  Recall from \eqref{eq.D_k_q} that 
$$D_k^{q}({\bf u}) = \sum_{I\in \mathbb J_k}M_{I}^{I}({\bf u}),$$
where  $M^I_J(\buu)\in \Oq$, for $I,J\in \JJ_k=\{\mbox{$k$-element subsets of $\JJ=\{1,\ldots,n\}$}\}$, is the quantum determinant of the $I\times J$ submatrix of $\buu$.

Assume that $I, J\in \JJ_
k$, whose elements are listed in
any chosen order as sequences
${\bf i}=(i_1,\cdots,i_k)$ and $
{\bf j}=(j_1,\cdots,j_k)$. That is, assume that $\{i_1,\ldots,i_k\}=I$ and $\{j_1,\ldots,j_k\}=J$.

Let us first prove
\be	
	\begin{tikzpicture}[baseline=(ref.base)]
		\fill[gray!20] (-1.3,-1) rectangle (1.3,1);
		\coordinate (src) at (0.3,0);
		\begin{scope}[edge]
			\foreach \y in {-0.5,0,0.5} {
				\draw[-o-={0.6}{>}] (src) -- +(1,\y);}
		\end{scope}	
		\draw [line width =1pt] (-0.3,0)--(0.3,0);
		\node [above] at
        (0,-0.05) {\scalebox{0.8}{$n-k$}};
		
		\coordinate (src) at (-0.3,0);
		\begin{scope}[edge]
			\foreach \y in {-0.5,0,0.5} {
				\draw[-o-={0.6}{<}] (src) -- +(-1,\y);}
	\end{scope}
	
	\path (-1.3,-0.5)node[left]{$i_1$} (-1.3,0)node[above,rotate=90]{...} (-1.3,0.5)node[left]{$i_k$};
	\path (1.3,-0.5)node[right]{$j_1$} (1.3,0)node[below,rotate=90]{...} (1.3,0.5)node[right]{$j_k$};
	\draw[wall,<-] (1.3,-1) -- +(0,2);
	\draw[wall,<-] (-1.3,-1) -- +(0,2);
	\node(ref) at (0,0) {\phantom{$-$}};
\end{tikzpicture}
=(-q)^{-\binom{k}{2}}(-1)^{\binom{n-k}{2}}[n-k]_q! (-q)^{\ell(\textbf{i})+\ell(\textbf{j})} M_{J}^{I}({\bf u}),
\label{eq_detq}
\ee
where $\ell(\sim)$ is defined in \eqref{eq.length_of_sequence}, and $[\sim]_q!$ in \eqref{eq.quantum_integer}.

In fact, we  have 
\begin{align*}
	&\begin{tikzpicture}[baseline=(ref.base)]
		\fill[gray!20] (-1.3,-1) rectangle (1.3,1);
		\coordinate (src) at (0.3,0);
		\begin{scope}[edge]
			\foreach \y in {-0.5,0,0.5} {
				\draw[-o-={0.6}{>}] (src) -- +(1,\y);}
		\end{scope}
		\draw [line width =1pt] (-0.3,0)--(0.3,0);
		\node [above] at 
        (0,-0.05){\scalebox{0.8}{$n-k$}};
		\coordinate (src) at (-0.3,0);
		\begin{scope}[edge]
			\foreach \y in {-0.5,0,0.5} {
				\draw[-o-={0.6}{<}] (src) -- +(-1,\y);}
	\end{scope}
	\path (-1.3,-0.5)node[left]{$i_1$} (-1.3,0)node[above,rotate=90]{...} (-1.3,0.5)node[left]{$i_k$};
	\path (1.3,-0.5)node[right]{$j_1$} (1.3,0)node[below,rotate=90]{...} (1.3,0.5)node[right]{$j_k$};
	\draw[wall,<-] (1.3,-1) -- +(0,2);
	\draw[wall,<-] (-1.3,-1) -- +(0,2);
	\node(ref) at (0,0) {\phantom{$-$}};
	\end{tikzpicture}\\
=&(-1)^{\binom{n}{2}}q^{\frac{1}{2n}\big( \binom{n-k}{2}-\binom{k}{2} \big)}
(-q)^{\ell(\textbf{i})-\binom{k}{2}} \sum_{\sigma:[1;n-k]\rightarrow  \bar{I}^c}(-q)^{\binom{n-k}{2} -\ell(\sigma)}
\begin{tikzpicture}[baseline=(ref.base)]
\fill[gray!20] (-1,-1) rectangle (1,1);
\coordinate (src) at (0,0);
\begin{scope}[edge]
\foreach \y in {-0.5,0,0.5} {
\draw[-o-={0.6}{>}] (src) -- +(-1,\y);
\draw[-o-={0.6}{>}] (src) -- +(1,\y);}
\end{scope}
\path (-1,-0.5)node[left]{$\sigma(n-k)$} (-1,0)node[above,rotate=90]{...} (-1,0.5)node[left]{$\sigma(1)$};
\path (1,-0.5)node[right]{$j_1$} (1,0)node[below,rotate=90]{...} (1,0.5)node[right]{$j_k$};
\draw[wall,<-] (1,-1) -- +(0,2);
\draw[wall,<-] (-1,-1) -- +(0,2);
\node(ref) at (0,0) {\phantom{$-$}};
\end{tikzpicture}\\
=&(-1)^{\binom{n}{2}}q^{\frac{1}{2n}\big( \binom{n-k}{2}-\binom{k}{2} \big)}
(-q)^{\ell(\textbf{i})-\binom{k}{2}} \\
& \quad \cdot \sum_{\sigma:[1;n-k]\rightarrow  \bar{I}^c}(-q)^{\binom{n-k}{2} -\ell(\sigma)}
(-1)^{\binom{n}{2}}(-q)^{-\ell(\sigma)+\ell(\textbf{j})}
q^{\frac{1}{2n}\big(- \binom{n-k}{2}+\binom{k}{2} \big)}
M^I_{J}(\textbf{u})\\
=&(-q)^{\ell(\textbf{i})+\ell(\textbf{j})}(-q)^{\binom{n-k}{2}- \binom{k}{2}} M_{J}^{I}(\textbf{u})\sum_{\sigma:[1;n-k]\rightarrow  \bar{I}^c}q^{-2\ell(\sigma)}\\
=&(-q)^{-\binom{k}{2}}(-1)^{\binom{n-k}{2}}[n-k]_q! (-q)^{\ell(\textbf{i})+\ell(\textbf{j})} M_{J}^{I}(\textbf{u}),
\end{align*}
where the first equality comes from \cite[Lemma 4.6]{LY23}, the second equality  
from  \cite[Lemma 4.13]{LY23}, and the last equality from $q^{\binom{n-k}{2}} \sum_{\sigma:[1;n-k] \to \bar{I}^c} q^{-2\ell(\sigma)} = [n-k]_q!$
(see \cite[Lemma 3.1]{Wan23}). 
Here, we used the following notations of \cite{LY23}: let $[1;n-k] = \{1,2,\ldots,n-k\}$, and for each $I \subset \JJ$ 
let
$$
\bar{I} = \{ \bar{i} \, | \, i \in I \}, \qquad
I^c = \JJ \setminus I, \qquad
\bar{I}^c = (\bar{I})^c,
$$
where $\bar{i} = n+1-i$. This proves \eqref{eq_detq}.

Observe now that
	\begin{align*}
		\Theta_c({\bf a}_k') &= \sum_{I\in\mathbb J_k}\sum_{\{i_1,\cdots,i_k\} = I}\begin{tikzpicture}[baseline=(ref.base)]
			\fill[gray!20] (-1.3,-1) rectangle (1.3,1);
			\coordinate (src) at (0.3,0);
			\begin{scope}[edge]
				\foreach \y in {-0.5,0,0.5} {
					\draw[-o-={0.6}{>}] (src) -- +(1,\y);}
			\end{scope}
			\draw [line width =1pt] (-0.3,0)--(0.3,0);
			\node [above] at
            (0,-0.05){\scalebox{0.7}{$n-k$}};
			\coordinate (src) at (-0.3,0);
			\begin{scope}[edge]
				\foreach \y in {-0.5,0,0.5} {
					\draw[-o-={0.6}{<}] (src) -- +(-1,\y);}
		\end{scope}
		\path (-1.3,-0.5)node[left]{$i_1$} (-1.3,0)node[above,rotate=90]{...} (-1.3,0.5)node[left]{$i_k$};
		\path (1.3,-0.5)node[right]{$i_1$} (1.3,0)node[below,rotate=90]{...} (1.3,0.5)node[right]{$i_k$};
		\draw[wall,<-] (1.3,-1) -- +(0,2);
		\draw[wall,<-] (-1.3,-1) -- +(0,2);
		\node(ref) at (0,0) {\phantom{$-$}};
		\end{tikzpicture} \qquad (\because\mbox{\eqref{eq-alphak-loop}, Theorem \ref{t.splitting2}}) \\
		&=(-q)^{-\binom{k}{2}}(-1)^{\binom{n-k}{2}}[n-k]_q!\sum_{I\in\mathbb J_k}\sum_{\{i_1,\cdots,i_k\} = I} q^{2\ell(\textbf{i})} M_{I}^{I}(\textbf{u}) \quad (\because\mbox{\eqref{eq_detq}; for ${\bf i} = (i_1,\ldots,i_k)$}) \\			&=(-q)^{-\binom{k}{2}}(-1)^{\binom{n-k}{2}}[n-k]_q!\sum_{I\in\mathbb J_k}M_{I}^{I}(\textbf{u})\sum_{\{i_1,\cdots,i_k\} = I} q^{2\ell(\textbf{i})} \\
		&=(-q)^{-\binom{k}{2}}(-1)^{\binom{n-k}{2}}[n-k]_q![k]_q!q^{\binom{k}{2}}\sum_{I\in\mathbb J_k}M_{I}^{I}(\textbf{u}) \qquad (\because\mbox{\cite[Lemma 3.1]{Wan23}}) \\
		&=(-1)^{\binom{k}{2}+\binom{n-k}{2}}[n-k]_q![k]_q! 
        D_k^q({\bf u}).
	\end{align*}
    From \eqref{eq-alphak-loop}, we get the desired \eqref{eqal}.

    This finishes the proof of Proposition 
    \ref{r-Thet_alpha_k}, hence also that of Theorem \ref{thm.cutting_homomorphism_and_module-trace}.
\end{proof}

\def\SeA{{\cS_\heta(\An)}}
\def\SeB{{\cS_\heta(\B)}}

As a result, while keeping Lemma \ref{r.indep} in mind, we get the following, which was proved in \cite{DS}.
\bcor
\label{r-inj2} Assume that a non-zero complex number $\hxi
$
satisfies $[n]_\xi!\neq 0$, where $\xi= \hxi^{2n^2}$.
Then $\Theta_c: \cS_\hxi(\An) \to \cS_\hxi(\B)$ is injective.
\ecor

\subsection{Frobenius homomorphism 
for the annulus
}
Now we show that the Frobenius homomorphism, when restricted to the image of the cutting homomorphism $\Theta_c : \cS(\An) \to \cS(\PP_2)$ in \eqref{eq.cutting_homomorphism_of_annulus}, is equivalent to the Adams operation (\OldS\ref{subsection-root-Adams}). More precisely:

\begin{proposition}[annulus Frobenius homomorphism]\label{r-L-comp-F}
	Assume 
    that a root of unity $\home \in \mathbb{C}$ satisfies
    $[n]_\omega!\neq 0$, where $\omega= \home^{2n^2}$. Let $N=\ord(\omega^2),\, \heta=\home^{N^2}$ (as in \OldS\ref{ss.assumptions_on_roots_of_unity}).
    Then there exists a unique $\mathbb C$-algebra embedding $$\Psi\colon\SeA\rightarrow \SoA$$ 
    that satisfies the following property:
    \begin{enuma}
        \item The following diagram commutes:
\begin{align*}
		\raisebox{10mm}{\xymatrix{
				 \Rep_\eta \ar[r]^-{\kappa} \ar[d]_{\Ad} & \SeA \ar[d]^{\Psi} \\
				\Rep_\omega \ar[r]^-{\kappa} & \SoA,
		}}
	\end{align*}
    where the left vertical map $\Ad: \Rep_\eta \to \Rep_\omega$, the Adams operation, is defined in  
    \OldS\ref{subsection-root-Adams}.
    \end{enuma}
Moreover, this map $\Psi$ satisfies the following properties as well:
    \begin{enumerate}[label={\rm (\alph*)}]
        \setcounter{enumi}{1}
        \item The following diagram commutes:
\begin{align}\label{eq.repres_Frob}
		\raisebox{10mm}{\xymatrix{
				 \SeA \ar[r]^-{\Theta_c} \ar[d]_{\Psi} & \SeB= \Oe \ar[d]^{\Phi} \\
				\SoA \ar[r]^-{\Theta_c} & \SoB=\Oo,
		}}
	\end{align}
    where the right vertical map $\Phi : \Oe \to \Oo$ is the Frobenius homomorphism for the quantized function algebras in \OldS\ref{ss.Frobenius_map_for_Oq}.

    \item For each $1\leq k\leq n-1$, we have 
    $$\Psi([
    {\bf a}_k]_{\heta}) = \bar P_{N,k}([
    {\bf a}_1]_{\home},\cdots,[
    {\bf a}_{n-1}]_{\home}),$$
    where ${\bf a}_i$ is as defined in \eqref{eq-alphak-loop}, and $\bar P_{N,k}$ is the reduced power elementary polynomial defined in \eqref{eq-def-barP}.
    \end{enumerate}

\end{proposition}
\bpr 

(a) The existence and uniqueness of a map $\Psi$ satisfying 
item (a) is obvious because $\kappa$ is an isomorphism (Theorem \ref{twice_sphere}). 
The map $\Psi$ is given as the following composition
$$\SeA\xrightarrow{\kappa^{-1}}
\Rep_\eta \xrightarrow{
\Ad}
\Rep_\omega\xrightarrow{\kappa} \SoA,$$
which is a $\mathbb C$-algebra embedding because 
the Adams operation map $\Ad$ is a 
$\mathbb C$-algebra embedding ($\because$\OldS\ref{subsection-root-Adams}).

(b) With the identifications $\kappa:\Rep_\eta \xrightarrow{\cong } \SeA$, $\kappa:\Rep_\omega \xrightarrow{\cong } \SoA$ (Theorem \ref{twice_sphere}), and 
item (a), 
diagram \eqref{eq.repres_Frob} becomes
\begin{align}\label{eq55}
		\raisebox{10mm}{\xymatrix{
				 \Rep_\eta \ar[r]^-{\Theta_c\circ \kappa} \ar[d]_{\Ad} & \Oe \ar[d]^{\Phi} \\
				\Rep_\omega \ar[r]^-{\Theta_c \circ \kappa} & \Oo
		}}
	\end{align}
	By 
    diagram \eqref{commmmm} of Theorem \ref{thm.cutting_homomorphism_and_module-trace} we have $\Theta_c\circ \kappa= \cT_q$. Hence 
    diagram \eqref{eq55} becomes 
    diagram \eqref{eq57}, which is commutative by 
    Theorem \ref{thm-L-comp-F}.

(c)    With the identifications $L_{\eta}(\varpi_
{k}) = \Lambda_{\eta}(\varpi_
{k})=[
{\bf a}_k]_{\heta}$ and 
    $L_{\omega}(\varpi_
    {k}) = \Lambda_{\omega}(\varpi_
    {k})=[
    {\bf a}_k]_{\home}$ (Theorems \ref{thm-rep-dotU} and \ref{twice_sphere}), 
equation \eqref{eq57a} established in \OldS\ref{subsection-root-Adams} implies (c). 

\epr

\begin{remark}
    Recall that ${\bf a} \in\cS(\mathsf{A})$ represented by the oriented loop in Figure \ref{fig1}(a). 
    It is trivial to check that 
    $$\Theta_c({\bf a}) = D_1^q({\bf u}).$$
    We use $\cev{{\bf a}}$ to denote the loop obtained from ${\bf a}$ by reversing its orientation. 
     	For each $1\leq i\leq n$, equations \eqref{e-Oq-ops} and \eqref{eq-Hopf-bigon} imply
 	$$
 	S(u_{ii})= {\det}_q(\buu^{ii})\text{ and }
 	S(u_{ii}) = S(
 	\raisebox{-.30\height}{
 	\begin{tikzpicture}
 		\fill[gray!20] (0,0) rectangle (1,1);
 		\draw[wall,<-] (0,0) -- (0,1);
 		\draw[wall,<-] (1,0) -- (1,1);
 		\draw[-o-={0.6}{>}] (0,0.5) --(1,0.5);
 		\node [left] at(0,0.5) {$i$};
 		\node [right] at(1,0.5) {$i$};
 	\end{tikzpicture}})
    =\raisebox{-.30\height}{
    	\begin{tikzpicture}
    		\fill[gray!20] (0,0) rectangle (1,1);
    		\draw[wall,<-] (0,0) -- (0,1);
    		\draw[wall,<-] (1,0) -- (1,1);
    		\draw[-o-={0.6}{<}] (0,0.5) --(1,0.5);
    		\node [left] at(0,0.5) {$\bar i$};
    		\node [right] at(1,0.5) {$\bar i$};
    \end{tikzpicture}}.
 	$$
 	Here $\buu^{ii}$ is the result of removing the $i$-th row and $i$-th column from $\buu$. Then we have 
 	\begin{align*}
 	\Theta_c(\cev{{\bf a}}) = \sum_{1\leq i\leq n}
 		\raisebox{-.30\height}{
 			\begin{tikzpicture}
 				\fill[gray!20] (0,0) rectangle (1,1);
 				\draw[wall,<-] (0,0) -- (0,1);
 				\draw[wall,<-] (1,0) -- (1,1);
 				\draw[-o-={0.6}{<}] (0,0.5) --(1,0.5);
 				\node [left] at(0,0.5) {$i$};
 				\node [right] at(1,0.5) {$i$};
 		\end{tikzpicture}}
 	= \sum_{1\leq i\leq n} {\det}_q(\buu^{\bar i\bar i}) =  D_{n-1}^{q}({\bf u}).
 	\end{align*}

 Note that the definitions of ${\bf a}_{k}$ (see \eqref{eq-alphak-loop}), $2\leq k\leq n-2$, do not need the invertibilities of $[n-1]_q$ and $[n]_q$. 
Although in \OldS\ref{sec-image-closed-loop} we require that the root of unity $\home$ satisfies $[n]_{
\omega}!\neq 0$, for Theorem \ref{thm11} (the main result in \OldS\ref{sec-image-closed-loop}), this restriction could be reduced to $[n-2]_{\home}!\neq 0$ by replacing $
{\bf a}_1$ and $
{\bf a}_{n-1}$ in our constructions with $
{\bf a}$ and $\cev{
{\bf a}}$ respectively. Notice that this replacement is justified by $\Theta_c(
{\bf a}) = D_1^q({\bf u})$ and 
$\Theta_c(\cev{
{\bf a}}) = D_{n-1}^q({\bf u})$ which we showed in the previous paragraph. 
So when $n=2,3$, we can get rid of the restriction $[n]_\omega!\neq 0$ for Theorem \ref{thm11}.

\end{remark}

\def\tB{{\widetilde \B}}
\def\tAn{{\widetilde \An}}
\def\CL{{\mathcal L}}
\def\hxi{{\hat \xi}}

\section{The image of 
a knot 
under the Frobenius homomorphism}
\label{sec-image-closed-loop}

In this section we prove 
the main result of the present paper,  describing the image of a framed oriented knot under the Frobenius homomorphism of stated ${\rm SL}_n$-skein algebras in terms of certain threading operations.  
We will show that the Frobenius homomorphism is functorial with respect to the embedding of the thickened annulus into a marked 3-manifold. 

The Frobenius homomorphism for (stated) ${\rm SL}_n$-skein algebras was first defined for
$n=2$ in \cite{BW16} using threading of Chebyshev polynomials along knots. Later works \cite{LP,BL22,KQ} show that powers of stated arcs would give a simpler construction, still for $n=2$. For $
\rm SL_n$-skeins, Bonahon and Higgins \cite{BH23} defined threadings of 
reduced power elementary polynomials along knots which are `transparent' (i.e. ``locally central") in the skein modules, and they conjectured the existence of a Frobenius homomorphism whose image of knots are given by these threadings.  
The  result of this section confirms the conjecture for essentially marked 3-manifolds. Earlier,
for 
$n=3$ Higgins \cite{Hig23} 
proved the result for the thickenings of surfaces 
having at least one puncture 
per connected component, 
and later for all marked 3-manifolds \cite{Hig24} with the restriction that the order of $\home$ is co-prime 
to 6.

For the thickenings of pb surfaces without boundary,  
with or without punctures, 
there is no known Frobenius homomorphism for $n\ge 4$. In this case, we show that the Frobenius homomorphism can be defined on a quotient of the ${\rm SL}_n$-skein algebra, called the projected ${\rm SL}_n$-skein algebra. 

\subsection{Notations and assumptions on $\home$ and $\hxi$} \label{sec-ass_o}
In this section, $\home\in \BC$ is a root 
of unity. Let 
$$\omega= \home^{2n^2}, \quad N= \ord(\omega^2), \quad \heta = \home^{N^2}, \quad \eta = \omega^{N^2},
$$
as in \OldS\ref{ss.assumptions_on_roots_of_unity}.
We further require that $$[n]_{\omega}!\neq 0.$$
We use the ring $R=\Zhq ( [n]_q!  )^{-1}$. 

Also, $\hxi\in \BC$ 
either (i) is not a root of 
unity, or (ii) is one of $\heta$ and $\home$.

\subsection{The image of 
a framed oriented knot under the Frobenius homomorphism}\label{subsec.image_of_closed_components} Recall that $\An$ is the 
open annulus, or the twice-punctured sphere (\OldS\ref{ss.stated_SL_n-skein_algebra_of_the_annulus}), and $\tAn=\An \times (-1,1)$ is its thickening.

A framed oriented knot $\alpha$ in a marked 3-manifold $\MN$ defines a unique (up to isotopy) embedding $f_{\alpha}: \tAn\embed \mathring M$ (see \OldS\ref{subsec:punctured_bordered_surface_and_n-web}). The map $\alpha \mapsto f_{\alpha}$ is a bijection between isotopy classes of framed oriented knots in $\MN$ and morphisms from $\tAn$ to $\MN$ in the category $\mathbb{M}$ (\OldS\ref{ss.marked}). Note that $\cN$ does not play a role here, since a morphism from $\tAn$ to $\MN$ is also a morphism from $\tAn$ to $M$. Denote by $(f_{\alpha})_* : \cS(\An) \to \cS\MN$ the induced map, and by $(f_{\alpha})_{*,\hxi}$ the evaluation of $\hq$ at $\hxi$.

\begin{definition}\label{def.threading_of_element}
Let $\MN$ be a marked $3$-manifold, which is not necessarily assumed to be an essentially marked $3$-manifold.
For $z\in \cS(\An;R)$ and a framed oriented knot $\alpha$ in $\MN$ define the {\bf result of threading $z$ along $
\alpha$} as
$$
\alpha *_\hq z := (f_
{\alpha})_*(z) \in \cS(M,\cN;R).
$$
For $z\in \cS_{\hxi}(\mathsf{A})$ define the {\bf result of threading $z$ along $
\alpha$} as
\be 
\alpha{*_\hxi} z := (f_
{\alpha})_{*, \hxi}(z) \in \cS_{\hxi}\MN.
\ee
\end{definition}

This is similar to the construction of satellite knots. 

The following `functoriality' theorem will let us understand the image of a framed oriented knot in an essentially marked $3$-manifold $\MN$ under the Frobenius homomorphism $\Phi : \cS_\heta\MN \to \cS_\home\MN$ (of Theorem \ref{thmFrob}) via the image of a core curve in the annulus $\An$ under the annulus Frobenius homomorphism $\Psi : \cS_\heta(\An) \to \cS_\home(\An)$ (of Proposition \ref{r-L-comp-F}).
\bthm[functoriality of Frobenius homomorphisms for embedding a thickened annulus] \label{thm10} 
Assume that $\MN$ is an essentially marked 3-manifold (\OldS\ref{subsec:essentially_marked_3-manifolds}), and that
$f: \tAn \hookrightarrow \MN$ is a morphism in the category $\mathbb{M}$. Then the following diagram is commutative:
\be 
\label{eq-dia10}
\begin{tikzcd}
		\cS_{\hat{\eta}}(\An)  \arrow[r, "f_{*,\heta}"]
		\arrow[d, "\Psi"]  
		&    \cS_{\hat{\eta}}\MN\arrow[d, "
        \Phi"] \\
	\cS_{\hat{\omega}}(\An)  \arrow[r, "f_{*,\home}"] 
		&  	\cS_{\hat{\omega}}\MN,  
	\end{tikzcd}
\ee
where we use notations for $\heta$ and $\home$ in \OldS\ref{sec-ass_o}, the left vertical map $\Psi$ is from Proposition \ref{r-L-comp-F}, and the right vertical map $\Phi$ is from Theorem \ref{thmFrob}(a).

Put it differently, if  $\alpha$ is a framed oriented knot in  $\MN$ and  $z\in 
\cS_\heta(\An)$, then
\be 
\Phi(\alpha *_\heta z) = \alpha *_\home \Psi(z).
\label{eq10}
\ee
\ethm
\def\tPone{{\widetilde{\PP_{1,1}}}}
\bpr 
First, consider the the special case when $\MN$ is the thickening $\widetilde{\PP_{1,1}}$ of the punctured monogon $\PP_{1,1}$, obtained from the monogon $\PP_1$ by removing one interior puncture. The interior of $\PP_{1,1}$ is the annulus $\An$, and $f: \widetilde{\An} \to \widetilde{\PP_{1,1}}$ is the thickening of $\iota: \An \embed \PP_{1,1}$. The embedding $\iota$ is first map in the following chain of operations:

\be \text{Chain I:} \quad 
\begin{tikzpicture}[baseline=0cm,every node/.style={inner sep=2pt}]
\draw[wall,fill=gray!20,dotted] (0,0) circle[radius=0.8];
    \draw[fill=white] (0,0) circle[radius=0.1];
    \draw [red] (0,0.1)--(0,0.8);
    \node [right] at (0,0.4) {$c$};
\end{tikzpicture}
\xrightarrow{\iota}
\begin{tikzpicture}[baseline=0cm,every node/.style={inner sep=2pt}]
\draw[wall,fill=gray!20] (0,0) circle[radius=0.8];
    \draw[fill=white] (0,0) circle[radius=0.1];
    \draw[fill=white] (0,0.8) circle[radius=0.1];
    \draw [red] (0,0.1)--(0,0.7);
    \node [right] at (0,0.4) {$c$};
\end{tikzpicture}
\xrightarrow{\Cut_c}
\begin{tikzpicture}[baseline=0cm,every node/.style={inner sep=2pt}]
\draw[draw=white,wall,fill=gray!20] (-0.8,-0.8)--(0,0.8)--(0.8,-0.8)--cycle;
\draw [wall,red] (-0.8,-0.8)--(0,0.8);
\draw [wall,red] (0,0.8)--(0.8,-0.8);
    \draw [wall] (-0.8,-0.8)--(0.8,-0.8);
    \draw[fill=white] (-0.8,-0.8) circle[radius=0.1];
    \draw[fill=white] (0,0.8) circle[radius=0.1];
    \draw[fill=white] (0.8,-0.8) circle[radius=0.1];
\end{tikzpicture}, 
\qquad \An \xrightarrow {\iota} \PP_{1,1} \xrightarrow {\Cut_c} \PP_3.
\label{eqI}
\ee  
The diagram \eqref{eq-dia10} is the left square of the following diagram
\be 
\label{eq-dia10a}
\begin{tikzcd}
		\cS_{\hat{\eta}}(\An)  \arrow[r, "\iota_{*,\heta}"] \arrow[d, "\Psi"]  
		&    \cS_{\hat{\eta}}(\PP_{1,1})\arrow[d, "\Phi"]  \arrow[r, "\Theta_c"] 
		&    \cS_{\hat{\eta}}(\PP_{3})\arrow[d, "\Phi"]
		\\
	\cS_{\hat{\omega}}(\An)  \arrow[r, "\iota_{*,\home}"] 
		&  	\cS_{\hat{\omega}}(\PP_{1,1}) \arrow[r, "\Theta_c"] 
		&\cS_{\hat{\omega}}(\PP_{3}) 
	\end{tikzcd}
\ee
Note that both rows are maps of stated ${\rm SL}_n$-skein algebras 
induced by Chain I 
\eqref{eqI}.

By Theorem \ref{thmFrob}(c), the right square is commutative. Both maps $\Theta_c$ are injective (Theorem \ref{t.splitting2}). Hence we will have the commutativity of the left square if we can show that the composition of the two squares, which we call 
the outer diagram, is commutative.

Consider another chain of operations
\be \text{Chain II:} \quad 
\begin{tikzpicture}[baseline=0cm,every node/.style={inner sep=2pt}]
\draw[wall,fill=gray!20,dotted] (0,0) circle[radius=0.8];
    \draw[fill=white] (0,0) circle[radius=0.1];
    \draw [red] (0,0.1)--(0,0.8);
    \node [right] at (0,0.4) {$c$};
\end{tikzpicture}
\xrightarrow{\Cut_c}
\begin{tikzpicture}[baseline=0cm,every node/.style={inner sep=2pt}]
\draw[draw=white,wall,fill=gray!20] (-0.8,-0.8)--(0,0.8)--(0.8,-0.8)--cycle;
\draw [wall,red] (-0.8,-0.8)--(0,0.8);
\draw [wall,red] (0,0.8)--(0.8,-0.8);
    \draw [wall,dotted] (-0.8,-0.8)--(0.8,-0.8);
    \draw[fill=white] (-0.8,-0.8) circle[radius=0.1];
    \draw[fill=white] (0,0.8) circle[radius=0.1];
    \draw[fill=white] (0.8,-0.8) circle[radius=0.1];
\end{tikzpicture}
\xrightarrow{\jmath}
\begin{tikzpicture}[baseline=0cm,every node/.style={inner sep=2pt}]
\draw[draw=white,wall,fill=gray!20] (-0.8,-0.8)--(0,0.8)--(0.8,-0.8)--cycle;
\draw [wall,red] (-0.8,-0.8)--(0,0.8);
\draw [wall,red] (0,0.8)--(0.8,-0.8);
    \draw [wall] (-0.8,-0.8)--(0.8,-0.8);
    \draw[fill=white] (-0.8,-0.8) circle[radius=0.1];
    \draw[fill=white] (0,0.8) circle[radius=0.1];
    \draw[fill=white] (0.8,-0.8) circle[radius=0.1];
\end{tikzpicture},
\qquad \An \xrightarrow {\Cut_c} \PP_{2} \xrightarrow {\jmath} \PP_3.
\label{eqII}
\ee
Correspondingly we have the following diagram, where the rows come from Chain II \eqref{eqII}:
\be
\label{eq-dia10b}
\begin{tikzcd}
		\cS_{\hat{\eta}}(\An)  \arrow[r, "\Theta_c"] \arrow[d, "\Psi"]  
		&    \cS_{\hat{\eta}}(\PP_{2})\arrow[d, "\Phi"]  \arrow[r, "\jmath_{*,\heta}"] 
		&    \cS_{\hat{\eta}}(\PP_{3})\arrow[d, "\Phi"]
		\\
	\cS_{\hat{\omega}}(\An)  \arrow[r, "\Theta_c"] 
		&  	\cS_{\hat{\omega}}(\PP_{2}) \arrow[r, "\jmath_{*,\home} "] 
		&\cS_{\hat{\omega}}(\PP_{3}) 
	\end{tikzcd}
\ee
The left square is commutative by Theorem \ref{r-L-comp-F}(b), while  the right square is commutative by Theorem \ref{thmFrob}(b). Hence the outer diagram is commutative.

Since the cutting homomorphism commutes with the embedding, by comparing Chain I and Chain II we get
\be 
\Theta_c \circ \iota_* = \jmath_*\circ \Theta_c.
\label{eq-com6}
\ee 
Thus the outer diagram of \eqref{eq-dia10b} equals 
the outer diagram of \eqref{eq-dia10a}. This shows the commutativity of the outer diagram of \eqref{eq-dia10a}
, and hence the theorem in our special case of $\widetilde{\PP_{1,1}}$.

For a general case of $\MN$, assume that $f= f_\gamma$ for a framed oriented knot $\gamma$ in $\MN$. Choose an embedded path $a$ in $M$ connecting a point on $\gamma$ and a point in $\cN$. Then a small neighborhood of $\gamma \cup a$ is diffeomorphic to $\tPone$; see Figure \ref{fig-web-neighborhood}. The theorem follows from the special case by functoriality (\OldS\ref{ss.functor}, Theorem \ref{thmFrob}(b)).
\epr

We record here a consequence of \eqref{eq-com6}.
\bpro \label{r-inj6}
Let $\xi$ be a non-zero complex number satisfying 
$[n]_\xi!\neq 0$.  Then 
 $\iota_{*,\hxi}: \cS_
 \hxi(\An) \to \cS_
 \hxi(\PP_{1,1})$ is  injective.
\epro

\bpr Consider  
equation \eqref{eq-com6}, with $\hq$ evaluated at $\hxi$. 
The map $\Theta_c$$:\cS_\hxi(\An) \to \cS_\hxi(\PP_2)$ in the right hand side is injective by Lemma \ref{r-inj2}, and
the map $\jmath_{*,\hxi} : \cS_\hxi(\PP_2) \to \cS_\hxi(\PP_3)$ is injective by Lemma \ref{lem_P3}. Hence the right hand side is injective, so it follows that $\iota_{*,\hxi}$ in the left hand side 
is injective.
\epr

\def\SqMN{{\cS\MN}}

\subsection{Explicit formulas} 
We write down explicitly the formula for the image of an oriented framed knot  
under the Frobenius homomorphism, and of some more elements.

Recall 
from \eqref{eq-alphak-loop} of \OldS\ref{ss.stated_SL_n-skein_algebra_of_the_annulus} the definition of ${\bf a}_k \in \SqA$, where $R= \Zhq([n]_q!  )^{-1}$.  

Following \cite{BH23}, we define the threading of multi-variable polynomials:
\begin{definition}\label{def.threading_of_polynomial}
 Let $
 P\in R[y_1, \dots, y_{n-1}]$ and $
 \gamma$ be a framed oriented 
 knot in $\MN$. 
 Define the {\bf threading of $
 P$ along $
 \gamma$} by
\be
\nonumber
\gamma^{[P]} := \gamma *_\hq P({\bf a}_1,\ldots,{\bf a}_{n-1}) \in \cS(M,\cN;R),
\ee
where the (threading) operation $*_\hq$ is defined as in Definition \ref{def.threading_of_element}.

\end{definition}

The following theorem, which is one of the main results of the present paper, gives an explicit skein-theoretic description of the image under the Frobenius homomorphism $$\Phi : \cS_\heta\MN \to \cS_\home\MN$$ (of Theorem \ref{thmFrob}), for an essentially marked $3$-manifold $\MN$, of a stated framed $\cN$-arc (\OldS\ref{subsec:punctured_bordered_surface_and_n-web}), an oriented framed link, and an oriented framed link colored with fundamental weights of ${\rm SL}_n$ (in the language of Cautis-Kamnitzer-Morrison skein algebra; see Theorem \ref{twice_sphere}). In particular, it resolves a conjecture of Bonahon and Higgins \cite{BH23}.
\bthm
\label{thm11} Let $\MN$ be an essentially marked 3-manifold (\OldS\ref{subsec:essentially_marked_3-manifolds}). Let $\home$, $\heta$ and $N$ be as in \OldS\ref{sec-ass_o}. 
\begin{enumerate}[label={\rm (\arabic*)}]
\item  If $\alpha$ is a stated framed $\cN$-arc in $\MN$, then $$
\Phi(
[\alpha]_\heta) = \left [ \alpha^{(N)}  \right]_\home,$$
where for $x\in \cS\MN$ and a non-zero complex number $\hxi$ we denote by $[x]_\hxi \in \cS_\hxi\MN$ the image of $x$ under the natural map $\cS\MN \to \cS_\hxi\MN$, and $\alpha^{(N)}$ means the $N$-parallel copies of $\alpha$ as defined in \OldS\ref{subsec:main_result_of_Frobenius_section}.

\item If $\alpha$ is an oriented framed knot in $\MN$ , then
\be 
\Phi([\alpha]_\heta) = \left[ \alpha^{[ \bar P_{N,1}  ]} \right]_\home,
\label{eq-Phi}
\ee
where $\alpha^{[ \bar P_{N,1}]} \in \cS(M,\cN;R)$ is the result of threading (Definition \ref{def.threading_of_polynomial}) the reduced power elementary polynomial $\bar P_{N,1}$ (\eqref{eq-def-barP} of \OldS\ref{ss.representations_of_sln}) along $\alpha$.

\item Suppose that $\alpha= \alpha_1 \sqcup \dots \sqcup \alpha_{r}$ is a stated $n$-web in $\MN$, where each $\alpha_j$ is either a stated framed $\cN$-arc or a framed oriented knot. For each $j$ choose a regular neighborhood $U_j$ of $\alpha_j$ in $M$ such that $U_1,\ldots,U_{r}$ are mutually disjoint, and $(U_j, U_j \cap \cN)$ is isomorphic to the thickened bigon $\widetilde{\PP_2}$ if $\alpha_j$ is a stated framed $\cN$-arc and to the thickened annulus $\widetilde{\An}$ if $\alpha_j$ is a framed oriented knot. For each $j$, write $\Phi([\alpha_j]_\heta)$ as a linear combination $\sum_{i_j} c_{i_j,j} [\alpha_{i_j,j}]_\home$ with $c_{i_j,j} \in \mathbb{C}$ and a stated $n$-web $\alpha_{i_j,j}$ in $(U_j, U_j \cap \cN) \subset (M,\cN)$. Then $\Phi([\alpha]_\heta)$ is given by the `term-by-term multilinear disjoint union' of $\Phi([\alpha_j]_\heta)$, $j=1,\ldots,r$, in the sense that
\be
\label{Phi_for_disjoint_union}
\Phi([\alpha]_\heta) = \sum c_{i_1,1} c_{i_2,2} \cdots c_{i_{r},r} [ \alpha_{i_1,1} \sqcup \alpha_{i_2,2} \sqcup \cdots \sqcup \alpha_{i_{r},r}]_\home,
\ee
where the sum is over the indices $i_1,\ldots,i_{r}$.

 \item For a framed oriented knot $\alpha$ in $\MN$ and for any $k=1,\ldots,n-1$,
\be 
\Phi(\alpha *_\heta [{\bf a}_k]_\heta) = \left[ \alpha^{
[\bar P_{N,k} 
]}\right]_\home,
\ee
where ${\bf a}_k$ is as in \eqref{eq-alphak-loop}, the (threading) operation $*_\heta$ is as defined in Definition \ref{def.threading_of_element}, and $
\alpha^{[\bar{P}_{N,
k}]}\in \cS(M,\cN;R)$ is the result of threading (Definition \ref{def.threading_of_polynomial}) the reduced power elementary polynomial $\bar P_{N,k}$ (\eqref{eq-def-barP} of \OldS\ref{ss.representations_of_sln}) along $
\alpha$.
\end{enumerate}
\ethm

\begin{proof}
    
    We first show the items (1), (2) and (4).
    Using Theorem \ref{thm10}, the functoriality theorem, the proof boils down to showing the corresponding statements for the thickening of the annulus $\An$, which are proved in Theorem \ref{thmFrob}(a) and Proposition \ref{r-L-comp-F}(c).

    It remains to show (3). For each $j$ such that $\alpha_j$ is a framed oriented knot, choose a path $p_j$ in $M$ connecting $\cN$ to $\alpha_j$ as we did in \OldS\ref{ss.Proof_of_Theorem_6_1a_for_3-manifolds} in the proof of Theorem \ref{thmFrob}(a), so that $p_j$'s are mutually disjoint. Then let $U_j' \supset U_j$ be a regular neighborhood of $\alpha_j \cup p_j$ such that $(U_j', U_j' \cap \cN)$ is isomorphic to the thickening $\widetilde{\PP_{1,1}}$ of the punctured monogon which appeared in the proof of Theorem \ref{thm10} . Let $\fS_j = \PP_{1,1}$. For each $j$ such that $\alpha_j$ is a stated framed $\cN$-arc, let $U_j'=U_j$ and $\fS_j = \PP_2$. Since each $\fS_j$ is an essentially bordered pb surface, from Theorem \ref{thmFrob}(a) we have a Frobenius homomorphism which we denote by $\Phi_j : \cS_\heta(\fS_j) \to \cS_\home(\fS_j)$. Note that $\cS(\sqcup_j \fS_j)$ is naturally isomorphic to $\otimes_j \cS(\fS_j)$. By using Theorem \ref{thmFrob}(b) for each $j$ for the embedding $\widetilde{\fS_j} \to \sqcup_j \widetilde{\fS_j}$, we can see that the Frobenius homomorphism $\Phi : \cS_\heta(\sqcup_j \fS_j) \to \cS_\home(\sqcup_j \fS_j)$ for $\sqcup_j \fS_j$ (which exists by Theorem \ref{thmFrob}(a)) is given on each $j$-th tensor factor $
    \cS_\heta(\fS_j)$ of $\cS_\heta(\sqcup_j \fS_j) = \otimes_j \cS_\heta(\fS_j)$ as $\Phi_j$. Now Theorem \ref{thmFrob}(b) for the embedding $\sqcup_j \widetilde{\fS_j} \stackrel{\sim}{\to} \cup_j U'_j \hookrightarrow \MN$ (of essentially marked $3$-manifolds) yields the sought-for result, proving (3).
\end{proof}

\def\hxi{{\hat \xi}}
\def\Sx{{\cS_\hxi}}

\def\CVect{\BC\text{-}{\mathsf{Vect}}}

\subsection{Category language}
Let $\MM''$ be the monoidal category whose objects are marked 3-manifolds where each component is either essentially marked or isomorphic to  the thickened annulus. The tensor products and morphisms are the same as in the category $\MM$ of marked 3-manifolds (\OldS\ref{ss.marked}).

Let $\CVect$ be the  monoidal category of $\BC$-vector 
spaces.
For each non-zero $\hxi\in \BC$, the stated ${\rm SL}_n$-skein module gives a  monoidal functor 
$$ \cS_\hxi: \MM'' \to \CVect,\quad  \MN \to \Sx\MN.$$
Theorems \ref{thmFrob} and \ref{thm10} imply that the Frobenius homomorphism gives a natural transformation 
$$ \Phi_{\heta \to \home}: \cS_\heta \Longrightarrow \cS_\home,$$
which 
coincides with (i) $\Phi:\Oe \to \Oo$ of \OldS\ref{ss.Frobenius_map_for_Oq} (coming from \cite{PW}) for the thickened bigon and (ii) the Adams operation $\Psi:\Rep_\eta \to \Rep_\home$ in \OldS\ref{subsection-root-Adams} for the thickened annulus. Since for $\MN\in \MM''$ the stated ${\rm SL}_n$-skein module $\SMN$ is spanned by stated framed $\cN$-arcs and framed oriented links, see \cite[Proposition 4.2]{LS21}, such a natural transformation, if exists, is unique.  
Theorem \ref{thmFrob}(b) and Theorem \ref{thm10} show
that such a natural transformation does exist.

\subsection{Geometric description of the image of a knot}\label{subsec-geomatric-description} We now show how the 
reduced power elementary polynomial $\bar P_{m,k}$ can be expressed by a simple diagram.

For $1\leq k\leq n-1$ and a positive integer $m$,  
define 
$$
{\bf a}_{m,k} = (-1)^{m(\binom{k}{2}+\binom{n-k}{2})}\frac{1}{((n-k)!k!)^{m}} {\bf a}_{m,k}'\in \cS_{1}(\An),$$
where the $n$-web ${\bf a}_{m,k}'$ in $
\widetilde{\An}$ is given in Figure \ref{fig-beta}. In terms of the Cautis-Kamnitzer-Morrison skein algebra, $
{\bf a}_{
m,k}$ can be thought of as being proportional to the $
m$-times-going-around loop ${\bf a}$ colored by the weight $\varpi_k$. As $\cS_{1}(\An)= \BC[[
{\bf a}_1]_1, \dots, [
{\bf a}_
{n-1}]_1]$ from Theorem \ref{twice_sphere}, 
each $x\in \cS_{1}(\An)$ determines a polynomial $P_x$ in $n-1$ variables such that $$x = [P_x(
{\bf a}_1, \dots, 
{\bf a}_
{n-1})]_1$$ in $\cS_{1}(\An)$ (see \eqref{eq-alphak-loop} for ${\bf a}_k$).

 \begin{figure}
     \centering
     \def\svgwidth{0.3\textwidth}
     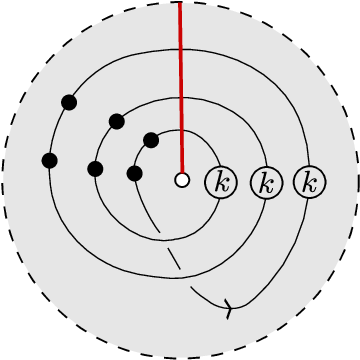
     \caption{The black $n$-web is ${\bf a}_{m,k}'$ ($m=3$ in the above picture), where each circle $k$ represents $k$ strands without crossings, the black dots are sinks or sources. The  red arc is the cutting arc $c$, meeting ${\bf a}'_{m,k}$ at $mk$  
     points}.
     \label{fig-beta}
 \end{figure}

\begin{proposition}\label{prop-twice-classical}
\begin{enuma}
\item
We have $\bar P_{m,k} = P_{{\bf a}_{m,k}  }  $. In other words, in $\cS_{1}(\An)$ we have 
	$$
    {\bf a}_{
    m,k} =\left[ \bar P_{
    m,k}(
    {\bf a}_1,\cdots,
    {\bf a}_{n-1})\right]_1.$$

    \item
    Consequently, with the assumptions of Theorem \ref{thm10}, 
	we have
	\be 
	\Phi(
    \alpha *_{\hat\eta} 
    {\bf a}_k) = 
    \alpha *_\home F_{1 \to \home} (
    {\bf a}_{N,k}).
	\ee
	where $F_{1 \to \home}: \cS_1(\An) \to \SoA$ is the $\BC$-algebra isomorphism given by $[
    {\bf a}_k]_1 
    \mapsto [
    {\bf a}_k]_\home$ for $k=1, \dots, n-1$.
\end{enuma}
\end{proposition}
\begin{proof}
	
(a) 	Since 
$\bar P_{
m,k}(D_1^1({\bf u}),\cdots,D_{n-1}^1({\bf u}))= D_k^1({\bf u}^
{m})$ ($\because$Lemma \ref{lem-PD})
and $\Theta_c(
{\bf a}_k) = D^q_k(\buu)$ by \eqref{eqal}, we have
	\begin{align*}
		\Theta_c([\bar P_{
        m,k}(
        {\bf a}_1,\cdots,
        {\bf a}_{n-1})]_1) = D_k^1(\textbf{u}^
        {m}).
	\end{align*}
	Since the injectivity of $\mathcal{T}_1$ in Proposition \ref{r-inj}, together with Theorem \ref{thm.cutting_homomorphism_and_module-trace}, implies the injectivity of $\Theta_c : \cS_1(\An) \to \cS_1(\PP_2)$, it suffices to show that $\Theta_c(
    {\bf a}_{
    m,k}) = D_k^1(\textbf{u}^
    {m})$.
	From  \eqref{eqal} and \cite[Lemma 4.9]{LY23} (where the latter is used to swap the states; see especially \cite[equations (99) and (100)]{LY23})
    we have 
	\begin{align*}		\Theta_c(
    {\bf a}_{
    m,k})=\sum_{I_1,\cdots, I_
    {m}\in\mathbb J_k} M^{I_2}_{I_1}(\textbf{u})M^{I_3}_{I_2}(\textbf{u})\cdots M^{I_
    {m}}_{I_{
    m-1}}(\textbf{u}) M^{I_{1}}_{I_
    {m}}(\textbf{u})
		=\sum_{I_1\in\mathbb J_k} M^{I_1}_{I_1}(\textbf{u}^
        {m})=D_k^1(\textbf{u}^
        {m}),
	\end{align*}
	where the second equality comes from 
	the following   linear algebra lemma.
	\begin{lemma}\label{lem-linear}
		Suppose that $A$ and $B$ are two square matrices of size $n$, and $I,J\in\mathbb J_k$  for some $1\leq k\leq n-1$. Then we have 
		$\det((AB)_{I,J})=\sum_{K\in\mathbb J_k}\det(A_{I,K}) \det(B_{K,J}),$
        where $(AB)_{I,J}$ is the $I\times J$ submatrix of $AB$. \qed
	\end{lemma} 
\renewcommand{\qed}{}	\end{proof}

(b) 
Due to \eqref{eq10} of Theorem \ref{thm10}, the functoriality theorem, it suffices to show that 
\begin{align}
    \Psi([{\bf a}_k]_{\heta}) = F_{1 \to \home}({\bf a}_{N,k})\in\SoA.
\end{align}
From part (a)  
for $m=N$ we have
\begin{align*}
    F_{1 \to \home}({\bf a}_{N,k})
    = F_{1 \to \home}\left(\left[ \bar P_{
    N,k}({\bf a}_1,\cdots,{\bf a}_{n-1})\right]_1\right)
    =\left[ \bar P_{
    N,k}({\bf a}_1,\cdots,{\bf a}_{n-1})\right]_{\home}
    =\Psi([{\bf a}_k]_{\heta}), 
\end{align*}
where the last equality holds by Proposition \ref{r-L-comp-F}(c). \qed

\def\MN{(M,\mathcal{N})}

\def\zS{\cS_{\hat{\eta}}}
\def\qS{\cS_{\hat{q}}}
\def\cK{\mathcal K}

\subsection{The Frobenius homomorphism for pb surfaces without boundary}\label{subsec.Frobenius_without_boundaries} 

 In this subsection $\fS= \Sigma_{g,n}$, the result of removing $n$ points from an oriented surface of genus $g$. Since $\fS$ is not essentially bordered, we have not constructed the Frobenius homomorphism $\Phi : \cS_\heta(\fS) \to \cS_\home(\fS)$ for a surface.
Here  we define the Frobenius homomorphism for a quotient of $\SS$, known as  the projected ${\rm SL}_n$-skein algebra introduced in \cite{LS21}.

Assume that $\fS$ has at least one punctures, and $p$ is one of them. 
Let $c_p$ be a trivial ideal arc at $p$. This means the two endpoints of $c_p$ are both $p$ and $c_p$ bounds an embedded monogon $D_p$; see Figure \ref{fig_cp}. By removing the interior of $D_p$, from $\fS$ we get 
essentially bordered pb surface $\fS'$, where $c_p$ now becomes a boundary edge.  
With the identification $\fS=\fS\setminus D_p$ 
there is a unique morphism $f: \fS \embed \fS'$, such that $f$ is the identity outside a small neighborhood of $D_p$. One can imagine that $f$ `enlarges' the puncture $p$.

For a non-zero $\hat \xi\in \BC$ let $\cK_{\hat \xi, p}$ be the kernel of the induced map $$f_*: \Sx(\fS) \to \Sx(\fS').$$

\begin{figure}
	\centering
	\begin{tikzpicture}
\draw[draw=white,fill=gray!20]  (-1,-1) rectangle (1,1);
    \draw[wall,fill=gray!20] (0,0) circle[radius=0.5];
    \draw[fill=white] (0.5,0) circle[radius=0.1];
    \node at (0.8,0) {$p$};
    \node [above] at (-0.7,0.1) {$c_p$};
\end{tikzpicture}
	\caption{The arc $c_p$ and the enclosed monogon $D_p$.}\label{fig_cp}
\end{figure}
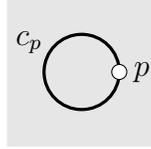

\def\SxsS{\cS^*_\hxi(\fS)}
\def\SxS{\cS_{\hat \xi}(\fS)}
\def\Ssx{\cS_\hxi^{*}}
\def\Sso{\cS_\home^{*}}
\def\Sse{\cS_\heta^{*}}
\def\SxA{\cS_\hxi(\An)}

\def\Ker{\text{Ker}}
In \cite[Section 9]{LS21}, it is proved that  $\cK_{\hxi,p}$ does not depend on the choice of a puncture $p$, and we denote it by $\cK_{\hxi,\fS}$. Define the {\bf projected stated ${\rm SL}_n$-skein algebra} by 
\begin{align}\label{def-projected-skein-punctured}
    \SxsS:= \SxS/\cK_{\hxi,\fS}.
\end{align}

By \cite[Corollary 9.2]{LS21}, we have the following.

\begin{lemma}\label{r.kernel}
For any ideal arc $c$ of $\fS$, the cutting homorphism $\Theta_c\colon\cS_\hxi(\fS) \to \cS_\hxi(\Cut_c(\fS))$ of  Theorem \ref{t.splitting2} has kernel $\cK_{\hxi,\fS}$ and hence descends to an embedding
   \begin{align}
       \label{eq.cutting_homomorphism_projected}
       \Theta_c^{*} : \SxsS\embed 
       \cS_\xi(\Cut_c(\fS)) .
   \end{align}

\end{lemma}

Note that $\overline{\cS}(\fS) = \SS$ since $\fS$ has empty boundary. For any triangulation $\lambda$ of $\fS$ the two $X$-version quantum traces are the same.

\begin{lemma} 
    Assume that $\fS= \Sigma_{g,n}$ has  an ideal triangulation $\lambda$.
    The $X$-version quantum trace map $\tr^X_\lambda\colon
    \cS_{\hat \xi}(\fS)\rightarrow
    \mathcal X_{\hat \xi}(\fS,\lambda)$ of Theorem \ref{thm-X}(b)
    induces a quantum trace map for the projected skein algebra
    \begin{align}
        \label{eq.X-quantum_trace_projected}
        \tr^{X*}_\lambda\colon
    \cS_{\hat \xi}^{*}(\fS)\rightarrow
    \mathcal X_{\hat \xi}(\fS,\lambda).
    \end{align}
\end{lemma}

\begin{proof} By cutting $\fS$ along the  edges in $\lambda$ we get a collection $\mathcal{F}_\lambda$ of triangles. 
In \cite[\OldS 12.2]{LY23}  it is shown that we can identify
$\XS$ with a subset of $ \bigotimes_{\tau \in \mathcal{F}_\lambda} \bsX(\tau)$ such 
that the quantum trace $\tr^X_\lambda$ is the composition 
$$ 
\cS_\hq(\fS)  \xrightarrow {\Theta} \bigotimes_{\tau \in \mathcal{F}_\lambda} {\cS}_\hq(\tau)
\xrightarrow {\pr } \bigotimes_{\tau \in \mathcal{F}_\lambda} \overline{\cS}_\hq(\tau) \xrightarrow {\ot \btr^X_\tau} \bigotimes_{\tau \in \mathcal{F}_\lambda} \bsX_\hq(\tau).
$$
Here ${\Theta}$ is the composition of the cutting homomorphisms for all the interior edges of $\lambda$, and $\pr$ is the tensor product of all the projections $ {\cS}_\hq(\tau)\onto \overline{{\cS}}_\hq(\tau)$. Since $\cK_{\hxi,\fS}$ is in the kernel of $\Theta$ by Lemma \ref{r.kernel}, the map $\tr^X_\lambda$ descends to the quotient $\cS^*_\hxi(\fS)$ as claimed by the lemma.

\end{proof}

We denote $\Sigma_g= \Sigma_{g,0}$. 
The natural embedding $\Sigma_{g,1} \embed \Sigma_g$ induces a surjective algebra homomorphism ${\bf P}\colon\cS(\Sigma_{g,1})\onto \cS(\Sigma_g)$ with kernel
	 generated by $a-a'$, where $a$ and $a'$ are 
     $n$-web diagrams  
     in $\Sigma_{g,1}$ related by sliding over the puncture, see Figure \ref{eq-relation-point-p}. 
Define $$\mathcal K_{\hxi,\Sigma_g}:=
{\bf P}(\mathcal K_{\hxi,\Sigma_{g,1}})\subset \cS_{\hxi}(\Sigma_{g}).$$ Since ${\bf P}$ is surjective, we have that $\mathcal K_{\hxi,\Sigma_g}$ is an ideal of $\cS_{\hxi}(\Sigma_g)$.
Define the {\bf projected ${\rm SL}_n$-skein algebra} of 
$\Sigma_g$ by 
\begin{align}\label{def-projected-skein-closed}
    \cS_{\hxi}^{*}(\Sigma_g):=\cS_{\hxi}(\Sigma_g)/ \mathcal K_{\hxi,\Sigma_g}.
\end{align}

    \begin{figure}
         \centering
         \raisebox{-.35\height}{
\begin{tikzpicture}
\draw[draw=white,fill=gray!20]  (-1,-1) rectangle (1,1);
    \draw[fill=white] (0,0) circle[radius=0.1];
    \node [right] at (0.1,0) {$p$};
    \node [left] at (-1,0) {$a=$};
    \draw (1,0) parabola bend (0,-0.5) (-1,0);
\end{tikzpicture}}
\;\;\;\;\;\;\;\;\;\;
\raisebox{-.35\height}{
\begin{tikzpicture}
\draw[draw=white,fill=gray!20]  (-1,-1) rectangle (1,1);
    \draw[fill=white] (0,0) circle[radius=0.1];
    \node [right] at (0.1,0) {$p$};
    \node [left] at (-1,0) {$a'=$};
    \draw (1,0) parabola bend (0,0.5) (-1,0);
\end{tikzpicture}}
         \caption{$n$-web diagrams $a$ and $a'$ in  
      $\Sigma_{g,1}$ related by sliding over the puncture $p$}
         \label{eq-relation-point-p}
     \end{figure}
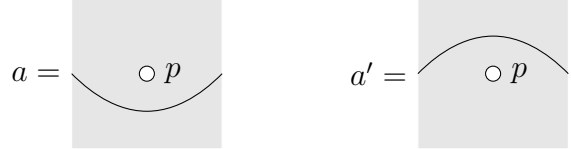

Since 
    ${\bf P}\colon\cS_{\hat \xi}(\Sigma_{g,1})\onto \cS_{\hat \xi}(\Sigma_g)$ sends 
$ \mathcal K_{\hxi,\Sigma_{g,1}}$ to $\mathcal K_{\hxi,\Sigma_g}$, it induces a quotient map
    $${\bf P}^{*}\colon\cS_{\hat \xi}^{*}(\Sigma_{g,1})\onto \cS_{\hat \xi}^{*}(\Sigma_g).$$

\begin{lemma}\label{lem-P-star}
    
 The kernel of $\bf P^*$  is
	 generated by ${\bf pr}(a) - {\bf pr}(a')$, where $a,a'\in \cS_{\hat \xi}(\Sigma_{g,1})$ 
     are 
     $n$-web diagrams 
     in $\Sigma_{g,1}$ related by sliding over the puncture, 
     as in Figure \ref{eq-relation-point-p},
     and ${\bf pr}$ denotes the projection from $\cS_\hxi(\Sigma_{g,1})$ to $\cS_\hxi^*(\Sigma_{g,1})$.
\end{lemma}

\begin{proof}
We have 
\begin{align*}
    \cS_\hxi^*(\Sigma_{g})=
    \cS_\hxi(\Sigma_{g})/{\bf P}(\mathcal K_{\hxi,\Sigma_{g,1}})
    =\dfrac{\cS_\hxi(\Sigma_{g,1})/{\rm ker}\,{\bf P}}{\mathcal K_{\hxi,\Sigma_{g,1}}/(\mathcal K_{\hxi,\Sigma_{g,1}}\cap {\rm ker}\,{\bf P})}
    =\dfrac{\cS_\hxi^{*}(\Sigma_{g,1})}{(\mathcal K_{\hxi,\Sigma_{g,1}} + {\rm ker}\,{\bf P})/\mathcal K_{\hxi,\Sigma_{g,1}}}.
\end{align*}
It is easy to see that $(\mathcal K_{\hxi,\Sigma_{g,1}} + {\rm ker}\,{\bf P})/\mathcal K_{\hxi,\Sigma_{g,1}}={\bf pr}({\rm ker}\,{\bf P})$.
This completes the proof since ${\rm ker}\,{\bf P}$ is
	 generated by $a-a'$, where $a,a'\in \cS_{\hat{\xi}}(\Sigma_{g,1})$ 
     are 
     $n$-web diagrams  
     in $\Sigma_{g,1}$ related by sliding over the puncture.

\end{proof}

For ease of notation, we often omit writing the projection map ${\bf pr} : \cS_\hxi(\fS) \to \cS_\hxi^*(\fS)$, and for an element $x$ of $\cS_\hxi(\fS)$, we still denote ${\bf pr}(x) \in \cS_\hxi^*(\fS)$ by $x$.

 We now establish the Frobenius homomorphism for the projected stated ${\rm SL}_n$-skein algebras for $\Sigma_{g,n}$.
 
 \def\Sfs{{\mathsf{Surf}^*}}
 \def\Sstr{{\mathsf{Surf}^{*, \mathsf{st}}}}
 \def\CAlg{{\BC\text{-}\mathsf{Alg}   }}

\begin{theorem}\label{thm-main-pb-surface-star}

Let $\fS=\Sigma_{g,n}$ for $g,n\in \BN$.
There is a unique $\BC$-algebra map 
$$
\Phi^*: \Sse(\fS) \to \Sso(\fS)
$$
such that if $\alpha$ is a simple loop diagram on $\fS$ then $\Phi^{*}([\alpha]_\heta) = [\bar P_{N,1}(\alpha)]_\home$.

Moreover, the map $\Phi^*$ has the following properties.
\begin{enumerate}
\item  If $
\alpha\subset \widetilde{\fS}=\fS\times(-1,1)$ is a framed oriented knot, 
then $
\Phi^*([
\alpha]_\heta) =\left[  
\alpha^ { [\bar P_{N,1} ] } \right]_\home$, 
where $
\alpha^{[ \bar P_{N,1}]} \in \cS(\fS)$ is the result of threading (Definition \ref{def.threading_of_polynomial}) the reduced power elementary polynomial $\bar P_{N,1}$ (\eqref{eq-def-barP} of \OldS\ref{ss.representations_of_sln}) along $
\alpha$. More generally, if $
\alpha= 
\alpha_1 \sqcup \dots \sqcup 
\alpha_
{r}\subset \widetilde{\fS}$, where 
each $
\alpha_j$ is a framed oriented knot, the $\Phi^*([
\alpha]_\heta)$ is given as the `term-by-term (multilinear) disjoint union' of $\Phi^*([
\alpha_j]_\heta)$, $j=1,\ldots,
r$, in the sense as in \eqref{Phi_for_disjoint_union}.

\item 
$\Phi^*$ commutes with the cutting homomorphism in \eqref{eq.cutting_homomorphism_projected}, in the sense that
\be \Phi \circ \Theta_c^* = \Theta ^*_c \circ \Phi^*.
\ee

\item 
$\Phi^*$
is compatible with the Frobenius homomorphisms $\Phi^\bT$ \eqref{eqFrDef} for quantum tori via 
the quantum traces for projected skein algebras in \eqref{eq.X-quantum_trace_projected}, meaning that for any ideal triangulation $\lambda$ of $\fS$, one has
\begin{align*}
\tr^{X*}_\lambda \circ 
\Phi^* = \Phi^\bT \circ \tr^{X*}_\lambda. 
\end{align*}

\end{enumerate}
\end{theorem}

\bpr 
{\bf Case 1:} $n >0$
Let $p$ be a puncture of $\fS$, with $\fS'$  as above. 
Recall the `puncture-enlarging' embedding $f: \fS \hookrightarrow \fS'$ and its induced map $f_*: \cS_\hxi(\fS) \to \cS_\hxi(\fS')$. By definition, $\Ssx(\fS)$ is identified with the image of $f_*$, hence 
with the subspace $A_\hxi(\fS') \subset \Sx(\fS')$ spanned by all $n$-web diagrams in $\fS'$ not touching $c_p$. Equation \eqref{Phi_for_disjoint_union} shows that if an  $n$-web diagram $\al$ in $\fS'$ does not touch $c_p$, then its image $\Phi([\al]_\heta)$ is a linear combination of $n$-web diagrams not touching $c_p$. This implies that $\Phi ( A_\heta(\fS') ) \subset A_\home(\fS').$

 Let $
\Phi^*$ be the restriction of $
\Phi$ onto $A_\heta (\fS')$. Then (1)--(3) follow from the corresponding properties of $
\Phi$ in Theorem \ref{thmFrob}.

{\bf Case 2:} $\fS=\Sigma_g$.
Lemma \ref{lem-P-star} implies that the kernel of 
${\bf P}^{*}\colon\cS_{\hat \xi}^{*}(\Sigma_{g,1})\onto \cS_{\hat \xi}^{*}(\Sigma_g)$ is
	 generated by $a-a'$, where $a, a'\in \cS_{\hat \xi}(\Sigma_{g,1})$ 
     are 
     $n$-web diagrams
     in $\Sigma_{g,1}$ related by sliding over the puncture, 
     as in Figure \ref{eq-relation-point-p}.
Because of relation \eqref{e.sinksource}, we can assume that both $a$ and $a'$ are oriented link diagrams.
For any $z\in \cS_\home(\An)$, the threadings $a *_\home z$ and $a' *_\home z$ (defined in Definition \ref{def.threading_of_element})  are also related by the sliding 
moves, hence they are equal under the projection  $\cS_{\hat \omega}^{*}(\Sigma_{g,1})\onto \cS_{\hat \omega}^{*}(\Sigma_g)$. Thus  $\Phi^*\colon \cS_\heta^{*}(\Sigma_{g,1})\to \cS_\home^{*}(\Sigma_{g,1})$ 
 descends to a map
	$$\Phi^*\colon \cS_\heta^{*}(\Sigma_g) \to \cS_\home^{*}(\Sigma_g)$$
	with all the required properties.
\epr

\begin{remark}\label{rem-810} In \cite{LS21} it is conjecture that $\cK_{
\hxi,\Sigma_{g,n}}=0$ for all $g,n\in \BN$.
    If the conjecture is true, then we have  the sought-for Frobenius homomorphism $
\Phi$ for $\SS$
for any pb surface $\fS$.
The conjecture was confirmed  for $n=2$ \cite{Le:triangulation}  and $n=3$ \cite{Hig23}.
It is shown in \cite{wang2024statedTQFT} that $\cK_{
\hxi,\Sigma_{g,n}}=0$ for general $n$ when $\xi=\hat{\xi}^{2n^2}=1$.
Note that $\eta=\heta^{2n^2}=\pm 1$ in Theorem \ref{thm-main-pb-surface-star}.
Then we have $\Sse(\fS)=\Se$ when $\eta=1$.
\end{remark}

\section{Root of unity transparency for knots}\label{sec-root-of-unity-transparency}

One of the very motivations of studying the Frobenius homomorphism $\Phi : \cS_\heta(\fS) \to \cS_\home(\fS)$ for stated ${\rm SL}_n$-skein algebras, as one can see e.g. in \cite{BW16} for $n=2$, is because it is helpful in the investigation of the center of $\cS_\home(\fS)$. Indeed, the image of $\Phi$ provides important elements in the center. More generally, for marked 3-manifolds $\MN$, the image of the Frobenius homomorphism $\Phi : \cS_\heta\MN \to \cS_\home\MN$ provides important elements of $\cS_\home\MN$ that are {\bf transparent}, which is a local version of the notion of commutativity (Remark \ref{rem.transparency}).

In this section we obtain a class of transparent elements of $\cS_\home\MN$ and $\cS_\home(\fS)$. We rely heavily on \cite{BH23}, but we establish and suggest a further development, using the results of the present paper on the Frobenius homomorphisms.

\subsection{Notations}

Let $\home \in \mathbb{C}$ be a root of unity. 
Recall that $\omega=\home^{2n^2}$
and  $\ord(\omega^2)=N$. Let $$d=\ord(\omega^{2N/n}).$$ Since $\omega^{2N}=1$, we have $d |n$. Let $\heta = \home^{N^2}$.

Let $\fS$ be a pb surface. The algebraic intersection gives rise to a bilinear form
$$
\la \cdot, \cdot \ra_\BZ: H_1(\fS;\BZ) \ot_{\BZ} H_1(\fS, \pfS; \BZ) \to \BZ.
$$

We identify $H_1(\An; \BZ_n)$ with $\BZ_n$ so that the element represented by the core curve  is identified with the generator $1\in \BZ_n$. 

Recall from \OldS\ref{sec-grading} the definition of $\cS(\fS)_k$ for $k\in H_1(\fS,\partial \fS;\mathbb Z_n)$; one defines $\cS_\home(\fS)_k$ this way as well. In the case $\fS = \An$, we have $H_1(\fS,\partial \fS;\mathbb Z_n)=\mathbb{Z}_n$, so $\cS_\home(\An)_k$ makes sense for $k\in\{1,\ldots,n\}=\mathbb{Z}_n$. For each divisor $d$ of $n$, we define
\be 
\nonumber
\cS^{(d)}_\home(\An) = \bigoplus_{k \in\{1,\ldots,n\}, \, d |k} \cS_\home(\An)_k.
\ee

\vspace{2mm}

\def\SoS{{\cS_\home(\fS)}}
\def\SqS{{\cS_{\hq}(\fS)}}
\def\SeS{{\cS_\heta(\fS)}}
\def\SdeS{\cS^{(d)}_\heta(\fS) }
\def\tPo{{\tilde \Phi_\home}}

\subsection{Root of unity transparency}
The following statement is based heavily on \cite{BH23}, even though the {\it stated} ${\rm SL}_n$-skein algebra  is not discussed there.

\begin{theorem}\label{thm-main} Assume that $\hat \omega\in \BC$ is a root of unity such that $[n]_{\omega}!\neq 0$. 
Assume that $\gamma$ is a framed oriented knot in a marked 3-manifold $\MN$, and $z\in \SeA_k$ where $k\in \BZ_n =\{0,1,\ldots,n-1\} = H_1(\An;\BZ_n)$; see \OldS\ref{sec-grading} for $\cS_\heta(\An)_k$. 

 In $\cS_{\hat \omega}\MN$, we have the following relations:
 \be 
 \raisebox{-.10in}{
\begin{tikzpicture}
\tikzset{->-/.style=
{decoration={markings,mark=at position #1 with
{\arrow{latex}}},postaction={decorate}}}
\filldraw[draw=white,fill=gray!20] (-0.7,-0.5) rectangle (2,0.5);
\draw(0,0) --(-0.7,0);
\draw (0,0) --(0.4,0);
\draw[decoration={markings, mark=at position 0.95 with {\arrow{>}}},postaction={decorate}](0.8,0) --(2,0);
\draw(-0.3,0.1) --(-0.3,0.5);
\draw[decoration={markings, mark=at position 0.5 with {\arrow{<}}},postaction={decorate}](-0.3,-0.5) --(-0.3,-0.1);
\filldraw[draw=black,fill=gray!20](-0.1,-0.3) rectangle (1.7,0.3);
\node at(0.8,0) {$\gamma*_{\home}\Psi(z)$};
\end{tikzpicture}}
=
 \omega^{\frac{2kN}{n}}
\raisebox{-.10in}{
\begin{tikzpicture}
\tikzset{->-/.style=
{decoration={markings,mark=at position #1 with
{\arrow{latex}}},postaction={decorate}}}
\filldraw[draw=white,fill=gray!20] (-0.7,-0.5) rectangle (2,0.5);
\draw (-0.4,0) --(-0.7,0);
\draw[decoration={markings, mark=at position 0.96 with {\arrow{>}}},postaction={decorate}](-0.2,0) --(2,0);
\draw[decoration={markings, mark=at position 0.2 with {\arrow{<}}},postaction={decorate}](-0.3,-0.5) --(-0.3,0.5);
%
\filldraw[draw=black,fill=gray!20](-0.1,-0.3) rectangle (1.7,0.3);
\node at(0.8,0) {$\gamma*_{\home}\Psi(z)$};
\end{tikzpicture}}
,\;
\raisebox{-.10in}{
\begin{tikzpicture}
\tikzset{->-/.style=
{decoration={markings,mark=at position #1 with
{\arrow{latex}}},postaction={decorate}}}
\filldraw[draw=white,fill=gray!20] (-0.7,-0.5) rectangle (2,0.5);
\draw(0,0) --(-0.7,0);
\draw (0,0) --(0.4,0);
\draw[decoration={markings, mark=at position 0.95 with {\arrow{>}}},postaction={decorate}](0.8,0) --(2,0);
\draw[decoration={markings, mark=at position 0.8 with {\arrow{>}}},postaction={decorate}](-0.3,0.1) --(-0.3,0.5);
\draw(-0.3,-0.5) --(-0.3,-0.1);
\filldraw[draw=black,fill=gray!20](-0.1,-0.3) rectangle (1.7,0.3);
\node at(0.8,0) {$\gamma*_{\home}\Psi(z)$};
\end{tikzpicture}}
=
 \omega^{-\frac{2kN}{n}}
\raisebox{-.10in}{
\begin{tikzpicture}
\tikzset{->-/.style=
{decoration={markings,mark=at position #1 with
{\arrow{latex}}},postaction={decorate}}}
\filldraw[draw=white,fill=gray!20] (-0.7,-0.5) rectangle (2,0.5);
\draw (-0.4,0) --(-0.7,0);
\draw[decoration={markings, mark=at position 0.96 with {\arrow{>}}},postaction={decorate}](-0.2,0) --(2,0);
\draw[decoration={markings, mark=at position 0.8 with {\arrow{>}}},postaction={decorate}](-0.3,0) --(-0.3,0.5);
\draw (-0.3,-0.5) --(-0.3,0);
\filldraw[draw=black,fill=gray!20](-0.1,-0.3) rectangle (1.7,0.3);
\node at(0.8,0) {$\gamma*_{\home}\Psi(z)$};
\end{tikzpicture}}
 \label{eq-trans}
 \ee
In particular, 
	\be 
	\raisebox{-.10in}{
\begin{tikzpicture}
\tikzset{->-/.style=
{decoration={markings,mark=at position #1 with
{\arrow{latex}}},postaction={decorate}}}
\filldraw[draw=white,fill=gray!20] (-0.7,-0.5) rectangle (2,0.5);
\draw (0,0) --(-0.7,0);
\draw (0,0) --(0.4,0);
\draw[decoration={markings, mark=at position 0.85 with {\arrow{>}}},postaction={decorate}](0.8,0) --(2,0);
\draw(0,0.1) --(0,0.5);
\draw [decoration={markings, mark=at position 0.5 with {\arrow{<}}},postaction={decorate}](0,-0.5) --(0,-0.1);
\filldraw[draw=black,fill=gray!20](0.3,-0.3) rectangle (1.3,0.3);
\node at(0.8,0) {$\gamma^{\bar P_{N,k}}$};
\end{tikzpicture}}= \omega^{\frac{2kN}{n}}
	\raisebox{-.10in}{
\begin{tikzpicture}
\tikzset{->-/.style=
{decoration={markings,mark=at position #1 with
{\arrow{latex}}},postaction={decorate}}}
\filldraw[draw=white,fill=gray!20] (-0.7,-0.5) rectangle (2,0.5);
\draw (-0.1,0) --(-0.7,0);
\draw (0.1,0) --(0.4,0);
\draw[decoration={markings, mark=at position 0.8 with {\arrow{>}}},postaction={decorate}](0.8,0) --(2,0);
\draw[decoration={markings, mark=at position 0.92 with {\arrow{>}}},postaction={decorate}](0,0.5) --(0,-0.5);
%
\filldraw[draw=black,fill=gray!20](0.3,-0.3) rectangle (1.3,0.3);
\node at(0.8,0) {$\gamma^{\bar P_{N,k}}$};
\end{tikzpicture}},\;
	\raisebox{-.10in}{
\begin{tikzpicture}
\tikzset{->-/.style=
{decoration={markings,mark=at position #1 with
{\arrow{latex}}},postaction={decorate}}}
\filldraw[draw=white,fill=gray!20] (-0.7,-0.5) rectangle (2,0.5);
\draw(0,0) --(-0.7,0);
\draw (0,0) --(0.4,0);
\draw[decoration={markings, mark=at position 0.85 with {\arrow{>}}},postaction={decorate}](0.8,0) --(2,0);
\draw[decoration={markings, mark=at position 0.8 with {\arrow{>}}},postaction={decorate}](0,0.1) --(0,0.5);
\draw(0,-0.5) --(0,-0.1);
\filldraw[draw=black,fill=gray!20](0.3,-0.3) rectangle (1.3,0.3);
\node at(0.8,0) {$\gamma^{\bar P_{N,k}}$};
\end{tikzpicture}}= \omega^{-\frac{2kN}{n}}
	\raisebox{-.10in}{
\begin{tikzpicture}
\tikzset{->-/.style=
{decoration={markings,mark=at position #1 with
{\arrow{latex}}},postaction={decorate}}}
\filldraw[draw=white,fill=gray!20] (-0.7,-0.5) rectangle (2,0.5);
\draw (-0.1,0) --(-0.7,0);
\draw (0.1,0) --(0.4,0);
\draw[decoration={markings, mark=at position 0.8 with {\arrow{>}}},postaction={decorate}](0.8,0) --(2,0);
\draw[decoration={markings, mark=at position 0.8 with {\arrow{>}}},postaction={decorate}](0,0) --(0,0.5);
\draw (0,-0.5) --(0,0);
\filldraw[draw=black,fill=gray!20](0.3,-0.3) rectangle (1.3,0.3);
\node at(0.8,0) {$\gamma^{\bar P_{N,k}}$};
\end{tikzpicture}}.
	\label{eq-trans1}
	\ee
	Here $\raisebox{-.10in}{
\begin{tikzpicture}
\tikzset{->-/.style=
{decoration={markings,mark=at position #1 with
{\arrow{latex}}},postaction={decorate}}}
\filldraw[draw=white,fill=gray!20] (-0.3,-0.5) rectangle (2,0.5);
\draw(0,0) --(-0.3,0);
\draw (0,0) --(0.4,0);
\draw[decoration={markings, mark=at position 0.95 with {\arrow{>}}},postaction={decorate}](0.8,0) --(2,0);
\filldraw[draw=black,fill=gray!20](-0.1,-0.3) rectangle (1.7,0.3);
\node at(0.8,0) {$\gamma*_{\home}\Psi(z)$};
\end{tikzpicture}}$ represents $\gamma *_\home \Psi(z)$ (Definition \ref{def.threading_of_element}; see Proposition \ref{r-L-comp-F} for $\Psi : \cS_\heta(\An) \to \cS_\home(\An)$) and
	$\raisebox{-.10in}{
\begin{tikzpicture}
\tikzset{->-/.style=
{decoration={markings,mark=at position #1 with
{\arrow{latex}}},postaction={decorate}}}
\filldraw[draw=white,fill=gray!20] (0,-0.5) rectangle (2,0.5);
\draw[decoration={markings, mark=at position 0.85 with {\arrow{>}}},postaction={decorate}](0.8,0) --(2,0);
\draw(0,0) --(0.8,0);
\filldraw[draw=black,fill=gray!20](0.3,-0.3) rectangle (1.3,0.3);
\node at(0.8,0) {$\gamma^{\bar P_{N,k}}$};
\end{tikzpicture}}$ represents $\gamma^{[\bar P_{N,k}]}$ (Definition \ref{def.threading_of_polynomial}; see \eqref{eq-def-barP} of \OldS\ref{ss.representations_of_sln} for $\bar P_{N,k}$).
\end{theorem}

\begin{proof} 
Let us prove the first identity of  \eqref{eq-trans}, as the second is obtained from the first by the reflection anti-involution in Proposition \ref{r.reflection} of \OldS\ref{sec.reflection}. 

By Theorem \ref{twice_sphere}, $z$ is a polynomial in ${\bf a}_1, \dots, {\bf a}_{n-1}$. Since ${\bf a}_i\in \SeA_i$, it is enough to consider the case when $z= {\bf a}_k$. In that case \eqref{eq-trans} is \eqref{eq-trans1}, which we prove now.

By the last two equations in the proof of  \cite[Proposition 11]{BH23},   there are elements $x_{N,k}, y_{N,k}\in \SMN$ such that 
\begin{align}
\raisebox{-.10in}{
\begin{tikzpicture}
\tikzset{->-/.style=
{decoration={markings,mark=at position #1 with
{\arrow{latex}}},postaction={decorate}}}
\filldraw[draw=white,fill=gray!20] (-0.7,-0.5) rectangle (2,0.5);
\draw (0,0) --(-0.7,0);
\draw (0,0) --(0.4,0);
\draw[decoration={markings, mark=at position 0.85 with {\arrow{>}}},postaction={decorate}](0.8,0) --(2,0);
\draw(0,0.1) --(0,0.5);
\draw [decoration={markings, mark=at position 0.5 with {\arrow{<}}},postaction={decorate}](0,-0.5) --(0,-0.1);
\filldraw[draw=black,fill=gray!20](0.3,-0.3) rectangle (1.3,0.3);
\node at(0.8,0) {$\gamma^{\bar P_{N,k}}$};
\end{tikzpicture}}&=  q ^{\frac {Nk} n + N} x_{N,k} +  q ^{\frac {Nk} n} y_{N,k}
\label{eq75a}
\\
\raisebox{-.10in}{
\begin{tikzpicture}
\tikzset{->-/.style=
{decoration={markings,mark=at position #1 with
{\arrow{latex}}},postaction={decorate}}}
\filldraw[draw=white,fill=gray!20] (-0.7,-0.5) rectangle (2,0.5);
\draw (-0.1,0) --(-0.7,0);
\draw (0.1,0) --(0.4,0);
\draw[decoration={markings, mark=at position 0.8 with {\arrow{>}}},postaction={decorate}](0.8,0) --(2,0);
\draw[decoration={markings, mark=at position 0.92 with {\arrow{>}}},postaction={decorate}](0,0.5) --(0,-0.5);
%
\filldraw[draw=black,fill=gray!20](0.3,-0.3) rectangle (1.3,0.3);
\node at(0.8,0) {$\gamma^{\bar P_{N,k}}$};
\end{tikzpicture}} &=  q ^{-\frac {Nk} n - N} x_{N,k} +  q ^{-\frac {Nk} n} y_{N,k}. \label{eq75b}
\end{align}
When $\hq=\home$ we have $\omega^N= \omega ^{-N}$. Then it is clear that 
$$\raisebox{-.10in}{
\begin{tikzpicture}
\tikzset{->-/.style=
{decoration={markings,mark=at position #1 with
{\arrow{latex}}},postaction={decorate}}}
\filldraw[draw=white,fill=gray!20] (-0.7,-0.5) rectangle (2,0.5);
\draw (0,0) --(-0.7,0);
\draw (0,0) --(0.4,0);
\draw[decoration={markings, mark=at position 0.85 with {\arrow{>}}},postaction={decorate}](0.8,0) --(2,0);
\draw(0,0.1) --(0,0.5);
\draw [decoration={markings, mark=at position 0.5 with {\arrow{<}}},postaction={decorate}](0,-0.5) --(0,-0.1);
\filldraw[draw=black,fill=gray!20](0.3,-0.3) rectangle (1.3,0.3);
\node at(0.8,0) {$\gamma^{\bar P_{N,k}}$};
\end{tikzpicture}} = \omega^{2kN/n} \raisebox{-.10in}{
\begin{tikzpicture}
\tikzset{->-/.style=
{decoration={markings,mark=at position #1 with
{\arrow{latex}}},postaction={decorate}}}
\filldraw[draw=white,fill=gray!20] (-0.7,-0.5) rectangle (2,0.5);
\draw (-0.1,0) --(-0.7,0);
\draw (0.1,0) --(0.4,0);
\draw[decoration={markings, mark=at position 0.8 with {\arrow{>}}},postaction={decorate}](0.8,0) --(2,0);
\draw[decoration={markings, mark=at position 0.92 with {\arrow{>}}},postaction={decorate}](0,0.5) --(0,-0.5);
%
\filldraw[draw=black,fill=gray!20](0.3,-0.3) rectangle (1.3,0.3);
\node at(0.8,0) {$\gamma^{\bar P_{N,k}}$};
\end{tikzpicture}}.$$

	\end{proof}
	
	\brem We do not require $\MN$ to be essentially marked. Even though in \cite{BH23} only non-marked 3-manifolds are considered, the proof of \eqref{eq75a} and \eqref{eq75b} involves only $n$-web diagrams in a small neighborhood of the union of $\gamma$ and the shaded rectangle, and hence is valid in $\SMN$.
	\erem

\begin{remark}
	We have restriction $[n]_\omega!\neq 0$ for Theorem \ref{thm-main} and its consequences. 
For $n=2$ the restriction can be lifted \cite{BW16,Le15,BL22}. It is not too difficult to lift this restriction also when $n=3$, see \cite[Proposition 9.16 and Remark 9.17]{kim2024unicity}. 
\end{remark}

\def\Frob{\mathsf{Fro}_\home}

\subsection{Transparent elements}
An immediate corollary of Theorem \ref{thm-main} is the following.
\bpro 
\label{r-trans5}
Assume that in Theorem \ref{thm-main} we have $z\in \cS_\heta^{(d)}(\An)$. Then $\gamma *_\home \Psi(z)$ is {\bf transparent} in $\SoM$. This means that if $\gamma'$ is isotopic to $\gamma$ in $M$, and $\al$ is a stated $n$-web disjoint from both $\gamma$ and $\gamma'$, then in $\SoM$ we have 
$$  \al 
\sqcup (\gamma *_\home \Psi(z)) =   \al 
\sqcup (\gamma' *_\home \Psi(z)), $$
where $\sqcup$ in both sides is the `term-by-term (bilinear) disjoint union' in the sense as in \eqref{Phi_for_disjoint_union}.
\epro

\begin{remark}\label{rem.transparency}
    The transparency condition means that while we isotopy $\gamma$ to $\gamma'$, we can `pass through' other $n$-webs $\alpha$, which explains the terminology `transparent'. For a surface $\fS$, transparent elements of $\cS_\home(\fS)$, which is an algebra, are central.
\end{remark}

 Let $\fS$ is a pb surface. 
By Proposition \ref{r-trans5}, the set
\be 
\Frob'(\fS) := \BC\text{-span}\{ \gamma* \cS^{(d)}_\heta(\An) \,\mid\, 
\gamma: \text{ framed oriented knots in   } \tfS  \}
\label{eqFrobp}
\ee
is  in the center of $\SoS$. The next result, also an easy corollary of Theorem \ref{thm-main}, shows that the center might be bigger.

Let $\gamma$ be an oriented framed link in 
$\widetilde{\fS}$. Define 
\be \tPo(\gamma) := \left[ \gamma^{[\bar P_{N,1}]}\right]_\home \in \SoS,
\ee
where the right hand side is defined as threading in Definition \ref{def.threading_of_polynomial}. 
Note that  $\tPo(\gamma) \in \SoS$ is well-defined for an isotopy class of $\gamma$. If $\fS$ is essentially bordered, then by \eqref{eq-Phi} of Theorem \ref{thm11}, we have
$$ \tPo(\gamma) = 
\Phi( [\gamma]_\heta  ),$$
showing that $\tPo(\gamma)$ is well-defined when $\gamma$ is considered as an element $\SeS$.

\bpro \label{thm-central}
Assume that a  framed oriented link $\gamma$ be  in the thickened pb surface $\tfS$  is {\bf $\BZ_d$-homology annihilator}, meaning that the homology class $[\gamma]\in H_1(\fS;\BZ)$ of $\gamma$   satisfies
$$  \la [\gamma], x \ra_\BZ = 0 \mod d \quad \text{for all} \ x\in H_1(\fS,\pfS;\BZ).  $$

Then  $\tPo(\gamma)= \left[ \gamma^{[\bar P_{N,1}]}\right]_\home$ is in the center of $\SoS$.

\epro

\bpr Let $\beta $ be a stated $n$-web diagram in $\fS$. Assume that the components of $\gamma$ are $
\gamma_1, \dots, \gamma_{r}$.
Since ${\bf a}_1 \in \SeA_1$, by Theorem \ref{thm-main},
$$ \left(
\bigsqcup_{j=1}^{r}
\gamma_j *_\home \Psi({\bf a}_1) 
\right) \beta = \omega^{ \frac{2N}n \la \gamma, \beta\ra  }\, \beta \left(
\bigsqcup_{j=1}^{r}
\gamma_j *_\home \Psi({\bf a}_1) \right),
$$
where $\bigsqcup_{j=1}^{r}$ stands for the `term-by-term (multilinear) disjoint union' in the sense of \eqref{Phi_for_disjoint_union}.
The assumption $\omega^{ \frac{2N}n \la \gamma, \beta\ra  }=1$ proves that $\beta$ and $\tPo(\gamma) = [\gamma^{[\bar{P}_{N,1}]}]_\home$ commute.
\epr

\brem 
We have the following more general construction of central elements: Assume that $z_j \in \SeA_{k_j}$ for $j=1, \dots, m$ such that $\sum_{j} k_j \gamma _j$ is a $\BZ_d$-homology annihilator. Then the same proof shows that $
\bigsqcup_j (  \gamma_j *_\home \Psi(z_j))$ is central. However, this seemingly more general construction does not give new central elements, due to the following.
\bpro Suppose that $\al$ is a framed oriented knot in $\MN$ and $z\in \SeA_k$. Then $\al *_\home \Psi(z)$ is a $\BC$-linear combination of elements of the form $\tPo(\gamma)$, where $\gamma$ are oriented framed links with $k$ components.
\epro

\bpr[Sketch of Proof] It is enough to consider the case $z= {\bf a}_k'$ given in Figure \ref{fig1}. Using the defining relation \eqref{e.sinksource}, in $\SeA$ we can express ${\bf a}_k'$ as a linear combination of framed oriented links. The functoriality of $\Psi$ (see Theorem \ref{thm10})  will complete the proof.
\epr
\erem

 Assume that $\fS$ is an essentially bordered pb surface. Then we have the Frobenius homomorphism $
 \Phi: \SeS \to \SoS$, and we have $\tPo(\gamma)= 
 \Phi( [\gamma]_\heta  ).$
 Let 
 $$\cS^{(d)}_\heta(\fS)\subset \SeS$$
  be the subspace spanned by all elements represented by 
  `closed' (i.e. endpointless) $n$-webs
  that are $\BZ_d$-homology annihilators. Then Proposition \ref{thm-central} can be reformulated as follows.
\bpro \label{prop.some_central_elements}
 Assume that $\fS$ is an essentially bordered pb surface.
The image of $\SdeS$ under the Frobenius homomorphism $\Phi : \cS_\heta(\fS)\to \cS_\home(\fS)$ is in the center of $\SoS$. \qed
\epro

 Assume now that $\fS=\Sigma_{g,m}$, the surface of genus $g$ with $m$ punctures removed. We don't exclude the case $m=0$.

For $\fS = \Sigma_{g,m}$ we do not have a Frobenius homomorphism for (stated) ${\rm SL}_n$-skein algebras yet, but we can define a multivalued substitution as follows. Assume that $x\in \SdeS$. Then $x$ can be  presented as a linear combination 
\be 
 x = \sum c_\gamma \gamma, \qquad c_\gamma \in \BC, \label{eq-pre}
 \ee
where each $\gamma$ is a framed oriented link 
that is a $\BZ_d$-homology annihilator.

\bcon\label{conj.Frobenius_for_Zd_elements} The sum $ \sum c_\gamma  \tPo(\gamma) \in \SoS$ depends only on $x$, but not on the choice of the presentation~\eqref{eq-pre}. 
\econ
If Conjecture \ref{conj.Frobenius_for_Zd_elements} is true, then we can define the Frobenius homomorphism by $\Phi_\home(x) = \sum c_\gamma \tPo(\gamma)$.

Without assuming this conjecture, define $\Phi_\home(x)$ as the set of all possible values of $ \sum c_\gamma \tPo(\gamma) \in \SoS$, when the presentations \eqref{eq-pre} vary. Let 
$$\Phi_\home(\SdeS  ) := \bigcup_{x\in \SdeS}\Phi_\home(x).$$
Then  Proposition \ref{thm-central} implies that
\bpro
Assume that $\fS=\Sigma_{g,m}$.
The set $\Phi_\home(\SdeS  )$ is in the center of $\SoS$. \qed
\epro

Meanwhile, there are obvious central elements of $\SoS$. 
An $n$-web diagram is 
{\bf peripheral} if, after applying an isotopy, it lies in 
an arbitrarily small neighborhood of the  punctures. Clearly every 
peripheral $n$-web diagram determines a central element of $\SoS$. 
\def\Frob{\mathsf{Fro}_\home }

\bcon Assume that $\fS=\Sigma_{g,m}$. Then the center of $\SoS$ is generated by $\Frob(\fS):=\Phi_\home(\SdeS  )$ and 
peripheral $n$-web diagrams.
\label{conj5}
\econ

Conjecture \ref{conj5} is proved for the case of ${\rm SL}_2$ in \cite{FKL}, where the proof is based on the existence of an explicit, geometric basis of $\SS$. For ${\rm SL}_3$ and $m \ge 1$, the Conjecture is proved in \cite{kim2024unicity}.
For some results on the center of $\SoS$, for ${\rm SL}_n$, see \cite{KW24}.

Let us compare Conjecture \ref{conj5} with \cite[Conjecture 16]{BH23}.
Let $N'= \ord(\omega^{2/n})=Nd$. 
Then \cite[Conjecture 16]{BH23} is the same as Conjecture \ref{conj5}, except that $\Frob(\fS)$ is replaced by 
$\Frob''(\fS)$, 
which is the subalgebra of $\SoS$ generated by all threading 
$\Psi_{N'} (\cS(\An))$
along all framed oriented knots. Note that $\Frob (\fS) \supset \Frob'(\fS)$, where the latter is the subalgebra of $\SoS$ generated by all threading 
$\Psi_N (\cS^{(d)}(\An))$
along all framed oriented knots. 
 Using $\Psi_{N'}= \Psi_N \circ \Psi_d$ one can  check that 
$\Psi_{N'} ( \cS(\An)) \subset \Psi_N (\cS^{(d)}(\An))$. The following arguments show that $\Psi_{N'} ( \cS(\An)) \neq \Psi_N (\cS^{(d)}(\An))$.

Let $n=d=3$, and hence $N'=3 N$. Identify $\cS_\home(A)= \cS_\heta(A) =\BC[x_1, x_2, x_3]^{S_3}= \BC[e_1,e_2] $
so that $e_1= \al_1$  and $e_2=\al_2$.  Since $\al_1+ \al_2=0$ in $H_1(\An,\BZ_3)$, we have $\Psi_N(e_1) \Psi_N(e_2) \in \Psi_N (\cS^{(d)}(\An)) $, and we will show that it is not in $\Psi_{N'} ( \cS(\An))$.

The map $\Psi_N: \BC[e_1, e_2] \to \BC[e_1, e_2]$, for any $N$, can be defined inductively by
$$ \Psi_0(e_i) =3, \Psi_1(e_i)= e_i, \Psi_2(e_i) = e_i^2 - 2 e_{ i'}, \Psi_m(e_i) = e_i \Psi_{m-1}(e_i) - e_{ i'} \Psi_{m-2}(e_i) + \Psi_{m-3}(e_i),$$
$\text{where}\   i' = 3-i.$

It follows that \be \Psi_N(e_1 e_2) = e_1^{N} e_2^{N} + (\text{lower total degree terms}).\label{eq.extr}\ee
Here the `total degree' of a term $e_1^{k_1}e_2^{k_2}$ is $k_1+k_2$. The set $ \{ \Psi_{N'}(e_1^k) \Psi_{N'}(e_2^l) \mid k,l \in \BN\}$ is a basis of $\Psi_{N'} (\BC[e_1, e_2] )$. We have
\begin{align}\label{eq-threading-n-prime}
    \Psi_{N'}(e_1^k) \Psi_{N'}(e_2^l) =  e_1^{kN'} e_2^{lN'} + (\text{lower total degree terms}).
\end{align}
It follows from equations \eqref{eq.extr} and \eqref{eq-threading-n-prime} that $\Psi_N(e_1 e_2)$  is not in the span of $ \{ \Psi_{N'}(e_1^k) \Psi_{N'}(e_2^l) \mid k,l \in \BN\}$.

\appendix

\def\bbb{{\mathbf b}}
\def\bbB{{\mathbf B}}
\section{Frobenius homomorphism for bigon}\label{appendix-bigon}

We provide a proof of Theorem \ref{Fro_Oq} of \OldS\ref{ss.Frobenius_map_for_Oq}, as promised. 
 Let $q\in \BC^*$, $\tilde A$ and $B$  be $\BC$-algebras. 
Assume that $\avec  x:=(x_1, \dots, x_n)\in \tilde A^n$ is a {\em a $q$-point}, meaning that $ x_j x_i =q x_i x_j$ for $i<j$, and furthermore algebraically independent, meaning that the subalgebra $\Pol(x_1, \dots, x_n)$ generated by $x_1, \dots, x_n$ has
$\{x^{\bk} \mid \bk \in \BN^n\}$ as a basis.  Let $\bbb= (b_{ij})_{i,j=1}^n$ be a matrix with entries in $B$. Define $\avec  y= \bbb \ot \avec  x $ and $\avec  z= \bbb^T \ot \avec  x$, i.e.
$$y_i = \sum _j b_{ij} \ot x_j, z_i = \sum _j b_{ji} \ot x_j \in B \ot \Pol(x_1,\dots, x_n).$$ The following is one of the definitions of $q$-quantum matrices.
\blem \label{rqM}
$\bbb$ is a $q$-matrix if and only if both $\avec  y, \avec  z$ are $q$-points.
\elem

If $q$ is a primitive root of 1 of order $N$, then from the quantum binomial formula,
\be (x_1 + \dots + x_n)^N = x_1^N + \dots + x_n^N. \label{eqbinom}
\ee

\bpro There exists a unique algebra monomorphism $\Phi\colon \mathcal M_\eta (n)\to\mathcal M_\omega(n) $ given by $\Phi(u_{ij})= u_{ij}^N$.
\epro
\bpr Let $q=\omega, \avec  X=(x^N_1, \dots, x^N_n), \avec  Y=(y^N_1, \dots, y^N_n), \avec  Z=(z^N_1, \dots, z^N_n), \bbB= (u^N_{ij}) $. Then 

(*) $\avec  X, \avec  Y, \avec  Z$ are  $\eta$-points, and  $\avec  X$ is algebraically independent.

The $q$-commutations show that $(u_{i1} \ot x_1, \dots, u_{in} \ot x_n)$ is a $\omega^2$-point. By \eqref{eqbinom},
$$ 
y_i^N=  \sum _j u_{ij}^N \ot x_j^N.
$$
Thus $\avec  Y = \bbB \ot \avec  X$. Similarly $\avec  Z = \bbB^T \ot \avec  X$. By Lemma \ref{rqM},  $\bbB$ is an $\eta$-matrix.

\epr

Let $\xi$ be a non-zero complex number.
For any $x\in\Oq$, we will use $[x]_{\xi}$ to denote the image of $x$ under the projection
$\Oq\rightarrow\mathcal O_\xi$.
Define the monoid
\begin{equation}\label{eq.Gamma}
\Gamma = \Mat_n(\BN)/ (\Id).
\end{equation}
Here $\Mat_n(\BN)=\BN^{n\times n}$ is an additive monoid, and $(\Id)$ is the submonoid generated by the identity matrix. Two matrices $m,m'\in \Mat_n(\BN)$ determine the same element in $\Gamma$ if and only if $m-m'= k \Id$ for $k\in \BZ$. Each $m\in \Gamma$ has a unique lift $\hat m\in \Mat_n(\BN)$, called the \term{minimal representative}, such that $\min_{i} \hat m_{ii}=0$. Note that $\Gamma\cong \BN^{n^2-n}\oplus \BZ^{n-1}$, hence it is a submonoid of a free abelian group.

\bpro[\cite{PW}] \label{r.basisOq}
Let  $\ord$ be a linear order on  $\JJ^2$.
We have the following:
\begin{enuma}
    \item  the set
\begin{equation*}
B^\ord:=\left\{b^{\ord}(m):= \prod_{(i,j)\in \JJ ^2} u_{ij}^{\hat m _{ij}} \mid m \in \Gamma = \Mat_n(\BN)/(\Id)\right\}
\end{equation*}
where the product is taken with respect to the order $\ord$, is a free $R$-basis of $\Oq$. 

\item for any non-zero complex number $\xi$, the set
\begin{equation*}
[B^\ord]_{\xi}:=\left\{[b^\ord(m)]_{\xi} \mid m \in \Gamma = \Mat_n(\BN)/(\Id)\right\}
\end{equation*}
is a basis of $\mathcal O_\xi$. 
\end{enuma}

\epro

\begin{lemma}
    $\Phi\colon \mathcal M_\omega (n)\to\mathcal M_\eta(n) $ induces an injective $\mathbb C$-algebra homomorphism
    $\Phi\colon \mathcal O_\eta \to\mathcal O_\omega$.
\end{lemma}

\begin{proof}
    Since $\Phi(\det_{\eta}({\bf u}))
    =\det_{\omega}({\bf u})$, we have that 
    $\Phi\colon \mathcal M_\omega (n)\to\mathcal M_\eta(n) $ induces a  $\mathbb C$-algebra homomorphism
    $\Phi\colon \mathcal O_\eta \to\mathcal O_\omega$.
    It is a trivial that 
    $\Phi$ maps $[B^\ord]_{\eta}$ injectively to $[B^\ord]_{\omega}$. 
    Then Proposition \ref{r.basisOq}(b) shows that $\Phi\colon \mathcal O_\eta \to\mathcal O_\omega$ is injective.
\end{proof}

\section{Proofs of Lemmas \ref{lem-height-Nparallel} and \ref{lem-arcs}}\label{Appendix-2}
\begin{proof}[Proof of Lemma \ref{lem-height-Nparallel}]
We only prove the second the identity since the following proving technique also applies for the first one.
We suppose that the local picture with a state $i$ (resp. $j$) is a part of the $\cN$-arc $a$ (resp. $b$).

{\bf Case 1} when $a\neq b$:
Since $\omega^2$ is a primitive $N$-th root of unity, then we have 
$1+\omega^2+\cdots+\omega^{2(N-1)}=0$. 
Then \cite[Lemma 7.6]{Wan23} implies 
$$
\raisebox{-.40\height}{
		\begin{tikzpicture}
			\draw[color=gray!20,fill=gray!20] (-1,-1.2) rectangle (1,1.2);
			\draw[wall,<-] (1,-1.2)-- (1,1.2);
			\draw[line width =1pt] (-1,-0.8)--(1,0.8);
			\draw[color=gray!20,fill=gray!20] (-0.1,-0.08) rectangle (0.1,0.08);
			\draw[line width =1pt] (-1,0.8)--(1,-0.8);
			\draw[fill=gray!20] (-0.7,0.2) rectangle (-0.3,0.6);
			\node  at (-0.5,0.4) {$N$};
			\node [right] at (1,0.8) {$i$};
			\node [right] at (1,-0.8) {$j$};
			\draw[fill=black] (0.5,0.4) circle[radius=0.1] ;
			\draw[fill=white] (0.5,-0.4) circle[radius=0.1] ;
		\end{tikzpicture}
	}=\omega^{\frac{1}{n}-\delta_{\bar{i},j}}
	\raisebox{-.40\height}{
		\begin{tikzpicture}
			\draw[color=gray!20,fill=gray!20] (-1,-1.2) rectangle (1,1.2);
			\draw[wall,<-] (1,-1.2)-- (1,1.2);
			\draw[line width =1pt] (-1,-0.4)--(1,-0.4);
			\draw[color=gray!20,fill=gray!20] (-0.1,-0.08) rectangle (0.1,0.08);
			\draw[line width =1pt] (-1,0.4)--(1,0.4);
			\draw[fill=gray!20] (-0.7,0.2) rectangle (-0.3,0.6);
			\node  at (-0.5,0.4) {$N$};
			\node [right] at (1,0.4) {$j$};
			\node [right] at (1,-0.4) {$i$};
			\draw[fill=white] (0.5,0.4) circle[radius=0.1] ;
			\draw[fill=black] (0.5,-0.4) circle[radius=0.1] ;
		\end{tikzpicture}
	}.
	$$
Applying the above identity $N$ times, we have 
$$
\raisebox{-.40\height}{
		\begin{tikzpicture}
			\draw[color=gray!20,fill=gray!20] (-1,-1.2) rectangle (1,1.2);
			\draw[wall,<-] (1,-1.2)-- (1,1.2);
			\draw[line width =1pt] (-1,-0.8)--(1,0.8);
			\draw[color=gray!20,fill=gray!20] (-0.1,-0.08) rectangle (0.1,0.08);
			\draw[line width =1pt] (-1,0.8)--(1,-0.8);
			\draw[fill=gray!20] (-0.7,0.2) rectangle (-0.3,0.6);
			\node  at (-0.5,0.4) {$N$};
			\draw[fill=gray!20] (-0.7,-0.2) rectangle (-0.3,-0.6);
			\node  at (-0.5,-0.4) {$N$};
			\node [right] at (1,0.8) {$i$};
			\node [right] at (1,-0.8) {$j$};
			\draw[fill=black] (0.5,0.4) circle[radius=0.1] ;
			\draw[fill=white] (0.5,-0.4) circle[radius=0.1] ;
		\end{tikzpicture}
	}=\eta^{\frac{1}{n}-\delta_{\bar{i},j}}
	\raisebox{-.40\height}{
		\begin{tikzpicture}
			\draw[color=gray!20,fill=gray!20] (-1,-1.2) rectangle (1,1.2);
			\draw[wall,<-] (1,-1.2)-- (1,1.2);
			\draw[line width =1pt] (-1,-0.4)--(1,-0.4);
			\draw[color=gray!20,fill=gray!20] (-0.1,-0.08) rectangle (0.1,0.08);
			\draw[line width =1pt] (-1,0.4)--(1,0.4);
			\draw[fill=gray!20] (-0.7,0.2) rectangle (-0.3,0.6);
			\node  at (-0.5,0.4) {$N$};
			\draw[fill=gray!20] (-0.7,-0.2) rectangle (-0.3,-0.6);
			\node  at (-0.5,-0.4) {$N$};
			\node [right] at (1,0.4) {$j$};
			\node [right] at (1,-0.4) {$i$};
			\draw[fill=white] (0.5,0.4) circle[radius=0.1] ;
			\draw[fill=black] (0.5,-0.4) circle[radius=0.1] ;
		\end{tikzpicture}
	}.
$$

 \begin{figure}
     \centering
     \includegraphics[scale=0.8]{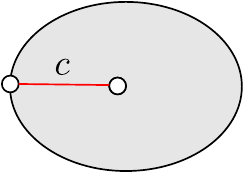}
     \caption{The ideal arc $c$ in the once punctured monongon.}
     \label{fig-ideal-monogon}
 \end{figure}

{\bf Case 2} when $a=b$. 
Suppose that the two endpoints of $a$ are contained in the component $e$ of $\cN$. Then the regular open neighborhood of $a\cup e$  is diffeomorphic with the thickening of the once punctured monogon $\PP_{1,1}$.
Then it suffices to show the following
\begin{align}\label{eq-height-change-P4}
    \begin{array}{c}\includegraphics[scale=0.7]{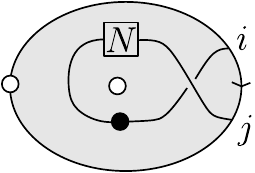}\end{array}
    =\eta^{\frac{1}{n}-\delta_{\bar i,,j}}
    \begin{array}{c}\includegraphics[scale=0.7]{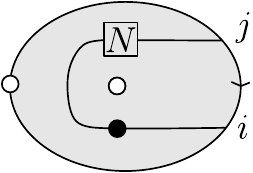}\end{array}.
\end{align}
We use $c$ to denote the horizontal ideal arc connecting the two punctures of the once punctured monogon, see Figure \ref{fig-ideal-monogon}. Then we have 
\begin{align*}
    \Theta_c \left(\begin{array}{c}\includegraphics[scale=0.7]{N-monogon-cross.pdf}\end{array}\right)
  &=\sum_{t\in\mathbb J}
  \begin{array}{c}\includegraphics[scale=0.75]{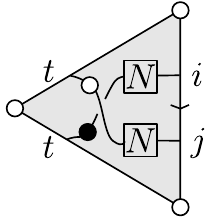}\end{array}
  \quad (\because \mbox{Lemma \ref{rcut}})\\
  &=\eta^{\frac{1}{n}-\delta_{\bar i,,j}}
  \sum_{t\in\mathbb J}
  \begin{array}{c}\includegraphics[scale=0.75]{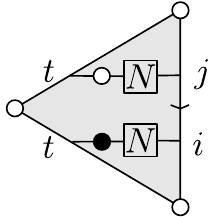}\end{array}
  \quad (\because \mbox{Case 1})\\
  &= \eta^{\frac{1}{n}-\delta_{\bar i,,j}}
   \Theta_c \left(\begin{array}{c}\includegraphics[scale=0.7]{N-monogon-non-cross.pdf}\end{array}\right)
   (\because \mbox{Lemma \ref{rcut}}).
\end{align*}
Then we have equation \eqref{eq-height-change-P4} because $\Theta_c$ is injective.

\end{proof}

Recall that $\Oq \cong \cS(\PP_2)$ as $\Zhq$ algebras;
see Theorem \ref{thm-iso-Oq-P2}. 
Then the Hopf algebra structure of $\Oq$ induces a  Hopf algebra structure of 
$\cS(\PP_2)$ (see \cite{LS21} for geometric interpretation of this Hopf algebra structure).

Given a marking $\beta$ of a marked $3$-manifold $(M, \cN)$, consider its closed disk neighborhood $D$ in $\partial M$, disjoint from the other markings of $(M, \cN)$. By pushing the interior of $D$ inside $M$, we get a new disk $D'$ which is properly embedded in $M$. Cutting $(M, \cN)$ along $D'$, we get a new marked 3-manifold $(M', \cN')$ isomorphic to $(M, \cN)$, and another marked 3-manifold bounded by $D$ and $D'$. The latter, after removing the common boundary of $D$ and $D'$, is isomorphic to the thickening of the bigon, with $\beta$ considered its right face marking. Hence, this construction yields an $\Zhq$-linear cutting map.
\[
\Delta_\beta\colon\cS\MN \to \cS\MN\otimes_{\Zhq} \Oq.
\]
This is a right coaction of $\Oq$ on $\cS\MN$ \cite{LS21}.

\newcommand{\wallcup}[5][]{
\twowallpic{#2}{left}{#4}{#5}{
 \ifthenelse{\equal{#1}{w}}{\draw[wall]  (\xwidth,0) --(\xwidth,1);}{}
\draw[edge, -o-={0.8}{#3}] (0,\yb) ..controls (\xwidth,\yb) and (\xwidth,\ya) .. (0,\ya);
}}

\begin{lemma}\cite[equation (56)]{LS21}\label{appendix-lem-1}
In $\cS\MN$, we have
    $
    \raisebox{-.30in}{
		\begin{tikzpicture}			\draw[color=gray!20,fill=gray!20] (-0.3,-1) rectangle (0.7,1);
        \draw [wall,<-] (-0.3,-1)-- (-0.3,1);
			\draw [color = black, line width =1pt](-0,-0.5)--(-0.3,-0.5);
			\draw [color = black, line width =1pt](-0,0.5)--(-0.3,0.5);
			\draw [color = black, line width =1pt] (0 ,-0.5) arc (-90:90:0.5);
			\filldraw[fill=white] (-0.05,0.5) circle (0.1);
            \node [left] at(-0.3,0.5) {$i$};
            \node [left] at(-0.3,-0.5) {$j$};
	\end{tikzpicture}}
    =\delta_{\bar i,j}\mathbbm{c}_{\bar i}^{-1}.
    $
\end{lemma}

\begin{lemma}\cite[Lemma 4.9(a)]{LY23}\label{zero}
	In $\cS\MN$, we have
	$\raisebox{-.20in}{
		\begin{tikzpicture}
			\filldraw [gray!20](-0.5,-0.5) rectangle (1,0.5);
			\draw[wall,-o-={0.1}{<}] (1,-0.5) -- (1,0.5);
			\draw[edge,-o-={0.6}{>}] (0.5,0)--(1,0.3);
			\draw[edge,-o-={0.6}{>}] (0.5,0)--(1,-0.3);
			\draw[line width =0.8pt](-0.5,0)--(0.5,0);
			\node [above] at 
            (0,-0.1)
            {$n-2$};
			\node [right] at (1,0.3) {$i$};
			\node [right] at (1,-0.3) {$i$};
	\end{tikzpicture}}
	=
	\raisebox{-.20in}{
		\begin{tikzpicture}
			\filldraw [gray!20](-0.5,-0.5) rectangle (1,0.5);
			\draw[wall,-o-={0.1}{<}] (1,-0.5) -- (1,0.5);
			\draw[edge,-o-={0.6}{<}] (0.5,0)--(1,0.3);
			\draw[edge,-o-={0.6}{<}] (0.5,0)--(1,-0.3);
			\draw[line width =0.8pt](-0.5,0)--(0.5,0);
			\node [above] at 
            (0,-0.1) {$n-2$};
			\node [right] at (1,0.3) {$i$};
			\node [right] at (1,-0.3) {$i$};
	\end{tikzpicture}}
	=0$.
\end{lemma}

 \cite[equation (51)]{LS21} and Lemma \ref{zero}  imply the following.
\begin{lemma}\label{appendix-lem-2}
In $\cS(\PP_2)$, we have
\begin{align*}
    \left(
    \raisebox{-.30in}{
		\begin{tikzpicture}			\draw[color=gray!20,fill=gray!20] (-0.3,-1) rectangle (0.7,1);
        \draw [wall,<-] (-0.3,-1)-- (-0.3,1);
        \draw [wall,<-] (0.7,-1)--(0.7,1);
			\draw [color = black, line width =1pt](-0,-0.5)--(-0.3,-0.5);
			\draw [color = black, line width =1pt](-0,0.5)--(-0.3,0.5);
			\draw [color = black, line width =1pt] (0 ,-0.5) arc (-90:90:0.5);
			\filldraw[fill=white] (-0.05,0.5) circle (0.1);
            \node [left] at(-0.3,0.5) {$i$};
            \node [left] at(-0.3,-0.5) {$j$};
	\end{tikzpicture}}\;
    \right)^N
    = q^{\frac{N(N-1)(n-1)}{2n}} 
    \raisebox{-.30in}{
		\begin{tikzpicture}			\draw[color=gray!20,fill=gray!20] (-1,-1) rectangle (0.7,1);
        \draw [wall,<-] (-1,-1)-- (-1,1);
			\draw [wall,<-] (0.7,-1)--(0.7,1);
			\draw [line width =1pt] (-1,-0.5)--(-0.7,-0.5);
			\draw [color = black, line width =1pt](-0,-0.5)--(-0.3,-0.5);
			\draw[fill=gray!20] (-0.7,-0.3) rectangle (-0.3,-0.7);
			\node at(-0.5,-0.5) {$N$};
			\draw [color = black, line width =1pt] (-1,0.5)--(-0.7,0.5);
			\draw [color = black, line width =1pt](-0,0.5)--(-0.7,0.5);
			\draw [color = black, line width =1pt] (0 ,-0.5) arc (-90:90:0.5);
			\filldraw[fill=white] (-0.5,0.5) circle (0.1);
            \node [left] at(-1,0.5) {$i$};
            \node [left] at(-1,-0.5) {$j$};
	\end{tikzpicture}}.
\end{align*}

\end{lemma}

\begin{proof}[Proof of Lemma \ref{lem-arcs}]

We use $\beta$ to  denote the component of $\cN$ involved in equation \eqref{eq-N-copies-cap-wall}.  Lemma \ref{rcut} implies that 
\begin{align*}
    &\Delta_\beta\left( 
    \raisebox{-.30in}{
		\begin{tikzpicture}			\draw[color=gray!20,fill=gray!20] (-1,-1) rectangle (0.7,1);
			\draw [wall] (0.7,-1)--(0.7,1);
			\draw [line width =1pt] (-1,-0.5)--(-0.7,-0.5);
			\draw [color = black, line width =1pt](-0,-0.5)--(-0.3,-0.5);
			\draw[fill=gray!20] (-0.7,-0.3) rectangle (-0.3,-0.7);
			\node at(-0.5,-0.5) {$N$};
			\draw [color = black, line width =1pt] (-1,0.5)--(-0.7,0.5);
			\draw [color = black, line width =1pt](-0,0.5)--(-0.7,0.5);
			\draw [color = black, line width =1pt] (0 ,-0.5) arc (-90:90:0.5);
			\filldraw[fill=white] (-0.5,0.5) circle (0.1);
	\end{tikzpicture}}\;
    \right)\\
    =&
    \sum_{i,j\in\mathbb J}
    \raisebox{-.30in}{
		\begin{tikzpicture}			\draw[color=gray!20,fill=gray!20] (-1.3,-1) rectangle (0,1);
			\draw [wall,<-] (0,-1)-- (0,1);
			\draw [line width =1pt] (-1.3,-0.5)--(-0.7,-0.5);
			\filldraw[fill=black] (-1,-0.5) circle (0.1);
			\draw [color = black, line width =1pt](-0,-0.5)--(-0.3,-0.5);
			\node [right] at(-0,-0.5) { $j$};
			\draw[color=black] (-0.5,-0.5) circle (0.2);
			\draw [line width =1pt] (-1.3,0.5)--(-0.7,0.5);
			\draw [color = black, line width =1pt](-0,0.5)--(-0.3,0.5);
			\node [right] at(-0,0.5) { $i$};
			\draw[fill=gray!20] (-0.7,0.3) rectangle (-0.3,0.7);
			\node at(-0.5,0.5) {$N$};
			\draw[fill=gray!20] (-0.7,-0.3) rectangle (-0.3,-0.7);
			\node at(-0.5,-0.5) {$N$};
			\filldraw[fill=white] (-1,0.5) circle (0.1);
	\end{tikzpicture}}
    \otimes_{\mathbb C}
    \raisebox{-.30in}{
		\begin{tikzpicture}			\draw[color=gray!20,fill=gray!20] (-1,-1) rectangle (0.7,1);
        \draw [wall,<-] (-1,-1)-- (-1,1);
			\draw [wall,<-] (0.7,-1)--(0.7,1);
			\draw [line width =1pt] (-1,-0.5)--(-0.7,-0.5);
			\draw [color = black, line width =1pt](-0,-0.5)--(-0.3,-0.5);
			\draw[fill=gray!20] (-0.7,-0.3) rectangle (-0.3,-0.7);
			\node at(-0.5,-0.5) {$N$};
			\draw [color = black, line width =1pt] (-1,0.5)--(-0.7,0.5);
			\draw [color = black, line width =1pt](-0,0.5)--(-0.7,0.5);
			\draw [color = black, line width =1pt] (0 ,-0.5) arc (-90:90:0.5);
			\filldraw[fill=white] (-0.5,0.5) circle (0.1);
            \node [left] at(-1,0.5) {$i$};
            \node [left] at(-1,-0.5) {$j$};
	\end{tikzpicture}}
    \quad (\because \mbox{Lemma \ref{rcut}})\\
    =& \omega^{-\frac{N(N-1)(n-1)}{2n}} \sum_{i,j\in\mathbb J}
    \raisebox{-.30in}{
		\begin{tikzpicture}			\draw[color=gray!20,fill=gray!20] (-1.3,-1) rectangle (0,1);
			\draw [wall,<-] (0,-1)-- (0,1);
			\draw [line width =1pt] (-1.3,-0.5)--(-0.7,-0.5);
			\filldraw[fill=black] (-1,-0.5) circle (0.1);
			\draw [color = black, line width =1pt](-0,-0.5)--(-0.3,-0.5);
			\node [right] at(-0,-0.5) { $j$};
			\draw[color=black] (-0.5,-0.5) circle (0.2);
			\draw [line width =1pt] (-1.3,0.5)--(-0.7,0.5);
			\draw [color = black, line width =1pt](-0,0.5)--(-0.3,0.5);
			\node [right] at(-0,0.5) { $i$};
			\draw[fill=gray!20] (-0.7,0.3) rectangle (-0.3,0.7);
			\node at(-0.5,0.5) {$N$};
			\draw[fill=gray!20] (-0.7,-0.3) rectangle (-0.3,-0.7);
			\node at(-0.5,-0.5) {$N$};
			\filldraw[fill=white] (-1,0.5) circle (0.1);
	\end{tikzpicture}}
    \otimes_{\mathbb C}
     \left(
    \raisebox{-.30in}{
		\begin{tikzpicture}			\draw[color=gray!20,fill=gray!20] (-0.3,-1) rectangle (0.7,1);
        \draw [wall,<-] (-0.3,-1)-- (-0.3,1);
        \draw [wall,<-] (0.7,-1)--(0.7,1);
			\draw [color = black, line width =1pt](-0,-0.5)--(-0.3,-0.5);
			\draw [color = black, line width =1pt](-0,0.5)--(-0.3,0.5);
			\draw [color = black, line width =1pt] (0 ,-0.5) arc (-90:90:0.5);
			\filldraw[fill=white] (-0.05,0.5) circle (0.1);
            \node [left] at(-0.3,0.5) {$i$};
            \node [left] at(-0.3,-0.5) {$j$};
	\end{tikzpicture}}\;
    \right)^N
    \quad (\because \mbox{Lemma \ref{appendix-lem-2}})\\
    =&\omega^{-\frac{N(N-1)(n-1)}{2n}} 
    \omega^{-\frac{N(n-1)}{2n}} (-\omega)^{-N(n-\bar{i})}
    \sum_{i\in\mathbb J}
    \raisebox{-.30in}{
		\begin{tikzpicture}			\draw[color=gray!20,fill=gray!20] (-1.3,-1) rectangle (0,1);
			\draw [wall,<-] (0,-1)-- (0,1);
			\draw [line width =1pt] (-1.3,-0.5)--(-0.7,-0.5);
			\filldraw[fill=black] (-1,-0.5) circle (0.1);
			\draw [color = black, line width =1pt](-0,-0.5)--(-0.3,-0.5);
			\node [right] at(-0,-0.5) { $\bar i$};
			\draw[color=black] (-0.5,-0.5) circle (0.2);
			\draw [line width =1pt] (-1.3,0.5)--(-0.7,0.5);
			\draw [color = black, line width =1pt](-0,0.5)--(-0.3,0.5);
			\node [right] at(-0,0.5) { $i$};
			\draw[fill=gray!20] (-0.7,0.3) rectangle (-0.3,0.7);
			\node at(-0.5,0.5) {$N$};
			\draw[fill=gray!20] (-0.7,-0.3) rectangle (-0.3,-0.7);
			\node at(-0.5,-0.5) {$N$};
			\filldraw[fill=white] (-1,0.5) circle (0.1);
	\end{tikzpicture}}
    \otimes_{\mathbb C}
    \raisebox{-.30in}{
		\begin{tikzpicture}			\draw[color=gray!20,fill=gray!20] (-0.3,-1) rectangle (0.7,1);
        \draw [wall,<-] (-0.3,-1)-- (-0.3,1);
        \draw [wall,<-] (0.7,-1)--(0.7,1);
	\end{tikzpicture}}
     \quad (\because \mbox{Lemma \ref{appendix-lem-1}})\\
    =&\eta^{-\frac{n-1}{2n}}(-\eta)^{\bar i-n}
    \sum_{i\in\mathbb J}
    \raisebox{-.30in}{
		\begin{tikzpicture}			\draw[color=gray!20,fill=gray!20] (-1.3,-1) rectangle (0,1);
			\draw [wall,<-] (0,-1)-- (0,1);
			\draw [line width =1pt] (-1.3,-0.5)--(-0.7,-0.5);
			\filldraw[fill=black] (-1,-0.5) circle (0.1);
			\draw [color = black, line width =1pt](-0,-0.5)--(-0.3,-0.5);
			\node [right] at(-0,-0.5) { $\bar i$};
			\draw[color=black] (-0.5,-0.5) circle (0.2);
			\draw [line width =1pt] (-1.3,0.5)--(-0.7,0.5);
			\draw [color = black, line width =1pt](-0,0.5)--(-0.3,0.5);
			\node [right] at(-0,0.5) { $i$};
			\draw[fill=gray!20] (-0.7,0.3) rectangle (-0.3,0.7);
			\node at(-0.5,0.5) {$N$};
			\draw[fill=gray!20] (-0.7,-0.3) rectangle (-0.3,-0.7);
			\node at(-0.5,-0.5) {$N$};
			\filldraw[fill=white] (-1,0.5) circle (0.1);
	\end{tikzpicture}}
    \otimes_{\mathbb C}
    \raisebox{-.30in}{
		\begin{tikzpicture}			\draw[color=gray!20,fill=gray!20] (-0.3,-1) rectangle (0.7,1);
        \draw [wall,<-] (-0.3,-1)-- (-0.3,1);
        \draw [wall,<-] (0.7,-1)--(0.7,1);
	\end{tikzpicture}}
     \quad (\because \mbox{$(-\omega)^{N}=-\eta$})
\end{align*}
Then we have equation \eqref{eq-N-copies-cap-wall} since $\Delta_\beta$ is a right coaction of $\cS_{\hat \omega}(\PP_2)$ on $\cS_{\home}\MN$.

\end{proof}

\bibliographystyle{hamsalpha}
\bibliography{biblio}

\end{document}